\documentclass[12pt]{amsart}
\usepackage{amsthm}
\usepackage{amssymb}
\usepackage{amsmath}
\usepackage{amscd}
\usepackage{eufrak}
\usepackage[dvips]{graphicx}

\makeatletter

\newcommand{\C}{\mathbb{C}}

\newcommand{\R}{\mathbb{R}}

\newcommand{\SL}{\operatorname{SL}}
\newcommand{\SU}{\operatorname{SU}}
\newcommand{\PSL}{\operatorname{PSL}}

\newcommand{\slg}{\operatorname{sl}}

\newcommand{\su}{\operatorname{su}}

\newcommand{\mathcalD}{{\mathcal D}}

\newcommand{\ord}{\operatorname{ord}}

\renewcommand{\Re}{\operatorname{Re}}

\renewcommand{\Im}{\operatorname{Im}}

\renewcommand{\Im}{\operatorname{Im}}
\newcommand{\ef}{\mathfrak{f}}   

   \newtheorem{theorem}{Theorem}[section]
   \newtheorem{proposition}[theorem]{Proposition}
   \newtheorem{corollary}[theorem]{Corollary}
   \newtheorem{lemma}[theorem]{Lemma}
   \newtheorem{defn}[theorem]{Definition}
 \theoremstyle{remark}
   \newtheorem{example}[theorem]{Example}
   \newtheorem{remark}[theorem]{Remark}
\numberwithin{equation}{section}

\begin{document}
\begin{titlepage}
\begin{center}
{\Huge\bf Discrete Constant Mean} \\  \vspace{10pt}
{\Huge\bf Curvature Surfaces} \\ \vspace{10pt}
{\Huge\bf via Conserved Quantities} \\ 
\vspace{36pt}
{\LARGE Wayne Rossman} \\
\vspace{96pt}

\newpage

\section*{Forward}

\vspace{0.4in}

These notes are about discrete constant mean curvature surfaces 
defined by an approach related to 
integrable systems techniques.  We introduce the 
notion of discrete constant mean curvature surfaces by 
first introducing properties of 
smooth constant mean curvature surfaces.  We 
describe the mathematical structure of the smooth 
surfaces using conserved quantities, which can 
be converted into a discrete theory in a natural way.  

\vspace{0.4in}

About referencing: 
We do not attempt to give a complete reference list, and omit 
what is already referenced in \cite{wisky}.  We 
list only articles referenced in the body of the text, or that were 
written after \cite{wisky} was published, or were otherwise not 
included in the reference list in \cite{wisky}, or that were 
referenced in \cite{wisky} but need to be updated.  

\vspace{0.4in}

About using quaternions: In following with the historical 
development of the field, we use a model that involves quaternions.  However, the 
use of a more standard model has some advantages, as it can be applied in more 
general dimensions and settings (see Chapter 
\ref{discretemaximalsurface} here, for example), and sometimes 
gives less cluttered computations.  It would be a good exercise to convert 
this text into one involving a 
more standard quaternion-free model, but we do not do that here 
(see \cite{BHRS}), and instead only make 
occasional comments about this.  

\vspace{0.4in}

Acknowledgements: Primary thanks must go to Udo Hertrich-Jeromin, 
who carefully and patiently taught the author more than half 
of the material in this text.  The author also owes thanks to 
many others for numerous mathematical tips: Fran Burstall, Tim Hoffmann, 
Boris Springborn, Ulrich Pinkall, Masaaki Umehara, Kotaro Yamada, 
Takeshi Sasaki, Masaaki Yoshida, Masatoshi Kokubu, 
Shoichi Fujimori, Shimpei Kobayashi, Yusuke Kinoshita 
and Tetsuhiro Tachiyama amongst them.

\vspace{0.4in}

Also, the author would like to express his gratitude to the Kyushu University 
GCOE program for giving him the opportunity to present the material in this 
text in a short course at Kyushu University in February of 2010.  

\vspace{0.4in}

It goes without saying that I, the author, am solely responsible for 
choices of approaches and for any possible errors.

\end{center}
\end{titlepage}

\pagenumbering{roman}
\setcounter{page}{2}

\tableofcontents

\newpage

\pagenumbering{arabic}

\section{Motivations for studying CMC surfaces}
\label{chapter0}

These notes are about surfaces of constant mean curvature, or, more briefly, 
"CMC" surfaces.  In particular, we will focus on discrete versions of 
CMC surfaces.  However, it is useful to first take a close look at 
the smooth case, so let us start there.  

Smooth CMC surfaces can be thought of as mathematical models for 
soap films, or we might say that they 
are "mathematically perfect" soap films.  
Saying that CMC surfaces are models for soap films is certainly not a rigorous 
mathematical definition, but it is a good starting point for 
appreciating why CMC surfaces are interesting objects.  In fact, it would be 
impossible to explain why mathematicians have put so much effort 
into understanding 
CMC surfaces without discussing soap films, or interfaces between 
fluids, or some 
other similar idea.  Even though modern-day research on CMC 
surfaces might not always 
relate immediately to soap films, the notion of soap films is 
invariably lurking in 
the background.  So let this be our first informal definition: 
\begin{quote} {\em CMC surfaces are soap films.}  
\end{quote} 
In fact, CMC surfaces are defined to be 
those surfaces whose mean curvature is constant, as their name suggests.  
But we save a rigorous definition of mean curvature for later.  This rigorous 
definition is locally equivalent to the above informal definition, 
and we also explain 
this later.  

\subsection{Soap films} 
A soap film forms a surface that minimizes area with respect to some given 
constraints, and it is the constraints that determine which soap film will be 
formed.  Let us give some examples, all of which can be physically 
constructed if 
one has the necessary ingredients: 
\begin{enumerate}
\item If one puts a circular wire ring 
into a fluid soap solution and then extracts it, one obtains a soap film that 
is a flat planar disk with this ring as its boundary.  Here 
the only constraint 
on this soap film is its boundary, which is fixed to be a round 
circle, and this 
boundary constraint then determines the resulting soap film (the flat planar 
disk).  
\item Blowing sufficiently hard on the above flat round disk in item (1) above 
would cause this soap film to break free of the circular ring and become a 
free floating round sphere.  (This activity is a common pastime for 
young children.)  This sphere contains a pocket of air of a certain 
volume, and since this air cannot escape to the other side of the 
soap film, this volume is fixed.  Here the only constraint on this 
soap film is the 
fixed volume it contains.  With respect to this volume constraint, the 
soap film minimizes its area, and the round sphere is the unique 
shape that accomplishes this.  
\item Taking two circular wire rings of the same radius, we can produced two 
flat soap films in the shapes of round disks, as in item (1).  
Putting these two disks together so that they coincide and then 
pulling them slightly 
apart in the direction perpendicular to the planes they lie in 
results in a soap film 
that has three smooth pieces meeting along a singular round 
circle.  Two of the 
smooth pieces are surfaces of revolution and are reflections of 
each other across the 
plane that is midway between the two parallel planes containing 
the two circular wire 
rings.  The third smooth piece is a flat planar disk contained 
in that plane of 
reflection.  If one pushes a dry pointed object (such as a pencil) 
into the third 
smooth piece, then the soap film will instantly pop into a single smooth 
anvil-shaped surface of revolution.  This last soap film is 
called a catenoid.  
It is determined by its boundary constraint, which is two 
fixed circular wire rings.  
\item Taking the catenoidal soap film in the previous example, 
we can place two 
flat plastic disks so that they fill the planar regions 
inside the two boundary 
circular wires.  We have then trapped air inside the catenoid.  
Making a small hole 
in one of the plastic disks and pumping more air into this 
interior region (through that hole), the 
sides of the anvil-shaped catenoid will expand to accommodate 
the increase of volume 
inside.  If just the right amount of air is pumped in (and if 
the two boundary circular 
wires are not too far from each other), the soap film will 
become exactly a portion of 
a round cylinder.  Thus the round cylinder can be made 
using a soap film.  In this 
case there are two constraints.  One constraint is the 
fixed boundary (two circular 
wire rings in parallel planes), and the other is the fixed volume (inside the 
cylinder).  Other surfaces of revolution can be made from soap films in this 
way by pumping air into the interior region, and these surfaces turn out to be 
portions of Delaunay surfaces, which we have described in detail 
in \cite{wisky}.  

\end{enumerate}
These examples show that the flat plane, the round sphere, the 
catenoid and the round 
cylinder are all CMC surfaces.  

\begin{figure}[phbt]
\begin{center}
\includegraphics[width=1.0\linewidth]{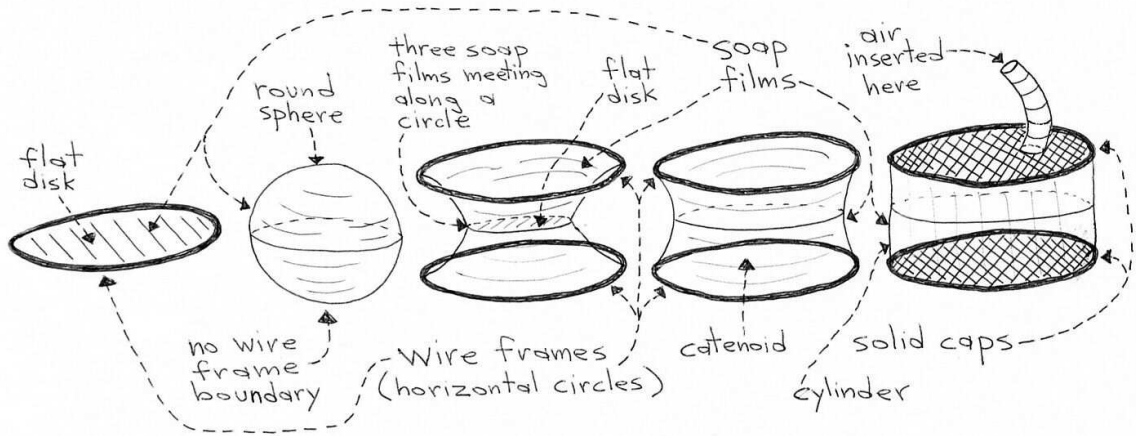}
\caption{The soap films described in items (1), (2), (3) and (4) at the 
beginning of Chapter \ref{chapter0}.}
\end{center}
\end{figure}

Amongst the four examples above, only the second and fourth ones 
have any volume 
constraints.  The volume constraints in these two cases are that the volume to 
one side of the surface is constrained to be a fixed quantity.  
In the case that there are only boundary constraints and no volume constraints 
(as in the first and 
third examples), the resulting soap film is a special case of a 
CMC surface that is 
called a minimal surface.  Thus the flat plane and catenoid are 
minimal surfaces.  
In the case that there are volume constraints (as in the second and 
fourth examples), the resulting soap film is a non-minimal CMC 
surface.  Thus the 
round sphere and round cylinder are non-minimal CMC surfaces.  

\subsection{Interfaces} 
More generally, CMC surfaces are models for the interface between two 
distinct uniform fluids.  For example, when one pours 
some lighter-than-water oil into a cup of water, the oil will rise to the top 
and the interface between the oil and the 
water will become a flat horizontal plane, a minimal surface.  If one has 
two types of oils of equal density that do not like to interact 
with each other, and 
one puts a small amount of one type into a glass container filled 
with the other type, then 
the first type will take the shape of a round ball floating in 
the other type.  Since this 
ball is round, the interface between the two oils will be the CMC surface 
that is a round sphere.  (In the presence of gravity, the interface between 
two distinct uniform non-interacting fluids can be a more general 
type of surface called a capillary surface, not always a 
CMC surface.  Robert Finn has done much work on capillary surfaces; see 
\cite{Finn1}, \cite{Finn2}, \cite{Finn3}, \cite{Finn4} for nice 
introductions to the subject.)

\subsection{Variational property} 
That soap films minimize area with respect to some given 
constraints is called a {\em variational} property, because this 
minimization property 
can be rephrased in the following way: 
If one continuously {\em varies} (deforms) the soap film so 
that its given constraints are 
preserved, then the area of the soap film will increase.  Thus soap films 
minimize area under continuous {\em variations} that preserve the 
constraints.  
Once we give a formal definition of CMC surfaces, we will see that 
CMC surfaces 
are a larger class of surfaces than soap films, in part because 
CMC surfaces include 
nonphysical objects called "unstable" soap films, and so the 
above statement is not 
strictly true for CMC surfaces.  However, this is a technical 
point that we can 
ignore for the moment, and simply note that the above variational 
property turns out 
to still be true for small pieces of CMC surfaces: If one continuously 
varies a sufficiently small portion of a CMC surface so that 
its given constraints are 
still preserved, then the area of the varied surfaces will be larger 
than that of the 
original CMC surface.  Thus we can give a second definition for CMC surfaces 
that is still informal, but is intuitively useful: 
\begin{quote} {\em CMC surfaces are surfaces that locally 
minimize area with respect to 
boundary and volume constraints.}  \end{quote}
We will describe the meaning of an "unstable" CMC surface in more 
detail in Section 
\ref{section1.6}, and we will see some examples there.  

\begin{figure}[phbt]
\label{fig:examplesofsoapfilms}
\begin{center}
\includegraphics[width=0.20\linewidth]{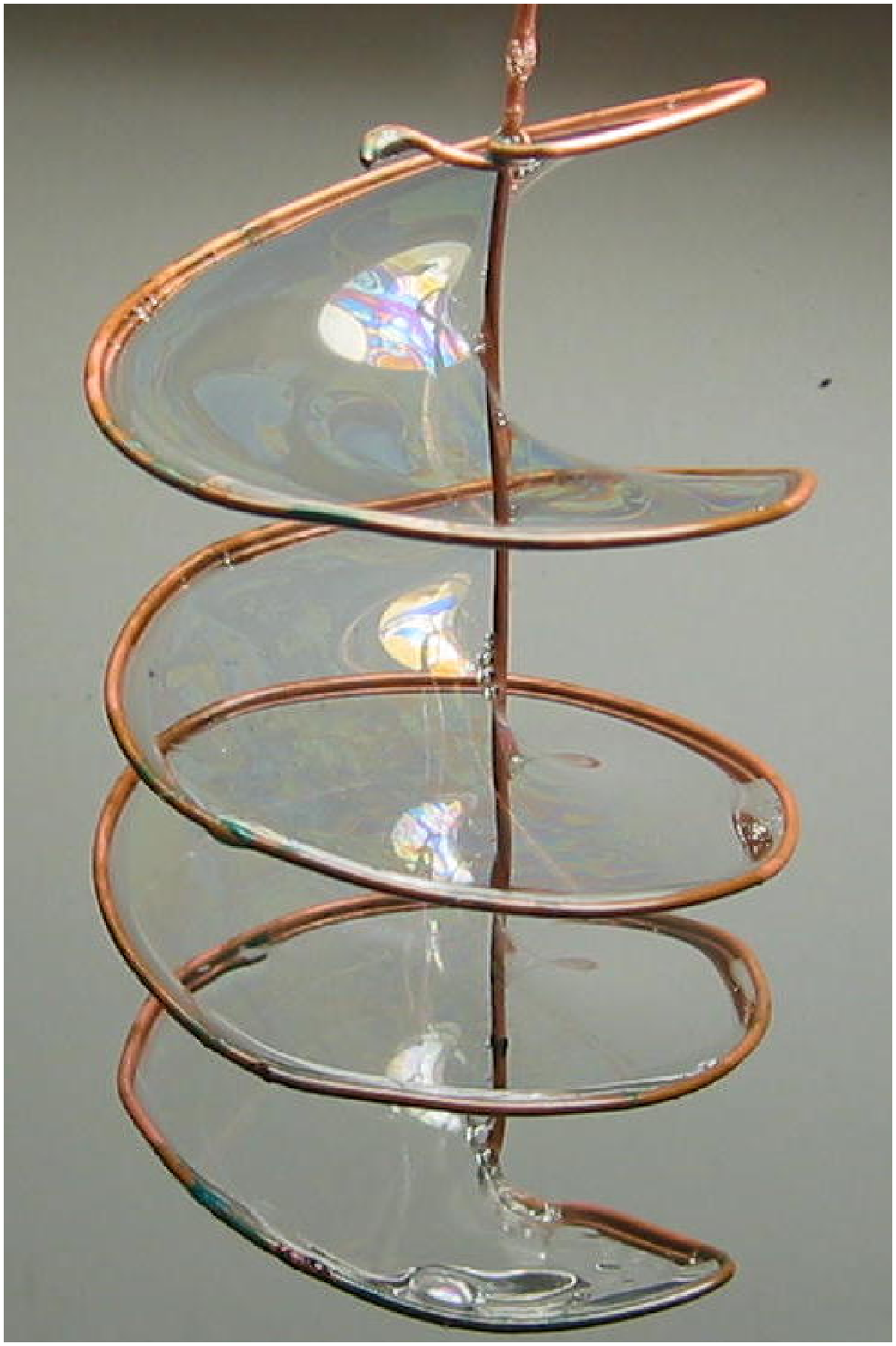}
\includegraphics[width=0.3\linewidth]{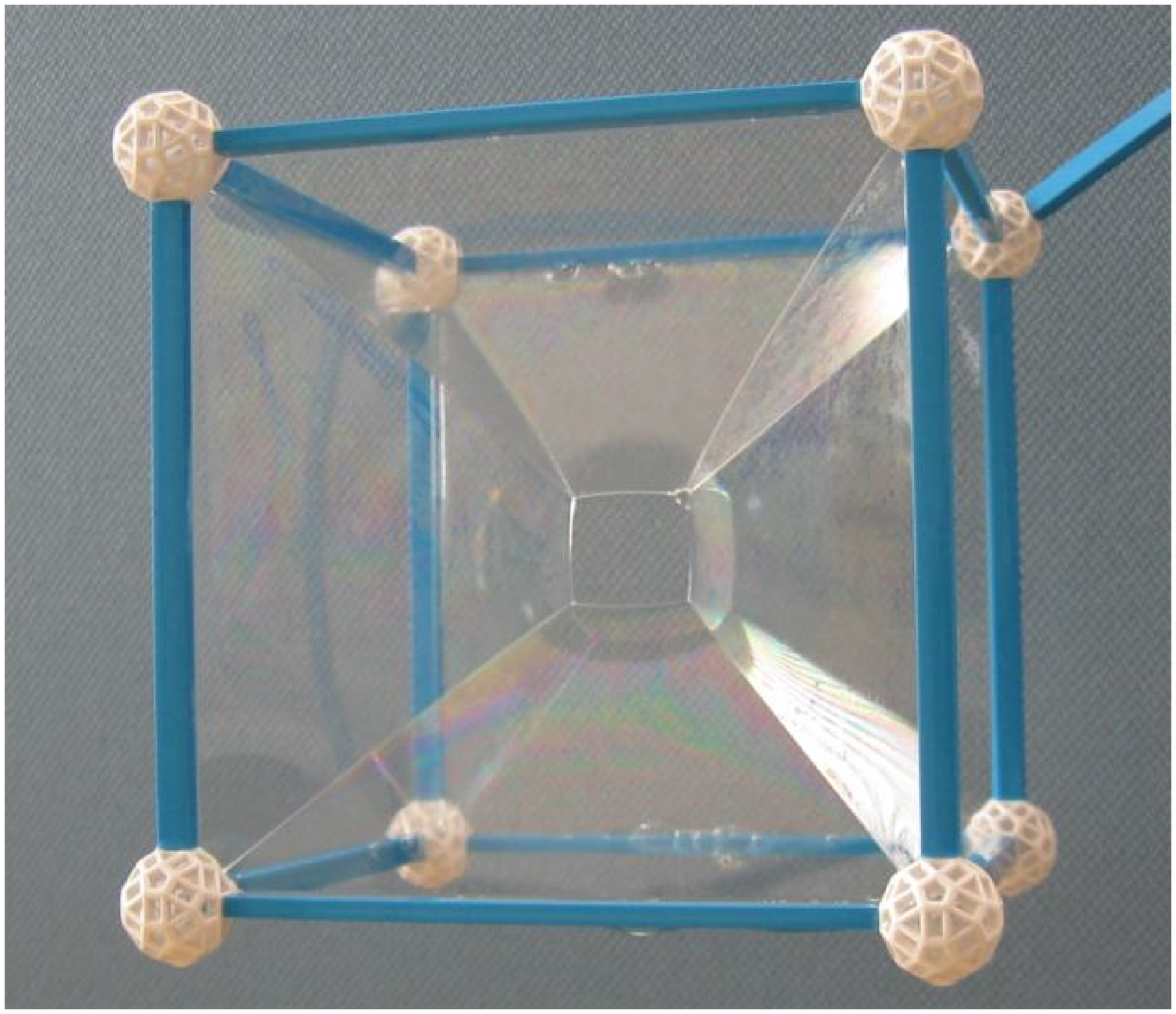}
\includegraphics[width=0.27\linewidth]{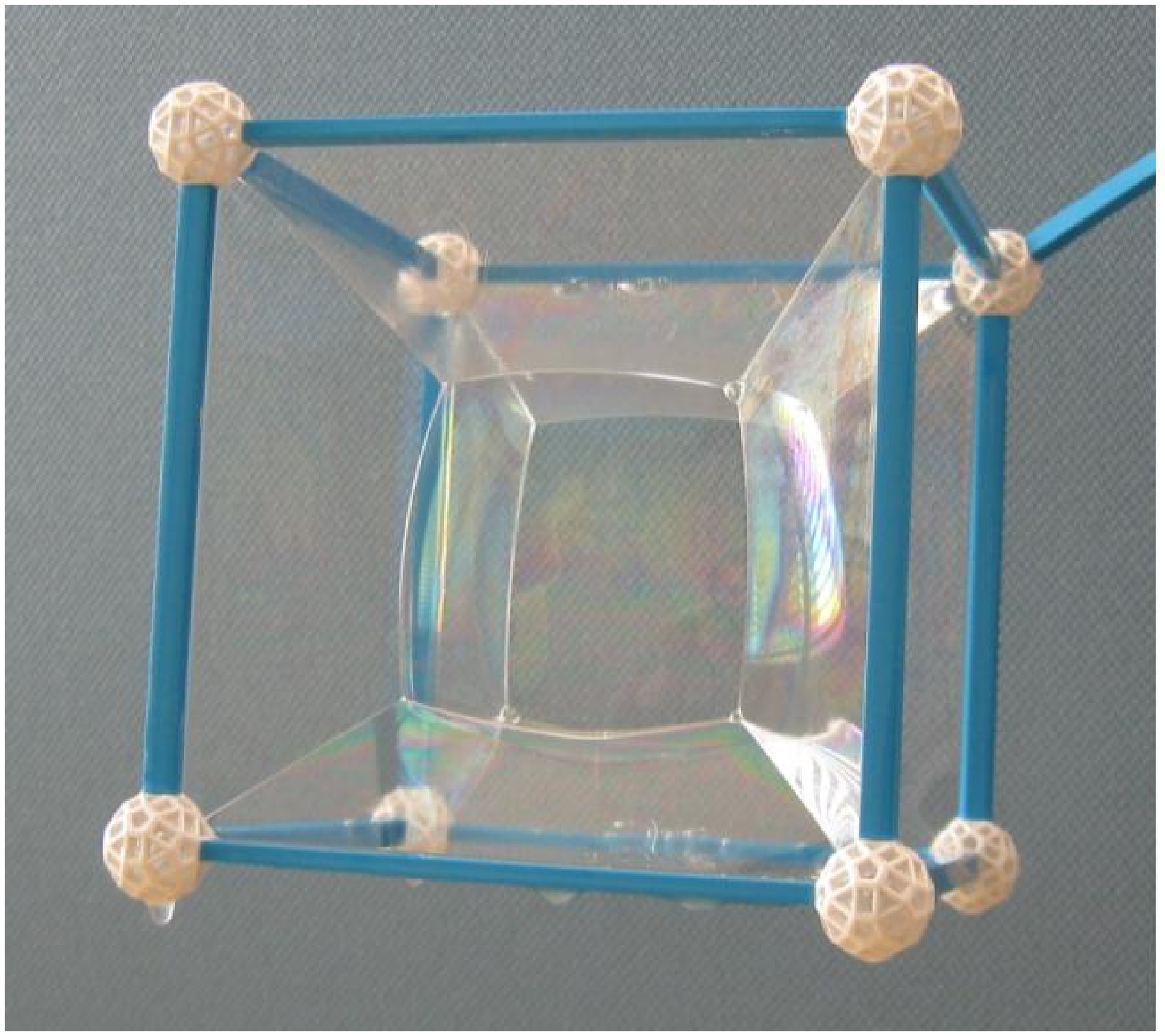}
\includegraphics[width=0.24\linewidth]{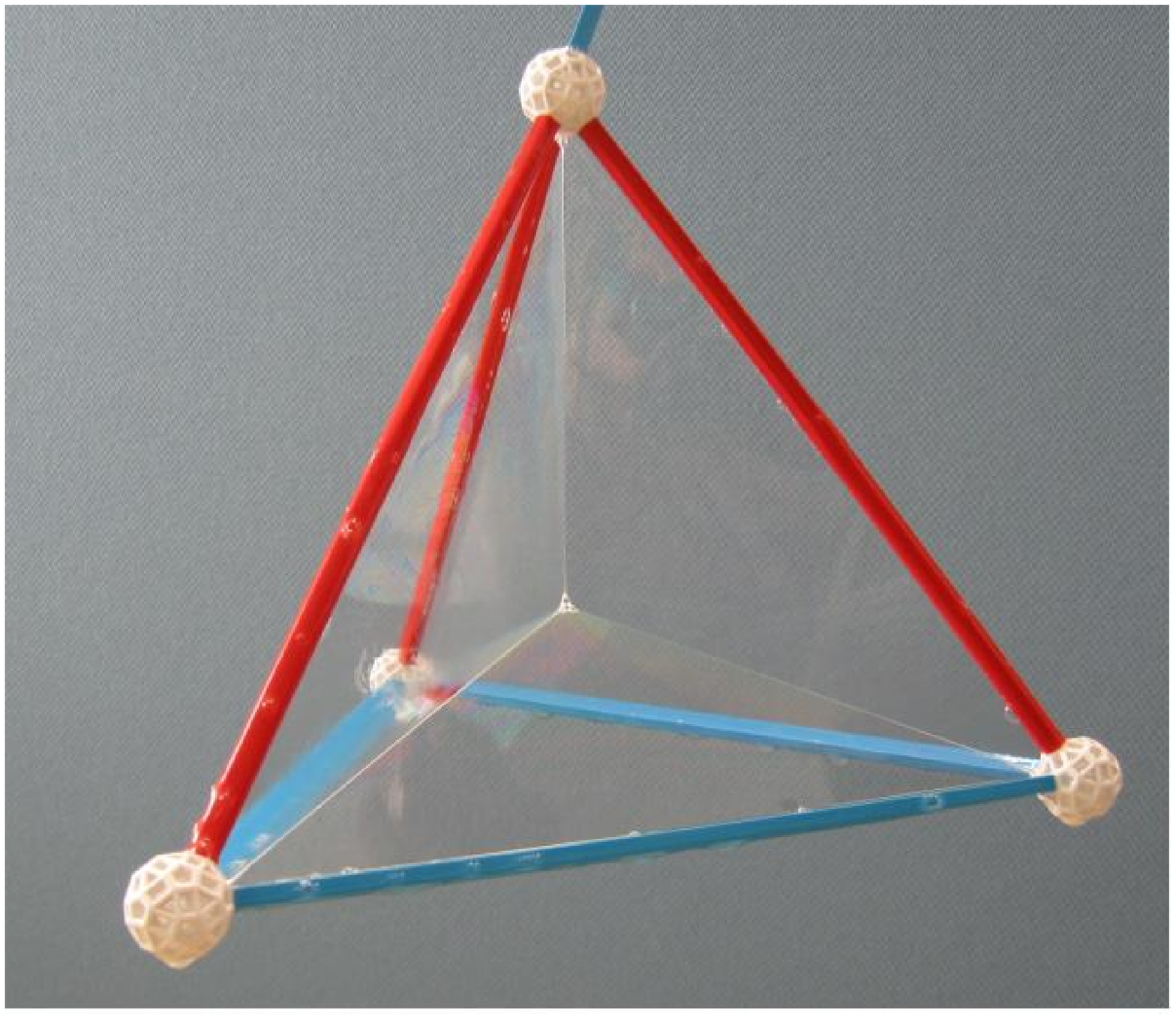}
\includegraphics[width=0.33\linewidth]{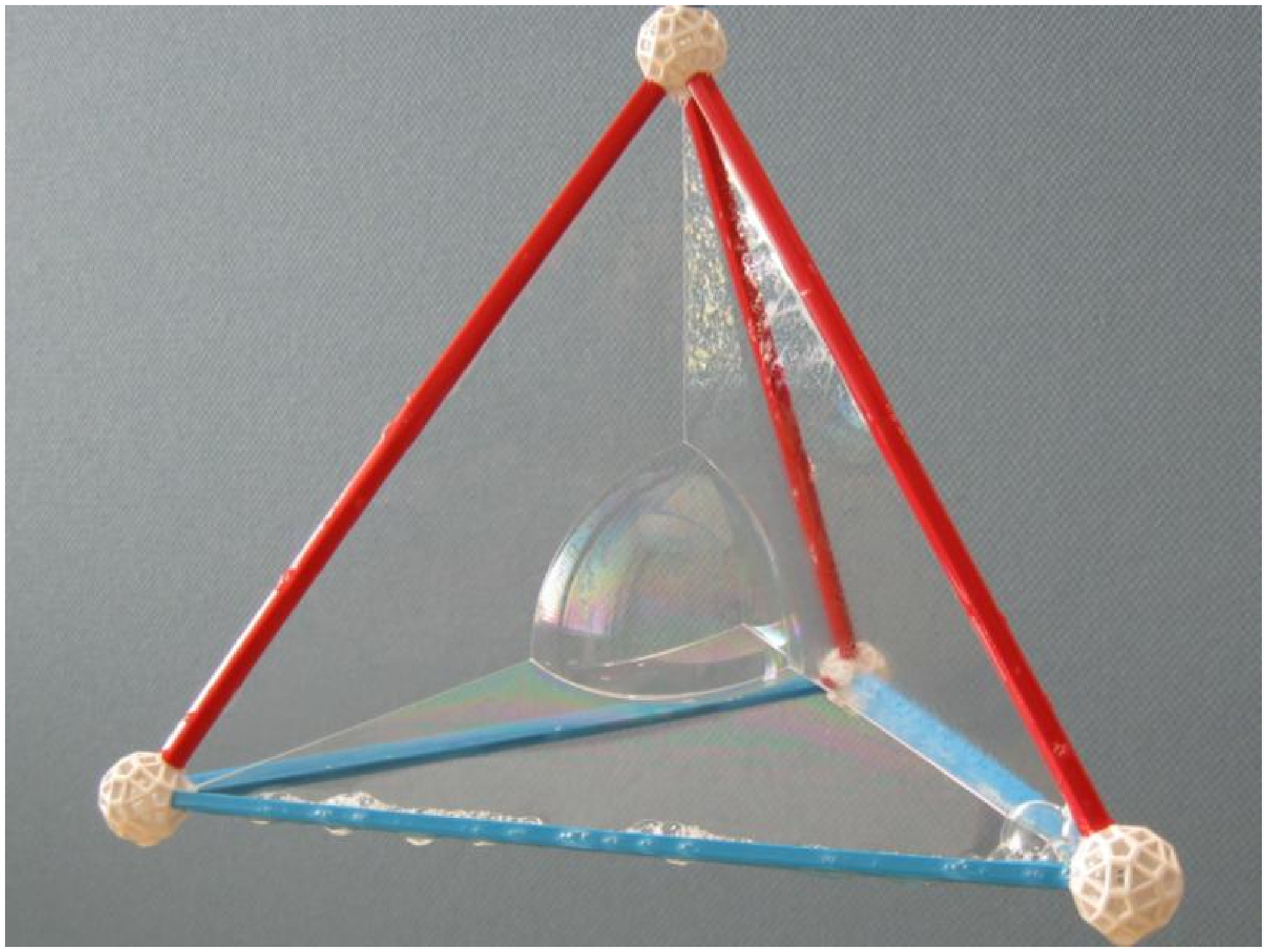}
\includegraphics[width=0.3\linewidth]{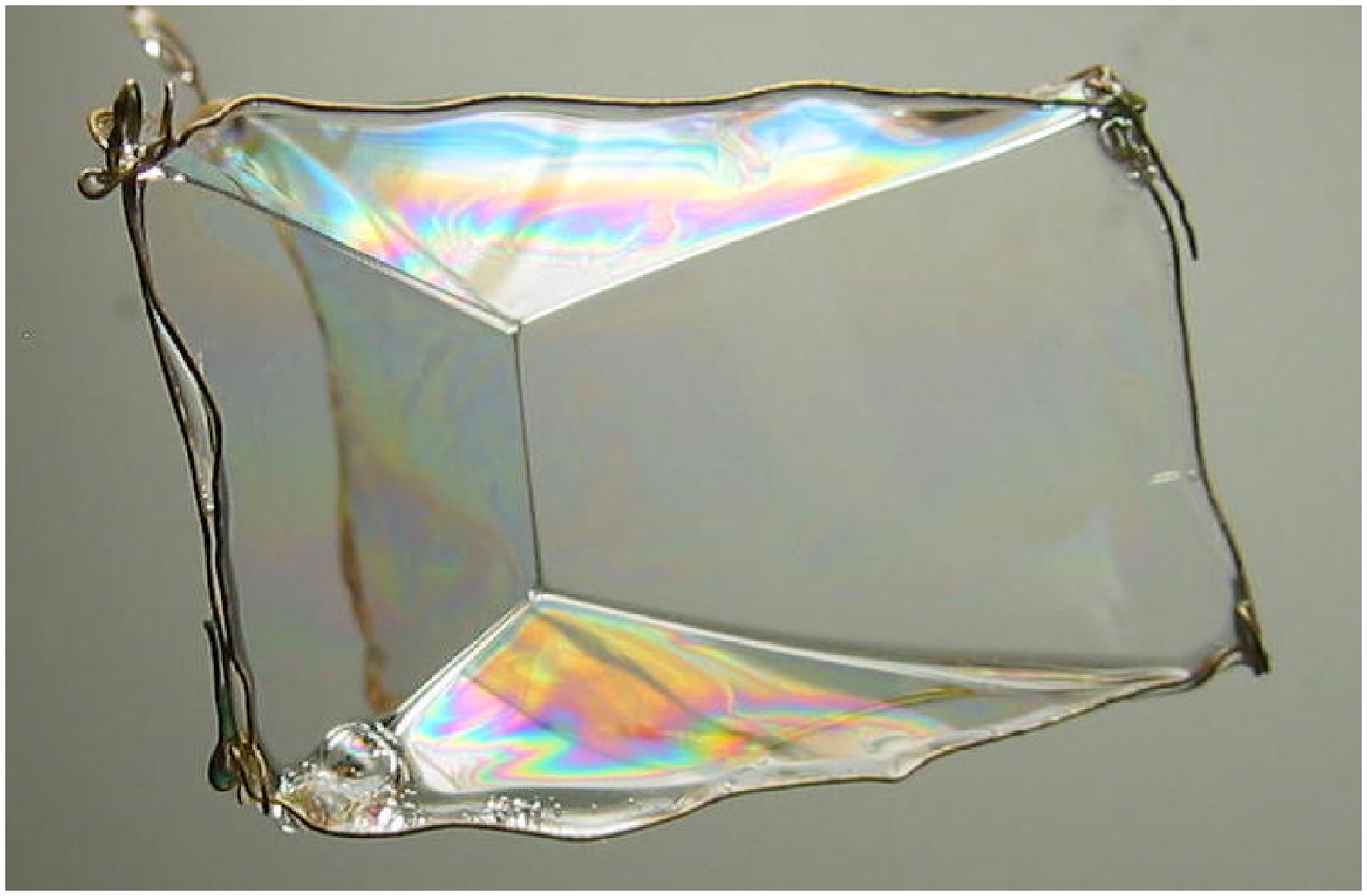}
\includegraphics[width=0.225\linewidth]{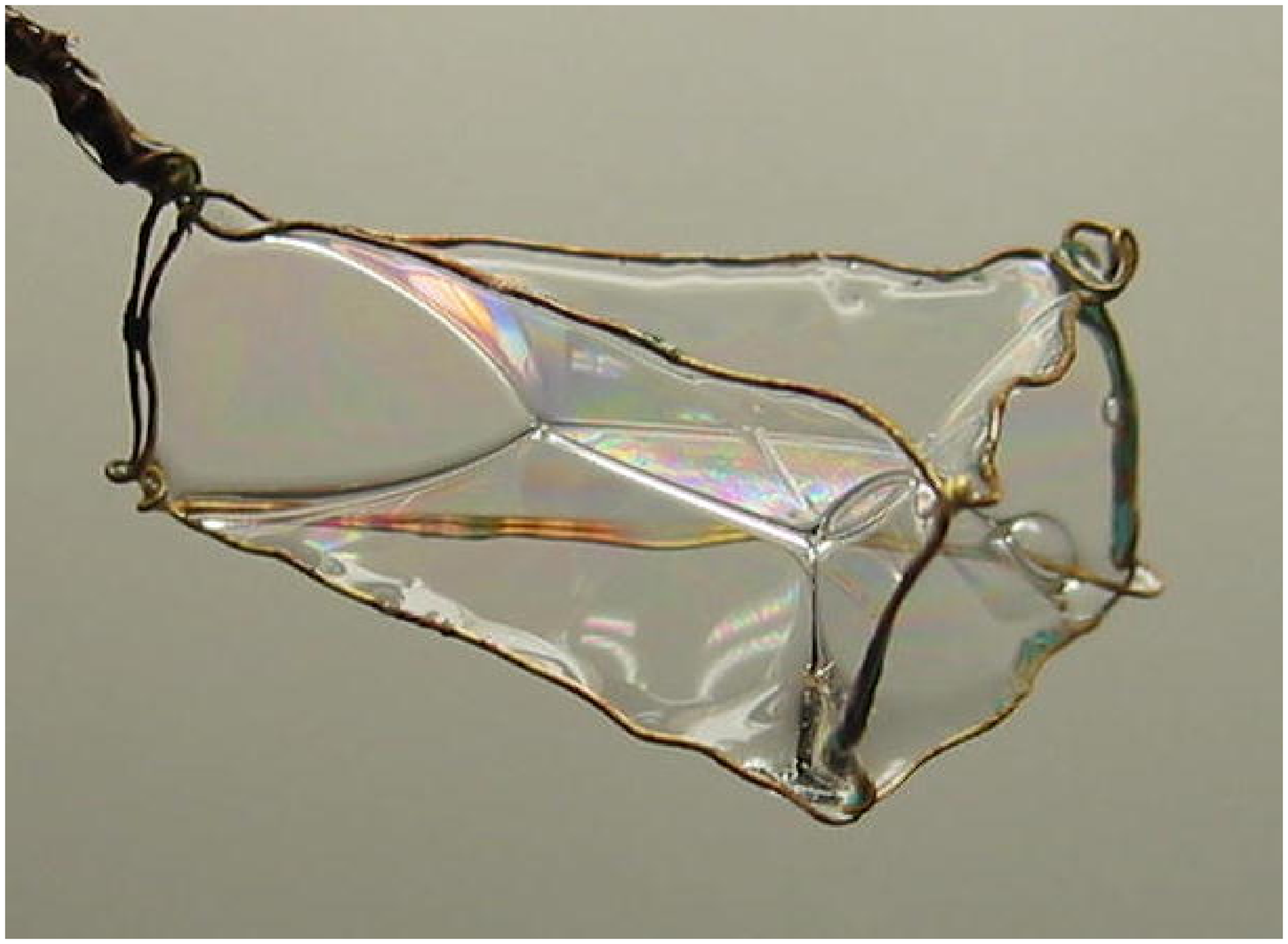}
\includegraphics[width=0.25\linewidth]{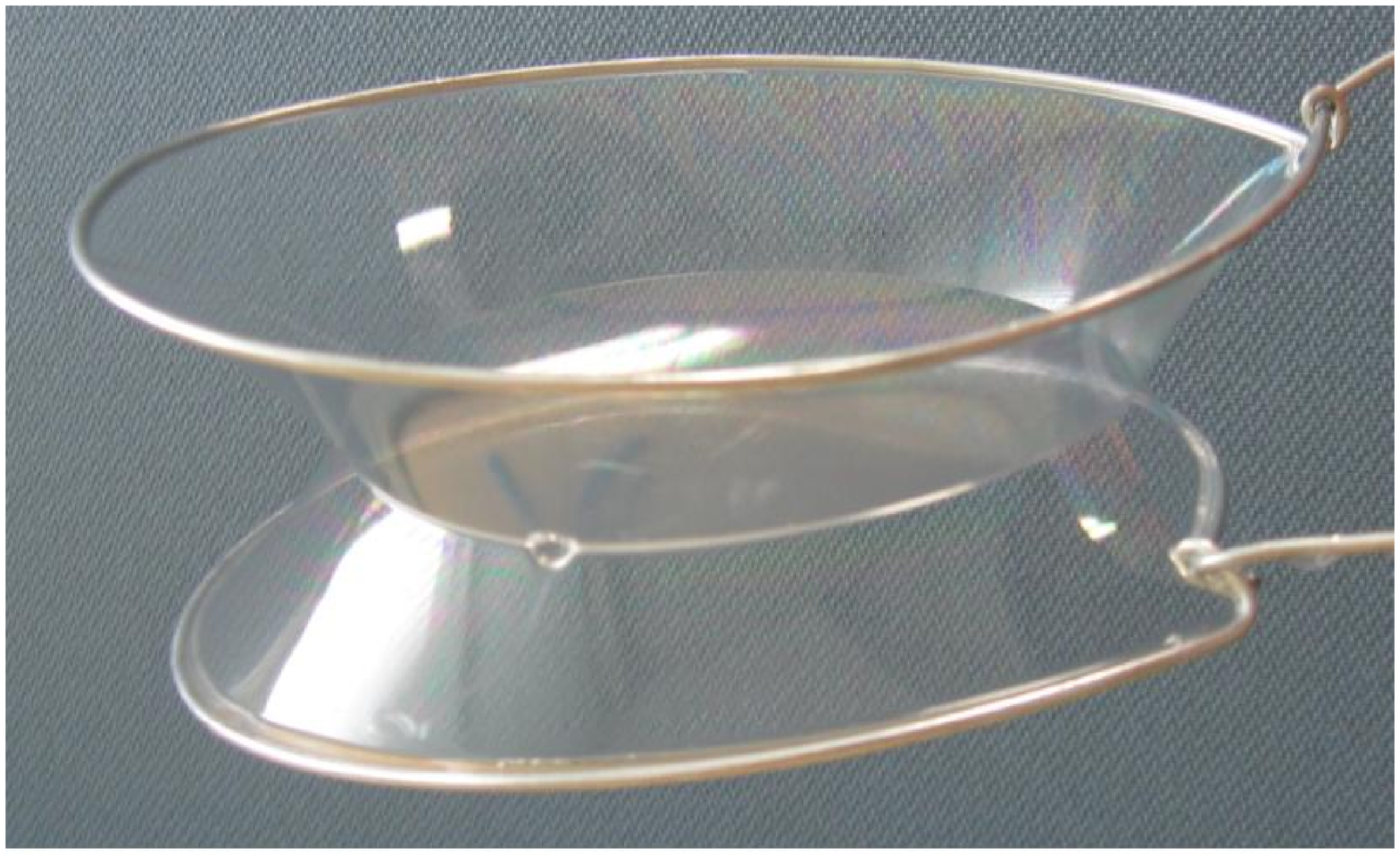}
\includegraphics[width=0.2\linewidth]{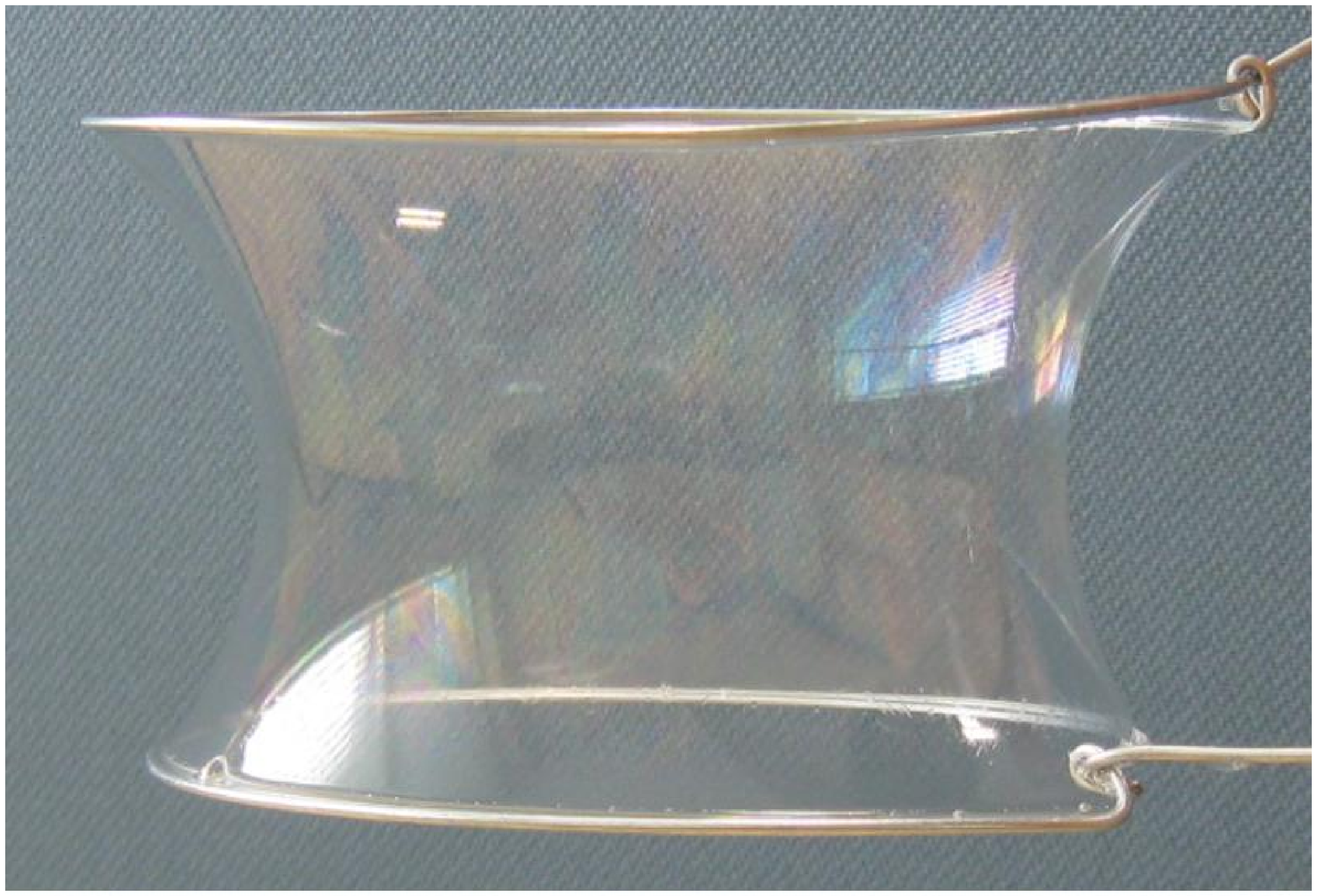}
\includegraphics[width=0.24\linewidth]{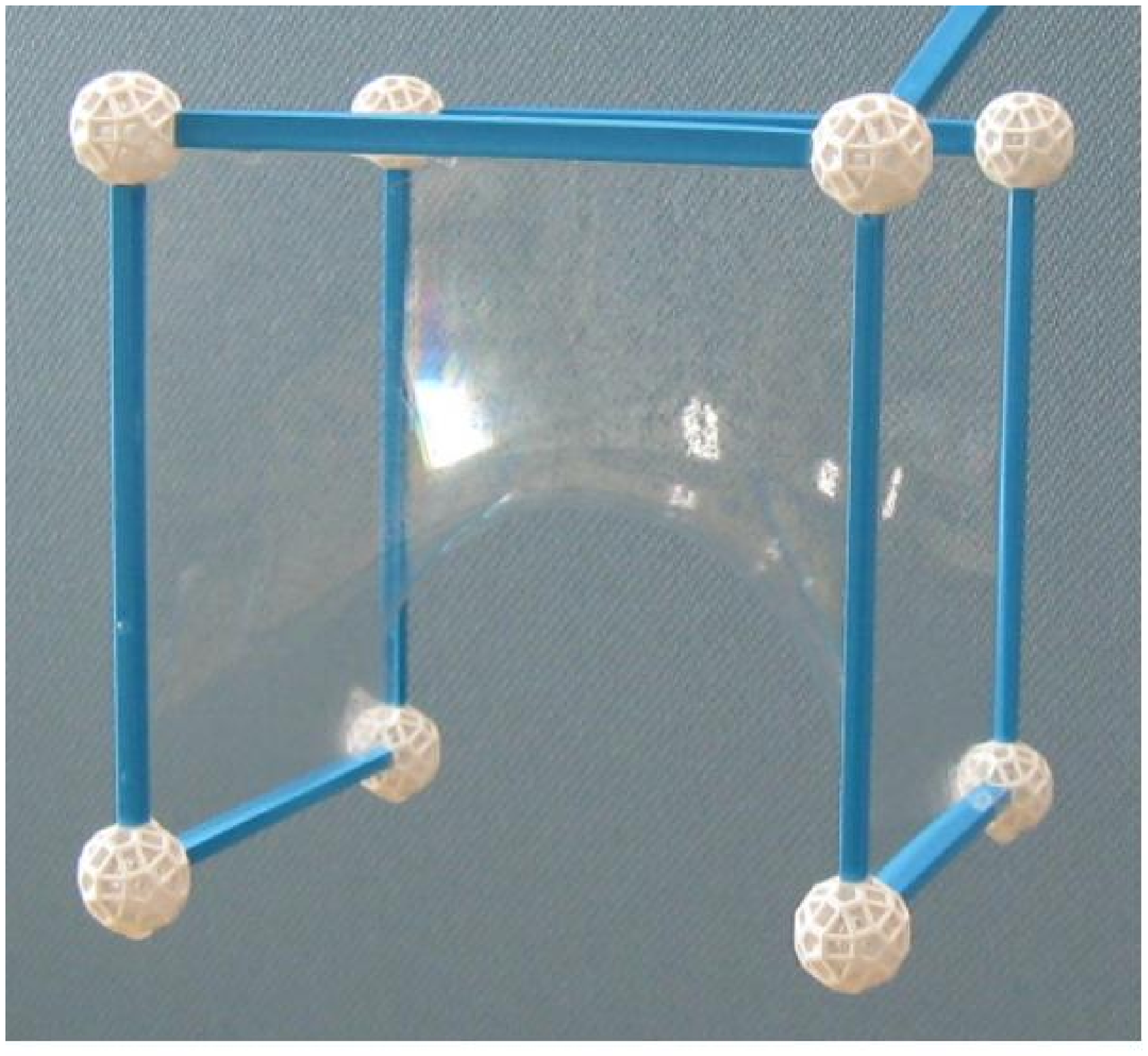}
\includegraphics[width=0.24\linewidth]{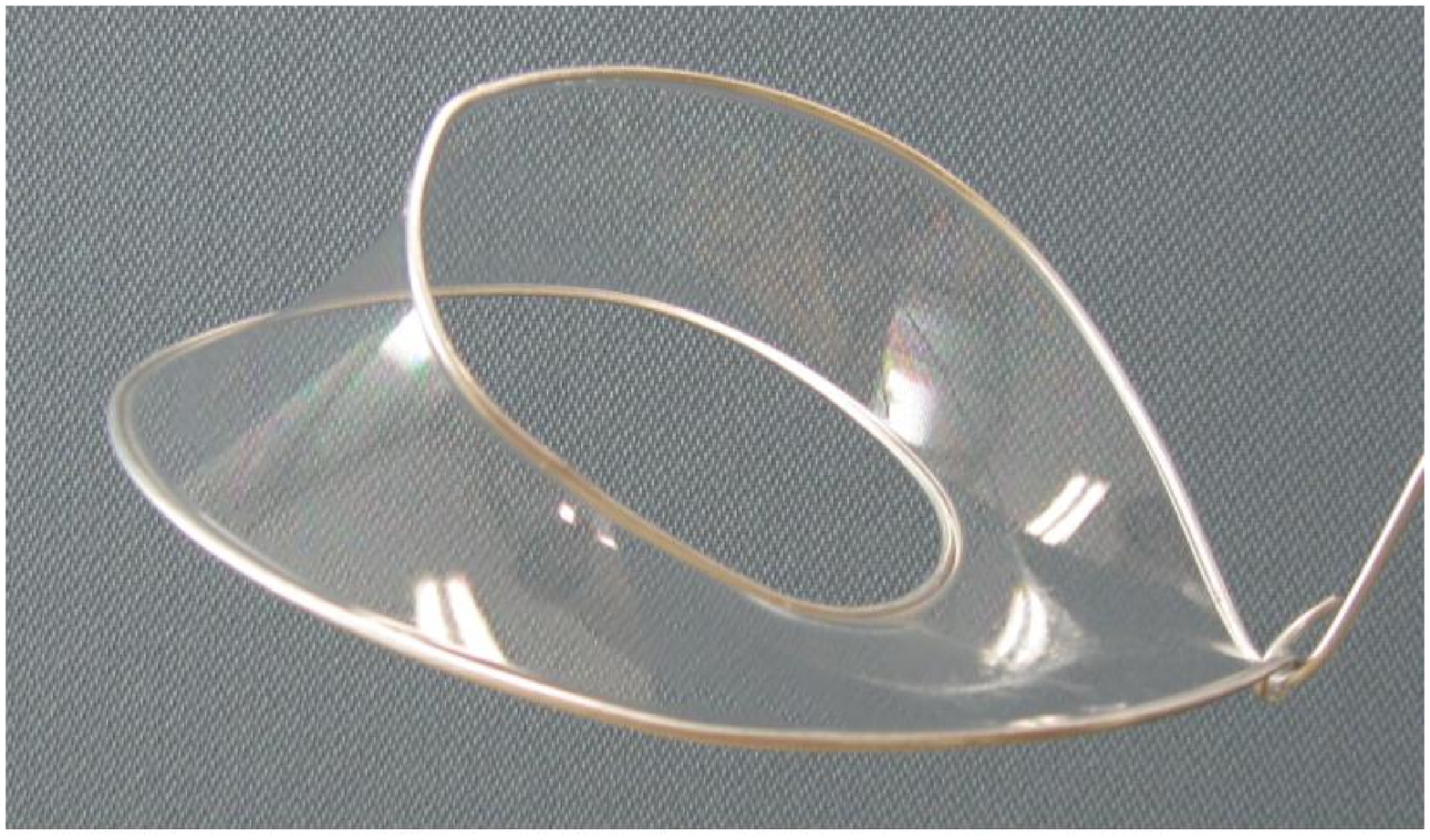}
\includegraphics[width=0.24\linewidth]{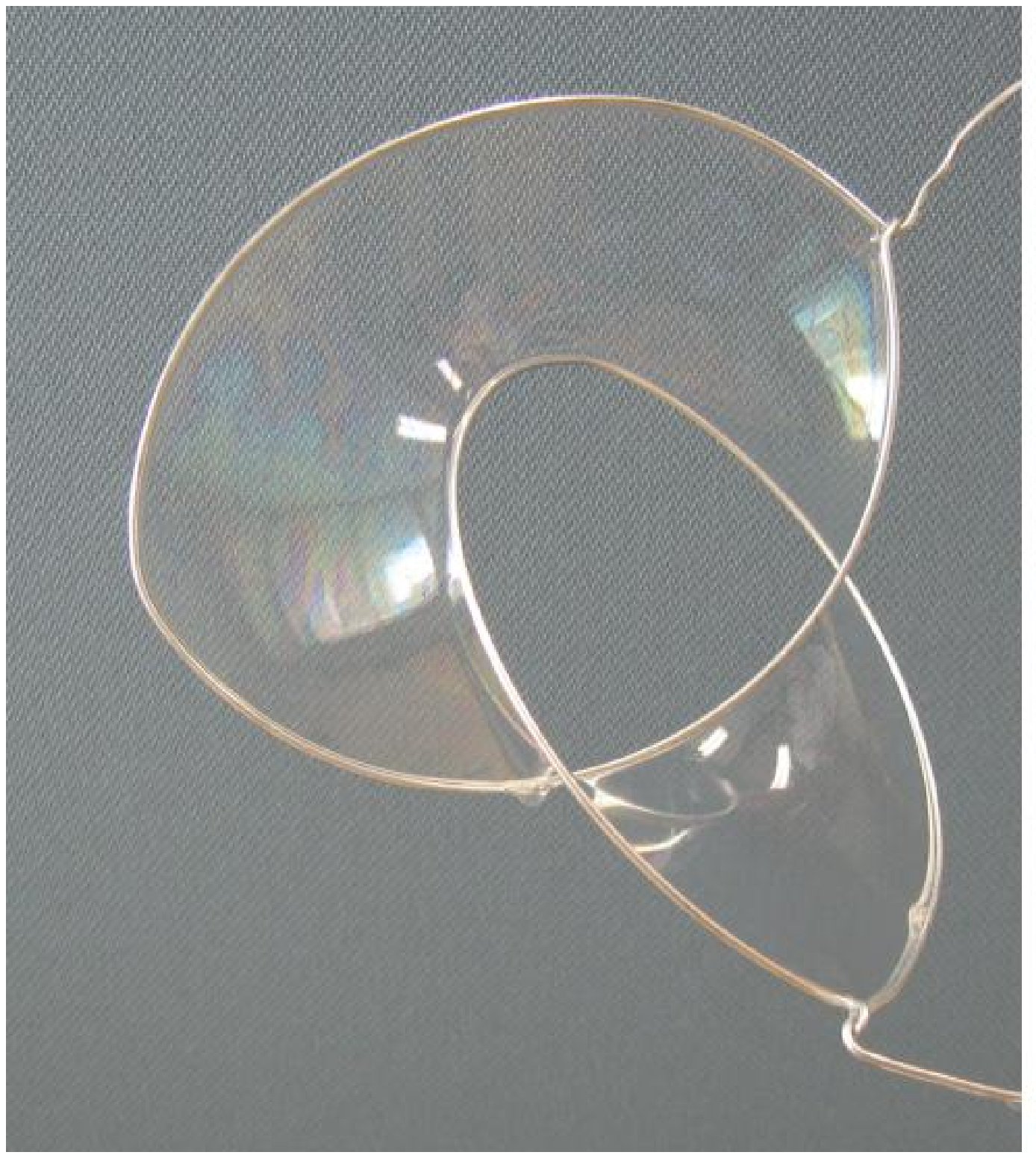}
\includegraphics[width=0.24\linewidth]{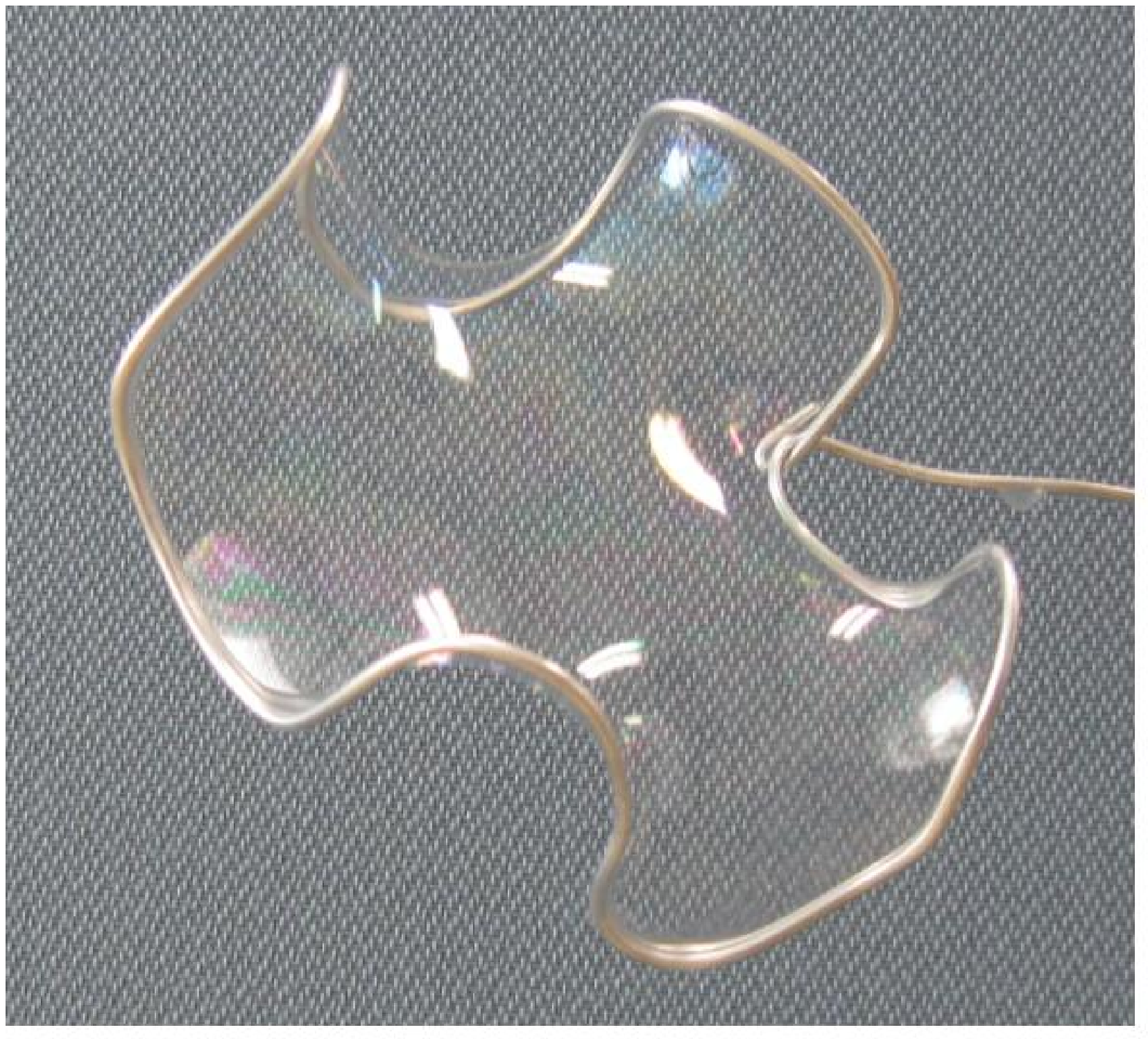}
\includegraphics[width=0.24\linewidth]{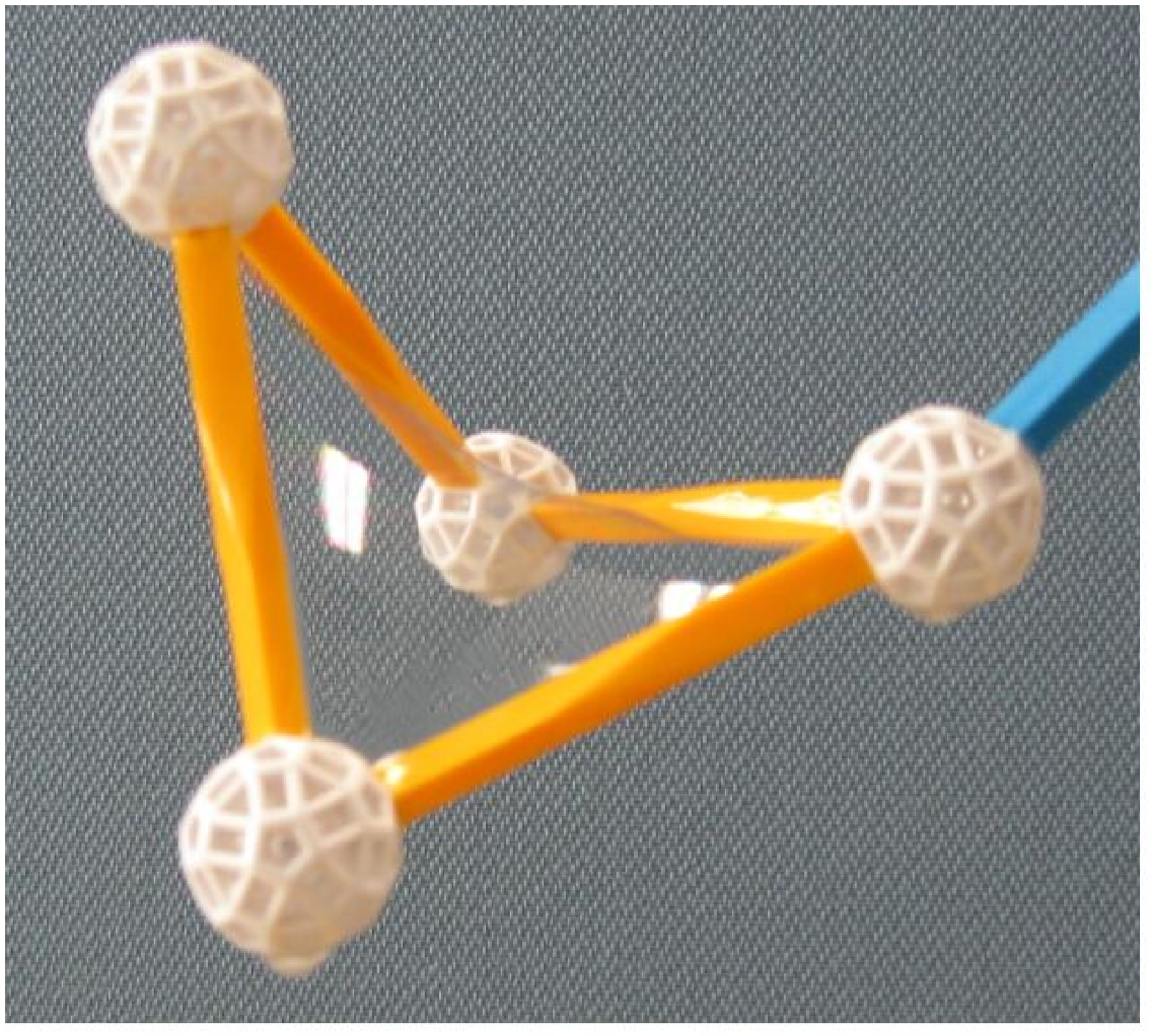}
\caption{Examples of soap films.  Whenever surfaces come together 
along a singular edge, they meet in threes and come together 
at $120$ degree angles, and whenever those singular edges meet at a 
singular vertex, they meet in fours and come together at the 
tetrahedral angle (approximately $109$ degrees).}
\end{center}
\end{figure}

\subsection{Connections with other fields} 
Because CMC surfaces model soap films and 
interfaces between fluids, they have connections to 
physics, chemistry and polymer science.  In fact, sometimes 
new examples of these surfaces 
are discovered by people in these other fields rather than by differential 
geometers.  (One example of this are the minimal surfaces found by Fischer and 
Koch \cite{FischerKoch}, see Figure 3.4.10 in 
\cite{wisky}.)  CMC surfaces have 
connections with biology as well, and an example of this is that some forms 
of coral take shapes resembling the triply periodic Schwarz P minimal surface 
in Figure 3.  CMC surfaces are even sometimes 
connected to architecture, as can be seen by looking at 
the Olympic Stadium in Munich, which has sheets resembling minimal surfaces.  
Thus is it clear that CMC surfaces have connections to fields 
outside of mathematics, 
and this is certainly one of the reasons why we study them.  

\begin{figure}[phbt]
\label{fig-schwarzP}
\includegraphics[width=0.33\linewidth]{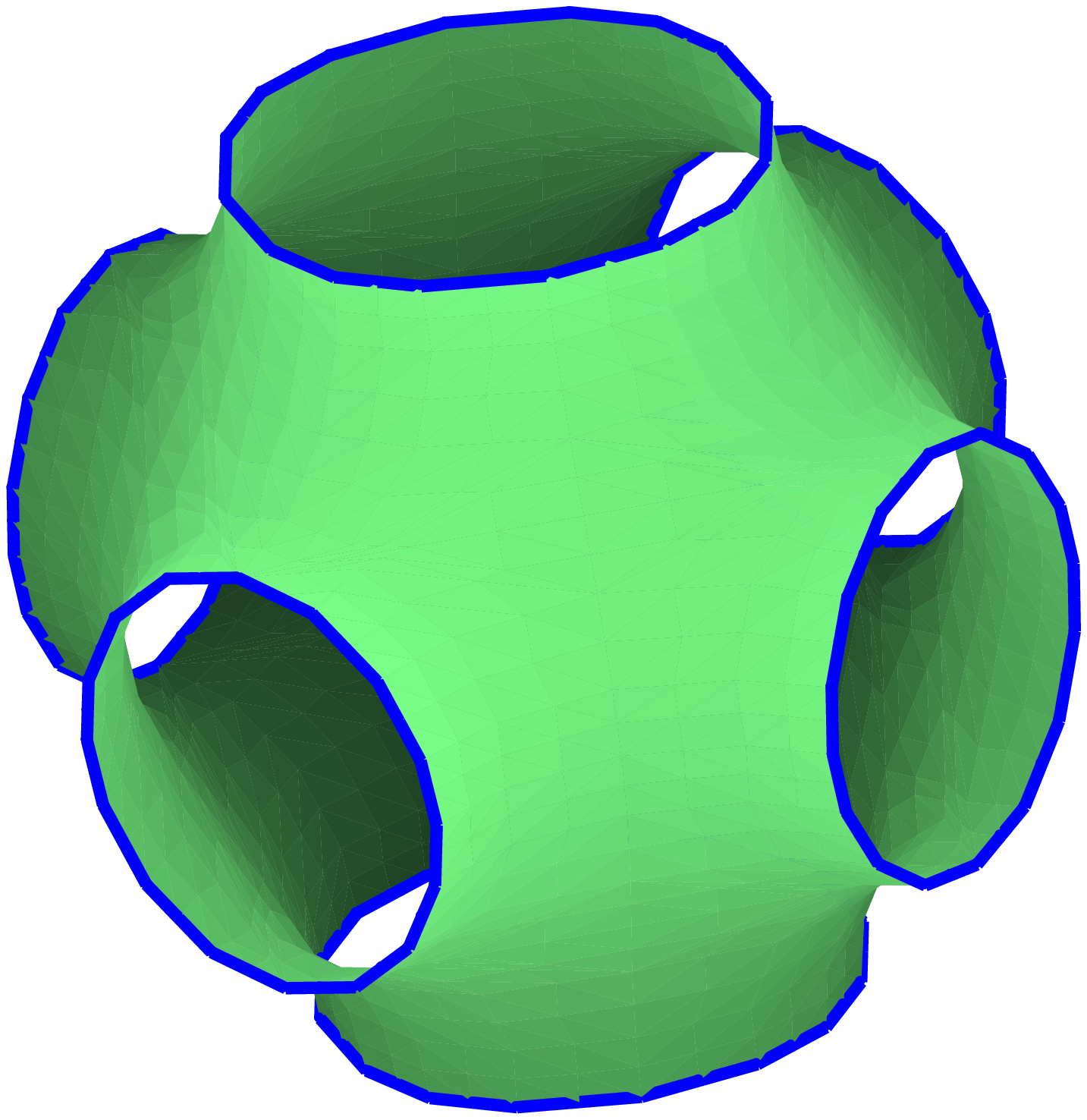}
\includegraphics[width=0.33\linewidth]{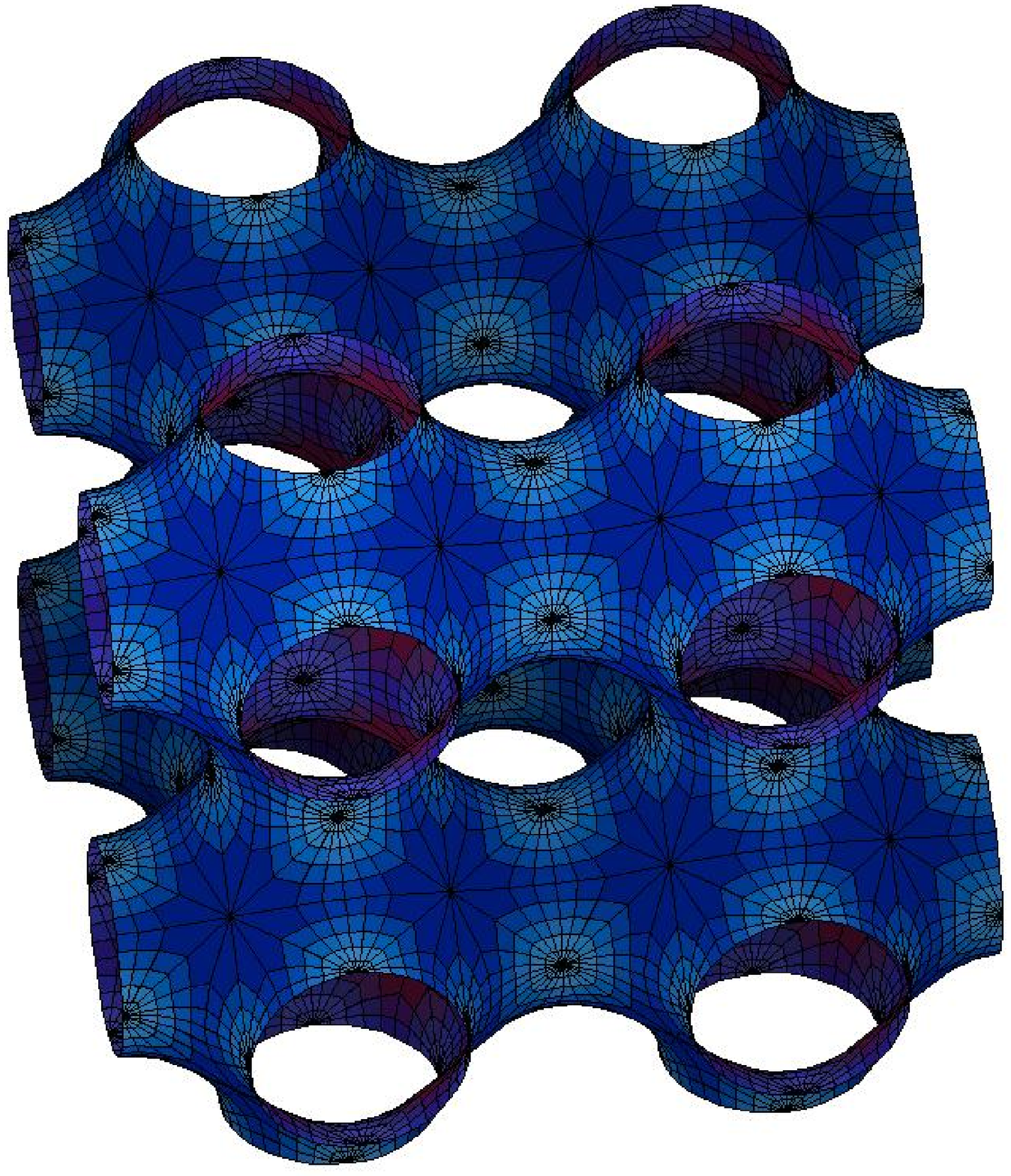}
\caption{The minimal triply-periodic Schwarz P surface.} 
\end{figure}

\subsection{Connections within mathematics} 
Other reasons for studying CMC surfaces are that they have 
a rich mathematical structure and have interesting relations 
to other fields within mathematics.  Although 
minimal and CMC surfaces are topics of geometry, they are also 
fundamental examples in the calculus of variations, as is clear from the 
variational property that we described above.  Thus minimal and 
CMC surfaces are 
closely connected to the calculus of variations (although we will explore this 
connection only briefly in Section \ref{section1.6}).  

Minimal surfaces are also strongly related to the field of complex 
analysis via a theorem 
called the Weierstrass representation (this representation 
was given in \cite{wisky}).  This representation provides a 
way to describe all minimal surfaces using pairs of complex-analytic functions 
defined on Riemann surfaces.  As a result, the theory of minimal surfaces has 
a rich mathematical structure and has many easily accessible examples.  
A number of the simpler examples were described in \cite{wisky}.  

Also, by making use of an additional parameter (called the 
spectral parameter), 
one can describe non-minimal CMC surfaces as well in terms of 
complex-analytic functions defined on Riemann surfaces 
(see \cite{wisky}).  Hence again we have 
a connection to the field of complex analysis.  Furthermore, away 
from isolated special points (umbilics), non-minimal 
CMC surface theory is equivalent to the sinh-Gordon 
equation.  This equation appears prominently in the theory of 
integrable systems, so CMC surfaces are 
also clearly connected to that field.  In fact, the essential 
idea behind the DPW method, which we focused on in 
\cite{wisky}, comes from the theory 
of integrable systems. The DPW method is a method for constructing 
CMC surfaces using 
loop group techniques coming from the theory of integrable systems.  
Finally, we note that both the minimal and non-minimal 
CMC surface equations are well-known partial differential equations, so the 
connection of these surfaces to the field of partial differential equations is 
evident.  

Applying the techniques of these other fields of mathematics to CMC 
surfaces gives these surfaces a rich mathematical structure and gives us the 
means to describe many examples of CMC surfaces, as we saw in \cite{wisky}.  

\subsection{Non-Euclidean ambient spaces} 
When we move to studying CMC surfaces in spaces other than Euclidean $3$-space 
$\R^3$, the connections 
to chemistry, polymer science, biology and architecture certainly largely 
disappear, but connections to physics still remain -- 
and the strong connections to other fields within mathematics remain 
completely intact, 
as we can find other ambient spaces for which the rich mathematical structure 
of CMC surfaces and their connections to other mathematical fields 
carry over.  
In some ways the mathematical structure carries over in an analogous way from 
the case of $\R^3$, but in some ways the structure changes in 
interesting ways.  
The behavior of the direction perpendicular to the surface (the Gauss map) can 
behave quite differently in other 3-dimensional ambient spaces, and the global 
properties of the CMC surfaces can be markedly different.  In this text, we 
will study CMC surfaces (and some other types of surfaces as well) in 
the spaces $\mathbb{S}^3$, $\mathbb{H}^3$ 
and $\mathbb{R}^{2,1}$ that we will define later in this text.  

\subsection{Discrete CMC surfaces} 
Recently, finding discrete analogs of smooth objects has become 
an important theme in mathematics, appearing in a variety 
of places in analysis and geometry.  So it is natural 
to consider discrete analogs of smooth minimal and CMC 
surfaces.  But there is no single definitive approach; the 
definition one chooses depends on which properties of smooth 
minimal and CMC surfaces one wishes to emulate in the 
discrete case.  

One can define a discrete minimal surface in Euclidean 
$3$-space ${\mathbb{R}^{3}}$ to be a piecewise linear 
triangulated surface that is critical for area with respect 
to any compactly-supported boundary-fixing continuous 
piecewise-linear variation (of its vertices) that 
preserves its simplicial structure, see 
\cite{PinkPolt}.  Then one can define 
discrete CMC surfaces the same way, but adding the condition 
that the variations must preserve volume to one side of the 
surface, as in \cite{PR}.  These definitions are clearly 
imitating the variational properties that smooth minimal 
and CMC surfaces have.  This results in discrete surfaces 
with the right variational properties, but without the 
elegant "holomorphic" structure that the corresponding 
smooth surfaces have.  Examples of a discrete 
catenoid and Delaunay surface 
made via this approach are shown on the left-hand side of 
Figure 4.  We will not take this approach in these notes.  

One could instead use discretized versions of integrable systems to 
define discrete minimal and CMC surfaces, in analogy to integrable 
systems properties of smooth minimal and CMC surfaces, as 
Bobenko and Pinkall did 
(\cite{BobPink}, \cite{BobPink2}).  These discrete surfaces 
are formed from planar quadrilaterals.  This approach gives 
discrete minimal and CMC surfaces with "discrete holomorphic" mathematical 
structures corresponding to the "smooth holomorphic" structures of 
the corresponding smooth minimal and CMC surfaces.  This approach has the 
advantage of preserving the rich mathematical structure in the discrete 
case, but it generally does not yield area-critical discrete 
surfaces with respect to vertex variations.  Examples of a discrete 
catenoid and Delaunay surface 
made via this approach are shown on the right-hand side of 
Figure 4.  These discrete surfaces and this 
approach are the central subject of this text.  

\subsection{Prerequisites} 
Before discussing more about CMC surfaces, we need to 
define some mathematical objects that will facilitate the discussion.  
We begin in Section \ref{ambientspaces}, 
as promised above (after a brief introduction 
to variational properties in Section \ref{section1.6}), 
with the ambient spaces that will appear in this text.  

Although we already have defined in \cite{wisky}, or 
will define here, everything that we need to rigorously discuss CMC surfaces, 
in fact it would be hard for the reader to appreciate the signifigance of the 
discussions here without at least a bit of experience with 
differential geometry.  
We assume that the reader is already somewhat familiar with basic differential 
geometry.  There are many good textbooks on basic differential 
geometry and surface theory, for example: \cite{doCarmo1}, \cite{doCarmo2}, 
\cite{Gray}, \cite{Hopf}, \cite{kn}, \cite{La1}, \cite{N}, \cite{O} and \cite{Sp}.  

\begin{figure}[phbt]
\label{catenoids-two-ways}
\begin{center}
\begin{tabular}{ccc}
\includegraphics[width=0.4\linewidth]{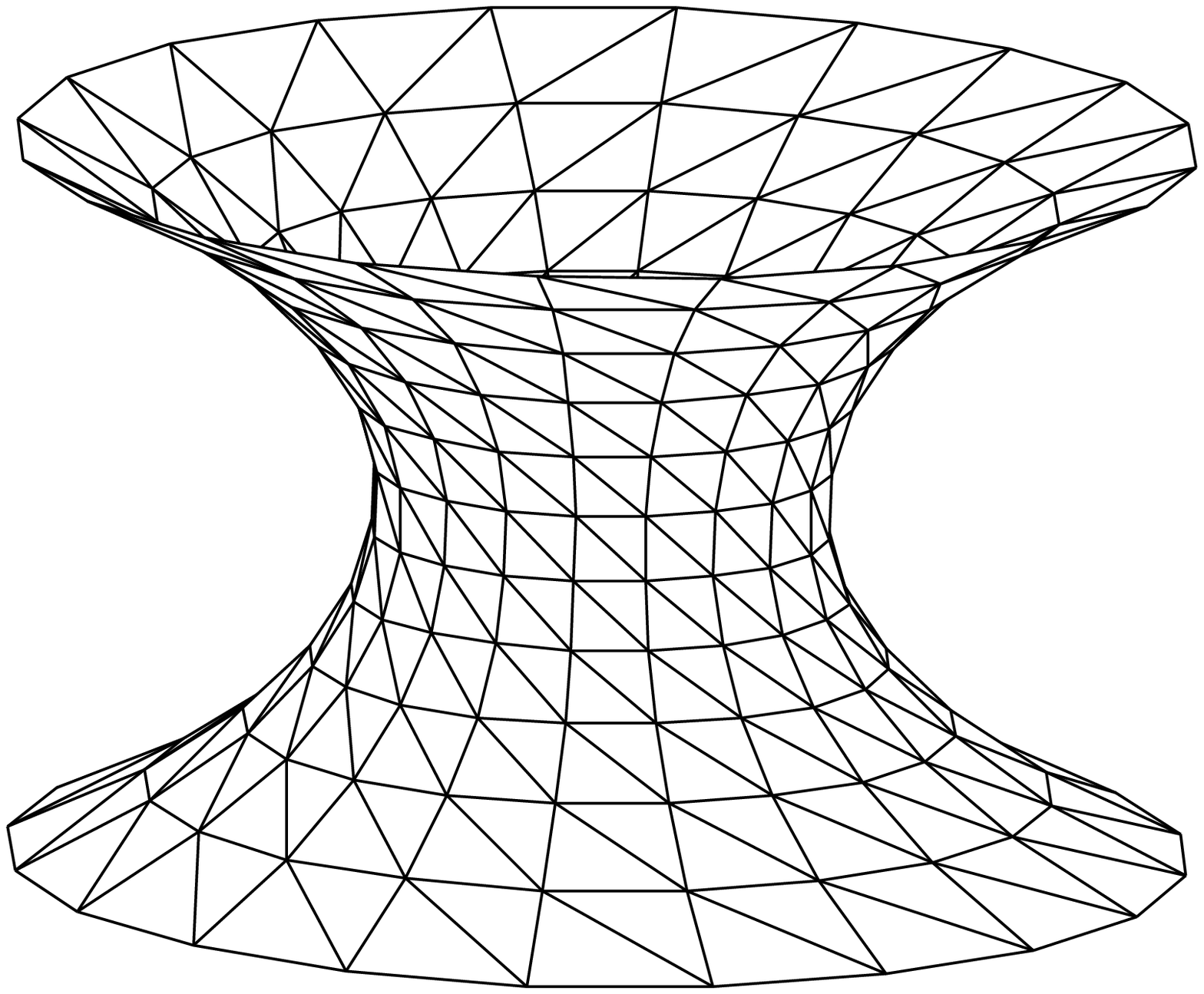} &
\includegraphics[width=0.475\linewidth]{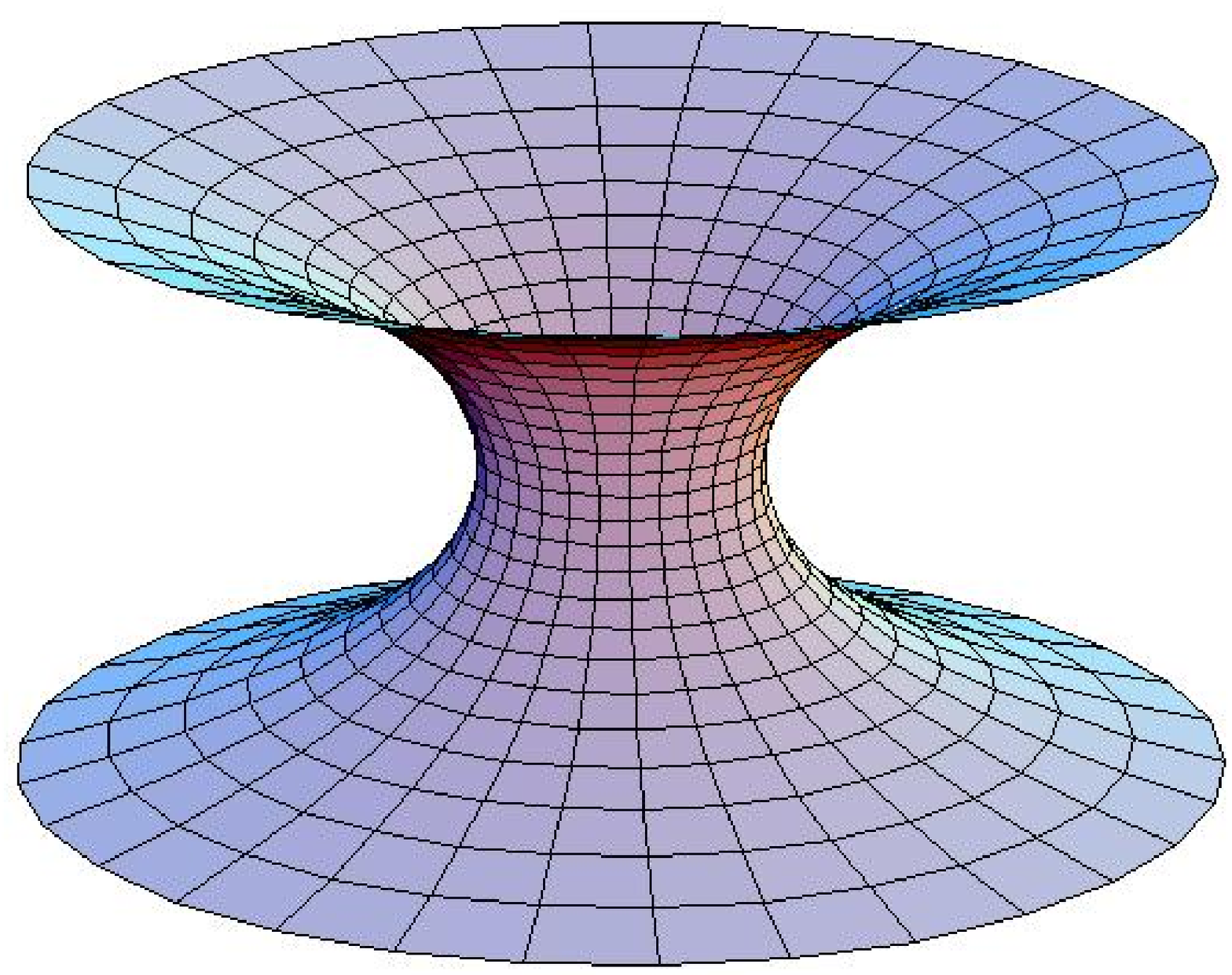} &
\end{tabular}
\begin{tabular}{ccc}
\includegraphics[width=0.8\linewidth]{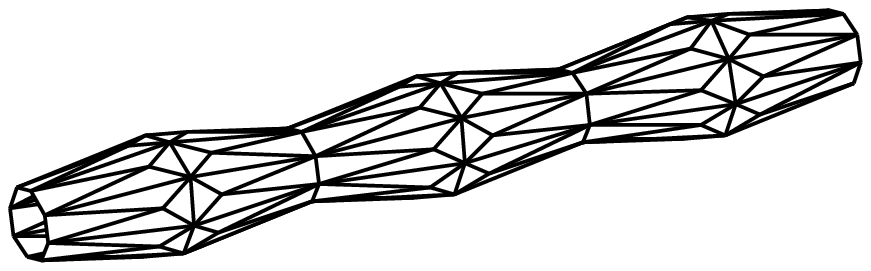} &
\includegraphics[width=0.15\linewidth]{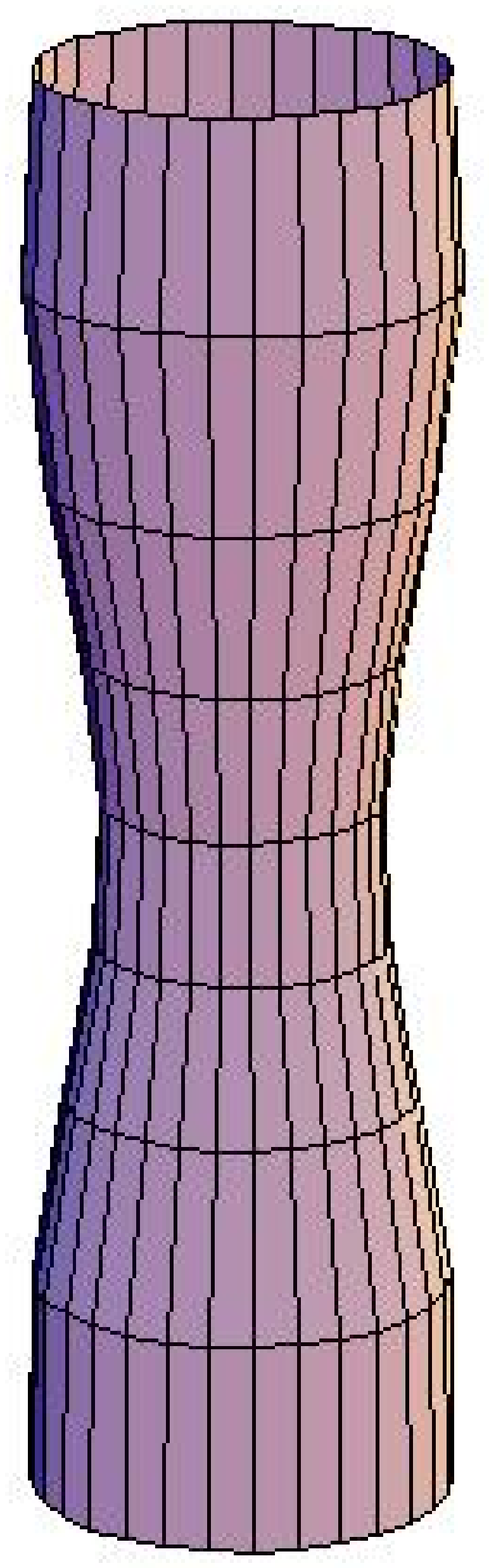} &
\end{tabular}
\caption{Discrete minimal catenoids and Delaunay surfaces.  The example in the 
upper left (resp. lower left) is discrete minimal (resp. discrete CMC) 
with respect to the variational approach.  The example in the 
upper right (resp. lower right) is discrete minimal (resp. discrete CMC) 
with respect to an integrable systems approach.  
}
\end{center}
\end{figure}

\section{Smooth CMC surfaces and their variational 
properties}\label{section1.6}

We defined mean curvature $H$ and CMC surfaces in \cite{wisky}.  The 
definition there states that surfaces for which $H$ is constant are 
CMC surfaces, and that minimal surfaces are those CMC surfaces 
with mean curvature $H=0$.  In this section, we consider why, with these 
definitions, minimal and CMC surfaces are models for soap films.  

The first and second variation formulas here are important for 
understanding how 
CMC and minimal surfaces are models for soap films, and in turn 
for understanding why 
we are interested in such surfaces.  However, since these formulas will not be 
directly used later in this text, we content ourselves with stating them 
without proof, and with stating some other properties without proof as well.  
Furthermore, to simplify the discussion a bit, we restrict 
ourselves in this section 
to the case that the ambient space is $\mathbb{R}^3$.  (Analogous 
properties hold for 
the minimal and CMC surfaces in the other ambient spaces we 
consider, with slightly different formulas.)  

Let \[ f : \Sigma \to \mathbb{R}^3 \] be an immersion of a 
$2$-dimensional domain 
$\Sigma$ in the $(u,v)$-plane $\mathbb{R}^2$ (i.e. the plane 
$\mathbb{R}^2$ with Cartesian coordinates $u$ and $v$) 
into $\mathbb{R}^3$ with induced metric $g$ and with 
unit normal vector $\vec N = \vec N(u,v)$.  We first note that another 
equivalent way to define the mean curvature $H$ at $f(p)$ is 
as the average of the normal curvatures \[ -\langle \vec{v}, 
D_{\vec{v}} \vec{N} \rangle \] 
(intuitively, the 
normal curvature measures the rate at which the surface bends 
toward $\vec{N}$, in the direction $\vec{v}$) 
in all tangent directions \[ \vec{v} \in {\mathcal S} = 
\{ \vec{w} \in T_p \Sigma \, | \, g(\vec{w},\vec{w})=1 \} \; , \] where 
the average is computed by integrating 
$-\langle \vec{v}, D_{\vec{v}} \vec{N} \rangle$ over $\mathcal S$ 
(with respect to an appropriate $1$-dimensional volume form on 
$\mathcal S$, which we do not describe 
explicitly here).  Thus, for example, a minimal surface has 
average normal curvature zero at every point, and this suggests a physical 
interpretation, for which we quote \cite{HM}: 
\begin{quote}
\cite{HM}: 
``Loosely speaking, one imagines the surface as made up of very 
many rubber bands, stretched out in all directions; on a minimal surface the 
forces due to the rubber bands balance out, and the surface does not need to 
move to reduce tension.''  
\end{quote}
To say this more rigorously, suppose $\Sigma$ is a compact domain in 
the $(u,v)$-plane, and 
define a {\em smooth boundary-fixing 
variation} of the immersion $f(\Sigma)$ to be a $C^\infty$ 
map $f_t: (-1,1) \times \Sigma \to \mathbb{R}^3$ with three properties: 
\begin{enumerate}
\item $f_t(\cdot): \Sigma \to \mathbb{R}^3$ is an immersion for all 
$t \in (-1,1)$, 
\item $f_0 = f$ on $\Sigma$, 
\item $f_t|_{\partial \Sigma} = f|_{\partial \Sigma}$ for all 
$t \in (-1,1)$.  
\end{enumerate}
We call \[ \tfrac{d}{dt} f_t |_{t=0} \] the 
{\em variation vector field} of $f_t$ at $t=0$.  

Note that $\mbox{Area}(f_t(\Sigma)) = \int_\Sigma dA_t$, where 
$dA_t = \sqrt{g_{t,11} g_{t,22} - 
g_{t,12}^2} dudv$ is the volume element (the area $2$-form) 
of the metric $g_t=(g_{t,ij})$ induced by the immersion $f_t$ with respect to 
the coordinates $(u,v)$ of $\Sigma$.  It turns out that 
(see, for example, \cite{La1}) the 
first variation formula for smooth boundary-fixing variations is then 
\begin{equation}\label{1rstvar} 
\left. \frac{d}{dt} \mbox{Area}(f_t(\Sigma)) \right|_{t=0} = 
- \int_\Sigma \left\langle H \vec{N}, 
\left. \frac{d}{dt} f_t \right|_{t=0} \right\rangle dA_{0} \; . 
\end{equation}  
In particular, minimal surfaces (with $H \equiv 0$) are critical for area 
amongst all smooth boundary-fixing variations on any compact domain 
$\Sigma$, and we could 
have defined them this way.  Actually, when the subdomain $\bar \Sigma$ 
of $\Sigma$ is small enough, 
not only is $f(\bar \Sigma)$ critical for area, it is also the 
unique least-area surface with boundary $f(\partial \bar \Sigma)$, hence 
"minimal" surface is a natural name for such surfaces.  
Indeed, minimal surfaces are a natural $2$-dimensional generalization of 
$1$-dimensional geodesics, because geodesic segments of 
sufficiently short length are the least-length paths 
from one endpoint of the segment to the other (see Section 1.1 of 
\cite{wisky}).  Furthermore, 
although longer geodesics might not be least-length between their 
endpoints, they 
are still always critical for length amongst all smooth 
variations of the path fixing 
the endpoints (again, see Section 1.1 of 
\cite{wisky}).  This is 
completely analogous to the variational properties of minimal 
surfaces.  

Similarly, a nonminimal CMC surface could be defined 
as an immersion $f:\Sigma \to \mathbb{R}^3$ such that 
$f(\Sigma)$ is critical for area amongst all 
smooth boundary-fixing variations that keep the volume on one side of the 
surface unchanged: the derivative of this volume with respect to $t$, 
at $t=0$, is 
\[ \int_\Sigma \left\langle \vec{N}, \left. \frac{d}{dt} 
f_t \right|_{t=0} \right\rangle dA_t \; , \] so if the volume is unchanging 
with respect to $t$, and hence 
$\int_\Sigma \langle \vec{N}, \left. \frac{d}{dt} f_t 
\right|_{t=0} \rangle dA_t=0$, and if 
$H$ is constant, then Equation \eqref{1rstvar} implies 
(also, see \cite{BdCE}, for example) 
\[ \left. \frac{d}{dt} \mbox{Area}(f_t(\Sigma)) \right|_{t=0} =0 \; . \]  
Variations that preserve volume to one side of $f_t|_\Sigma$ are called 
{\em volume-preserving} variations.  This is a natural restriction to make 
for non-minimal CMC surfaces, as the example in item (2) of 
Section \ref{chapter0} shows.  
If the round sphere soap film described there were 
allowed to deform in a way that did not preserve the volume inside of it, it 
would reduce its area by simply reducing its radius, and shrink 
down to a single point 
with no area.  But clearly this does not happen, and the reason it does not 
happen is because of this volume constraint.  

We conclude that minimal surfaces in 
$\mathbb{R}^3$ are surfaces that are critical for area with respect to smooth 
variations that fix their boundaries, and 
CMC surfaces are critical for area with respect to smooth 
variations that fix their boundaries and fix the volume to one side 
of the surfaces.  This is why minimal and CMC surfaces model 
physical soap films, which always move to minimize area.  
Minimal surfaces model soap films 
not enclosing bounded pockets of air, as such films 
are area minimizing for all boundary-fixing 
variations.  Nonminimal CMC surfaces model soap films enclosing 
bounded pockets of air, as such films are area minimizing only 
for variations that keep the air pockets' volumes fixed.  

These variational properties in the Euclidean case similarly hold 
for other ambient spaces, such as $\mathbb{S}^3$ and $\mathbb{H}^3$ 
(see Section \ref{ambientspaces}, see also \cite{wisky}).  

The second variation formula for volume-preserving variations 
of CMC surfaces (\cite{bc}, \cite{Che}, \cite{Silv}, \cite{La1}) is 
(we may ignore the volume-preserving condition when the CMC surface is 
minimal) 
\begin{equation}\label{2ndvar} 
\left. \frac{d^2}{dt^2} \text{Area}(f_t(\Sigma)) \right|_{t=0} = 
\int_\Sigma h \cdot L(h) dA_0 \; , \end{equation} where \[ L(h) = 
- \triangle h - (4H^2-2K)h \;\; \text{ and } \;\; 
h = \langle \tfrac{d}{dt} f_t |_{t=0} , \vec{N}|_{t=0} \rangle
\; , \] with Gaussian curvature 
$K$ (see \cite{wisky}) and Laplace-Beltrami operator 
\[ \triangle h = 
\frac{1}{\frak{G}} \left( \partial_u (\frak{G} g^{11} \partial_u h) + 
\partial_u (\frak{G} g^{12} \partial_v h) + 
\partial_v (\frak{G} g^{21} \partial_u h) + 
\partial_v (\frak{G} g^{22} \partial_v h) \right) \; , \]
where $\frak{G}=\sqrt{g_{11}g_{22}-g_{12}^2}$, and $g^{-1} = 
(g^{ij})_{i,j=1,2}$ is the inverse matrix of $g=g_0=(g_{ij})_{i,j=1,2}$.  

Since the first derivative $\tfrac{d}{dt} 
\text{Area}(f_t(\Sigma)) |_{t=0}$ is zero for 
CMC surfaces with respect to the appropriate variations, the sign of the 
second derivative \eqref{2ndvar} determines whether a variation increases or 
decreases the area.  If 
there exists a variation $f_t$ so that \eqref{2ndvar} becomes negative, 
then the minimal or CMC surface will not be area minimizing with 
respect to the appropriate 
space of variations.  If, on the other hand, \eqref{2ndvar} is 
positive for every 
nontrivial variation $f_t$ with respect to the appropriate 
variation space, then the minimal or CMC surface will be 
locally area minimizing in the space of variations.  

The four examples of soap films described at the beginning of 
Section \ref{chapter0} are 
examples of minimal and CMC surfaces that are area-minimizing.  If they had 
not been area-minimizing we never would 
have been able to construct them with soap films 
in the first place.  However, not all of these four examples 
extend (analytically) to larger CMC 
surfaces that are area-minimizing (even though any CMC extensions 
are certainly still 
area-critical, by the first variation formula \eqref{1rstvar}).  
The first example, the flat disk, can be extended to a 
complete flat plane, which is a minimal surface.  The complete flat plane is 
area-minimizing in the sense that any compact region $\Sigma$ 
within it is area-minimizing 
(with respect to the compact region's boundary) 
and can be made as a soap film with a planar wire frame in the shape of 
its boundary.  In particular, \eqref{2ndvar} will always be 
positive for any 
nontrivial smooth boundary fixing variation of any such compact region 
$\Sigma$.  The second example, the round sphere, is already complete and 
so cannot be extended at all.

It is the third and fourth examples that extend to surfaces which are 
not area-minimizing.  Let us consider the fourth example first.  
The fourth example is a round cylinder, and, up to a rigid motion 
of $\mathbb{R}^3$, we can represent it by the immersion 
\[ f(u,v) = (r \, \cos u, r \, \sin u, r \, v) \] for 
$(u,v) \in \Sigma 
= [0,2\pi] \times [0,\tfrac{d}{r}] \subset \mathbb{R}^2$ for some constants 
$r,d \in \mathbb{R}^+$.  This is a portion of a cylinder with radius $r$ and 
height $d$.  The induced fundamental forms (see \cite{wisky}) are 
\[ g = r^2 dz d\bar z \;\;\; \text{ and } \;\;\; b = -\tfrac{r}{4} dz^2 
-\tfrac{r}{2} dzd\bar z -\tfrac{r}{4} d\bar z^2 \; , \;\;\; 
\text{with} \;\;\; z = u + i v \; , \;\;\; i = \sqrt{-1} \; . \]  
So $K=0$ and 
$H=\frac{-1}{2r}$, and the right-hand side of \eqref{2ndvar} is 
\begin{equation}\label{cyl2ndff} 
\int_0^{\tfrac{d}{r}} \int_0^{2\pi} h \cdot L_{cyl}(h) dudv \; , 
\;\;\;\;\;\; L_{cyl}(h) = -h_{uu}-h_{vv}-h \; . \end{equation}  
This second derivative 
of area can be negative for some boundary-fixing volume-preserving 
variation if and only 
if $d>2 \pi r$, and there is a 
reason why $2 \pi r$ is the height beyond which the cylinder becomes only 
area-critical instead of area-minimizing.  We will not fully 
explain the reason 
here (we refer the reader to \cite{bc} for a rigorous explanation), 
but we will give a 
clue as to why this is so.  Consider the function 
\[ h=h(v)=\sin \tfrac{2 \pi r v}{d} \; . \]  
It has these properties: 
\begin{itemize}
\item $h|_{v=0}=h|_{v=d/r}=0$ (an infinitesimal "boundary-fixing" property), 
\item $\int_0^{d/r} h dv = 0$ (an infinitesimal "volume-preserving" property), 
\item $L_{cyl}(h) = \mu h$, where 
\[ \mu = \frac{4 \pi^2 r^2 - d^2}{d^2} \; . \]
\end{itemize}
Thus $h$ is an eigenfunction of the operator $L_{cyl}$ with eigenvalue 
$\mu$, and $\mu < 0$ precisely when $d>2 \pi r$.  So if we choose a 
rotationally symmetric variation based on this function $h$ (i.e. a 
rotationally symmetric variation whose 
variation vector field at $t=0$ is $h \cdot \vec N$, where $\vec N = 
(\cos u, \sin u, 0)$ is the unit normal vector to $f=f(u,v)$, 
see \cite{bc}), the integrand in 
the second variation formula \eqref{cyl2ndff} will become 
negative precisely when 
$d>2 \pi r$.  We conclude that a cylindrical tube of radius $r$ and 
height $d>2 \pi r$ cannot be made as a physical film.  

The third example of a soap film from 
Section \ref{chapter0} is a catenoid.  
The profile curve for a catenoid is the hyperbolic cosine function, so a 
catenoid can be parametrized as 
\[ f(u,v) = (\cosh v \cos u , \cosh v \sin u , v ) \; , \] with 
\[ (u,v) \in \Sigma = [0,2 \pi] \times [-d,d] \subset 
\mathbb{R}^2 \] for some $d \in \mathbb{R}^+$.  Here $2 d$ is the 
distance between 
the two boundary circles.  Let $d_0 \approx 1.2$ be the unique 
positive solution to 
$d_0 \sinh d_0 = \cosh d_0$.  Then the catenoid $f$ will be 
area-minimizing if $d < d_0$ 
and will not be area-minimizing (i.e. only area-critical) if $d > d_0$.  Hence 
if we extend the value $d$ past $d_0$, the catenoid will 
no longer be constructable with a soap film.  

\begin{figure}[phbt]
\begin{center}
\includegraphics[width=1.0\linewidth]{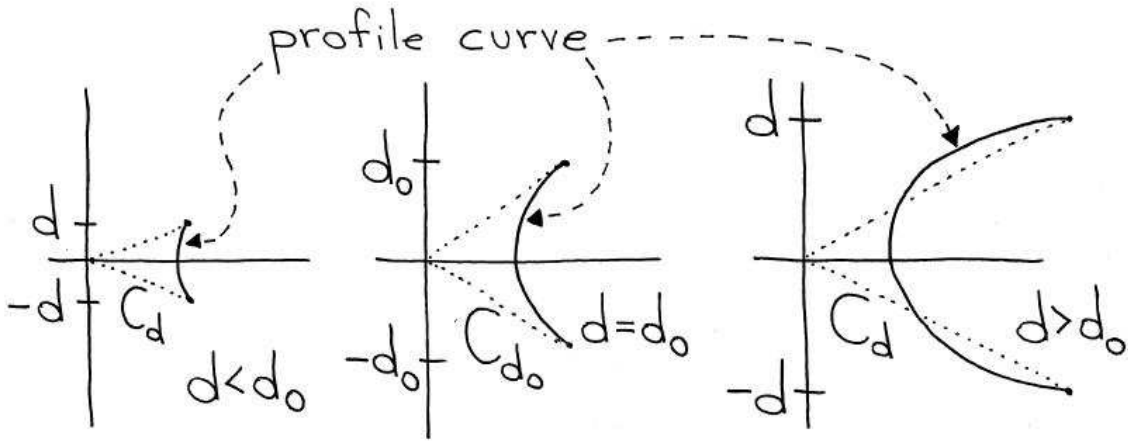}
\caption{The profile curve on the left (resp. middle, right) 
creates a stable (resp. weakly stable, unstable) catenoid.}
\end{center}
\end{figure}

Again, we will not explain here why $d_0$ is the precise value beyond 
which the catenoid becomes 
non-area-minimizing, but, again, we will give a hint why this is so.  The 
value $d_0$ actually has a geometric interpretation, as follows: For 
each $v > 0$, consider the cone 
\[ C_v = \left\{ \left. 
\left( x,y,\pm \frac{v}{\cosh v} \sqrt{x^2+y^2} \right) 
\, \right| \, x,y \in \mathbb{R} \right\} \; . \]  Then the cone $C_v$ 
intersects 
the catenoid tangentially (i.e. a non-transversal non-empty 
intersection) if and only if 
$v = d_0$.  When $d < d_0$, any homothety of $\mathbb{R}^3$ 
centered at the origin 
$(0,0,0)$ will move the catenoid to another catenoid disjoint 
from the first one, 
while this is not the case when $d > d_0$.  
These facts are related to the 
question of whether there exists a boundary-fixing 
variation $f_t$ of the catenoid $f$ that has negative 
second derivative of area (we do not 
need the "volume-preserving" property here, as the catenoid 
is a minimal surface).  For 
a complete explanation of this, a good source is \cite{choe}.  

\subsection{Steiner points}\label{steinerpoints}

Minimal surfaces minimize area (at least locally) with respect to 
their boundary curves, thus, as noted above, 
they model soap films that do not 
surround bounded pockets of air.  One could consider the analogous 
phenonemon, but one dimension lower.  Instead of 
trying to connect $1$-dimensional things like sets of curves 
(i.e. the wire frames that we 
use to make soap bubbles) with area-minimizing surfaces, 
we could try to connect $0$-dimensional things such as finite 
sets of points, 
and instead of connecting them with $2$-dimensional surfaces, we 
would connect them with $1$-dimensional curves, and instead of 
trying to minimize the areas 
of the $2$-dimensional surfaces, we would minimize lengths of 
the $1$-dimensional curves.  As we saw in Figure 2, the 
area-minimizing soap films can have $1$-dimensional singular 
curves where three sheets of a soap film come together at 
equal angles.  When the dimension is reduced by $1$ as above, 
the singular curves are replaced with {\em Steiner points}, which are 
singular points at which three curves (actually straight line 
segments) come together at equal angles.  

\begin{figure}[phbt]
\label{fig:steiner}
\begin{center}
\includegraphics[width=0.6\linewidth]{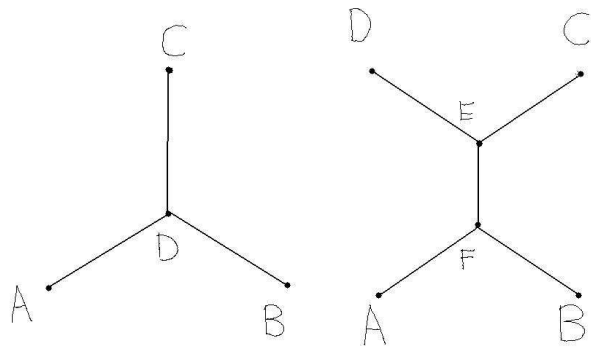} 
\end{center}
\caption{Examples of Steiner points in length-minimizing 
planar graphs.}
\end{figure}

To demonstrate this, let us consider the following examples: 

\begin{example}
Imagine you have two cities, call them 
city $A$ and city $B$, on a flat region of land, where 
no mountains or lakes or other obstructions exist, and you want 
to build a road (or collection of roads) that connects 
the two cities.  Suppose further that you want to 
minimize the total length of the road (or roads).  

You would, of course, just build one road along the straight line from city 
$A$ to city $B$.  (The mathematicial statement would be that the 
shortest path between two points is a straight line.)
\end{example}

\begin{example}
Now imagine that there are three 
cities, city $A$, city $B$ and city $C$, and that those three cities 
lie at the three vertices of an equilateral triangle.  Now you 
want to build roads with minimal total 
length so that all three cities are connected, i.e. so that you can 
drive from any one city to any other of the three.  

Suppose that the length of each side of the triangle is $\ell$.  
If you just build a straight road from city $A$ to city $B$, 
and another straight road from city $A$ to city $C$, then you 
would not need to build any road from city $B$ to city $C$, as 
you could already get from city $B$ to city $C$ by passing through 
city $A$.  The total length of the roads would be $2 \ell$.  

But this is not actually the best solution.  The best way 
is to make a new city $D$ at the center of 
the triangle, and then make three straight-line roads, one from 
each of the cities $A$, $B$ and $C$ directly to city $D$.  Now 
the sum of the three lengths of these roads would be 
$\sqrt{3} \ell$, which is strictly less than $2 \ell$, and this 
is the best way.  See Figure 6.  
\end{example}

The city $D$ in the previous example is what we call a Steiner point.  
It is an added point that is used to minimize total length of roads.  

\begin{example}
Now imagine that you have four cities, cities $A$, $B$, $C$ and $D$,  
at the vertices of a square with sides of length $\ell$, in 
sequencial order around the square.  
To connect these four cities so that the 
road length is minimized, you might first think of building three 
straight-line roads, each of length $\ell$, one from city $A$ to 
city $B$, one from city $B$ to city $C$ and one from city $C$ to 
city $D$.  
(Note that you now do not need a road from city $A$ to city $D$, just 
as in the previous example).  Then the total length of roads is 
$3 \ell$.

But this is not the best way.  A better 
way would be to put a city $E$ at the center of the square 
and draw roads directly from 
each of the four original cities to the new city $E$.  Then the 
roads form an "X" and the length is now $2 \sqrt{2} \ell$, which is 
less.  

But this is still not the best solution.  The best solution 
is to actually have two new cities $E$ and $F$ (i.e. two 
Steiner points), 
and then to draw in roads as in the second picture of Figure 6.  
The two Steiner points are placed in this picture so that 
the angle between any two roads meeting at a Steiner point is always 
exactly 120 degrees.  One can now check the total length of the 
roads is strictly less than $2 \sqrt{2} \ell$, and this is the 
best solution.  Note that there are two different ways to choose a 
least-length solution.  
\end{example}

In the above three examples, we have seen how Steiner points 
help us to find the least-length collection of "1-dimensional" 
curves (i.e. roads) that connects some points (cities) together.  
This is analogous to the way singular points (and singular 
curves) can appear on area-minimizing surfaces.  

\section{Ambient spaces} 
\label{ambientspaces}

CMC surfaces always exist in some larger ambient space.  
In the soap-film examples 
we described in Chapters \ref{chapter0} and \ref{section1.6}, we were 
assuming that the CMC surfaces lie in the Euclidean $3$-space 
$\mathbb{R}^3$.  We encountered CMC surfaces in other 
non-Euclidean ambient spaces in \cite{wisky}.  
Also, there is a description of general Riemannian and Lorentzian manifolds 
in \cite{wisky}.  Here we give two examples of ambient spaces: we 
describe hyperbolic 3-space, like in \cite{wisky}, 
but in a bit more detail; we also briefly describe de Sitter $3$-space.  
Minkowski ($n+1$)-space $\mathbb{R}^{n,1}$ 
and spherical $3$-space $\mathbb{S}^3$ also 
appear in these notes, and we assume the reader is already familiar with 
those spaces (they are described in \cite{wisky}).  

\subsection{Hyperbolic $3$-space $\mathbb{H}^3$}
\label{hyperbolic3space}

Hyperbolic $3$-space $\mathbb{H}^3$ is the unique simply-connected 
$3$-dimensional complete Riemannian manifold with 
constant sectional curvature $-1$.  However, it can be described by a 
variety of models, each with its own advantages: the Minkowski 
space model, the Poincare ball model, the Hermitian 
matrix model, the Klein ball model and the upper-half-space model.  

We define $\mathbb{H}^3$ by way of the Minkowski $4$-space 
$\mathbb{R}^{3,1}$ with its Lorentzian metric $g_{\mathbb{R}^{3,1}}$ 
of signature $(+++-)$, by 
taking the upper sheet of the two-sheeted hyperboloid 
\[ \mathfrak{M} = \left\{(x_1,x_2,x_3,x_0) \in \mathbb{R}^{3,1} \, \left| \, 
x_{0}^2-\sum_{j=1}^{3} x_j^2 = 1 \, , \; x_{0} > 0 \right. \right\} \; , \]  
with metric $g$ given by the restriction of $g_{\mathbb{R}^{3,1}}$ 
to the tangent spaces of this $3$-dimensional upper sheet.  
We call this $\mathfrak{M}$ 
the {\em Minkowski model for hyperbolic $3$-space}.  Although the 
metric $g=g_{\mathbb{R}^{3,1}}$ is Lorentzian and 
therefore not positive definite, 
the restriction of $g$ to this upper sheet is actually positive 
definite, so $\mathfrak{M}$ is a Riemannian manifold.  

The isometry group of 
$\mathfrak{M}$ can be described using the matrix group 
\[ O_+(3,1) = \{ A = (a_{ij})_{i,j=1}^4 \in O(3,1) \, | \, a_{44} > 0 \} 
\; . \] 
For $A \in O_+(3,1)$, the map 
\[ \mathbb{R}^{3,1} \ni \vec x \to (A (\vec x)^t)^t \in \mathbb{R}^{3,1} \] 
is an isometry of $\mathbb{R}^{3,1}$ that preserves $\mathfrak{M}$, 
hence it is an isometry of $\mathfrak{M}$.  In fact, 
all isometries of $\mathfrak{M}$ can be described this way.  

The following lemma tells us that the Minkowski model for hyperbolic $3$-space
is indeed the true hyperbolic $3$-space.  

\begin{lemma}\label{H3hascurvatureminus1}
$\mathfrak{M}$ is a simply-connected 
$3$-dimensional complete Riemannian manifold with constant 
sectional curvature $-1$.  
\end{lemma}

Since this lemma implies $\mathfrak{M}$ is really the hyperbolic $3$-space 
$\mathbb{H}^3$, we will in fact sometimes refer to this Minkowski space model 
$\mathfrak{M}$ simply as $\mathbb{H}^3$.  

\begin{proof}
It is clear that $\mathfrak{M}$ is simply-connected.
Let us now check that it has constant sectional curvature $-1$.  

For any point $p \in \mathfrak{M}$, there exists a matrix $A \in \text{SO}_3 = 
O(3) \cap \{ A \in M_{3 \times 3} (\mathbb{R}) 
\, | \, \det A = + 1 \}$ such that the $4 \times 4$ matrix 
\[ \begin{pmatrix} 
 &  &  & 0 \\
 & A &  & 0 \\
 &  &  & 0 \\
0 & 0 & 0 & 1
\end{pmatrix} \in O_+(3,1) \]
preserves $\mathfrak{M}$ and
maps $p$ to a point of the form $(0,0,\sinh(s),\cosh(s))$, $s \in \mathbb{R}$.
Then the matrix
\[ 
\begin{pmatrix}
1 & 0 & 0 & 0 \\
0 & 1 & 0 & 0 \\
0 & 0 & \cosh(-s) & \sinh(-s) \\
0 & 0 & \sinh(-s) & \cosh(-s)
\end{pmatrix} \in O_+(3,1) \; \] is an isometry of $\mathbb{R}^{3,1}$ that 
preserves $\mathfrak{M}$ and maps the point 
$(0,0,\sinh(s),\cosh(s))$ to the point
$(0,0,0,1)$.
Thus one can move an arbitrary point of $\mathfrak{M}$ to the 
point $(0,0,0,1)$ by an isometry of $\mathfrak{M}$.  Now, 
if ${\mathcal V}_1, {\mathcal V}_2$ are any two $2$-dimensional subspaces 
of the $3$-dimensional tangent space 
$T_{(0,0,0,1)}(\mathfrak{M})$, there exists a matrix $A \in O_+(3,1)$ 
representing an isometry of $\mathfrak{M}$ fixing $(0,0,0,1)$ such that 
$d\psi_{(0,0,0,1)}({\mathcal V}_1) = {\mathcal V}_2$.  
Therefore this model has constant sectional curvature, by Lemma 
1.1.6 in \cite{wisky}.  
Thus to see that $\mathfrak{M}$ has constant sectional curvature $-1$, 
one need only
check that this is the value of the sectional curvature of a single fixed 
$2$-dimensional 
subspace of $T_{(0,0,0,1)}(\mathfrak{M})$.  This can be done using 
Equation (1.1.10) or Equation (1.1.14) in \cite{wisky}, 
and we leave this computation to the reader.  

Finally, we argue that $\mathfrak{M}$ is complete.  
Intersecting $\mathfrak{M}$ with the plane 
$\{x_1 = x_2 = 0 \}$, we obtain a curve that can be parametrized with 
unit speed by $\alpha(s) = (0,0,\sinh(s),\cosh(s))$, i.e. 
this parametrization is unit speed with respect to the metric 
$g$ of $\mathfrak{M}$.  
Since the domain of $\alpha(s)$ is all $s 
\in \mathbb{R}$, this curve $\alpha(s)$ is complete.  
And since any geodesic segment in $\mathfrak{M}$ 
can be moved by an isometry
to $\alpha(s), \; 0 \leq s \leq a$ for some value of $a$, we know that 
any geodesic 
segment can be extended to a geodesic of infinite length.  
Therefore $\mathfrak{M}$ 
is complete.  This completes the proof of the lemma.  
\end{proof}

\begin{remark}
In fact, the functions $\sinh(s)$ and $\cosh(s)$ can be 
defined by the condition that the curve 
$\alpha(s) = (0,0,\sinh(s),\cosh(s))$ with 
$\alpha(0) = (0,0,0,1)$ and $\frac{d\alpha}{ds}(0) = (0,0,1,0)$ is 
a unit-speed geodesic with respect to the metric $g$ of 
$\mathfrak{M} = \mathbb{H}^3$.  This condition 
implies that the curve $\alpha(s)$ satisfies
$x_3^2 - x_{0}^2 = -1$, so $\cosh^2(s) - \sinh^2(s) = 1$, and by
differentiation of $x_3^2 - x_{0}^2 = -1$ with respect to $s$ we have 
\[ \frac{\tfrac{dx_{0}}{ds}}{\tfrac{dx_3}{ds}} = 
\frac{x_3}{x_{0}} \Longrightarrow
\frac{\tfrac{d}{ds}(\cosh(s))}{\tfrac{d}{ds}(\sinh(s))} = 
\frac{\sinh(s)}{\cosh(s)} \; . \]
Since $|\alpha^\prime(s)|^2 = (\frac{dx_3}{ds})^2 -
(\frac{dx_{0}}{ds})^2 = 1$, it follows that 
\[ \frac{d}{ds}(\sinh(s)) = \cosh(s) \; , 
\;\;\; \frac{d}{ds}(\cosh(s)) = \sinh(s) 
\; . \]
Now that we know how to differentiate $\cosh(s)$ and $\sinh(s)$, we know
the power series expansions of these functions about $s=0$.  Comparing these
series with the power series expansions 
for $e^s$ and $e^{-s}$ about $s=0$, we conclude that 
\[
\cosh(s) = \frac{e^s + e^{-s}}{2} \; ,
\;\;\; 
\sinh(s) = \frac{e^s - e^{-s}}{2} \; ,
\]
which of course are the standard definitions of $\cosh(s)$ and $\sinh(s)$.
(An analogous analysis can be carried out for the sine and
cosine functions on the unit circle in the Euclidean plane, 
viewing that unit circle 
as a geodesic in the unit sphere $\mathbb{S}^2$ in 
the natural extension of $\mathbb{R}^2$ to $\mathbb{R}^3$.) 
\end{remark}

Because the isometry group of $\mathfrak{M}$, which we have 
noted we may call simply $\mathbb{H}^3$, is the matrix group 
$O_+(3,1)$, the image of the geodesic 
$\alpha(t)=(0,0,\cosh t,\sinh t)$ under an 
isometry of $\mathbb{H}^3$ always lies in a $2$-dimensional 
plane of $\mathbb{R}^{3,1}$ 
containing the origin.  Thus we can conclude that the image of any 
geodesic in $\mathbb{H}^3$ is formed by the 
intersection of $\mathbb{H}^3$ with a $2$-dimensional plane in 
$\mathbb{R}^{3,1}$ 
which passes through the origin $(0,0,0,0)$ of $\mathbb{R}^{3,1}$.  

The Minkowski model is perhaps the best model of 
$\mathbb{H}^3$ for understanding 
the isometries and geodesics of $\mathbb{H}^3$.  
However, since the Minkowski model 
lies in the $4$-dimensional space $\mathbb{R}^{3,1}$, 
we cannot use it to view graphics 
of surfaces in $\mathbb{H}^3$.  So we would like to 
have models that can be viewed 
on the printed page.  We would also like to have a 
model that uses $2 \times 2$ matrices 
to describe $\mathbb{H}^3$, as this is more compatible 
with the DPW method described in \cite{wisky}, and the 
discussion in Sections \ref{lastsubsectionA} and 
\ref{lastsubsectionB} here.  With this in mind, 
we now give some other possible models for $\mathbb{H}^3$.  

\subsection{The Klein model} 
Let $\mathcal K$ be the $3$-dimensional ball in $\mathbb{R}^{3,1}$ lying in the
hyperplane $\{ x_0 = 1 \}$ with radius $1$ and center at $(0,0,0,1)$.
By Euclidean stereographic projection from the origin $(0,0,0,0) \in 
\mathbb{R}^{3,1}$ of the Minkowski model $\mathfrak{M}$ for $\mathbb{H}^3$
to $\mathcal K$, one has the Klein model $\mathcal K$ for
$\mathbb{H}^3$.  $\mathcal K$ is 
given the metric that makes this stereographic 
projection an isometry.  Since the 
geodesics of $\mathbb{H}^3$ in the Minkowski model 
are formed by the intersections of 
$\mathbb{H}^3$ with $2$-dimensional planes in 
$\mathbb{R}^{3,1}$ which pass through 
the origin, it is clear that after projection 
to $\mathcal K$, the geodesics become 
Euclidean straight lines in the Klein model, 
and this is the advantage of the Klein model.  
However, the disadvantage of the Klein 
model is that its metric is not conformal to the Euclidean metric (we defined 
conformality in \cite{wisky}, and we also define it here in Definition 
\ref{defn:conformality4.4}).  

\subsection{The Poincare model} 
Let $\mathcal P$ be the $3$-dimensional ball in $\mathbb{R}^{3,1}$ lying
in the hyperplane 
$\{ x_0 = 0 \}$ with radius $1$ and center at the origin $(0,0,0,0)$.
By Euclidean stereographic projection from the point $(0,0,0,-1) \in 
\mathbb{R}^{3,1}$ of the Minkowski model for $\mathbb{H}^3$ to $\mathcal P$, 
one has the Poincare model 
$\mathcal P$ for $\mathbb{H}^3$.  This stereographic 
projection is 
\begin{equation}\label{H3toPoincar} (x_1,x_2,x_3,x_0) \in \mathbb{H}^3 \to 
\left( \frac{x_1}{1+x_0}, \frac{x_2}{1+x_0}, \frac{x_3}{1+x_0}, 0 \right) \in 
\mathcal{P} \; . \end{equation}  
$\mathcal P$ is given the metric $g$ that makes this stereographic 
projection an isomety.  Since the fourth coordinate is identically zero 
in the Poincare model, we 
can simply remove it and view the Poincare model as the Euclidean unit ball 
\[ B^3 = \{(x_1,x_2,x_3) \in \mathbb{R}^3 \; | \; \; 
x_1^2+x_2^2+x_3^2 < 1\} \] in 
$\mathbb{R}^3$.  One can compute that the metric 
\begin{equation}\label{hypmet} 
g = \left( \frac{2}{1 - x_1^2-x_2^2-x_3^2} \right)^2 
(dx_1^2+dx_2^2+dx_3^2) \end{equation}  
is the one that will make the stereographic 
projection \eqref{H3toPoincar} an isometry.  By 
either Equation (1.1.10) or (1.1.14) in \cite{wisky}, the sectional 
curvature is constantly $-1$.  This metric $g$ in 
\eqref{hypmet} is written as a function times the 
Euclidean metric $dx_1^2+dx_2^2+dx_3^2$, and this means 
that the Poincare model's metric is conformal to the Euclidean metric.  From 
this it follows that angles between vectors in the tangent spaces are the 
same from the viewpoints of both the 
hyperbolic and Euclidean metrics, and this 
is why we prefer this model when showing 
graphics of surfaces in hyperbolic $3$-space.  
However, distances are clearly not Euclidean.  In fact, the boundary 
\[ \partial B^3 = \{(x_1,x_2,x_3) \in \mathbb{R}^3 \; | \; \; 
x_1^2+x_2^2+x_3^2 = 1 \} \] of the Poincare model is infinitely far from any 
point in $B^3$ with respect to the hyperbolic metric $g$ in 
\eqref{hypmet}.  For example, consider the curve 
\[ c(t) = (t,0,0) \; , \;\;\; t \in [0,1) \] in the Poincare model.  Its 
length is \[ \int_0^1 
\sqrt{g(c^\prime(t),c^\prime(t))} dt = \int_0^1 \tfrac{2 dt}{1-t^2} = + \infty 
\; . \]  Thus the point $(0,0,0)$ is infinitely far from the boundary point 
$(1,0,0)$ in the Poincare model.  For this reason, 
the boundary $\partial B^3$ is often called 
the {\em ideal boundary at infinity}.  

Unlike the Klein model, geodesics in the 
Poincare model are not Euclidean straight 
lines.  Instead they are segments of 
Euclidean lines and circles that intersect the 
ideal boundary $\partial B^3$ at right angles.  

Important examples of surfaces in 
$\mathbb{H}^3$ are described in \cite{wisky}, 
using the Poincare ball model: totally geodesic hypersurfaces (also called 
hyperbolic planes), hyperspheres, spheres and horospheres.  

\begin{figure}[phbt]
\begin{center}
\includegraphics[width=0.83\linewidth]{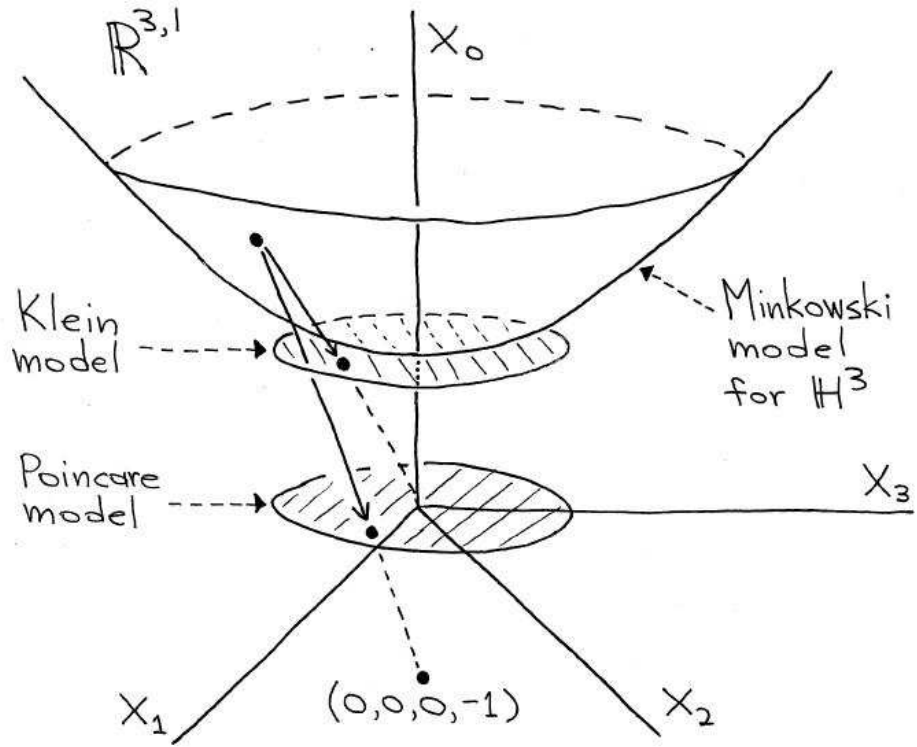}
\caption{The Klein, Poincare and Minkowski space 
models for $\mathbb{H}^3$.}
\end{center}
\end{figure}

\subsection{The upper-half-space model} 
One can obtain the upper-half-space model 
$\mathcal{U}$ for $\mathbb{H}^3$ 
from the Poincare model $\mathcal{P}$ 
by the M\"obius transformation of $\mathbb{R}^3$ 
which maps the unit ball $B^3$ (with the Poincare metric) centered at the 
origin to the upper half $\{ x_3 > 0 \}$ of $\mathbb{R}^3$ 
and maps the origin $(0,0,0)$ 
to $(0,0,1)$ and fixes $\partial B^3 \cap \{x_3 = 0 \}$.  
This map is 
\[ \mathcal{P} \ni (x_1,x_2,x_3) \to 
\frac{(2 x_1,2 x_2,1-x_1^2-x_2^2-x_3^2)}{x_1^2+x_2^2+(x_3-1)^2} \in 
\mathcal{U} \; . \] 
The metric induced on the upper-half-space by this 
transformation is \[ g = \frac{1}{x^2_{3}}(dx_1^2+dx_2^2+dx_3^2) \; , \]  
where we now view $(x_1,x_2,x_3)$ as coordinates 
of the model $\mathcal{U}$, i.e. $x_1,x_2 \in \mathbb{R}$ and $x_3 > 0$.  
Thus, like the Poincare model, the upper-half space model 
$\mathcal{U}$ is again conformal to 
Euclidean space.  And because M\"obius 
transformations preserve angles and also the 
set of circles and lines, again the 
geodesics are Euclidean lines and circles that intersect the 
ideal boundary at infinity $\{ x_3 = 0 \}$ at right angles.  
The isometries of the model $\mathcal{U}$ 
are generated by horizontal Euclidean
translations, Euclidean rotations about vertical axes,
Euclidean dilations about points in the plane 
$\{x_3 = 0\}$, and Euclidean inversions through Euclidean spheres (and
planes) intersecting the plane $\{x_3 = 0\}$ orthogonally.  

\subsection{The Hermitian matrix model} 
The Hermitian matrix model is a convenient model 
for applying the DPW method.  Unlike the other four 
models above, which can be used for hyperbolic spaces of any dimension, the 
Hermitian model can be used 
only when the hyperbolic space is $3$-dimensional.  

We first recall the following definitions: 
The group $\SL_2\!\mathbb{C}$ is all $2 \times 2$ matrices 
with complex 
entries and determinant $1$, with matrix multiplication as the group 
operation.  The vector space $\slg_2\!\mathbb{C}$ consists of all 
$2 \times 2$ complex matrices with trace $0$, with the vector space operations 
being matrix addition and scalar multiplication.  
(In Section \ref{liegroups} we will see that $\SL_2\!\mathbb{C}$ is a Lie 
group.  $\SL_2\!\mathbb{C}$ is $6$-dimensional.  Also, 
$\slg_2\!\mathbb{C}$ is the associated Lie 
algebra, thus is the tangent space of $\SL_2\!\mathbb{C}$ at the identity
matrix.  $\slg_2\!\mathbb{C}$ is also $6$-dimensional.)  
The group $\SU_2$ is the subgroup of matrices 
$F \in \SL_2\!\mathbb{C}$ such that $F \cdot F^*$ is the 
identity matrix, where 
$F^* = \bar{F}^t$.  Equivalently, 
\[ F = \left(
\begin{array}{cc}
p & -\bar{q} \\ q & \bar{p}
\end{array} \right) \; , \]
for some $p$, $q \in \mathbb{C}$ with $|p|^2 + |q|^2 = 1$.  
(We will see that $\SU_2$ is a $3$-dimensional Lie subgroup, 
in Section \ref{liegroups}.)  

Finally, we define 
Hermitian symmetric matrices as matrices of the form 
\[ \left( \begin{array}{cc}
a_{11} & a_{12} \\ \overline{a_{12}} & a_{22} 
\end{array} \right) \; , \] where $a_{12} \in \mathbb{C}$ and $a_{11}, 
a_{22} \in \mathbb{R}$.  Hermitian symmetric matrices with determinant $1$ 
have the additional condition that 
$a_{11} a_{22} - a_{12} \overline{a_{12}} = 1$.  

The Minkowski $4$-space 
$\mathbb{R}^{3,1}$ can be mapped to the space of $2 \times 2$ Hermitian 
symmetric matrices by 
\[ \psi : (x_1,x_2,x_3,x_0) \longrightarrow 
\left(
\begin{array}{cc}
x_0 + x_3 & x_1 + ix_2 \\
x_1 - ix_2 & x_0 - x_3
\end{array}
\right) \; . \]
For $\vec{x} \in \mathbb{R}^{3,1}$, the metric in the Hermitean matrix form 
is given by \[ \langle \vec{x} , 
\vec{x} \rangle_{\mathbb{R}^{3,1}} = -\mbox{det}(\psi(\vec{x})) \; . \]  
Thus $\psi$ maps the Minkowski model for $\mathbb{H}^3$ to the set of 
Hermitian symmetric 
matrices with determinant $1$.  Any Hermitian symmetric 
matrix with determinant $1$ can be written as the product $F F^*$ for some 
$F \in \SL_2\!\mathbb{C}$, and $F$ is determined 
uniquely up to right-multiplication 
by elements in $\SU_2$.  That is, for $F,\hat F \in \SL_2\!\mathbb{C}$, 
we have 
$F F^* = \hat F \hat F^*$ if and only if $F = \hat F \cdot B$ for some 
$B \in \SU_2$.  
Therefore we have the Hermitian model 
\[ {\mathcal H} = \{ F F^* \; | \; F \in \SL_2\!\mathbb{C} \} \; , \;\;\; 
   \;\; F^* := \bar{F}^t \; , \] for $\mathbb{H}^3$, 
and $\mathcal H$ is given the metric so that $\psi$ is an isometry from 
the Minkowski model of $\mathbb{H}^3$ to $\mathcal H$.  

It follows that, when we compare the Hermitean 
matrix and Poincare models $\mathcal{H}$ and 
$\mathcal{P}$ for $\mathbb{H}^3$, 
the mapping \[ \left( \begin{array}{cc}
a_{11} & a_{12} \\ \overline{a_{12}} & a_{22} 
\end{array} \right) \in {\mathcal H} \to 
\left( \frac{a_{12}+\overline{a_{12}}}{2+a_{11}+a_{22}}, 
\frac{i(\overline{a_{12}}-a_{12})}{2+a_{11}+a_{22}}, 
\frac{a_{11}-a_{22}}{2+a_{11}+a_{22}} \right) \in {\mathcal P} \]
is an isometry from $\mathcal H$ to $\mathcal P$.  

The Hermitian model is actually very convenient for describing 
the isometries of 
$\mathbb{H}^3$.  Up to scalar multiplication by $\pm 1$, the 
group $\SL_2\!\mathbb{C}$ represents the isometry group of 
$\mathbb{H}^3$ in the 
Hermitian model $\mathcal H$ in the following way: A matrix 
$h \in \SL_2\!\mathbb{C}$ 
acts isometrically on $\mathbb{H}^3$ in the model $\mathcal H$ by 
\[x \in \mathcal{H} \to h \cdot x := h \, x \, h^* \in \mathcal{H} \; , \]
where $h^* = \bar{h}^t$.  The kernel of this action is $\pm I$, 
hence $\PSL_2\!\mathbb{C} = \SL_2\!\mathbb{C}/\{\pm I\}$ is the 
isometry group of $\mathbb{H}^3$.  

\subsection{De-Sitter $3$-space $\mathbb{S}^{2,1}$}

Finally, we briefly consider another 
ambient space, which will be a Lorentzian manifold, because it also 
has a $2 \times 2$ Hermitian matrix model.  
Consider the $1$-sheeted hyperboloid in $\mathbb{R}^{3,1}$
\[ \mathbb{S}^{2,1} = \left\{ (x_1,x_2,x_3,x_0) \in 
\mathbb{R}^{3,1} \, \left| \, 
\sum_{j=1}^{3} x_j^2 - x_0^2 = 1 \right. \right\} \] with 
the metric $g$ induced on its tangent 
spaces by the restriction of the metric from the Minkowski 
space $\mathbb{R}^{3,1}$.  
This Lorentzian manifold $\mathbb{S}^{2,1}$ is called 
{\em de-Sitter $3$-space}.  

De-Sitter $3$-space $\mathbb{S}^{2,1}$ is homeomorphic to 
$\mathbb{S}^2 \times \mathbb{R}$, 
so it is simply-connected, since both $\mathbb{S}^2$ and $\mathbb{R}$ are 
individually simply-connected.  And this space, 
like hyperbolic space $\mathbb{H}^3$, can also be written with a 
$2 \times 2$ matrix model: 
\[ \mathbb{S}^{2,1} = \left\{ X \in M_{2 \times 2} (\mathbb{C}) 
\, | \, X^*=X,\langle X,X 
\rangle_{\mathbb{R}^{3,1}} = 1 
\right\} = \left\{ \left. F \begin{pmatrix} 1 & 0 \\ 0 & -1 
\end{pmatrix} F^* \, \right| \, F \in \SL_2\!\mathbb{C} \right\} \; , \] 
where $\langle X,X \rangle_{\mathbb{R}^{3,1}} = -\det X$.  
We note that $\mathbb{S}^{2,1}$ has constant sectional curvature $+1$.

\subsection{Lie groups and algebras}\label{liegroups}
We have already seen some Lie groups, 
that are amongst the most basic 
matrix groups, so here we briefly review some basic facts about 
Lie groups and algebras.  

\begin{defn}\label{Liegroup}
A set $G$ is a {\em Lie group} if 
\begin{enumerate}
\item $G$ is a differentiable manifold of class $C^\infty$, 
\item $G$ is a group with respect to some group operation, denoted by $\cdot$, 
\item for each fixed $g_0 \in G$ and each variable $g \in G$, the maps 
$L_{g_0} : g \to g_0 \cdot g$ (left multiplication) and 
$R_{g_0} : g \to g \cdot g_0$ (right multiplication) are $C^\infty$ differentiable.  
\end{enumerate}
\end{defn}

\begin{defn}\label{Lialgebra}
The {\em Lie algebra} $\frak{G}$ associated to a Lie group $G$ is 
the tangent space of (the manifold) $G$ at the identity 
element $e$ of (the group) 
$G$, i.e. $\frak{G} = T_e G$.  The Lie algebra $\frak{G}$ is then a vector 
space under addition and scalar multiplication of vectors in $T_e G$.  
Furthermore, there is a 
bracket operation $\frak{G} \times \frak{G} \to \frak{G}$ defined as 
follows: 
\[ [X,Y](f)=X(Y(f))-Y(X(f)) \; , \] where $X,Y$ are arbitrary elements of 
$\frak{G}$ with canonical left-invariant extensions to vector fields on $G$, 
and $f:G \to \mathbb{R}$ is any smooth map.  
\end{defn}

\begin{remark}
$X$ being a left-invariant vector field means that 
$X$ is given by transportation by the derivative map of 
left multiplication in $G$, i.e. 
\[ X_g = (L_g)_* X_e \; , \] where $L_g:G \to G$ denotes left 
multiplication by $g$, as in part (3) of Definition 
\ref{Liegroup}.  In the case that $G$ is a matrix group, 
then $(L_g)_* X$ becomes simply $(L_g)_* X = g X$, and the 
above equation can be written as $X_g = g X_e$.  
\end{remark}

\begin{remark}
In the definition of the Lie bracket above, $X(f)$ and $Y(f)$ must be 
defined at more than just one point $e$ (in particular, in 
a neighborhood of $e$) in order for $Y(X(f))$ and $X(Y(f))$ to 
be defined.  But because we take the canonical left-invariant 
extensions of $X$ and 
$Y$, in fact $[X,Y]$ is determined by $X|_e$ and $Y|_e$ alone.  
\end{remark}

\begin{remark}\label{remm:bracketsandmatrices}
When $G$ is a matrix group, $X$ and $Y$ in $\frak{G}$ can be 
identified with matrices, and it turns out that $[X,Y]$ can be 
identified with the difference of matrix products $X \cdot Y - 
Y \cdot X$.  We will give an example of this in 
Example \ref{exla:fourthexample}.  
\end{remark}

\begin{example}
The first example we consider is $\text{SO}_3$, defined as follows: 
\[ \text{SO}_3 = 
\{A \in M_{3 \times 3}(\mathbb{R}) \, | A \cdot A^t = I , \, \det A = 1 
\} \; . \]  
The group operation is then matrix multiplication.  
This represents the group of rotations of $\mathbb{R}^3$ that 
fix the origin of $\mathbb{R}^3$, 
and the group operation then represents composition of rotations.  
When considering a conformal immersion $f:\Sigma \to 
\mathbb{R}^3$ defined on a $2$-dimensional Riemann surface 
$\Sigma$ with local coordinate $z=u+i v$, we can consider the 
three vectors (two being tangent to $f$, and the third being the unit 
normal vector to $f$) 
\[ \tfrac{f_u}{||f_u||}|_{f(p)} \; , \;\; 
\tfrac{f_v}{||f_v||}|_{f(p)} \; , \;\; 
\vec{N}|_{f(p)} \] 
to be an orthonormal frame of $T_{f(p)} \mathbb{R}^3$.  We can 
use an element of $\text{SO}_3$ 
to describe this orthonormal frame by choosing the 
unique element of $\text{SO}_3$ that 
rotates $(1,0,0)$ and $(0,1,0)$ and $(0,0,1)$ to 
$\tfrac{f_u}{||f_u||}|_{f(p)}$ and 
$\tfrac{f_v}{||f_v||}|_{f(p)}$ and $\vec{N}|_{f(p)}$, respectively.  
We denote the Lie algebra of $\text{SO}_3$ by $so_3$.  
\end{example}

\begin{example}
The second example we consider is $\SL_2\!\mathbb{C}$, defined as follows: 
\[ \SL_2\!\mathbb{C} = \{A \in M_{2 \times 2}(\C) \, | 
\det A = 1 \} \; . \]  Again the 
group operation is matrix multiplication, 
and the group operation represents composition of linear maps of 
$\mathbb{C}^2$ to itself.  In fact, 
$\SL_2\!\mathbb{C}$ is a double cover of $\text{SO}_3$, as we saw in 
Sections 2.4 and 3.2 in \cite{wisky}.  
We denote the Lie algebra of $\SL_2\!\mathbb{C}$ by $\slg_2\!\mathbb{C}$.  
\end{example}

\begin{example}
Our third example is a subgroup of $\SL_2\!\mathbb{C}$: 
\[ \SU_2 = \{A \in \SL_2\!\mathbb{C} \, | A \cdot \bar{A}^t = I \} \]\[ 
= \left\{ \left. \begin{pmatrix} p & q \\ -\bar q & \bar p \end{pmatrix} \, 
\right| p,q \in \mathbb{C} , p \bar p+q \bar q = 1 \right\} \; . \]  
The corresponding Lie algebra is denoted $\su_2$, 
and we explicitly compute $\su_2$ here:  Consider a curve 
$c(t):(-\epsilon,\epsilon) \to \SU_2$ given by 
\[ c(t) = \begin{pmatrix} p(t) & q(t) \\ -\overline{q}(t) & 
\overline{p}(t) \end{pmatrix} \] with $c(0) = I$.  Then, with $\prime$ 
denoting the derivative with respect to $t$, 
\[ c^\prime(0) = \begin{pmatrix} p^\prime(0) & q^\prime(0) \\ 
-\overline{q^\prime}(0) & \overline{p^\prime}(0) \end{pmatrix} \] 
is an arbitrary element of 
$\su_2 = T_I \SU_2$.  In general, for any square matrix $\mathcal{A}$, we have 
$(\det \mathcal{A})^\prime=\text{trace} 
(\mathcal{A}^\prime \cdot \mathcal{A}^{-1}) \det \mathcal{A}$ 
(see Lemma \ref{matrix-identity-lemma} below), so if 
$\det \mathcal{A}$ is identically $1$, then the trace of 
$\mathcal{A}^\prime \cdot \mathcal{A}^{-1}$ is $0$.  This implies that 
$c^\prime(0)$ is trace-free.  Then, because $p(0)=1$ and $q(0)=0$, 
the derivative with respect to $t$ of $p(t) \overline{p}(t) + 
q(t) \overline{q}(t) = 1$ implies $p^\prime(0) \in i \mathbb{R}$.  We 
conclude that $\su_2$ is the $3$-dimensional vector space 
\[ \su_2 = \left\{ \left.\frac{-i}{2}\begin{pmatrix}
        -x_3 & x_1+i x_2 \\
        x_1-i x_2 & x_3
     \end{pmatrix} \; \right| \; x_1,x_2,x_3 \in \mathbb{R} \right\} \; , \] 
which is isomorphic as a vector space to $\mathbb{R}^3$, and so 
$\su_2$ is a matrix model for $\mathbb{R}^3$.  
\end{example}

\begin{example}\label{exla:fourthexample}
Our fourth example $\SL_2\!\mathbb{R} $ is also a subgroup of 
$\SL_2\!\mathbb{C}$: 
\[ \SL_2\!\mathbb{R} = \{A \in M_{2 \times 2}(\R) \, | 
\det A = 1 \} \; , \] with associated Lie algebra 
\[ \text{sl}_2\mathbb{R} = \{A \in M_{2 \times 2}(\R) \, | 
\text{tr} A = 0 \} \; . \]
We now explicitly describe the bracket operation on 
$\text{sl}_2\mathbb{R}$, 
in order to provide an example for the claim in Remark 
\ref{remm:bracketsandmatrices}.  To determine the bracket operation, 
take the three curves 
\[ c_1(t) = \begin{pmatrix}
1-t & 0 \\ 0 & (1-t)^{-1} 
\end{pmatrix} \; , \;\;\; 
c_2(t) = \begin{pmatrix}
1 & t \\ 0 & 1
\end{pmatrix} \; , \;\;\; 
c_3(t) = \begin{pmatrix}
1 & 0 \\ t & 1
\end{pmatrix} \] 
in $\SL_2\!\mathbb{R}$ through the identity matrix at $t=0$.  
To move these curves to other points of $\SL_2\!\mathbb{R}$, 
we use matrix multiplication on the left, i.e. 
\[ c_{j,a,b,d} = \begin{pmatrix}
a & b \\ d & (1+bd)a^{-1}
\end{pmatrix} \cdot c_j(t) \] for $j=1,2,3$ and $a,b,d \in \mathbb{R}$.  
(For our purposes we may assume $a \neq 0$.)  Now $a,b,d$ represent 
coordinates for a region of $\SL_2\!\mathbb{R}$ considered as a 
$3$-dimensional manifold.  
In fact, we could regard the coordinate chart $\phi$ to be defined by 
\[ \phi^{-1} \left( 
\begin{pmatrix}
a & b \\ d & (1+bd)a^{-1} 
\end{pmatrix}
 \right) = (a,b,d) \; , \] as a map from a region of $\mathbb{R}^3$ 
to a region of $\SL_2\!\mathbb{R}$.  Now, for a function 
\[ f: \SL_2\!\mathbb{R} \to \mathbb{R} \; , \] the composite maps 
\[ f \circ \phi (a(1-t),b(1-t)^{-1},d(1-t)) \; , \]  
\[ f \circ \phi (a,at+b,d) \; , \]  
\[ f \circ \phi (a+bt,b,(1+bd)a^{-1}t+d) \]  
equal, respectively, 
\[ f(c_{j,a,b,d}(t)) \] for $j = 1,2,3$.  
Then, by the chain rule, we have
\[ \left. \frac{d}{dt} f(c_{j,a,b,d}(t)) \right|_{t=0} = 
\vec{v}_{c_{j,a,b,d}} (f \circ \phi (a,b,d)) \; , \]  
for the three resulting left-invariant vector fields 
\[ \vec{v}_{c_{1,a,b,d}}= -a \partial_a+b\partial_b-d\partial_d \; , \] 
\[ \vec{v}_{c_{2,a,b,d}}= a \partial_b \; , \] 
\[ \vec{v}_{c_{3,a,b,d}}= b \partial_a+(1+bd)a^{-1} \partial_d \; . \] 
Thus 
\[ \vec{v}_{c_{1,a,b,d}} \circ \vec{v}_{c_{2,a,b,d}} - \vec{v}_{c_{2,a,b,d}} 
\circ \vec{v}_{c_{1,a,b,d}} = -2 \vec{v}_{c_{2,a,b,d}} \; , \] 
\[ \vec{v}_{c_{1,a,b,d}} \circ \vec{v}_{c_{3,a,b,d}} - \vec{v}_{c_{3,a,b,d}} 
\circ \vec{v}_{c_{1,a,b,d}} = 2 \vec{v}_{c_{3,a,b,d}} \; , \] 
\[ \vec{v}_{c_{2,a,b,d}} \circ \vec{v}_{c_{3,a,b,d}} - \vec{v}_{c_{3,a,b,d}} 
\circ \vec{v}_{c_{2,a,b,d}} = - \vec{v}_{c_{1,a,b,d}} \; . \] 
Correspondingly, 
\[ c^\prime_1(0) = \begin{pmatrix}
-1 & 0 \\ 0 & 1 
\end{pmatrix} \; , \;\;\; 
c^\prime_2(0) = \begin{pmatrix}
0 & 1 \\ 0 & 0 
\end{pmatrix} \; , \;\;\; 
c^\prime_3(0) = \begin{pmatrix}
0 & 0 \\ 1 & 0 
\end{pmatrix} \; , \]
and 
\[ c^\prime_1(0) c^\prime_2(0) - 
c^\prime_2(0) c^\prime_1(0) = -2 c^\prime_2(0) \; , \]  
\[ c^\prime_1(0) c^\prime_3(0) - c^\prime_3(0) 
c^\prime_1(0) = 2 c^\prime_3(0) \; , \]  
\[ c^\prime_2(0) c^\prime_3(0) - c^\prime_3(0) 
c^\prime_2(0) = - c^\prime_1(0) \; . \]  
Thus the behavior of the bracket on vector fields is exactly the same as the 
behavior of commutators of matrices in the Lie algebra.  This is why we can 
use matrix multiplication to define the Lie bracket in the case of 
$\text{sl}_2\mathbb{R}$, and this is true for matrix Lie groups in general.  
\end{example}

We now prove an equation we used in the third example above: 

\begin{lemma}\label{matrix-identity-lemma}
For any square matrix $\mathcal{A} \in M_{n \times n}(\mathbb{C})$ 
with $\det A \neq 0$ 
that depends smoothly on some parameter $t \in \mathbb{R}$, we have 
\begin{equation}\label{matrix-identity-lemma-eqn} 
\tfrac{d}{dt} (\det \mathcal{A})=\text{trace} 
((\tfrac{d}{dt} \mathcal{A}) \cdot 
\mathcal{A}^{-1}) \det \mathcal{A} \; . \end{equation}  
\end{lemma}

Lemma \ref{matrix-identity-lemma} is easily proven in the 
case that $n=2$.  It is also easily seen 
for general $n$ when $\mathcal{A}$ is upper triangular.  Furthermore, if 
$\mathcal{A}$ satisfies Equation \eqref{matrix-identity-lemma-eqn}, it 
is easily seen that the conjugation $P \cdot \mathcal{A} \cdot P^{-1}$ 
also satisfies Equation \eqref{matrix-identity-lemma-eqn}, 
for any $P \in M_{n \times n}(\mathbb{C})$ with $\det P \neq 0$ 
that depends smoothly on $t$.  Because any square matrix can be conjugated 
into an upper triangular matrix (the Jordan canonical form), this provides a 
proof of Lemma \ref{matrix-identity-lemma}.  

One could also prove Lemma \ref{matrix-identity-lemma} by direct computation: 
Write $\mathcal{A}=(a_{ij})_{i,j=1}^n$.  Then let 
$\mathcal{A}_{i,b_1,...,b_n} := 
\mathcal{A}|_{\{a_{i1} \to b_1,...,a_{in} \to b_n\}}$ be 
the matrix with entries as in $\mathcal{A}$, except that the 
$i$'th row has been 
replaced with the row vector $(b_1 \, ... \, b_n)$.  Then 
\begin{equation}\label{eqn:pf-ofLem-9.9}
\det (\mathcal{A}_{i,b_1,...,b_n}) = \sum_{j=1}^n 
\det (\mathcal{A}_{i,0,...,0,b_j,0,...,0}) = 
\sum_{j=1}^n b_j \tilde a_{ij} \; , 
\end{equation}
where $b_j$ is the value in the $ij$'th position of 
$\mathcal{A}_{i,0,...,0,b_j,0,...,0}$, and where 
\[ \tilde{a}_{ij} = \det (\mathcal{A}_{i,0,...,0,1,0,...,0}) \]
(again, $1$ is the value in the $ij$'th position of 
$\mathcal{A}_{i,0,...,0,1,0,...,0}$).  Then, for any 
$k \in \{ 1,...,n \}$, we have ($\delta_{ki}$ is the 
Kronecker delta function) 
\[ \sum_{j=1}^n a_{kj} \tilde{a}_{ij} = \det 
(\mathcal{A}_{i,a_{k1},...,a_{kn}}) = 
\delta_{ki} \cdot \det (\mathcal{A}) \; . \]  Hence, for 
\[ \tilde{\mathcal{A}} := (\tilde{a}_{ij})_{i,j=1}^n \; , \] 
we have 
\[ \mathcal{A} \cdot \tilde{\mathcal{A}}^t = 
\det(\mathcal{A}) \cdot I_{n \times n} \; . \] 
So if $\mathcal{A}$ is regular, i.e. $\det(\mathcal{A}) \neq 0$, then 
\[ \mathcal{A}^{-1} = \tfrac{1}{\det(\mathcal{A})} 
\tilde{\mathcal{A}}^t \; . \]  
Thus we have 
\[ \text{tr}( (\tfrac{d}{dt} \mathcal{A}) \mathcal{A}^{-1} ) 
\cdot \det(\mathcal{A}) = 
\text{tr}( (\tfrac{d}{dt} \mathcal{A}) \tfrac{1}{\det(\mathcal{A})} 
\tilde{\mathcal{A}}^t ) 
\cdot \det(\mathcal{A}) = \]\[ \text{tr}((\tfrac{d}{dt} 
\mathcal{A}) \tilde{\mathcal{A}}^t) = 
\sum_{i=1}^n \sum_{j=1}^n \frac{da_{ij}}{dt} \tilde a_{ij} = \]\[ 
\sum_{i=1}^n 
\det(\mathcal{A}_{i,a^\prime_{i1},...,a^\prime_{in}}) 
= \frac{d}{dt} (\det (\mathcal{A})) \; , \;\;\; 
a_{ij}^\prime := \frac{da_{ij}}{dt} \; , 
\] where the second to the 
last equality above follows from Equation \eqref{eqn:pf-ofLem-9.9}, 
proving Lemma \ref{matrix-identity-lemma}.  

A third proof of this lemma can be given by using the following 
fact: If $A$ and $X$ are $n \times n$ matrices and $\epsilon$ is 
a real number close to zero, then 
\begin{equation}\label{eqn:Aexpand} 
A=I+\epsilon X + \mathcal{O}(\epsilon^2) \;\; \text{implies} \;\; 
\det(A) = 1 + \epsilon \cdot 
\text{tr} X + \mathcal{O}(\epsilon^2) \; . \end{equation} 
The argument is as follows: 
Write the Taylor expansion of $\mathcal{A}(s)$ at the value 
$s=t$ as 
\[ \mathcal{A}(s) = \mathcal{A}(t) + (s-t) \mathcal{A}^\prime(t) 
+ \mathcal{O}((s-t)^2) \; . \]  
Then 
\[ \mathcal{A}(s) (\mathcal{A}(t))^{-1} = 
I + (s-t) \mathcal{A}^\prime(t) (\mathcal{A}(t))^{-1} 
+ \mathcal{O}((s-t)^2) \; , \] and \eqref{eqn:Aexpand} implies 
\[ \det ( \mathcal{A}(s) (\mathcal{A}(t))^{-1} ) = 
1 + (s-t) \cdot \text{tr}(\mathcal{A}^\prime(t) (\mathcal{A}(t))^{-1})
+ \mathcal{O}((s-t)^2) \; . \]  Taking the derivative of this 
with respect to $s$ and then evaluating at $s=t$, we have 
\[ \frac{(\det (\mathcal{A}(s)))^\prime|_{s=t}}{\det (\mathcal{A}(t))} 
= \left. \left( \text{tr}(\mathcal{A}^\prime(t) (\mathcal{A}(t))^{-1}) 
+ \mathcal{O}(s-t) \right) \right|_{s=t} \; , \] so 
\[ \frac{(\det (\mathcal{A}(t)))^\prime}{\det(\mathcal{A}(t))} 
= \text{trace}(\mathcal{A}^\prime(t) (\mathcal{A}(t))^{-1}) 
\; , \] proving Lemma \ref{matrix-identity-lemma}.  

\section{Riemann surfaces and Hopf's theorem}\label{Riemsurfs}

\subsection{Riemann surfaces}

When the dimension of a differentiable manifold $M$ is two, then we have some 
special properties.  This is because the coordinate charts are maps 
from $\mathbb{R}^2$, 
and $\mathbb{R}^2$ can be thought of as the complex plane $\mathbb{C} \approx 
\mathbb{R}^2$.  Thus we can consider the notion of holomorphic functions on 
$M$.  This leads to the idea of 
Riemann surfaces and the beautiful theory associated with them.  
Part of the beauty 
of this theory is that Riemann surfaces can be described in a variety of 
different ways, but this is outside the scope of this text, and 
for our purposes it suffices to consider just two descriptions of 
Riemann surfaces.  

To distinguish $2$-dimensional manifolds from other manifolds, 
we will often denote 
them by $\Sigma$ instead of $M$.  

Suppose $\Sigma$ 
is a differentiable manifold of dimension $2$ with differentiable structure 
defined by a family 
\[ \{ \, (U_\alpha \, , \, \phi_\alpha : U_\alpha \to \Sigma) \, \} \] 
of coordinate charts.  Let $(u_\alpha,
v_\alpha)$ be the coordinates of $U_\alpha \subseteq \mathbb{R}^2$.  If $W := 
\phi_\alpha(U_\alpha) \cup 
\phi_\beta(U_\beta) \neq \emptyset$, then $u_\beta,v_\beta$ can be viewed as 
functions of 
the variables $u_\alpha,v_\alpha$ on $\phi_\alpha^{-1}(W)$ via the 
transition function 
$f_{\beta \alpha} = \phi_\beta^{-1} \circ \phi_\alpha: 
\phi_\alpha^{-1}(W) \to \phi_\beta^{-1}(W)$.  Associating $U_\alpha \subseteq 
\mathbb{R}^2$ with the corresponding region of $\mathbb{C}$ by defining the 
complex coordinate 
\[ z_\alpha = u_\alpha + i v_\alpha \] for each coordinate chart 
$(U_\alpha,\phi_\alpha)$, 
we can view $z_\beta$ as a function of $z_\alpha$ on $\phi_\alpha^{-1}(W)$.  
When $z_\beta$ is a holomorphic function of $z_\alpha$, we say that the 
transition function $f_{\beta\alpha}$ is holomorphic.  

\begin{defn}\label{defn:Riemannsurfaces}
A differentiable manifold $\Sigma$ of dimension $2$ with 
differentiable structure 
defined by a family $\{(U_\alpha,\phi_\alpha)\}$ of coordinate charts is a 
{\em Riemann surface} if the transition functions $f_{\beta\alpha}$ are 
all holomorphic.  We then say that the family 
$\{(U_\alpha,\phi_\alpha)\}$ forms a 
{\em complex structure} on $\Sigma$.  
\end{defn}

The simplest example of a Riemann surface is 
$\mathbb{C}$ itself.  In this case, 
we can choose a single coordinate 
$(U_\alpha,\phi_\alpha)$ to give the differential 
structure, where $U_\alpha = \mathbb{R}^2$ and 
$\phi_\alpha$ is the identity map.  
Then it is vacuously true that the transition functions are holomorphic.  

Another example is the unit sphere 
$\mathbb{S}^2$ (in $\mathbb{R}^3$).  The differential structure 
can be defined by a pair of stereographic projections, so we 
can use two coordinate neighborhoods $(U_\alpha,\phi_\alpha)$ and 
$(U_\beta,\phi_\beta)$ with $U_\alpha=U_\beta=\mathbb{R}^2$, and 
with $\phi_\alpha$ equal to the inverse of stereographic 
projection from the north pole $(0,0,1)$, and with $\phi_\beta$ 
equal to the inverse of stereographic projection from 
the south pole $(0,0,-1)$ composed with a reflection of 
$\mathbb{S}^2$ across a plane fixing both the north and south poles.  
Then the map $\phi_\beta^{-1} \circ \phi_\alpha$ is 
holomorphic, so $\mathbb{S}^2$ is a 
Riemann surface.  

One property of Riemann surfaces is that they are 
always orientable.  Before proving 
this, we first recall the definition of orientability.  
Given two differentiable 
functions $f,g$ from a 
$2$-dimensional differentiable manifold $\Sigma$ to 
$\mathbb{R}$, we define the 
wedge product of their differentials as follows: For 
a point $p \in \Sigma$ and 
$\vec{v},\vec{w} \in T_p \Sigma$, \[ df_p \wedge dg_p (\vec{v},\vec{w}) = 
\frac{1}{2} (df_p (\vec{v}) dg_p  
(\vec{w})-df_p (\vec{w}) dg_p  (\vec{v})) \; . \]  
(Note that the wedge product defined here is {\em not} 
the same as the symmetric 
product defined in Section 1.1 of \cite{wisky}.)  
Then, for coordinate neighborhoods $(U_\alpha,\phi_\alpha)$ and 
$(U_\beta,\phi_\beta)$ such that $W := 
\phi_\alpha(U_\alpha) \cup \phi_\beta(U_\beta) \neq \emptyset$, and naming the 
coordinates $(u_\alpha,v_\alpha)$ and $(u_\beta,v_\beta)$ on 
$\phi_\alpha^{-1}(W)$ and 
$\phi_\beta^{-1}(W)$, respectively, we say that $(U_\alpha,\phi_\alpha)$ and 
$(U_\beta,\phi_\beta)$ are oriented in the same way if 
\[ du_\alpha \wedge dv_\alpha = h_{\alpha\beta} du_\beta \wedge dv_\beta \] 
for some positive function $h_{\alpha\beta}:\phi_\alpha^{-1}(W) 
\to \mathbb{R}^+$.  

If the coordinate charts $\{(U_\alpha,\phi_\alpha)\}$ that comprise the 
differential structure of $\Sigma$ 
can be chosen so that they are all oriented the same way 
wherever they intersect, 
we say that the manifold $\Sigma$ 
is orientable, and the family $\{(U_\alpha,\phi_\alpha)\}$ 
is said to be oriented.  

\begin{lemma}
Any Riemann surface is orientable.  
\end{lemma}

\begin{proof}
Let $(U_\alpha,\phi_\alpha)$ and $(U_\beta,\phi_\beta)$ be 
two coordinate charts of a 
Riemann surface $\Sigma$ 
such that $W:=\phi_\alpha(U_\alpha) \cap \phi_\beta(U_\beta) 
\neq \emptyset$.  Let 
$(u_\alpha,v_\alpha)$ and $(u_\beta,v_\beta)$ be the 
coordinates of $\phi_\alpha^{-1}(W) 
\subseteq \mathbb{R}^2$ and $\phi_\beta^{-1}(W) 
\subseteq \mathbb{R}^2$, respectively.  
Noting that the differentials of $z_\alpha$, $\bar z_\alpha$, $z_\beta$ and 
$\bar z_\beta$ satisfy 
\[ dz_\alpha = du_\alpha + i dv_\alpha \; , \;\;\; d\bar z_\alpha = 
du_\alpha - i dv_\alpha 
\; , \]\[ dz_\beta = du_\beta + i dv_\beta \; , \;\;\; d\bar z_\beta = 
du_\beta - i dv_\beta 
\; , \] and also that, because $z_\beta$ is a 
holomorphic function of $z_\alpha$ on 
$\phi_\alpha^{-1}(W)$, the chain rule implies 
\[ dz_\alpha = \frac{dz_\alpha}{dz_\beta} dz_\beta \; , \]  
we have 
\[ du_\alpha \wedge dv_\alpha = \tfrac{i}{2} dz_\alpha \wedge d\bar z_\alpha = 
\tfrac{i}{2} \left| \tfrac{dz_\alpha}{dz_\beta} \right|^2 dz_\beta \wedge 
d\bar z_\beta = \left| \tfrac{dz_\alpha}{dz_\beta} \right|^2 du_\beta 
\wedge dv_\beta 
\; . \]  Since $\left| \tfrac{dz_\alpha}{dz_\beta} \right|^2 > 0$ 
for all $\alpha$ and 
$\beta$, we conclude that $\Sigma$ is an orientable manifold.  
\end{proof}

\begin{remark}
We saw in Remark 1.3.6 of \cite{wisky} that nonminimal CMC surfaces in 
an oriented ambient space 
are always orientable.  So when using Riemann surfaces as 
the domains for nonminimal 
CMC immersions, the fact that the Riemann surfaces are 
orientable is not in any way a restriction on the types of 
CMC immersions we can consider.  
\end{remark}

Riemann surfaces are in a one-to-one 
correspondence with conformal equivalence classes 
of orientable $2$-dimensional Riemannian 
manifolds, giving us a second way to describe 
Riemann surfaces.  In order to explain this we start with a definition.  

\begin{defn}\label{defn:conformality4.4}
Let $\Sigma$ be a $2$-dimensional orientable Riemannian manifold with 
differentiable structure determined by a family $\{ (U_\alpha,\phi_\alpha) \}$ 
of coordinate charts and with positive 
definite metric $g$.  For any coordinate chart 
$(U_\alpha,\phi_\alpha)$ with 
coordinates $(u_\alpha,v_\alpha)$ on $U_\alpha$, suppose that 
the metric $g$ can be written as 
\[ g = \begin{pmatrix} f_\alpha & 0 \\ 0 & f_\alpha \end{pmatrix} \] 
in matrix form 
for some positive function $f_\alpha: U_\alpha \to \mathbb{R}^+$, 
or equivalently, as a 
symmetric $2$-form \[ g= f_\alpha (du_\alpha^2 + dv_\alpha^2) \; . \]  
Then we say that 
$g$ is a {\em conformal} metric and 
the $(U_\alpha,\phi_\alpha)$ are {\em conformal} 
coordinate charts.  
\end{defn}

Generally, for a metric 
\[ g = g_{11} du_\alpha^2 + g_{12} du_\alpha dv_\alpha + 
g_{21} dv_\alpha du_\alpha + 
g_{22} dv_\alpha^2 \] 
written as a symmetric $2$-form using the $1$-forms $du_\alpha$ and 
$dv_\alpha$ (note that $g_{12}=g_{21}$ because the metric is symmetric and 
$g_{11},g_{22} > 0$ because 
the metric is positive definite), we can rewrite the metric 
using the complex $1$-forms $dz_\alpha$ and $d\bar z_\alpha$ instead: 
\begin{equation}\label{precursorcomplexifiedmetric} 
g = A dz_\alpha^2 + 2 B dz_\alpha d\bar z_\alpha + 
\bar A d\bar z_\alpha^2 \; , 
\end{equation} 
\[ A = \frac{g_{11}-g_{22}- 2 i g_{12}}{4} \; , \;\;\; 
B = \frac{g_{11}+g_{22}}{4} \; . 
\] 
If the metric $g$ is conformal, then 
$g_{12}=g_{21}=0$ and $f_\alpha=g_{11}=g_{22}$, 
so the metric becomes 
\[ g = f_\alpha dz_\alpha d\bar z_\alpha \] 
with respect to the complex coordinate 
$z_\alpha = u_\alpha + i v_\alpha$.  Since $f_\alpha$ is a positive function, 
we could also write this as 
\begin{equation}\label{complexconformalformforthemetric} 
g = 4 e^{2 \hat u_\alpha} dz_\alpha d\bar z_\alpha \end{equation} for some 
real-valued function $\hat u_\alpha$ defined on $U_\alpha$, as noted in 
Remark 1.3.1 of \cite{wisky}.  

\begin{theorem}\label{conformalityispossible}
Let $\Sigma$ be a $2$-dimensional orientable manifold with an 
oriented family $\{(U_\alpha,\phi_\alpha)\}$ of 
coordinate charts that determines the 
differentiable structure and with a positive definite 
metric $g$.  Assume further that the transition functions of 
$\{(U_\alpha,\phi_\alpha)\}$ are real-analytic.  
Then there exists another family of coordinate charts 
$\{(V_\beta,\psi_\beta)\}$ 
that determines the same differentiable structure 
and with respect to which the 
metric $g$ is conformal.  Additionally, 
$\{(V_\beta,\psi_\beta)\}$ is oriented and gives a complex 
structure on $\Sigma$, so $\Sigma$ becomes a Riemann surface.  
\end{theorem}

\begin{remark}
The condition in Theorem \ref{conformalityispossible} that 
the transition functions be real-analytic can be weakened, 
but we include this condition to simplify the proof and 
because it is satisfied in all of the applications of this 
theorem later in this text.  
\end{remark}

\begin{proof}
We are given coordinate charts 
$(U_\alpha,\phi_\alpha)$ with complex coordinates 
$z_\alpha=u_\alpha+i v_\alpha$ on the $U_\alpha$.  
We must show that there exists a 
family $\{(V_\beta,\psi_\beta)\}$ of coordinates 
with the given differentiable 
structure so that the metric can be written as in 
Equation \eqref{complexconformalformforthemetric} with respect to the complex 
coordinates $w_\beta = x_\beta + i y_\beta$ of the $V_\beta$.  

The metric $g$ can be written 
as in Equation \eqref{precursorcomplexifiedmetric} 
with respect to the $(U_\alpha,\phi_\alpha)$ coordinate charts, 
and if $A=0$ then $g$ is already conformal and we are finished by 
taking $(U_\alpha,\phi_\alpha)$ and $(V_\beta,\psi_\beta)$ to be equal.  
So without loss of generality 
we can assume $A \neq 0$.  Then we can write $g$ as 
\[ g = s 
(dz_\alpha+\mu d\bar z_\alpha) (d\bar z_\alpha+\bar \mu dz_\alpha) \; , 
\;\;\; s = \frac{2 B}{1+|\mu|^2} > 0 \; , \] where $\mu$ satisfies 
\[ |\mu| = \frac{B-\sqrt{B^2-|A|^2}}{|A|} < 1 \; . \]  
We need to find new coordinates $(x_\beta,y_\beta)$ for $V_\beta$ so that 
$w_\beta = x_\beta + i y_\beta$ satisfies 
\[ dw_\beta = \lambda (dz_\alpha+\mu d\bar z_\alpha) \] for some 
nonzero function $\lambda$.  
Then $g$ is written as $g=s |\lambda|^2 dw_\beta d\bar w_\beta$ and 
we will have that $g$ is a conformal 
metric with respect to the new coordinates $w_\beta$.  

The equation $dw_\beta = 
\lambda (dz_\alpha+\mu d\bar z_\alpha)$ is satisfied by a 
solution $w_\beta$ to the equation 
\[ \frac{\partial w_\beta}{\partial \bar z_\alpha} = \mu 
\frac{\partial w_\beta}{\partial z_\alpha} \; , \]  
and then we can take \[ \lambda = \frac{dw_\beta}{dz_\alpha} \; . \]  
This is the Beltrami equation, and $\mu$ is called 
the Beltrami coefficient.  The fact that the transition 
functions are real-analytic implies there 
exist solutions to this Beltrami equation.  
This can be proven using the 
Cauchy-Kowalewski theorem, but let us trust that such solutions 
exist, and then continue with the proof.  (Such solutions 
exist in more general settings as well, but we do not 
explore that here).  

We conclude that we have a family of coordinate 
charts so that $g$ is conformal, and 
it only remains to show that this new family 
$\{(V_\beta,\psi_\beta)\}$ is oriented on $\Sigma$ and 
determines a complex structure on $\Sigma$.  
This new family is oriented because 
the original family $\{(U_\alpha,\phi_\alpha)\}$ was 
oriented and 
\[ dx_\beta \wedge dy_\beta = \left| 
\frac{\partial w_\beta}{\partial z_\alpha} \right|^2 
(1-|\mu|^2) du_\alpha \wedge dv_\alpha \; , \] 
with $\left| \tfrac{\partial w_\beta}{\partial z_\alpha} \right|^2 
(1-|\mu|^2) > 0$.  

To see that this new family determines a 
complex structure on $\Sigma$, we need to see 
that $w_\beta$ is a holomorphic function of $w_\gamma$ wherever 
$W := \psi_\beta(V_\beta) 
\cap \psi_\gamma(V_\gamma) \neq \emptyset$.  Both coordinates $w_\beta$ and 
$w_\gamma$ are conformal, so 
\begin{equation}\label{eqninReimsurfsection} 
g = 4 e^{2 \hat u_\beta} dw_\beta d\bar w_\beta = 4 e^{2 \hat u_\gamma} 
dw_\gamma d\bar w_\gamma \end{equation} on $W$.  Because of the chain rule 
\[ dw_\beta = \frac{\partial w_\beta}{\partial w_\gamma} dw_\gamma + 
\frac{\partial w_\beta}{\partial \bar w_\gamma} d\bar w_\gamma \; , \;\;\; 
d\bar w_\beta = \frac{\partial \bar w_\beta}{\partial w_\gamma} dw_\gamma + 
\frac{\partial \bar w_\beta}{\partial \bar w_\gamma} d\bar w_\gamma \; , \]
the right-most equality in Equation 
\eqref{eqninReimsurfsection} can hold only if either 
\[ \frac{\partial w_\beta}{\partial \bar w_\gamma} = 0 \;\; \text{ or } \;\; 
\frac{\partial w_\beta}{\partial w_\gamma} = 0 \; . \]  
Since the change of coordinates is 
orientation-preserving, we conclude that the first 
of the two equations holds, and so 
$w_\beta$ is a holomorphic function of $w_\gamma$.  
\end{proof}

\begin{defn}
Let $\Sigma$ be a $2$-dimensional 
orientable differentiable manifold with a given 
differentiable structure.  Suppose that $\Sigma$ 
becomes a Riemannian manifold with respect 
to some metric $g$ and also with respect to 
some other metric $\tilde g$.  
If $g = f \tilde g$ for some positive 
function $f: \Sigma \to \mathbb{R}^+$, we say that 
the two metrics $g$ and $\tilde g$ are {\em conformally equivalent}.  
\end{defn}

Note that if the metric $g$ is a 
conformal metric, then $g$ is conformally equivalent 
to the flat metric $du_\alpha^2+dv_\alpha^2$ on each coordinate chart 
$(U_\alpha,\phi_\alpha)$.  

Conformal equivalence of the 
metrics is clearly an equivalence relation, so we can talk 
about conformal classes of metrics, as in the next corollary.  

\begin{corollary}
Conformal equivalence classes of metrics on an orientable $2$-dimensional 
manifold $\Sigma$ 
are in one-to-one correspondence with the complex structures on $\Sigma$.  
\end{corollary}

\begin{proof}
As we saw in the proof of Theorem \ref{conformalityispossible}, 
each positive definite metric on $\Sigma$ 
produces a complex structure on $\Sigma$.  
Following the arguments in that proof, we can also see that 
two conformally equivalent 
metrics will produce the same complex structure, and the 
corollary follows.  
\end{proof}

In this text, we will always be considering smooth 
CMC surfaces as real-analytic immersions of 
$2$-dimensional differentiable (real-analytic) manifolds $\Sigma$.  
Each immersion will determine 
an induced metric $g$ on $\Sigma$ that makes it a 
Riemannian manifold.  Theorem \ref{conformalityispossible} tells us that 
we can choose coordinates 
on $\Sigma$ so that $g$ is conformal.  
Thus without loss of generality we can restrict 
ourselves to those immersions that 
have conformal induced metric, and we will do this 
on every occasion possible.  

\subsection{The Hopf differential and Hopf theorem}
\label{hopf-diff-hopf-thm}

The Hopf differential $Qdz^2$, defined in \cite{wisky}, 
is of central importance.  We have already seen in \cite{wisky} 
that the Hopf differential can be used to decide if a conformal 
immersion parametrized by a complex coordinate $z$ has 
constant mean curvature, because the surface will have 
constant mean curvature if and only if $Q$ is holomorphic.  
The Hopf differential 
can also be used to determine the umbilic points of a surface, as 
we will now see: 

Let us assume that $\Sigma$ is a Riemann surface with a 
coordinate $z = u + i v$ and that $f$ is a conformal immersion 
from $\Sigma$ into 
$\mathbb{R}^3$.  (Theorem \ref{conformalityispossible} has 
told us that we can always assume $\Sigma$ is a Riemann surface 
and the immersion $f$ is conformal.)  Then the 
first and second fundamental forms are 
\begin{equation}\label{eqn:ratsing7} 
g = \begin{pmatrix}
g_{11} & g_{12} \\ g_{21} & g_{22} 
\end{pmatrix}
= 
\begin{pmatrix}
\langle f_u,f_u \rangle & \langle f_u,f_v \rangle \\ 
\langle f_v,f_u \rangle & \langle f_v,f_v \rangle 
\end{pmatrix}
= 4 e^{2 \hat u} 
\begin{pmatrix}
1 & 0 \\ 0 & 1 
\end{pmatrix} \end{equation} 
and 
\[ b = \begin{pmatrix}
b_{11} & b_{12} \\ b_{21} & b_{22} 
\end{pmatrix}
= 
\begin{pmatrix}
\langle b_{uu},N \rangle & \langle b_{uv},N \rangle \\ 
\langle b_{vu},N \rangle & \langle b_{vv},N \rangle 
\end{pmatrix}
\; , \]  
where $N$ is a unit normal vector to $f$.  
The Hopf differential function is 
\[ Q = \frac{1}{4} (b_{11}-b_{22}-ib_{12}-ib_{21}) 
= \langle f_{zz},N \rangle \; , 
\] where $\langle \cdot , \cdot \rangle$ is the 
complex bilinear extension of the metric of $\mathbb{R}^3$, 
and 
\[ \partial_z = \tfrac{1}{2} (\partial_u - i \partial_v)
\; , \;\;\; 
\partial_{\bar z} = \tfrac{1}{2} (\partial_u + i \partial_v) \] 
by definition.  
Then \[ b = Qdz^2+\tfrac{1}{2} (b_{11}+b_{22}) + 
\bar Q d\bar z^2 \; . \]  
Now, the shape operator is 
\[ g^{-1} b = \frac{1}{4 e^{2 \hat u}} \begin{pmatrix} 
\tfrac{1}{2} (b_{11}+b_{22}) +Q+\bar Q & i(Q-\bar Q) \\ 
i(Q-\bar Q) & \tfrac{1}{2} (b_{11}+b_{22}) -Q-\bar Q 
\end{pmatrix} \] with respect to the basis 
$f_u$ and $f_v$ of each tangent space of 
$f(\Sigma)$.  The two principal curvatures are 
then the two eigenvalues of this shape 
operator $g^{-1}b$, which can be computed and seen to be 
\[ \tfrac{1}{2} (b_{11}+b_{22})  + 2 |Q| \; , \;\;\; 
    \tfrac{1}{2} (b_{11}+b_{22})  - 2 |Q| \; . \]  

\begin{defn}
Let $\Sigma$ be a $2$-dimensional manifold.  
The {\em umbilic points} of an immersion 
$f : \Sigma \to \mathbb{R}^3$ are the points 
where the two principal curvatures are equal.  
\end{defn}

So, for example, every point of a flat 
plane or a round sphere is an umbilic point, 
and a cylinder has no umbilic points.  
One can check that a catenoid also has no umbilic points.  

Putting all this together, we have the following lemma: 

\begin{lemma}\label{umbIsQis0}
If $\Sigma$ is a Riemann surface and 
$f : \Sigma \to \mathbb{R}^3$ is a conformal 
immersion, then $p \in \Sigma$ is an umbilic point if and only if 
$Q=0$ at $p$.  
\end{lemma}

Thus the Hopf differential tells us where the umbilic points are.  
When $Q$ is holomorphic, it follows that $Q$ is 
either identically zero or is zero 
only at isolated points.  So, in the case of a 
CMC surface, if there are any points that are not umbilics, then 
all the umbilic points must be isolated.  

If every point is an umbilic, we say that 
the surface is {\em totally umbilic}, and then the 
surface must be a plane or a round sphere.  This is 
proven in \cite{doCarmo1}, for example.  
But let us include a proof here: 

\begin{lemma}\label{lemmaonumbilics}
Let $\Sigma$ be a Riemann surface and $f : \Sigma \to \mathbb{R}^3$ a 
totally umbilic conformal immersion.  Then $f(\Sigma)$ is part of a plane or 
sphere.  
\end{lemma}

\begin{proof}
Because $f$ is totally umbilic, 
the Hopf differential $Q$ is identically zero.  
So $Q$ is clearly holomorphic, and thus $H$ is constant, by the Codazzi 
equation (see Section 1.3 in \cite{wisky}).  
Let $u,v \in \R$ be local conformal coordinates for $f$, and 
$N=N(u,v)$ the unit normal of $f$.  
We first consider the case that $H$ is not zero, and show that 
\begin{equation}\label{eqn:ratsing8} 
\partial_u (f+H^{-1} N) = \partial_v (f+H^{-1} N) = 0 \; . 
\end{equation} 
This can be computed as follows, with $\hat u$ as defined in  
\eqref{eqn:ratsing7}: 
\[ \langle f_u+H^{-1} N_u , f_u \rangle = 4 e^{2\hat u}-H^{-1} 
\langle N , f_{uu} \rangle = 4 e^{2\hat u}-H^{-1} b_{11} = \]\[ = 
4 e^{2\hat u}-H^{-1} 
(\tfrac{1}{2} (b_{11}+b_{22}) + Q + \bar Q) = \]\[ = 4 e^{2\hat u}-
\tfrac{1}{2} H^{-1} (b_{11}+b_{22}) = 
4 e^{2\hat u}-4 e^{2\hat u} = 0 \; . \]  Similarly, 
\[ \langle f_u+H^{-1} N_u , f_v \rangle=0 \; , \;\;\; 
\langle f_v+H^{-1} N_v , f_u \rangle=0 \; , \;\;\; 
\langle f_v+H^{-1} N_v , f_v \rangle=0 \; , \]\[ 
\langle f_u+H^{-1} N_u , N \rangle=0 \; , \;\;\; 
\langle f_v+H^{-1} N_v , N \rangle=0 \; . \]  
($\langle f_u, N_v \rangle=\langle f_v, N_u 
\rangle=0$ because $g^{-1} b$ is diagonal on a conformally 
parametrized totally umbilic surface.)  It follows that 
\eqref{eqn:ratsing8} holds, and so 
$f(\Sigma)$ is part of a round sphere of radius $H^{-1}$ with 
constant center point $f+H^{-1} N$.  

In the case that $H=0$, 
to show that $f(\Sigma)$ is part of a plane, we need only 
show that $N_u=N_v=0$.  Similarly to the previous case 
where $H$ was not zero, one can compute that 
\[ \langle N , N_u \rangle = \langle N , N_v \rangle = 
\langle f_u , N_u \rangle = 
\langle f_u , N_v \rangle = \langle f_v , N_u \rangle = 
\langle f_v , N_v \rangle = 0 \; , \]  and the result follows.  
\end{proof}

\begin{remark}
We stated Lemma \ref{lemmaonumbilics} with the assumption 
that the immersion is conformal, but in fact the conformality condition 
is not required.  
\end{remark}

In the case that $\Sigma$ is a closed Riemann surface (i.e. compact 
without boundary), we can take this even further.  Orientable closed 
Riemann surfaces are classified 
by their genus.  For example, if $\Sigma$ is a 
sphere, then it has genus $0$; if it is a torus, then it has genus $1$.  
So if $\Sigma$ is a closed orientable 
Riemann surface, then it has a genus $\frak{g}$ for some $\frak{g} \in 
\mathbb{Z}^+ \cup \{ 0 \}$.  Since $f$ is a CMC immersion, 
the Hopf differential $Qdz^2$ (written here in terms of local coordinates 
$z$) is a holomorphic $2$-differential defined on 
$\Sigma$.  The order $\ord_p(Qdz^2)$ of 
$Qdz^2$ at each point $p \in 
\Sigma$ is defined to be the order of the function $Q$ at 
$p$ (i.e. if $Q=z^k$, then $Q$ has order $k$ at $z=0$).  
It is then well known, when $Q$ is not identically zero (see \cite{fark-kra}, 
for example), that 
\begin{equation}\label{eqn:ratsing9} 
\sum_{p \in \Sigma} \ord_p(Qdz^2) = 4 \frak{g} - 4 \; . \end{equation} 
Because $Qdz^2$ is holomorphic, we have $\ord_p(Qdz^2) \geq 0$ for all 
$p \in \Sigma$.  We conclude that if $\frak{g} = 0$, then 
either $Q$ is identically 
zero or $0 \leq \sum_{p \in \Sigma} \ord_p(Qdz^2) = - 4$.  
The second case certainly 
cannot hold, so $Q$ is identically zero.  So the surface is 
totally umbilic and must be a round sphere, 
and this proves Hopf's theorem \cite{Hopf}: 

\begin{theorem}\label{thmofHopf} {\bf (The Hopf theorem.)}  
If $\Sigma$ is a closed $2$-dimensional manifold of genus zero and if 
$f : \Sigma \to \mathbb{R}^3$ 
is a nonminimal CMC immersion, then $f(\Sigma)$ is a round sphere.  
\end{theorem}

\begin{remark}\label{no-minimal-compact-surfaces}
In fact, there do not exist any compact minimal surfaces without 
boundary in $\mathbb{R}^3$, 
and we will prove this using the maximum principle, in 
the next chapter.  
Therefore, without assuming that $f$ in the above theorem 
in nonminimal, the result 
would still be true.  
\end{remark}

Now let us consider the case that $\Sigma$ is a closed 
Riemann surface of genus 
$\frak{g} \geq 1$ and $f : \Sigma \to \mathbb{R}^3$ 
is a conformal CMC immersion (by Remark 
\ref{no-minimal-compact-surfaces}, because 
there do not exist any closed compact 
minimal surfaces in $\mathbb{R}^3$, $f$ is guaranteed to be nonminimal).  
In this case, $f(\Sigma)$ certainly cannot be a sphere, so 
$Q$ is not identically zero (by Lemma 
\ref{lemmaonumbilics}). It follows from \eqref{eqn:ratsing9} that, 
counted with 
multiplicity, there are exactly $4 \frak{g} -4$ umbilic points on the surface.
We conclude the following: 

\begin{corollary}
A closed CMC surface in $\mathbb{R}^3$ of 
genus $1$ has no umbilic points, and a closed 
CMC surface in $\mathbb{R}^3$ of genus strictly greater than $1$ must 
have umbilic points.  
\end{corollary}

\section{The maximum principle for CMC surfaces}
\label{max_princ}

Here we consider the maximum principle for smooth CMC surfaces.  
Roughly, this principle states that 
if one CMC $H$ surface lies locally to one side of another CMC $H$ surface, 
and if they touch tangentially with a common orientation at 
some interior point, then the two surfaces must coincide in a local 
neighborhood of that point.  

The result in the theory of partial differential equations behind this 
principle is the maximum principle for elliptic partial differential 
equations (see, for example, \cite{PW}).  The maximum 
principle for CMC surfaces is relevant to us 
here because it can tell us quite a lot about the kinds of surface one 
can hope (or cannot hope) to construct.  This is because, although it is 
stated locally, the maximum principle can give global results.  It then 
becomes a powerful tool for making global statements about 
CMC surfaces.  For example, one can easily prove the following theorems:  

\begin{theorem}\label{thm:maxprinc1}
Any complete minimal surface in $\mathbb{R}^3$ or $\mathbb{H}^3$ without 
boundary cannot be compact.
\end{theorem}  

\begin{proof} 
By way of contradiction, suppose that $M$ is 
the image of a compact minimal surface without boundary in $\mathbb{R}^3$ or 
$\mathbb{H}^3$.  Then there exists a geodesic 
plane $P=P_0$ that does not intersect $M$.  Translating $P$ in the direction 
of a geodesic perpendicular to it and toward 
$M$ at unit speed (along the geodesic) to make a family of parallel geodesic 
planes $P_t$, $t \geq 0$, and taking 
the smallest value $t_0$ of $t$ so that $P_{t_0} \cap M \neq \emptyset$, one 
has the first (necessarily tangential) contact of $M$ with $P_{t_0}$.  
Thus one has two minimal surfaces $M$ 
and $P_{t_0}$ each lying to one side of each other and touching 
tangentially at some point $p$.  
The maximum principle then implies that 
in a local neighborhood of $p$, $M$ is contained in the geodesic 
plane $P_{t_0}$.  Once an open set in a 
minimal surface is a geodesic plane, the entire surface must lie within that 
geodesic plane.  (This last sentence follows in the case of 
$\mathbb{R}^3$ from real analyticity of the frame as in Remark 
4.4.2 in \cite{wisky} with $H$ chosen to be 
zero.  It also follows from the fact that the stereographic projection of 
the Gauss map in the Weierstrass representation 
is both holomorphic as in Section 3.4 of \cite{wisky} and is 
constant on an open set, and thus is constant on all of $M$.  
Any surface with a constant Gauss map must lie in a plane.  
An argument along the same lines using an analog of Remark 
4.4.2 in \cite{wisky} applies 
in the case of $\mathbb{H}^3$ as well.)  Since $M$ is complete, we 
conclude that $M$ is an entire geodesic plane, but this contradicts the 
assumed compactness of $M$. 
\end{proof}

\begin{figure}[phbt]
\begin{center}
\includegraphics[width=1.0\linewidth]{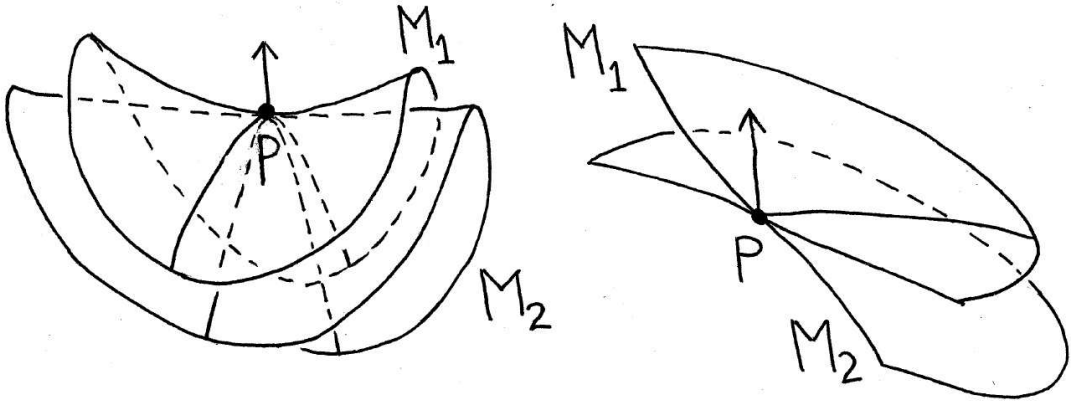}
\caption{The maximum principle (on the left) and the boundary point 
maximum principle (on the right).  In both cases, the surfaces $M_1$ and 
$M_2$ are tangential at $p$ and have the same constant mean curvature with 
respect to the normal direction $\vec{N}$ at $p$, and $M_1$ lies above 
$M_2$ as pictured here.  On the right hand side, the boundaries of $M_1$ and 
$M_2$ have a common tangent line at $p$.  The conclusion in the first 
case (left hand side) is that 
$M_1$ and $M_2$ must coincide in a neighborhood of the 
point $p$.  In the second case (right hand side), $M_1$ and $M_2$ will 
coincide in 
an open set whose closure contains $p$.}  
\end{center}
\end{figure}

\begin{theorem}\label{thm:maxprinc2}
The only embedded compact CMC surfaces in $\mathbb{R}^3$ and $\mathbb{H}^3$ 
are the round spheres.  
\end{theorem}  

This theorem can be proven using the Alexandrov reflection principle, which is an 
immediate consequence of the maximum principle (see, for example, 
\cite{KorKS}).  
Note that the embeddedness condition in Theorem \ref{thm:maxprinc2} is really 
necessary, as the CMC Wente tori show (see Chapter \ref{ashortchapter}).  

\begin{proof}
The Alexandrov reflection principle works in the following way: 
Consider the image of a compact embedded CMC surface $M$ in the ambient 
space $\mathbb{R}^3$ 
or $\mathbb{H}^3$.  Let $q$ be any fixed point in the ambient space, and 
let $\vec{v}$ be any unit 
vector in the tangent space of the ambient space at $q$.  
Let $\alpha_{\vec{v}}(t)$ be a geodesic in the ambient space such that 
$\alpha_{\vec{v}}(0)=0$ and 
$\frac{d}{dt}\alpha_{\vec{v}}(t)|_{t=0}=\vec{v}$.  Let $P_{\vec{v},t}$ 
be the 
uniquely determined geodesic plane containing $\alpha_{\vec{v}}(t)$ and 
perpendicular to $\frac{d}{dt}\alpha_{\vec{v}}(t)$.  Let 
\[ L_{\vec{v},t}^- = \cup_{s \leq t} P_{\vec{v},s} \; , \] 
\[ L_{\vec{v},t}^+ = \cup_{s \geq t} P_{\vec{v},s} \; . \] 
Let $t_0$ be the smallest value of $t$ such that 
$P_{t_0} \cap M \neq \emptyset$.  
Then $P_{t_0}$ lies to one side of $M$ and contacts $M$ tangentially.  
For $t > t_0$ and sufficiently close to $t_0$, the 
interior of the isometric reflection $R_t(M_{\vec{v},t}^-)$ of the portion 
$M_{\vec{v},t}^- = M \cap L_{\vec{v},t}^-$ of $M$ across the plane $P_t$ will 
not make any contact with the portion $M_{\vec{v},t}^+ = M \cap 
L_{\vec{v},t}^+$ of $M$, and nor will 
$R_t(M_{\vec{v},t}^-)$ and $M_{\vec{v},t}^+$ have any tangential contact along 
their common boundary.  One then 
continuously increases $t$ until one arrives at the smallest value $t_1$ where 
the reflection $R_{t_1}(M_{\vec{v},t_1}^-)$ of $M_{\vec{v},t_1}^-$ across 
$P_{t_1}$ and $M_{\vec{v},t_1}^+$ make a tangential contact at 
some point $p$ in $L_{\vec{v},t_1}^+$.  
Let us suppose for the moment that $p$ is in the interior of 
$L_{\vec{v},t_1}^+$.  
Since $t_1$ is the smallest such value, $R_{t_1}(M_{\vec{v},t_1}^-)$ lies 
locally to one side of $M_{\vec{v},t_1}^+$ near $p$.  Also, since 
$M$ is embedded, 
$R_{t_1}(M_{\vec{v},t_1}^-)$ and $M_{\vec{v},t_1}^+$ have the same 
orientation with 
respect to their mean curvature vectors at $p$.  Thus 
$R_{t_1}(M_{\vec{v},t_1}^-)$ and 
$M_{\vec{v},t_1}^+$ coincide in a neighborhood of $p$.  As in the proof of 
Theorem \ref{thm:maxprinc1}, real-analyticity of the frame implies that 
$R_{t_1}(M_{\vec{v},t_1}^-)$ and $M_{\vec{v},t_1}^+$ are globally identical in 
$L_{\vec{v},t}^+$.  Hence $M$ is invariant under isometric 
reflection across the 
geodesic plane $P_{\vec,t_1}$.  

When $p$ is not in the interior of $L_{\vec{v},t_1}^+$, it is in $P_{t_1}$.  
In this case we need a variant of the maximum principle for CMC surfaces, 
called the {\em boundary point} maximum principle for CMC surfaces.  This 
variant will be stated below and gives the same conclusion that $M$ is 
invariant under isometric reflection across the geodesic plane 
$P_{\vec,t_1}$.  

We conclude the proof by noting that the 
direction of $\vec{v}$ was arbitrary, so $M$ has a plane of reflective 
symmetry in every direction, and this is sufficient to conclude that 
$M$ is actually a round sphere.  
\end{proof}

\begin{figure}[phbt]
\begin{center}
\includegraphics[width=1.0\linewidth]{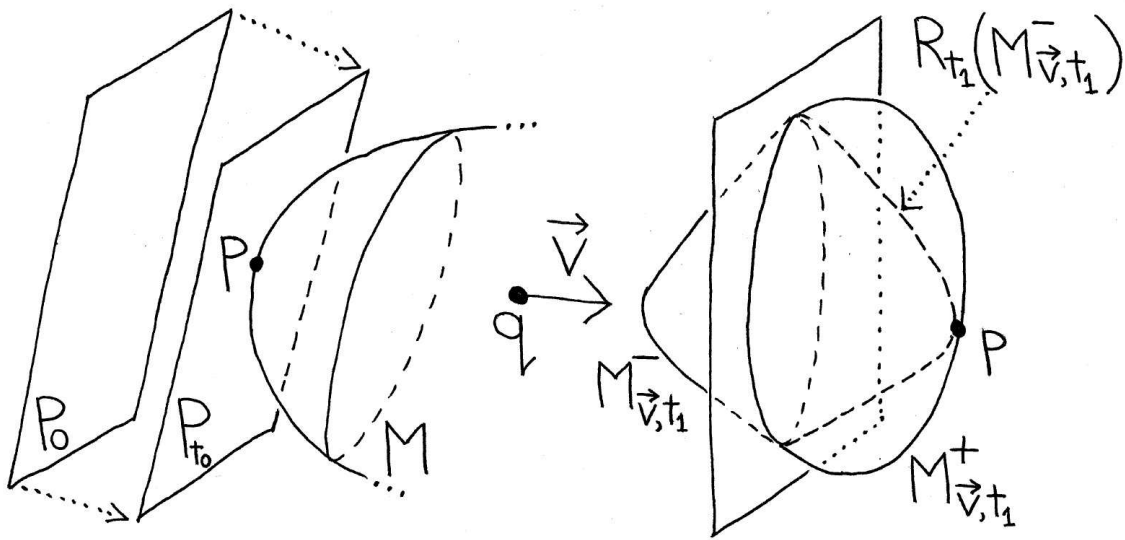}
\caption{The arguments in the proof of Theorem \ref{thm:maxprinc1} (on 
the left) and the proof of Theorem \ref{thm:maxprinc2} (on the right).}
\end{center}
\end{figure}

The maximum principle can also be applied to surfaces with boundary.  
For example, defining the convex hull of a set to be the smallest convex 
set that contains it, only can prove the following result similarly 
to the way Theorem \ref{thm:maxprinc1} was proven: 

\begin{theorem}\label{thm:maxprinc3}
The interior of any compact minimal surface in 
$\mathbb{R}^3$ or $\mathbb{H}^3$ with 
boundary must lie in the interior of the convex hull of its boundary.
\end{theorem}  

Many other results have been proven with the maximum principle, among them 
that any complete connected minimal surface in $\mathbb{R}^3$ with two 
embedded regular ends is a catenoid, proven by Schoen 
\cite{Sch}.  In addition, 
Korevaar, Kusner, Meeks, and Solomon (\cite{Meeks}, \cite{KorKS}), 
have proven that any complete nonminimal finite-topology embedded CMC surface 
with two ends in $\mathbb{R}^3$ is a Delaunay surface, and any 
surface of this type with three ends has a plane of reflective symmetry.  
Similar results for CMC surfaces in $\mathbb{H}^3$ can be found in 
\cite{KorKMS} and \cite{LeR}.  

We shall now prepare to give a formal statement and proof of the maximum 
principle for CMC surfaces.  For the sake of simplicity we shall at first 
assume that the ambient space is $\mathbb{R}^3$.  However, the 
arguments here will require only minor changes to become applicable for other 
ambient spaces as well.  For example, the arguments when the ambient 
space is $\mathbb{H}^3$ are very similar, and we will make some remarks 
about how to prove the $\mathbb{H}^3$ case in the final 
section of this chapter.  As the results we have given here are for 
$\mathbb{R}^3$ and $\mathbb{H}^3$, we shall restrict ourselves to a discussion 
of only those two cases.  

First we give some preliminaries on the maximum principle for 
elliptic equations in the next two sections.  Much of this material follows 
\cite{PW}.  

\begin{remark}
In this chapter, we choose to 
use $x$ and $x_j$ to represent independent variables, and symbols 
such as $a_{ij},b_j,f,f_j,\hat f,\hat f_j,g,g_j,h,u$ to represent 
dependent functions, 
which is different from the notations in the other chapters of this 
text.  This seems appropriate, however, since this chapter deals 
with objects of 
general diminesion, not just $2$-dimensional surfaces, and these 
notational choices are more 
standard in the general dimensional case.  
\end{remark}

\subsection{The maximum principle for elliptic equations 
of a single variable}

In order to get some intuition about the maximum principle for 
elliptic equations, we state and prove various versions of it 
in the case that there is only one independent variable.  

Let us begin with the simplest possible version of the maximum principle.  
We first consider the case that $u$ is a smooth function 
\[ u(x):[a,b] \rightarrow \mathbb{R} \] defined on the closed 
bounded interval $[a,b] \in \mathbb{R}$, and 
$L$ is the operator \[ L(u) = u^{\prime\prime} + g(x) 
u^{\prime} \] defined on functions $u$ as above, where $g(x)$ is a bounded
smooth function on $[a,b]$, and $\prime$ represents the derivative with 
respect to $x \in [a,b]$.  We now state the simplest possible version of 
the maximum principle: 

\begin{lemma}\label{lem:easymp}
(Simplified $1$-dimensional maximum principle) Let $u$, $g$ and 
$L$ be as above.  
If $L(u) > 0$ on $[a,b]$, then $u$ can attain its maximum value in $[a,b]$ 
only at the points $x=a$ or $x=b$.
\end{lemma}

\begin{proof}
Suppose that $u$ attains a local maximum at a point $c \in (a,b)$.  Then 
$u^{\prime}(c) = 0$ and $u^{\prime\prime}(c) \leq 0$, so $L(u)(c) \leq 0$, 
a contradiction.
\end{proof}

The above result was particularly easy, because we made the strong 
assumption that $L > 0$.  But there is a similar result in the case that 
we only assume $L \geq 0$, and then the proof is slightly more subtle 
(and in the application to CMC surfaces we have in mind, we 
will indeed only know that $L \geq 0$).  
In this case, $u$ can attain its maximum in the interior of $[a,b]$, but 
if it does, then $u$ must be a constant function: 

\begin{lemma}\label{lem:1dimlmp}
($1$-dimensional maximum principle) Let $u$, $g$ and $L$ be as above.  
Suppose that $L(u) \geq 0$ on $[a,b]$.  
If $u \leq M$ on $[a,b]$ for some constant $M \in \mathbb{R}$ and if there 
exists some $c \in (a,b)$ such that $u(c)=M$, then $u(x)=M$ for all $x 
\in [a,b]$.
\end{lemma}

\begin{proof}
Suppose there exists a $c \in (a,b)$ such that $u(c)=M$ and there 
exists a $d \in (a,b)$ such that $u(d) < M$.  Assume for now that $d > 
c$.  Because $g$ is bounded, we may choose a constant 
$\alpha >\max_{x \in [a,b]} \{|g(x)|\}$, and then we 
define $y(x) = e^{\alpha (x-c)} -1$.  
Note that $L(y(x)) > 0$.  It is possible to 
choose an $\epsilon$ such that $0 < \epsilon < \frac{M-u(d)}{y(d)}$, 
and then we define $w(x) = u+\epsilon y$.  $y$ is negative on 
$(a,c)$, so $w<M$ on $(a,c)$.  Note that $w(c)=M$ and $w(d)<M$.  So $w$ has 
an interior maximum in $(a,d)$ and $L(w)>0$.  This contradicts Lemma 
\ref{lem:easymp}.

In the case that $d<c$, we may use $y=e^{\alpha (c-x)}-1$ instead of
$y=e^{\alpha (x-c)}-1$ and produce a contradiction to Lemma 
\ref{lem:easymp} in the same way.  
\end{proof}

Now we consider a more general operator of the form 
\[ (L+h)(u) := u^{\prime\prime} + g u^{\prime} + h u \; , \] 
where $h=h(x)$ is a bounded smooth function on $[a,b]$.  Then the 
condition $L(u) \geq 0$ no longer implies that $u$ attains its maximum 
at either $x=a$ or $x=b$.  Here are two counterexamples: 

(1) Let $[a,b]=[0,\pi]$, let $h$ be identically $1$, let $g$ be 
identically $0$, and let 
$u=\sin(x)$.  Then $(L+h)(u)=u^{\prime\prime} + u=0$, and $u$ has an 
interior maximum of value $1$ at $x=\frac{\pi}{2}$ and 
is not maximized at the endpoints $a$ and $b$.  

(2) Let $[a,b]=[-1,1]$, let $h$ be identically $-1$, let $g$ be 
identically $0$, and let 
$u=-\cosh(x)$.  Then $(L+h)(u)=u^{\prime\prime} - u=0$, and $u$ has an 
interior maximum of value $-1$ at $x=\frac{\pi}{2}$ and is not 
maximized at the endpoints $a$ and $b$.  

These two examples show that nonzero $h$ can cause the operator 
$L+h$ to not satisfy the maximum principle, regardless of whether 
$h$ is positive or negative.  However, if 
we assume $h \leq 0$ and $\max_{x \in [a,b]}(u) \geq 0$, then we 
still have a maximum principle, as we now show: 

\begin{lemma}\label{lem:modeasymp}
(Modified simplified $1$-dimensional maximum principle) 
If $h \leq 0$ and $(L+h)(u) > 0$ 
on $[a,b]$, then $u$ cannot have a nonnegative maximum in the interior 
of $[a,b]$.  
\end{lemma}

\begin{proof}
Suppose that $c$ is an interior point of $[a,b]$ where $u$ has a 
nonnegative local maximum.  Then $u^{\prime}(c) = 0$, 
$u^{\prime\prime}(c) \leq 0$, $h(c)u(c) \leq 0$ imply 
$(L+h)(u) \leq 0$, a contradiction.  
\end{proof}

Again, if we only have $(L+h)(u) \geq 0$ then this statement 
above (Lemma \ref{lem:modeasymp}) is 
not true, but again the only exceptions are when $u$ is constant.  

\begin{lemma}\label{lem:mod1dimlmp}
(Modified $1$-dimensional maximum principle I) 
If $u$ satisfies $(L+h)(u) \geq 0$ with $h \leq 0$ on $[a,b]$, then if 
$u$ assumes a nonnegative maximum value $M$ at an interior point 
$c \in (a,b)$, then $u$ is identically equal to $M$.
\end{lemma}

\begin{proof}
Assume $M = \max_{x \in [a,b]}\{u\} \geq 0$ on $[a,b]$.  Assume 
there exists an 
interior point $c$ such that $u(c) = M$.  Also, assume there exists an 
interior point $d$ such that $u(d) < M$.  (Suppose for now that $d>c$.)
Because $g$ and $h$ are bounded, we can choose an $\alpha \in \mathbb{R}$ so 
that \[ \alpha^{2} + \alpha g + h 
(1-e^{-\alpha(x-c)}) > 0 \] for all $x \in [a,b]$.  
Then define $y(x) = e^{\alpha(x-c)}-1$, and note that 
$(L+h)(y) > 0$ on $[a,b]$.  Set $w=u+\epsilon y$ for some $\epsilon$ such that 
$0 < \epsilon < \frac{M-u(d)}{y(d)}$.  As 
$w<M$ on $(a,c)$, $w(c) = M$, $w(d) <M$, 
we have that $w$ has an interior maximum point in $(a,d)$.  Then, since 
$(L+h)(w) > 0$, we have a contradiction to Lemma \ref{lem:modeasymp}.  

Again, if $d<c$, we use $y = e^{\alpha(c-x)}-1$ instead 
of $y = e^{\alpha(x-c)}-1$. 
\end{proof}

Now let us consider a different modification of the maximum 
principle.  Here there 
will be no condition on the sign of $h$ (although 
$h$ is still assumed to be smooth and 
bounded).  Instead we will assume 
that $u$ attains 
a maximum value of precisely $0$ in the interior of the domain.  We shall also 
assume that $u$ is a real analytic function of the independent variable $x$.  

\begin{lemma}\label{lem:mod1dimlmpII}
(Modified $1$-dimensional maximum principle II) 
If a real analytic 
function $u \leq 0$ on $[a,b]$ satisfies $(L+h)(u) \geq 0$, and if 
$u(c)=0$ at an interior point 
$c \in (a,b)$, then $u$ is identically equal to $0$.  
\end{lemma}

\begin{proof}
Suppose that $u$ is not identically zero.  
Because $u(c)=u^\prime(c)=0$, we can expand $u$ at $x=c$ as 
\[ u = \sum_{j \geq 2} a_j (x-c)^{j+\ell} \] for some nonnegative 
integer $\ell$ and 
some $a_2 \neq 0$.  Because $u \leq 0$, we have 
\begin{equation}\label{modIIresult}
\ell \;\; \text{is an even integer, and } \, a_2 < 0 \; . \end{equation}  
Then $L(u)$ expands as 
\[ L(u) = (\ell+2)(\ell+1) a_2 (x-c)^{\ell} (1+\mathcal{O}(x-c)) \; . \]  
But then \eqref{modIIresult} implies $L(u) < 0$ for $x$ close to (but not 
equal to) $c$.  This contradiction proves the lemma.  
\end{proof}

\subsection{The maximum principle for elliptic equations 
in $n$ variables}

Now we consider the $n$-dimensional case, which is entirely analogous to the 
$1$-dimensional case above.  Let $(x_1,...,x_n)$ denote points 
in $\mathbb{R}^n$ 
and let $\mathcalD$ be an open bounded set in $\mathbb{R}^n$ with closure 
$\overline{\mathcalD}$.  We now consider a smooth function 
\[ u(x_1,...,x_n): \overline{\mathcalD} \to \mathbb{R} \; , \]
and we define the operator $L$ by \[ L(u) = \sum_{i,j=1}^n a_{ij}(x_1,...,x_n) 
\frac{\partial^{2}}{\partial x_{i} \partial x_{j}} + 
\sum_{j=1}^n b_{j}(x_1,...,x_n) \frac{\partial}{\partial x_j} \] defined on 
functions $u$ as above, where the coefficient functions 
\[ a_{ij}(x_1,...,x_n) \; , \;\; b_{j}(x_1,...,x_n) \] are bounded 
smooth functions on $\overline{\mathcalD}$, and 
$\frac{\partial}{\partial x_j}$ 
represents the partial derivative with respect to $x_j$.  

\begin{defn}
$L$ is {\em elliptic} in $\mathcalD$ if $(a_{ij})_{i,j=1}^n$ 
is a positive definite $n \times n$ matrix for all $x \in \mathcalD$; that is, 
if at each point in $\mathcalD$, \[ 
(y_1,...,y_n) (a_{ij}) (y_1,...,y_n)^t \geq \mu(x_1,...,x_n) \sum_{j=1}^n 
y_{j}^{2} \] for some positive function $\mu=\mu(x_1,...,x_n)$ 
on $\mathcalD$, and any $y_j \in \mathbb{R}$.  

$L$ is {\em uniformly elliptic} on $\overline{\mathcalD}$ if $\mu(x_1,...,x_n) 
\geq \mu_{0} > 0$ for all points in $\overline{\mathcalD}$, where $\mu_{0}$ is 
a fixed constant.  
\end{defn}

This definition is a natural generalization of the 
Laplacian $\triangle_1 u(x) = u^{\prime\prime}(x)$ in 
the definition of $L$ in the $1$-dimensional case, because of the following 
easily-computed fact: 
$(a_{ij})$ is positive definite at a point $p \in \overline{\mathcalD}$ if and 
only if there exists a linear 
transformation 
$\mathcal{A}:(x_1,...,x_n) \rightarrow (\tilde{x}_1,...,\tilde{x}_n)$ 
such that the second order part 
\[ \sum_{i,j=1}^n a_{ij}(x_1,...,x_n) 
\frac{\partial^{2}}{\partial x_{i} \partial x_{j}} \] 
of $L$ becomes the $n$-dimensional Laplacian \[ 
\triangle_n=
\sum_{j=1}^n \frac{\partial^{2}}{\partial \tilde x_{j}^2} \] at $\mathcal{A}(p)$.  

We state the following two results without proof, and refer the reader to 
\cite{GT}, \cite{PW} for full proofs.  However, we note that the 
ideas behind the 
proofs are like those in the above proofs for $1$ independent variable.  But 
in the case of $n$ independent variables, there is more bookkeeping 
involved in the computations, as expected by the greater number of 
independent variables.  

\begin{theorem}\label{thm:MPfull}
($n$-dimensional maximum principle) 
Let $u$ and $L$ be as above.  Suppose that $L(u) \geq 0$ and that $L$ is 
uniformly elliptic on $\overline{\mathcalD}$.  If $u$ attains a 
maximum value at a point in $\mathcalD$, then $u$ is a constant function.  
\end{theorem}

\begin{theorem}\label{thm:MPfullmod}
(Modified $n$-dimensional maximum principle I) 
Let $u$ and $L$ be as above.  Suppose that 
$(L+h)(u) = L(u) + h u \geq 0$ and that 
$L$ is uniformly elliptic on $\overline{\mathcalD}$, where 
$h \leq 0$ is bounded and smooth on $\overline{\mathcalD}$.  If 
$u$ attains a nonnegative maximum value at a 
point in $\mathcalD$, then $u$ is a constant function -- in 
particular, if $h$ is not identically zero, then $u$ must be 
identically zero.  
\end{theorem}

We also now state (without proof) a 
higher dimensional version of Lemma \ref{lem:mod1dimlmpII}, which could 
also be used to prove the maximum principle for 
CMC surfaces that follows.  We will not actually use it, as other 
forms of the maximum principle given here will suffice, but this next theorem 
is especially useful in proving the maximum principle for CMC surfaces 
when the ambient space is the $3$-sphere $\mathbb{S}^3$.  (We do not 
apply the maximum principle for CMC surfaces in $\mathbb{S}^3$ in this 
text.)  Since we would have two independent variables in the 
application of this theorem to CMC surfaces, we state the result here 
for only that case.  A proof can be found in H. Hopf's book \cite{Hopf}.  

\begin{theorem}\label{thm:mod2dimlmpII}
(Modified $2$-dimensional maximum principle II) 
Consider the operator 
\[ (L+h)(u) := \partial_{x_1} \partial_{x_1} u + \partial_{x_2} 
\partial_{x_2} u + 
g_1 \partial_{x_1} u + g_2 \partial_{x_2} u + h u \] for functions 
$u:\bar \mathcalD 
\to \mathbb{R}$ defined on the closure $\bar \mathcalD$ of an open bounded 
domain $\mathcalD$ of the $2$-dimensional $x_1x_2$-plane, where $g_1$, 
$g_2$ and $h$ are all smooth bounded functions defined on $\bar \mathcalD$.  
If a real analytic function $u \leq 0$ on $\bar \mathcalD$ satisfies 
$(L+h)(u) \geq 0$, and if $u(p)=0$ at a point 
$p \in \mathcalD$, then $u$ is identically equal to $0$.  
\end{theorem}

\subsection{Proof of the maximum principle for 
CMC surfaces in $\mathbb{R}^3$}  

Letting $\mathcalD$ be an open bounded domain in $\mathbb{R}^2$, 
and letting $f(x_1,x_2): \mathcalD \to \mathbb{R}$ be a smooth bounded 
function, we can consider the graph 
\[ \{ \hat{f}(x_1,x_2)=(x_1,x_2,f(x_1,x_2)) \in \mathbb{R}^3 \, | \, 
    (x_1,x_2) \in \mathcalD \} \] 
to be a smooth immersion $\hat f$ into $\mathbb{R}^3$.  
Choosing the unit normal to $\hat f$ to be the upward-pointing 
unit normal vector, 
we saw how to compute the mean curvature $H$ of this surface 
in Definition 1.3.5 in \cite{wisky}, 
as half the trace of the shape operator $S$.  
Because $\hat f$ is of the form $(x_1,x_2,f(x_1,x_2))$, one can 
easily compute that ($\delta_{ij}$ is the Kronecker delta function) 
\begin{equation}\label{eqn:Euclmeancurvature} 
H = \frac{1}{2}\mbox{trace}(S) = 
\frac{\sum_{i,j=1}^{2} f_{x_ix_j} (\delta_{ij}(1+
(f_{x_1})^2+(f_{x_2})^2)- f_{x_i}f_{x_j})}
{2(1+(f_{x_1})^2+(f_{x_2})^2)^{\frac{3}{2}}} \; , 
\end{equation} where 
$f_{x_i}$ denotes $\partial_{x_i} f$ and $f_{x_ix_j}$ denotes 
$\partial_{x_j}(\partial_{x_i} f)$.  

Now let $\hat{f}_1$ and $\hat{f}_2$ be two smooth oriented surfaces 
with boundary.  Suppose 
that the surface $\hat{f}_j$ can be written as a graph over a closed domain 
$\overline{\mathcalD}$ for $j=1,2$; that is, that 
\[ \hat{f}_j(x_1,x_2)=(x_1,x_2,f_j(x_1,x_2)) \] for 
$(x_1,x_2) \in \overline{\mathcalD}$ for some smooth bounded function 
$f_j: \overline{\mathcalD} \to \mathcal{R}$.  Furthermore, suppose that 
both $\hat{f}_1$ and $\hat{f}_2$ have the same constant mean curvature $H$ 
with respect to the orientations given by their upward pointing normals.  

\begin{defn}\label{defn:commontang}
We say that $\hat{f}_1$ {\em lies above} $\hat{f}_2$ if $f_1 \geq f_2$ for all 
points in $\overline{\mathcalD}$.  Then, if \[ 
p:=(x_1,x_2,f_1(x_1,x_2))=(x_1,x_2,f_2(x_1,x_2)) \] 
(i.e. $f_1(x_1,x_2)=f_2(x_1,x_2)$) for some 
point $(x_1,x_2) \in \overline{\mathcalD}$, and if one of the 
following two conditions 
\begin{enumerate}
\item $(x_1,x_2)$ is in the interior of $\overline{\mathcalD}$, or 
\item $(x_1,x_2)$ is in the boundary of $\overline{\mathcalD}$, and 
      the tangent planes of $\hat{f}_1$ and 
      $\hat{f}_2$ coincide at $p$, and furthermore the tangent lines of the 
      boundaries of $\hat{f}_1$ and $\hat{f}_2$ coincide at $p$
\end{enumerate}
holds, we say that $p$ is a 
{\em point of common tangency} of $\hat{f}_1$ and $\hat{f}_2$.  
\end{defn}

We are now ready to state the maximum principle for CMC 
surfaces in $\mathbb{R}^3$: 

\begin{proposition}\label{prop:maxprinc3} 
(The maximum principle for CMC surfaces in $\mathbb{R}^3$.)  
Let $\hat{f}_1$ and $\hat{f}_2$ be CMC $H$ graphs with respect 
to the orientation 
of upward pointing normals.  In particular, $H$ has 
the same value for both surfaces.  Suppose the following:
\newcounter{num}
\begin{list}%
{\arabic{num})}{\usecounter{num}\setlength{\rightmargin}{\leftmargin}}
\item $\hat{f}_1$ lies above $\hat{f}_2$.
\item $\hat{f}_1$ and $\hat{f}_2$ have a point $p$ of common tangency at which 
the first of the two enumerated items in Definition 
\ref{defn:commontang} holds.
\end{list}
Then $\hat{f}_1$ and $\hat{f}_2$ coincide in a neighborhood of $p$.
\end{proposition}

\begin{proposition}\label{prop:maxprinc4} 
(The boundary point maximum principle for CMC surfaces in $\mathbb{R}^3$.)  
Let $\hat{f}_1$ and $\hat{f}_2$ be CMC $H$ graphs with respect to 
the orientation 
of upward pointing normals, just as in Proposition \ref{prop:maxprinc3}.  In 
particular, $H$ has the same value for both surfaces.  Suppose the following:
\begin{list}%
{\arabic{num})}{\usecounter{num}\setlength{\rightmargin}{\leftmargin}}
\item $\hat{f}_1$ lies above $\hat{f}_2$.
\item $\hat{f}_1$ and $\hat{f}_2$ have a point $p$ of common tangency at which 
the second of the two enumerated items in 
Definition \ref{defn:commontang} holds.
\end{list}
Then $\hat{f}_1$ and $\hat{f}_2$ can be extended to surfaces that coincide in 
a neighborhood of $p$.
\end{proposition}

These two results are well known \cite{Al}, and we include a proof of 
just the first one here.  Proofs can also be found in \cite{Sch}, \cite{ER}.  

\begin{proof}
Applying a rigid motion of $\mathbb{R}^3$ if necessary, we may assume $p = 
(0,0,0)$ is the origin in $\mathbb{R}^3$ and 
that the common tangent plane of the two surfaces is the $x_1x_2$-plane 
$\{ x_3 = 0 \}$.  Hence $f_j(0,0)= 0$ and 
$(\partial_{x_1} f_j)(0,0) = (\partial_{x_2} f_j)(0,0) = 0$, for $j = 1,2$.  

Equation \eqref{eqn:Euclmeancurvature} and the fact that 
both surfaces have the same mean curvature imply that 
\begin{equation}\label{eqn:ratsing4} 
\sum_{i,j=1}^2 \left( w_{ij} 
\frac{\delta_{ij}(1+|\nabla f_2|^2)-(f_2)_{x_i}(f_2)_{x_j}}
{2 (1+|\nabla f_2|^2)^{\frac{3}{2}}}
+ \right. \end{equation}\[ (f_1)_{x_ix_j} \left( 
\frac{\delta_{ij}(1+|\nabla f_2|^2)-(f_2)_{x_i}(f_2)_{x_j}}
{2 (1+|\nabla f_2|^2)^{\frac{3}{2}}} \right.
- \]\[ \left. 
\left. \frac{\delta_{ij}(1+|\nabla f_1|^2)-(f_1)_{x_i}(f_1)_{x_j}}
{2 (1+|\nabla f_1|^2)^{\frac{3}{2}}} \right) \right) 
= 0  \; \; , \] 
where $|\nabla f_j|^2 = ((f_j)_{x_1})^2+((f_j)_{x_2})^2$ and 
\[ w := f_{2} - f_{1} \leq 0 \] with first derivatives 
$w_{j}=(f_{2})_{x_j} - (f_{1})_{x_j}$ and second derivatives 
$w_{ij}=(f_{2})_{x_ix_j} - (f_{1})_{x_ix_j}$.  
Defining $\beta_{ij}$ by 
\[ \beta_{ij}(u_1,u_2) = 
\frac{\delta_{ij}(1+u_1^2+u_2^2)-u_iu_j}
{2 (1+u_1^2+u_2^2)^{\frac{3}{2}}} \; , \]  
the intermediate value theorem tells us that 
\[ 
\beta_{ij}((f_2)_{x_1},(f_2)_{x_2}) - 
\beta_{ij}((f_1)_{x_1},(f_1)_{x_2}) = 
\]\[ \sum_{k=1}^2 
\left( \left( \frac{\partial}{\partial u_k} \beta_{ij} \right) 
((cf_2+(1-c)f_1)_{x_1},
(cf_2+(1-c)f_1)_{x_2}) \right) \cdot (f_2-f_1)_{x_k} \; \; , 
\] for some $c=c(i,j) \in [0,1]$.  
Equation \eqref{eqn:ratsing4} then has the form 
\begin{equation}\label{eqn:ratsing5} L w := \sum_{i,j=1}^2 
\left( a_{ij} w_{ij} + (f_1)_{x_ix_j} \sum_{k=1}^2 
\tilde{b}_{ijk} w_k \right) = 0 \; , \end{equation} 
where \[ \tilde{b}_{ijk} =
\left( \frac{\partial}{\partial u_k} \beta_{ij} \right)
((cf_2+(1-c)f_1)_{x_1},
(cf_2+(1-c)f_1)_{x_2}) \; . \]  Note that 
$a_{ij},\tilde{b}_{ij},\tilde{b}_{ijk}$ 
are all bounded functions.  Note also that $a_{ij} \approx 
\frac{\delta_{ij}}{2}$ in 
a small neighborhood of the origin $(x_{1},x_{2})=(0,0)$, and thus 
$(a_{ij})$ is a strictly positive definite $2 \times 2$ matrix in 
a small neighborhood of the origin.  

Since $w \leq 0$ in a small open neighborhood of $(x_{1},x_{2})=(0,0)$ and has 
a local maximum $w=0$ at $(0,0)$, it follows from the maximum 
principle Theorem \ref{thm:MPfullmod} (with $L$ as in 
\eqref{eqn:ratsing5} and $h$ identically equal to zero) 
that $w$ is identically 0 in a neighborhood of the origin.  
We conclude that $f_1 = f_2$ near $p$, and thus $\hat{f}_1$ and
$\hat{f}_2$ coincide in a neighborhood of $p$.  
\end{proof}

\subsection{The maximum principle for CMC 
surfaces in $\mathbb{H}^3$} 
One can give essentially the same proof for 
the maximum principle for CMC surfaces in 
other ambient spaces, such as $\mathbb{H}^3$.  (Some references for the 
maximum principle in the hyperbolic case
are \cite{KorKMS}, \cite{CL}, and references therein.)  Here we describe 
how one could prove the maximum principle for CMC surfaces in 
$\mathbb{H}^3$.  The arguments go along the same lines 
as above for $\mathbb{R}^3$, but some 
differences from the Euclidean case are the following: 
\begin{enumerate}
\item obviously the ambient space no longer has a Euclidean metric (here 
we will consider the Poincare model for $\mathbb{H}^3$, which is conformal 
to the Euclidean metric), and 
\item because the ambient space is not Euclidean, the equation for the 
mean curvature $H$ of a graph will change.  
\end{enumerate}
We now remark on each of these two items.  

{\em Regarding the first item:} 
For $\mathbb{H}^3$, we can use the Poincare ball model $\mathcal P$.  
This allows us to once again consider the two surfaces locally as graphs 
over the $x_1 x_2$-plane containing the origin.  Since the consideration 
is only local, the graphs will both lie in the unit ball $ \{ (x_1,x_2,x_3) 
\in \mathbb{R}^3 \, | \, x_1^2+x_2^2+x_3^2<1 \}$ that is the Poincare model.  
The notions of "point of common tangency" and "one surface lying above the 
other" do not change.  
The only difference is that now the ambient space has the metric 
\begin{equation}\label{eqn:ratsing6} 
\lambda^2 (dx_1^2 + dx_2^2 + dx_3^2) \; , \;\;\; \lambda = 
\frac{2}{1-x_{1}^{2}-x_{2}^{2}-x_{3}^{2}} \; , \end{equation} 
like in \eqref{hypmet}.  

{\em Regarding the second item:} 
Although this Poincare metric is not Euclidean, it is 
still conformal to the Euclidean 
metric, and this conformality will simplify the computation of the 
mean curvature $H$ for a graph in the Poincare model: 

\begin{lemma}
For a smooth immersion 
$\hat{f}(x_1,x_2)$ in $\mathcal{P}$ written as a graph 
\[ \hat{f}(x_1,x_2)=(x_1,x_2,f(x_1,x_2)) \in \mathcal{P} \] with 
$(x_1,x_2) \in \mathcalD \subset \mathcal{P} \cup \{x_3=0\}$, 
the mean curvature of $\hat{f}$ with respect to its ambient 
space $\mathcal{P} \approx \mathbb{H}^3$ is 
\begin{equation}\label{eqn:hypmeancurvature} 
\frac{H}{\lambda} - \frac{\lambda_{N}}{\lambda^{2}} \; , 
\end{equation} where 
\begin{enumerate}
\item $H$ is the Euclidean mean curvature as given in Equation 
\eqref{eqn:Euclmeancurvature}, 
\item $\lambda$ is the metric factor of the Poincare metric, as 
   given in \eqref{eqn:ratsing6}, 
\item $\lambda_{N}$ is the derivative of $\lambda$ with respect to 
the direction 
$N$, where $N$ is the unit normal vector to $\hat{f}$ with respect to the 
standard Euclidean space $(\mathbb{R}^3,dx_1^2+dx_2^2+dx_3^2)$.  
\end{enumerate}
\end{lemma}

\begin{remark}
In fact, the above formula for the mean curvature of a surface in 
$\mathbb{R}^3$ holds for any positive function $\lambda$, when 
$\mathbb{R}^3$ is given the metric $\lambda^2 (dx_1^2+dx_2^2+dx_3^2)$.  
But because our interest here is specifically 
in $\mathbb{H}^3$, we have 
fixed $\lambda$ to be the metric factor of the Poincare metric.  
\end{remark}

\begin{proof}
We will give this proof using the moving frames method.  

Note that with respect to the usual Euclidean metric, a surface that is 
a graph of the form $(x_1,x_2, f(x_1,x_2))$ has mean curvature $H$ as 
in Equation \eqref{eqn:Euclmeancurvature}.  For such a graph we 
define an orthonormal moving frame of vectors $e_1,e_2$ that is 
an oriented orthonormal frame of vectors for the
tangent space of the surface, and then define $e_3=N$ to be the unit
normal vector to the surface with the upward orientation.  We
define 1-forms $\omega^i$ and $\omega_i^j$ by 
\[ \omega^i(e_j) = \delta_{ij} \; \; , \; \; \; \nabla e_i = 
\sum_{j=1}^3\omega_i^j e_j \; \; . \]  
Note that the $\omega_i^j$ are skew symmetric, 
that is, $\omega_i^j = -\omega_j^i$.  Note also that we have the
structure equation \[ d\omega^i = \sum_{j=1}^3 
\omega^j \wedge \omega_j^i \; . \] 
We can then define the mean curvature as \[ H = \frac{1}{2} 
\sum_{i=1}^2 h_{ii} \; , \] where $h_{ij} = \langle
\nabla_{e_i}e_j,e_3 \rangle = \omega_j^3(e_i)$.  

If we now consider the same hypersurface, but with the
ambient metric $\lambda^2(dx_1^2+dx_2^2+dx_3^2)$, 
we can define an orthonormal moving frame in the
same way as above.  
We denote the orthonormal vectors and 1-forms and
mean curvature in this case by using the symbols 
$\hat{e}_i$ and $\hat{\omega}^i$ and 
$\hat{\omega}_i^j$ and $\hat{h}_{ij}$ and $\hat{H}$.  
Noting that we can take 
$\hat{e}_i = \frac{e_i}{\lambda}$ and $\hat{\omega}^i = 
\lambda \omega^i$, and using that $\hat{\omega}^i \wedge \omega^i = 0$, we
see that (with $\lambda_j = e_j(\lambda) = d\lambda (e_j)$ the 
derivative of $\lambda$ with respect to the direction $e_j$) 
\begin{eqnarray*}
\sum_{j=1}^3\hat{\omega}^j \wedge \hat{\omega}_j^i 
&=& d\hat{\omega}^i = d(\lambda \omega^i) 
 =  d\lambda \wedge \omega^i + \lambda d\omega^i 
 =  \sum_{j=1}^3\left(\lambda_j\omega^j \wedge \omega^i
 +  \lambda \omega^j \wedge \omega_j^i\right) \\
&=& \sum_{j=1}^3\Bigg(\lambda \omega^j \wedge 
    \left(\frac{\lambda_j}{\lambda}\omega^i + \omega_j^i\right)\Bigg) 
 =  \sum_{j=1}^3\Bigg(\hat{\omega}^j 
    \wedge \overbrace{\left(\frac{\lambda_j}{\lambda}\omega^i 
 -  \frac{\lambda_i}{\lambda}\omega^j 
 +  \omega_j^i\right)}^{\mbox{skew symmetric}}\Bigg) \; .  
\end{eqnarray*}
So we have $\hat{\omega}_j^i = (\lambda_j/\lambda)\omega^i - 
(\lambda_i/\lambda)\omega^j + \omega_j^i$.  Thus, for $i,j\le 2$, we have 
\[ \hat{h}_{ij} = \hat{\omega}_j^3(\hat{e}_i) = 
\left(\frac{\lambda_j}{\lambda}\omega^3 - 
\frac{\lambda_3}{\lambda}\omega^j + 
\omega_j^3\right)\left(\frac{e_i}{\lambda}\right) = \frac{h_{ij}}{\lambda}-
\frac{\lambda_3}{\lambda^2}\delta_{ij} 
\Longrightarrow \hat{H} = \frac{H}{\lambda} - 
\frac{\lambda_{3}}{\lambda^{2}} \; , \] 
where $\lambda_3=N(\lambda)=d\lambda (N)$ is the derivative of $\lambda$ with 
respect to $N=e_3$.  
\end{proof}


\section{Further motivations for studying CMC surfaces}\label{ashortchapter}

In Chapter \ref{Riemsurfs}, 
we gave Hopf's theorem showing that any closed CMC surface of genus $0$ in 
$\mathbb{R}^3$ must be a round sphere.  Hopf asked if 
every closed CMC surface must be a 
round sphere, without any initial assumption about 
the genus of the surface.  In effect, he asked if any closed CMC surface in 
$\mathbb{R}^3$ of any genus must in fact 
be of genus $0$ and thus be a round sphere.  

Evidence to support a positive answer to this question came from the maximum 
principle for CMC 
surfaces.  This principle gave a technique for showing that any 
embedded closed CMC 
surface in $\mathbb{R}^3$ must be a round sphere.  We gave a proof of 
this result in Chapter \ref{max_princ}.  
So from this, it follows that a closed CMC surface must be a 
round sphere if it is either of genus $0$ or is embedded.  

So any negative answer to Hopf's question would need to be an 
example that is both of positive genus 
and not embedded.  In 1986, H. Wente \cite {Wen:hopf} 
found exactly such examples, of genus 
$1$.  This discovery of a negative answer to Hopf's question gave 
impetus for further 
research on CMC surfaces.  Following Wente's discovery, 
U. Abresch \cite{Ab} in 1987 then published a more explicit representation, 
using elliptic 
functions, for closed CMC tori which contain a 
continuous family of planar principal 
curves.  (Principal curves in a surface are those 
whose tangent vectors are always a 
principal curvature direction, and planar curves 
are those that lie in some plane in 
$\mathbb{R}^3$.)  R. Walter \cite{Wal} (also published 
in 1987) found an explicit 
representation for those tori that Abresch considered, using special functions 
called the Jacobi sn and cn 
functions.  Walter's representation was developed using the fact that 
if one family of principal curves are all planar, then the perpendicular 
principal curves each lie in a sphere.  
Finally, J. Spruck \cite{Spruck} showed in 1988 
that these CMC tori considered by Abresch and 
Walter are exactly the same surfaces 
that Wente originally found.  

The works mentioned above and the development of 
the theory of integrable systems 
since the 1960's helped lead to the recognition that closed CMC tori could be 
studied by using techniques from the theory of integrable systems, and 
that closed CMC tori are special CMC surfaces in the sense that they 
are of "finite type".  This is what is shown in the works of 
U. Pinkall and I. Sterling \cite{ps} and A. Bobenko \cite{b1} \cite{Bob:tor} 
\cite{Bob:2x2}, from 1989 to 1991, and in these works all closed CMC tori in 
$\mathbb{R}^3$ were classified.  

We mention that also N. Kapouleas, in 1991 and 1995, constructed closed 
CMC surfaces for every genus $\frak{g} > 1$ \cite{Kap2}, \cite{Kap3}.  But 
Kapouleas used very different analytic techniques. Further 
developments in that 
direction have been presented recently by Kusner, Mazzeo, Pacard, 
Pollack and Ratzkin 
as well \cite{KusMazPol}, \cite{MP}, \cite{MP2}, \cite{MPP}, \cite{MPP2}, 
\cite{JRatz}.  

The way that the techniques of integrable systems were used in the 
works of Bobenko 
was to convert the problem of studying CMC surfaces into the 
language of $2 \times 2$ matrices.  The same approach was 
taken by Dorfmeister, Pedit and Wu when they developed the DPW method in their 
paper \cite{DPW} published in 1998.  While the language of 
$2 \times 2$ matrices might not seem 
so natural from the viewpoint of classical differential 
geometry in $\mathbb{R}^3$, it is very natural from the viewpoint of 
integrable systems, 
and is certainly convenient for describing the DPW method.  

The idea behind the DPW method is that 
the needed equations and their solutions 
can be found using holomorphic data and applying a splitting 
called the {\em Iwasawa 
decomposition} to maps from circles (loops) to $2 \times 2 $ 
matrices.  This idea 
dates back at least to I. M. Kri\v cever \cite{kr} (1980) and 
perhaps even earlier, 
and J. Dorfmeister, F. Pedit and 
H. Wu formulated it in a way that made the idea apply 
globally to CMC surfaces \cite{DPW}.  The DPW method was the central 
topic of \cite{wisky}.

Finally, we note that the integral systems approach to 
CMC surfaces also helped lead to notions of discrete CMC 
surfaces.  These notions preserve to a large extent the rich 
mathematical structure associated 
with smooth CMC surfaces, and this is exactly what we shall 
focus on for the remainder of this text, from Chapter 
\ref{sect:lcq-smooth} onward.  

But before that, we briefly give an aside on smooth surfaces in 
indefinite ambient spaces, in Chapter \ref{maximalsurface}.  

\section{Maximal surfaces in $\mathbb{R}^{2,1}$}
\label{maximalsurface}

In later chapters we consider surfaces in positive definite spaces 
such as $\mathbb{R}^3$, $\mathbb{S}^3$ (spherical $3$-space) and 
$\mathbb{H}^3$.  However, in this chapter we consider 
surfaces in a space that is not positive definite.  
We do this because we have not considered such a type of 
space yet, and it 
is informative to see the similarities and differences that occur in the 
indefinite case.  Here we choose maximal 
surfaces in $\mathbb{R}^{2,1}$ (Minkowski $3$-space), and 
because the $\mathbb{R}^{4,1}$ M\"obius geometric approach of 
later chapters does not work here, we investigate 
them in much the same way as we considered minimal surfaces in 
\cite{wisky}.  
It is possible to consider discrete versions of these surfaces 
\cite{KR-ocami}, and we come back to this in Chapter 
\ref{discretemaximalsurface}.  

We note that this chapter depends on Section 3.4 in \cite{wisky}, and 
we recommend that the reader look at that before reading this 
chapter.  We also note that this chapter and Chapter 
\ref{discretemaximalsurface} are independent of the 
other chapters here, so could be skipped over 
without affecting continuity of the text.  

Because the maximal surfaces here lie in a space that is not 
positive definite, they have interesting singularities.  The 
singular points can be cuspidal edges or swallowtails or cuspidal 
cross caps, generically, and could also be conical singularities, 
for example, less generically  (for related material, see, for example, 
\cite{FSUY}, \cite{KRSUY}, \cite{SUY1}, \cite{SUY2}, \cite{SUY3}, 
\cite{SUY4}, \cite{SY} and 
\cite{SY2}).  In fact, conical singularities and 
cuspidal edges and swallowtails exist on the surfaces shown in 
Figures 10, 11 and 12.  We will say more about why this happens 
below.  

\begin{figure}[phbt]
\begin{center}
\includegraphics[width=0.48\linewidth]{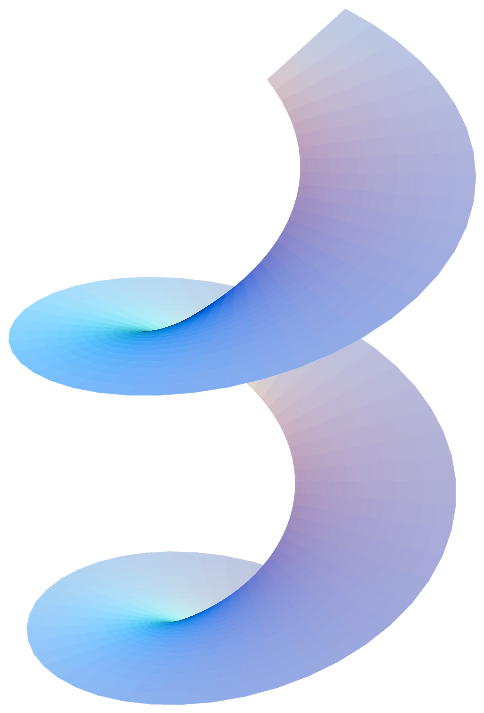}
\raisebox{8ex}{\includegraphics[width=0.48\linewidth]{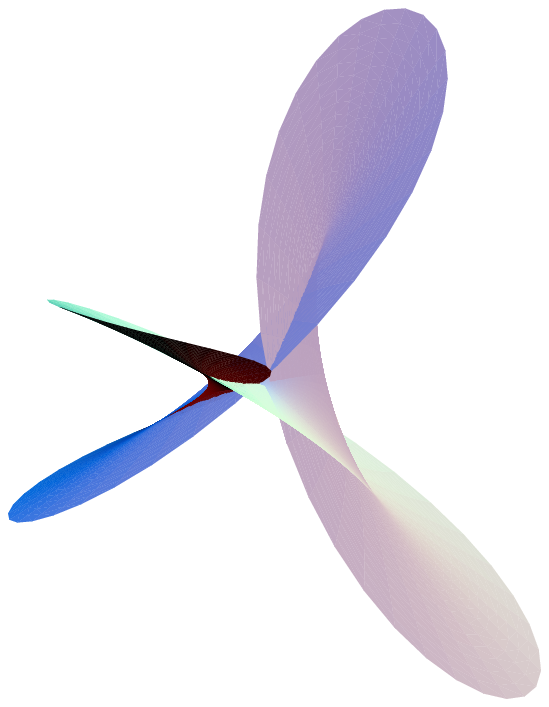}}
\vspace*{-40pt}
\caption{The maximal helicoid (on the left) and Enneper's maximal surface 
(on the right).  The maximal helicoid cousin is given by the representation of 
O.~Kobayashi with $(g,\eta)=(e^z,cie^{-z}dz)$, $c\in\mathbb{R}\setminus\{0\}$ 
on $\Sigma=\mathbb{C}$, like the data for a minimal helicoid in 
$\mathbb{R}^3$.  (This maximal helicoid is in fact contained in the 
image of a minimal helicoid as in Figure 3.4.3 
in \cite{wisky}.  See \cite{KO1} for a 
proof of this.) Enneper's maximal surface is given by the representation of 
O.~Kobayashi with $(g,\eta)=(z,cdz)$, $c\in\mathbb{R}\setminus\{0\}$ on 
$\Sigma=\mathbb{C}$, like the data for an Enneper's minimal surface in 
$\mathbb{R}^3$.  Graphics made by Hitomi Abe and Kouichi Shimose.}
\label{fg:MaxHeliEnn}
\end{center}
\end{figure}
\begin{figure}[phbt]
\begin{center}
\includegraphics[width=0.48\linewidth]{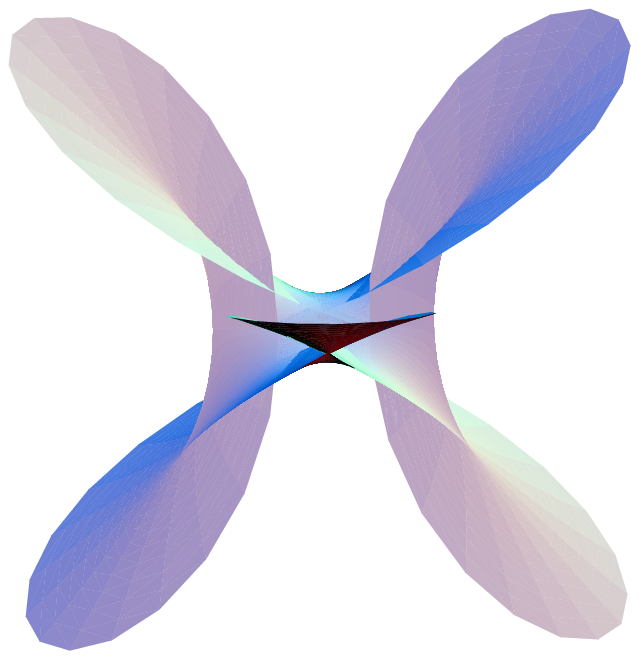}
\includegraphics[width=0.48\linewidth]{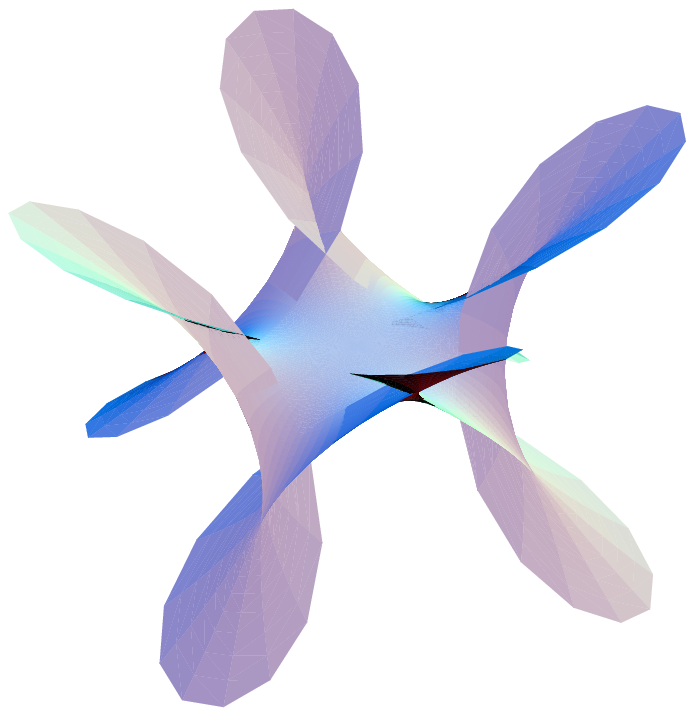}
\vspace*{-20pt}
\caption{The higher-order versions of Enneper's maximal surface are given by 
the representation of O.~Kobayashi with $(g,\eta)=(z^n,cdz)$, 
$c\in\mathbb{R} \setminus\{0\}$ on $\Sigma=\mathbb{C}$, like the data for 
higher-order versions of Enneper's minimal surface in $\mathbb{R}^3$.  
The left-hand side picture is drawn with $n=2$, and 
the right-hand side picture is drawn with $n=3$.  
Graphics made by Hitomi Abe and Kouichi Shimose.}
\label{fg:MaxEnn23}
\end{center}
\end{figure}
\begin{figure}[phbt]
\begin{center}
\includegraphics[width=0.46\linewidth]{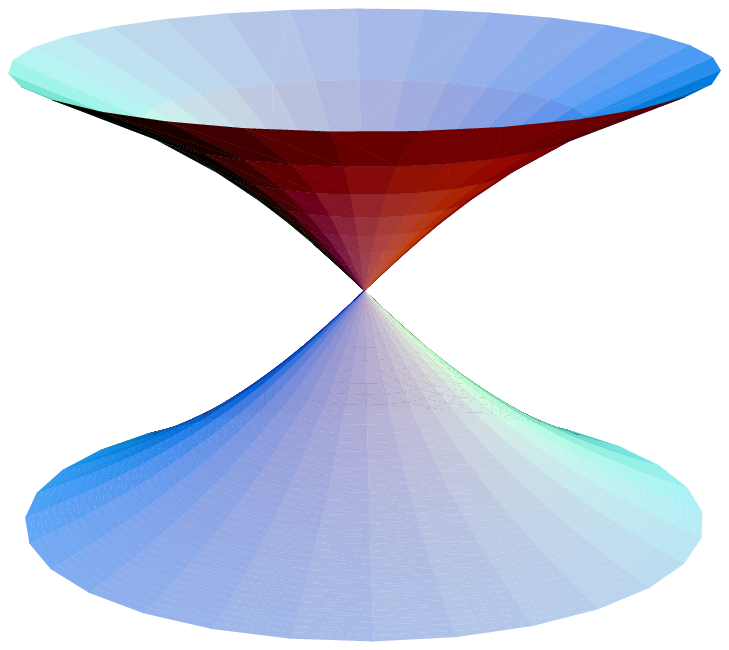}
\raisebox{3ex}{\includegraphics[width=0.52\linewidth]{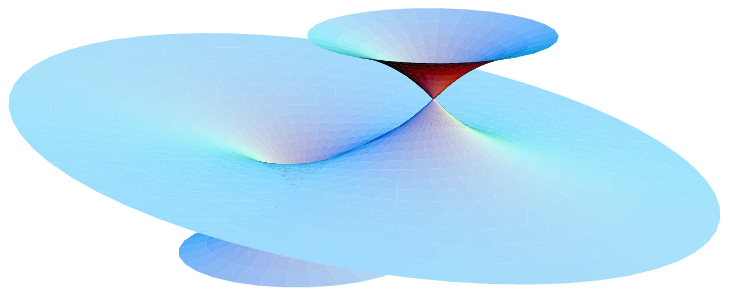}}
\vspace*{-20pt}
\caption{The maximal catenoid (on the left) and the maximal Lopez-Ros surface 
(on the right).  The maximal catenoid is given by the representation of 
O.~Kobayashi with $(g,\eta)=(z,cz^{-2}dz)$, $c\in\mathbb{R}\setminus\{0\}$ on 
$\Sigma=\mathbb{C}\setminus\{0\}$, like the data for a minimal catenoid in 
$\mathbb{R}^3$ (note that $\Sigma$ is not simply-connnected here, 
but the surface is a well-defined map from $\Sigma$ to $\mathbb{R}^{2,1}$).  
The maximal Lopez-Ros surface is given by the 
representation of O.~Kobayashi 
with $(g,\eta)=(\rho (z^2+3)/(z^2-1),\rho^{-1}dz)$, $\rho\in (0,\infty) 
= \mathbb{R}_+$ 
on $\Sigma=\mathbb{C}\setminus\{\pm 1\}$, like the data for a minimal 
Lopez-Ros surface in $\mathbb{R}^3$ (again we have a non-simply-connected 
domain).  Graphics made by Hitomi Abe and Kouichi Shimose}  
\label{fg:MaxCatLR}
\end{center}
\end{figure}

Let $\mathbb{R}^{2,1}=(\{(x_1,x_2,x_0)|x_j\in\mathbb{R}\},
\langle\cdot ,\cdot\rangle_{\mathbb{R}^{2,1}})$ be the $3$-dimensional 
Minkowski space with the Lorentz metric
\[
 \langle (x_1,x_2,x_0),(y_1,y_2,y_0)\rangle_{\mathbb{R}^{2,1}}
=x_1y_1 + x_2y_2 - x_0y_0 \; . 
\]
A surface in $\mathbb{R}^{2,1}$ is called a {\em spacelike} surface if the 
induced metric on the surface is positive definite.  
In this section we study spacelike surfaces in $\mathbb{R}^{2,1}$ whose mean 
curvature is identically zero (maximal surfaces).  
Furthermore, we establish O. Kobayashi's representation \cite{KO1} for these 
surfaces (see also \cite{mcnertney}), 
which is similar to the Weierstrass (Section 
3.4 in \cite{wisky}) and Bryant (Section 5.5 in \cite{wisky}) and 
G\'alvez, Mart\'{\i}nez and Mil\'an (Section 5.6 in \cite{wisky}) 
representations, 
and amongst these three previous representations is most similar to the 
Weierstrass representation.

Let \[ f:\Sigma\to\mathbb{R}^{2,1} \] 
be a conformally immersed spacelike surface, 
where $\Sigma$ is a simply-connected domain in $\mathbb{C}$ with complex 
coordinate $z$.  (Again, by Theorem 
\ref{conformalityispossible}, without loss of generality we may 
assume $f$ is conformal.)  
Then \[ \langle f_z,f_z\rangle =\langle f_{\bar{z}},f_{\bar{z}}\rangle =0 \; , 
\;\;\; \langle f_z,f_{\bar{z}}\rangle =2e^{2\hat u} \; , \] 
where $\hat u:\Sigma\to\mathbb{R}$ is defined this way 
and $\langle\cdot ,\cdot\rangle$ is the complex bilinear extension of 
$\langle\cdot ,\cdot\rangle_{\mathbb{R}^{2,1}}$.  
Let $N$ be a unit normal vector field of $f$.  (Note that $N$ is timelike, 
that is, $\langle N,N\rangle =-1$, since $f$ is spacelike.)
We choose $N$ so that it is future pointing, that is, so that the third 
coordinate of $N$ is positive.  
Then 
\begin{equation}
N:\Sigma\to\mathbb{H}^2:=
\{\vec n=(n_1,n_2,n_0)\in\mathbb{R}^{2,1} \, | 
\, \langle \vec n,\vec n\rangle =-1 \; , \;\;n_0>0\}
\label{eq:GMofL3}
\end{equation}
is the Gauss map of $f$.  

Note that the target space of the Gauss map is now $\mathbb{H}^2$, 
which is not compact (unlike the case of surfaces in 
$\mathbb{R}^3$, where the target of the Gauss map is the 
compact $\mathbb{S}^2$).  Singularities of the maximal 
surfaces occur when the Gauss map reaches the ideal boundary of 
$\mathbb{H}^2$.  

We have the following Gauss-Weingarten equations:
\begin{eqnarray*}
& f_{zz}=2\hat u_zf_z-QN \; ,\quad f_{z\bar z}=-2He^{2\hat u}N \; ,\quad
  f_{\bar z\bar z}=2\hat u_{\bar z}f_{\bar z}-\bar QN \; , & \\
& N_z=-Hf_z-\frac{1}{2}Qe^{-2\hat u}f_{\bar z} \; ,\quad
  N_{\bar z}=-Hf_{\bar z}-\frac{1}{2}\bar Qe^{-2\hat u}f_z \; , &
\end{eqnarray*}
where $Q:=\langle f_{zz},N\rangle$ is the Hopf differential 
function and 
$H=e^{-2\hat u}\langle f_{z\bar z},N\rangle /2$ is the mean curvature.  
Also, $(f_{zz})_{\bar z}=(f_{z\bar z})_z$ implies the following Gauss-Codazzi 
equations:
\begin{equation}\label{eq:GaussformaxinL3}
4 \hat u_{z\bar z}+Q\bar Qe^{-2\hat u}-4H^2e^{2\hat u}=0 \; ,\quad 
Q_{\bar z}=2H_z e^{2\hat u} \; .
\end{equation}
Note that the first equation here has different signs than the corresponding 
equation for minimal 
surfaces in $\mathbb{R}^3$ (see \cite{wisky}), although it is 
otherwise very similar.  

Now we identify $\mathbb{R}^{2,1}$ with the Lie algebra 
\[
\su_{1,1}
=\left\{\left.\begin{pmatrix}ia & b \\ \bar b & -ia\end{pmatrix}
        \right| a\in\mathbb{R},\;\;b\in\mathbb{C}\right\} \; ,
\]
of the Lie group 
\[
\SU_{1,1}
=\left\{\begin{pmatrix}\alpha & \beta \\ \bar\beta & \bar\alpha\end{pmatrix}
\left|\begin{array}{c}\alpha ,\beta\in\mathbb{C} \\ 
       \alpha\bar\alpha-\beta\bar\beta =1\end{array}\right.\right\} \; , 
\]
by identifying $(x_1,x_2,x_0)\in\mathbb{R}^{2,1}$ with the matrix 
\begin{equation}
 x_1\sigma_1+x_2\sigma_2+x_0i\sigma_3
=\begin{pmatrix}ix_0 & x_1-ix_2 \\ x_1+ix_2 & -ix_0\end{pmatrix} 
\in \su_{1,1} \; , 
\label{eq:L3identify}
\end{equation}
where $\sigma_1$, $\sigma_2$, 
$\sigma_3$ are the Pauli matrices defined in the beginning of 
Section 3.2 in \cite{wisky} (but the definition of the Pauli matrices 
is also apparent from \eqref{eq:L3identify} here).  
Note that the metric becomes
\[
\langle X,Y\rangle =\frac{1}{2}\text{trace}(XY) 
\] when considering $\mathbb{R}^{2,1}$ in this $2 \times 2$ 
matrix model.  

The following lemma is immediate:

\begin{lemma}
If $F\in\SU_{1,1}$, then 
$\langle X,Y\rangle =\langle FXF^{-1},FYF^{-1}\rangle$ 
for all $X,Y\in\mathbb{R}^{2,1}$.
\end{lemma}

We also have the following lemma:
\begin{lemma}\label{lm:rigidmotionL3}
There exists an $F\in\SU_{1,1}$ (unique up to sign $\pm F$) so that 
\[
e_1:=\frac{f_u}{|f_u|}=F\sigma_1F^{-1} \; ,\quad
e_2:=\frac{f_v}{|f_v|}=F\sigma_2F^{-1} \; ,\quad
N=Fi\sigma_3F^{-1} \; ,
\]
where $z=u+iv$ for $u,v\in\mathbb{R}$.
\end{lemma}

\begin{proof}
We first define the Lorentz group $\text{O}_{2,1}$ as the set of $3\times 3$ 
real valued matrices $A$ which satisfy $A^t I_{2,1}A=I_{2,1}$, 
and the proper Lorentz group as 
\[
\text{SO}_{2,1}^+:=\left\{\!\!\left.A=\begin{pmatrix}
a_{11} & a_{12} & a_{10} \\
a_{21} & a_{22} & a_{20} \\
a_{01} & a_{02} & a_{00} \end{pmatrix}\right|
\begin{array}{c}A\in\text{O}_{2,1} \; , \\\det A=1 \; , \\ 
a_{00}>0\end{array}\right\} \; , \;\; \text{where} \;\; 
I_{2,1}:=\begin{pmatrix}1 & 0 & 0 \\
                        0 & 1 & 0 \\
                        0 & 0 &-1 \end{pmatrix} \; .
\]
Then the correspondence $F\in\SU_{1,1}$ with 
$(F\sigma_1F^{-1},F\sigma_2F^{-1},Fi\sigma_3F^{-1})$ can be considered as the 
map $\varphi :\SU_{1,1}\to\text{SO}_{2,1}^+$ given by 
\[
\varphi (F)
:=\begin{pmatrix}
\alpha_1^2-\alpha_2^2-\beta_1^2+\beta_2^2 & 2\alpha_1\alpha_2+2\beta_1\beta_2 & 
2\alpha_1\beta_2+2\alpha_2\beta_1 \\ 
-2\alpha_1\alpha_2+2\beta_1\beta_2 & \alpha_1^2-\alpha_2^2+\beta_1^2-\beta_2^2 & 
2\alpha_1\beta_1-2\alpha_2\beta_2 \\
2\alpha_1\beta_2-2\alpha_2\beta_1 & 2\alpha_1\beta_1+2\alpha_2\beta_2 & 
\alpha_1^2+\alpha_2^2+\beta_1^2+\beta_2^2\end{pmatrix}
\]
via the identification \eqref{eq:L3identify}, where 
\[
F=\begin{pmatrix}
\alpha_1+i\alpha_2 & \beta_1+i\beta_2 \\
\beta_1-i\beta_2 & \alpha_1-i\alpha_2 \end{pmatrix} \; , \;\;\; 
\alpha_1,\alpha_2,\beta_1,\beta_2\in\mathbb{R} \; , \;\;\;
\alpha_1^2+\alpha_2^2-\beta_1^2-\beta_2^2=1 \; . 
\]
One can check that $\varphi$ is a surjective homomorphism, and that 
$\varphi (F_1)=\varphi (F_2)$ if and only if $F_1=\pm F_2$.  
This completes the proof.
\end{proof}

Therefore, choosing $F$ as in Lemma \ref{lm:rigidmotionL3}, we have 
$f_u=2e^{\hat u}F\sigma_1F^{-1}$ and $f_v=2e^{\hat u}F\sigma_2F^{-1}$, and so 
\[
f_z=2e^{\hat u}F\begin{pmatrix}0&0\\1&0\end{pmatrix}F^{-1} \; , \quad
f_{\bar z}=2e^{\hat u}F\begin{pmatrix}0&1\\0&0\end{pmatrix}F^{-1} \; .
\]
We define $U$ and $V$ by
\[
F_z=FU \; , \quad F_{\bar z}=FV \; .
\]
Then, similar to the computation in Section 
3.2 of \cite{wisky}, we have 
\[
U=\frac{1}{2}\begin{pmatrix}-\hat u_z&-iQe^{-\hat u}\\2iHe^{\hat u}&\hat u_z\end{pmatrix} \; ,\quad
V=\frac{1}{2}\begin{pmatrix}\hat u_{\bar z}&-2iHe^{\hat u}\\i\bar Qe^{-\hat u}&-\hat u_{\bar z}
             \end{pmatrix} \; . \]

Now we consider the case when $f$ is a {\em maximal} surface, that is, 
the mean curvature $H$ is identically zero.  
(Sufficiently small portions of $H=0$ surfaces in $\mathbb{R}^{2,1}$ actually 
locally maximize area with respect to arbitrary smooth boundary-fixing 
variations, rather than minimize area as they would in $\mathbb{R}^3$.  
Hence they are called maximal surfaces rather than minimal surfaces.  
See \cite{Ca} and \cite{ChengYau}, for example.)


Defining functions $a,b:\Sigma\to\mathbb{C}$ so that
\[ F= \begin{pmatrix} e^{-\hat u/2} \bar a & e^{-\hat u/2} b \\ 
e^{-\hat u/2} \bar b & e^{-\hat u/2} a \end{pmatrix} \] 
holds, then $a\bar{a}-b\bar{b}=e^{\hat u}$, because $F \in \SU_{1,1}$.  
Since $V=F^{-1}F_{\bar z}$, we have 
\[
\frac{1}{2}\begin{pmatrix}\hat u_{\bar z}&0\\i\bar Qe^{-\hat u}&-\hat u_{\bar z} \end{pmatrix}
=\frac{1}{2}\begin{pmatrix}
-\hat u_{\bar z}+2e^{-\hat u}(a\bar{a}_{\bar{z}}-b\bar{b}_{\bar{z}})&
2e^{-\hat u}(ab_{\bar{z}}-ba_{\bar{z}})\\2e^{-\hat u}(\bar{a}\bar{b}_{\bar{z}}-
\bar{b}\bar{a}_{\bar{z}})&-\hat u_{\bar{z}}+2e^{-\hat u}(\bar{a}a_{\bar{z}}-
\bar{b}b_{\bar{z}})
\end{pmatrix} \; .  \]
It follows that 
\[ \begin{pmatrix} \bar a & -\bar b \\ -b & a \end{pmatrix}
   \begin{pmatrix} a_{\bar z} \\ b_{\bar z} \end{pmatrix}
  =\begin{pmatrix} 0 \\ 0 \end{pmatrix}
\]
and so $a_{\bar{z}}=b_{\bar{z}}=0$; that is, $a$ and $b$ are 
holomorphic.  

Note that 
\[ f_z=2e^{\hat u} F \begin{pmatrix}0&0\\1&0\end{pmatrix} F^{-1}
= 2 \begin{pmatrix} ab&-b^2\\a^2&-ab \end{pmatrix} \; , \] 
which is holomorphic and is written as 
\[ f_z=(a^2-b^2,-i(a^2+b^2),-2iab) \] 
in the (complexification of the) 
standard $\mathbb{R}^{2,1}$ coordinates via the identification 
\eqref{eq:L3identify}.  
Since $f$ is real-valued, by Remark 3.4.2 in \cite{wisky}, we have 
\[ \Re \int f_z dz = \frac{1}{2} f + (c_1, c_2, c_0) \]
for some constant $(c_1,c_2,c_0)\in\mathbb{R}^{2,1}$.  
So, up to a translation, 
\begin{equation}\label{eq:Koba-rep}
\begin{array}{r@{=}l}
f=2 \Re \displaystyle\int f_z dz 
& 2\Re \displaystyle\int\limits_{\!\!\!} (a^2-b^2,-i(a^2+b^2),-2iab) dz \\
& \Re \displaystyle\int (1+g^2,i(1-g^2),-2g) \eta \; , \end{array}
\end{equation}
where $g=-ia/b$ and $\eta=-2b^2 dz$.
This is the Weierstrass-type representation for maximal surfaces as in 
\cite{KO1}.  We have just shown the following:

\begin{theorem} (The representation of O. Kobayashi \cite{KO1}) 
Any maximal immersion from a simply-connected domain $\Sigma$ into 
$\mathbb{R}^{2,1}$ can be given the parametrization 
\eqref{eq:Koba-rep}, using a meromorphic function 
$g:\Sigma\to\mathbb{C}$ and holomorphic 1-form $\eta$ on $\Sigma$.  
\end{theorem}

Also, the metric of the maximal surface is expressed as 
\[
(1-g\bar g)^2\eta\bar\eta =4e^{2\hat u}dzd\bar z \; . 
\]
Note that $g\bar g>1$, 
since $a\bar a-b\bar b>0$.  
Note also that we have one minus sign on the left side, unlike the plus 
sign we would have in the case of minimal surfaces in $\mathbb{R}^3$, 
see \cite{wisky}.  This minus sign here allows for singularities, because 
when $|g|$ approaches $1$ (i.e. the Gauss map approaches the ideal 
boundary of $\mathbb{H}^2$), the metric degenerates and singularities 
occur.  

The normal is 
\[
N=Fi\sigma_3F^{-1}=ie^{-\hat u}
\begin{pmatrix}a\bar a+b\bar b & -2\bar ab \\ 2a\bar b & -(a\bar a+b\bar b)
\end{pmatrix} \; , 
\]
which is written as 
\begin{eqnarray*}
N &=& 
e^{-\hat u}\left( i(a\bar b-\bar a b),a\bar b+\bar a b,a\bar a+b\bar b\right) \\
  &=& 
\left( \frac{-g-\bar g}{g\bar g -1},i\frac{g-\bar g}{g\bar g -1},
       \frac{g\bar g +1}{g\bar g -1}\right)
\end{eqnarray*}
in the standard $\mathbb{R}^{2,1}$ coordinates via the identification 
\eqref{eq:L3identify}.  
Thus the function 
\[
g:\Sigma\to\mathcal{C}
 :=(\mathbb{C}\cup\{\infty\})\setminus\{z\in\mathbb{C}\;|\;|z|\le 1\}
\]
is the composition of the Gauss map with stereographic projection from 
$\mathbb{H}^2$ (as in \eqref{eq:GMofL3}) to $\mathcal{C}$, and, as noted 
above, singularities of the surface occur when $|g|$ becomes $1$.  

\begin{remark}
If we assume that the mean curvature $H$ is a non-zero constant, 
then we have the Sym-Bobenko type formula (see \cite{BRS}) 
\[ f(z,\bar z)= \left. \frac{-i}{2H}\left(F 
\begin{pmatrix}
1 & 0 \\ 0 & -1 
\end{pmatrix}
             F^{-1}
             +2\lambda(\partial_\lambda F)F^{-1}\right) 
             \right|_{\lambda = 1} \; . 
\]
\end{remark}

\begin{remark}
For the purpose of considering 
spacelike CMC surfaces in $\mathbb{R}^{2,1}$ via the DPW method, 
we make this remark: Birkhoff splitting for the $\mathbb{R}^{2,1}$ case 
is the same as in Theorem 4.2.4 in \cite{wisky}, because $\SU_{1,1}$ and 
$\SU_2$ are both real forms of $\SL_2\!\mathbb{C}$.  
($\SU_{1,1}$ is a noncompact 
real form and $\SU_2$ is the compact real form of $\SL_2\!\mathbb{C}$.) 
However, the product of loop groups 
$\Lambda\SU_{1,1}\times\Lambda_+^\mathbb{R}\SL_2\!\mathbb{C}$ 
is now only an open dense subset of $\Lambda\SL_2\!\mathbb{C}$, 
so we do not have a global Iwasawa splitting available for 
a DPW style construction of spacelike surfaces in $\mathbb{R}^{2,1}$, 
see \cite{BRS}, \cite{In} and \cite{Kell}.  (When the ambient space 
is $\mathbb{R}^3$, i.e. in the $\SU_2$ case, 
there is a global Iwasawa splitting.)

The non-globalness of the Iwasawa splitting is directly related to 
singularities on the surfaces, because singularities occur 
precisely where the Iwasawa splittings associated with the surfaces 
leave 
$\Lambda\SU_{1,1}\times\Lambda_+^\mathbb{R}\SL_2\!\mathbb{C}$ 
(see \cite{BRS}).  
\end{remark}

\begin{remark}
$SU_{1,1}$ is isomorphic as a group to $SL_2\mathbb{R}$.  For example, 
\[
\text{SU}_{1,1} \ni 
\begin{pmatrix}
p_1+ip_2 & q_1+iq_2 \\ q_1-iq_2 & p_1-ip_2
\end{pmatrix}
\leftrightarrow
\begin{pmatrix}
p_1- q_1 & p_2 + q_2 \\ -p_2+q_2 & p_1+ q_1
\end{pmatrix}
\in \text{SL}_2\mathbb{R} 
\]
is one such isomorphism, where $p_j,q_j \in \mathbb{R}$ satisfy 
$p_1^2+p_2^2-q_1^2-q_2^2 = 1$.  As a result, either 
$\text{su}_{1,1}$ or $\text{sl}_2\mathbb{R}$ can be used to 
represent $\mathbb{R}^{2,1}$.  (Both ways of representing 
$\mathbb{R}^{2,1}$ can be found in the references 
\cite{Ino-next3}, \cite{Ino-next1}, \cite{Ino-next2} and 
\cite{In-toda}.)
\end{remark}

\section{Linear conserved quantities for 
smooth CMC surfaces}\label{sect:lcq-smooth}
\label{chaponsmoothCMCsurf}

In the next chapter we introduce an approach to discrete CMC surfaces 
coming from \cite{BHRS}.  
But to motivate that discussion, in this chapter 
we first explain a result of Burstall and Calderbank \cite{BurstCald} 
for the case of smooth CMC surfaces.  We begin by describing 
the $3$-dimensional space forms using the $5$-dimensional 
Minkowski space $R^{4,1}$.  

\subsection{Minkowski $5$-space}\label{subsec:Mink3ps} 
We give  a $2 \times 2$ matrix formulation for Minkowski $5$-space.  
Let $H$ denote the quaternions and $\Im H$ the imaginary quaternions.  
(We use $H$ to denote both the quaternions and the 
mean curvature of surfaces, but this should not create any confusion, as 
it will always be clear from context which meaning $H$ has in 
each case.)  
\begin{equation}\label{star8point1} 
\mathbb{R}^{4,1} = \left\{ \left. X=\begin{pmatrix} x & x_\infty \\ 
           x_0 & -x \end{pmatrix} \, \right| \, x \in \Im H , 
           x_0, x_\infty \in \mathbb{R} \right\} \end{equation} 
with Minkowski metric $\langle X,Y \rangle$ such that 
\begin{equation}\label{star8point2} 
\langle X,Y \rangle \cdot I = -\tfrac{1}{2} (XY+YX) \; , 
\end{equation} 
$I$ = identity matrix.  
This metric has signature 
$(+,+,+,+,-)$ with respect to the (orthonormal) basis 
\[ \begin{pmatrix} i & 0 \\ 0 & -i \end{pmatrix} \; , \;\; 
\begin{pmatrix} j & 0 \\ 0 & -j \end{pmatrix} \; , \;\;
\begin{pmatrix} k & 0 \\ 0 & -k \end{pmatrix} \; , \;\;
\begin{pmatrix} 0 & 1 \\ -1 & 0 \end{pmatrix} \; , \;\;
\begin{pmatrix} 0 & 1 \\ 1 & 0 \end{pmatrix} \; . \]
If we set $x_4=\tfrac{1}{2} (x_\infty-x_0)$, 
$x_5=\tfrac{1}{2} (x_\infty+x_0)$, we can write $X$ as 
\[ X= x_1 \begin{pmatrix} i & 0 \\ 0 & -i \end{pmatrix} + x_2 
\begin{pmatrix} j & 0 \\ 0 & -j \end{pmatrix} + x_3 
\begin{pmatrix} k & 0 \\ 0 & -k \end{pmatrix} + x_4 
\begin{pmatrix} 0 & 1 \\ -1 & 0 \end{pmatrix} + x_5 
\begin{pmatrix} 0 & 1 \\ 1 & 0 \end{pmatrix} \; , \] 
where $x=x_1i+x_2j+x_3k$, and then we have the correspondence 
$X \leftrightarrow (x_1,x_2,x_3,x_4,x_5)$ to the more usual way 
\[ \{ \xi=(x_1,x_2,x_3,x_4,x_5) \in \mathbb{R}^5 \, | \, 
||\xi||=\text{sgn}(\delta) \sqrt{|\delta|}, \, 
\delta = x_1^2+ x_2^2+ x_3^2+ x_4^2 - x_5^2 \} \]
of denoting $\mathbb{R}^{4,1}$.  The $4$-dimensional light cone is 
\[ L^4=\{ X \in \mathbb{R}^{4,1} \, | \, ||X||=0 \} \; . \]  
We can make the $3$-dimensional space forms as follows: 
A space form $M$ is \[ M = \{ X \in L^4 \, | \, 
\langle X,Q \rangle=-1 \} \] for any nonzero $Q \in R^{4,1}$.  
It will turn out that $M$ has constant sectional 
curvature $\kappa$, where 
$Q^2=\kappa \cdot I$, so without loss of generality we can obtain any 
space form by choosing 
\begin{equation}\label{choiceofQ} Q = \begin{pmatrix} 0 & 1 \\ 
\kappa & 0 \end{pmatrix} \; , \end{equation}
and then, after appropriately scaling $x$, 
and letting $\Im H \cup \{ \infty \}$ denote the one point 
compactification of $\Im H$, we can write 
\begin{equation}\label{star8point4} 
M = \left\{ \left. X=\frac{2}{1-\kappa x^2} \cdot 
\begin{pmatrix} x & -x^2 \\ 1 & -x \end{pmatrix} \, \right| \, 
x=x_1 i + x_2 j +x_3 k \in \Im H \cup \{ \infty \} \, , \; x^2 \neq 
\kappa^{-1} \right\} \, , \end{equation}
which is equivalent to $\{ (x_1,x_2,x_3) \in \mathbb{R}^3 \cup 
\{ \infty \} \, | x_1^2 + x_2^2 + x_3^2 \neq -\kappa^{-1}\}$.  
Note that when $\kappa<0$, $M$ 
becomes two copies of hyperbolic 3-space with sectional curvature 
$\kappa$.  Also, note the following property:
\[
 1-\kappa x^2 \;\; \text{is never zero for points in $M$.}
\]
$M$ is called a {\em quadric}, because it is determined by a quadratic
equation (for the light cone $L^4$) and a linear equation (for the 
hyperplane slicing through $L^4$ that produces $M$).  

\begin{remark}\label{rem:scaling}
Given any \[ \alpha \begin{pmatrix} x & -x^2 \\ 1 & -x 
\end{pmatrix} \] living in the projectivized light cone $PL^4$, 
for any real scalar $\alpha$, we can uniquely choose $\alpha$ 
so that we get a point in $M$, and so sometimes we can 
neglect the real scalar 
$\alpha$, or simply freely choose any $\alpha$ we like.  
\end{remark}

\begin{remark}
We have also used the letter $Q$ to denote the Hopf differential 
function.  Wherever we think this might cause confusion, 
we change the notation for the Hopf differential function to $\hat Q$.  
\end{remark}

The tangent space of $M$ at $X$ is 
\[ T_XM = \left\{ \mathcal{T}_a=\frac{2}{(1-\kappa x^2)^2} \cdot 
\begin{pmatrix} a+\kappa x a x & -x a-a x \\ 
\kappa (x a+ a x) & -a-\kappa x a x \end{pmatrix} \right\} \; , \] 
for $a \in \Im H$.  When $X=X(t) \in M$ is a smooth 
function of a real variable $t$, and when $\prime$ denotes 
differentiation with respect to $t$, we have 
\[ X^\prime = \mathcal{T}_{x^\prime} \; . \]
A computation gives 
\begin{equation}\label{conf-equivalence} 
\langle \mathcal{T}_{a} , \mathcal{T}_{b} \rangle = 
\frac{-4}{(1-\kappa x^2)^2} \text{Re}(ab) \; , \end{equation} 
\[ || \mathcal{T}_{a} || = 1 \Leftrightarrow |a|=\tfrac{1}{2} 
|1-\kappa x^2| \; . \]  Also, 
\begin{equation}\label{Xprimeprime-eqn} 
X^{\prime\prime} = \mathcal{T}_{\tfrac{2 \kappa 
(x x^\prime+x^\prime x)}{1 - \kappa x^2} 
\cdot x^\prime+x^{\prime\prime}} + 
\frac{4 (x^\prime)^2}{(1-\kappa x^2)^2} \cdot 
\begin{pmatrix} \kappa x & -1 \\ 
\kappa & -\kappa x \end{pmatrix} \; . \end{equation} 
Note that generally $X^{\prime\prime}$ is not contained in $T_X M$.  

The following lemma follows from \eqref{conf-equivalence}.  

\begin{figure}[phbt]
\begin{center}
\includegraphics[width=1.0\linewidth]{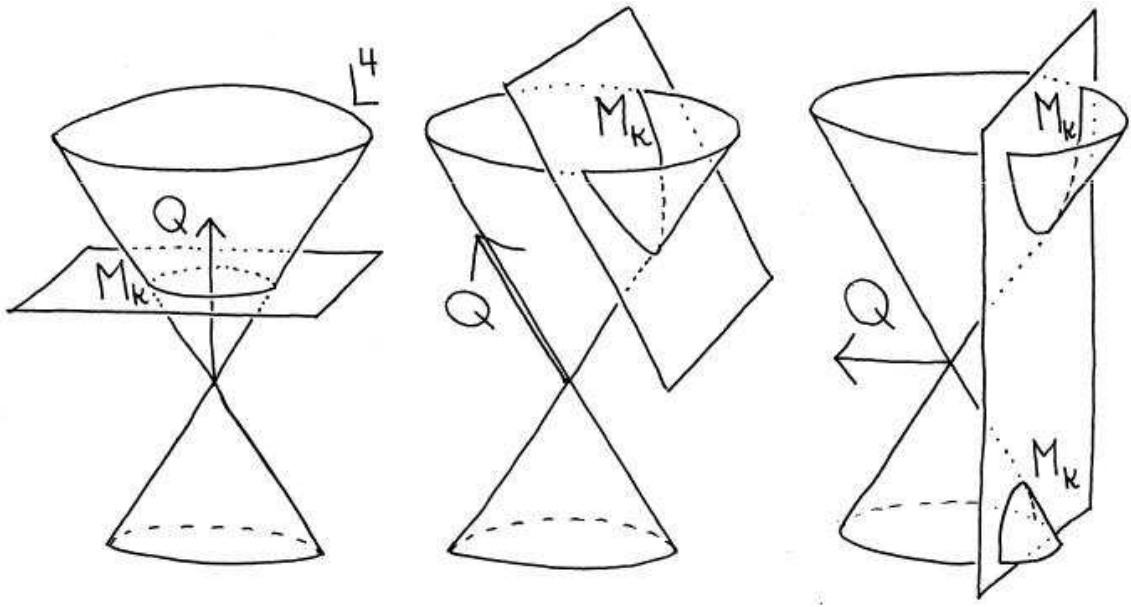}
\caption{Three choices of $\kappa$ ($\kappa>0$, 
$\kappa=0$, $\kappa<0$) giving the space forms $\mathbb{S}^3$, 
$\mathbb{R}^3$ and (two copies of) $\mathbb{H}^3$.}
\end{center}
\end{figure}

\begin{lemma}
The $M$ determined by the $Q$ in \eqref{choiceofQ} has 
constant sectional curvature $\kappa$.  
\end{lemma}

We see from \eqref{conf-equivalence} that the collection of 
$M$ given by the above choice \eqref{choiceofQ} 
for $Q$, for various $\kappa$, are 
all conformally equivalent (or M\"obius equivalent).  
In fact, the map $M \ni X \to x \in \text{Im}H \approx 
\mathbb{R}^3$ is 
stereographic projection when $\kappa \neq 0$.  (See Figure 15.)  

\subsection{Smooth surfaces in space forms} 

We now consider surfaces in the space forms.  Let 
\[ x=x_1(u,v)i+x_2(u,v)j+x_3(u,v)k \leftrightarrow X = X(u,v) \in M 
\] be a surface in $M$.  
(In this chapter we will use $x$ and $X$ to denote surfaces.)  
Assume $(u,v)$ is a conformal curvature-line coordinate system 
(every CMC surface can be parametrized this way, away from 
umbilic points).  We call 
such coordinates {\em isothermic} coordinates.  

Note that $x_1$, $x_2$ and $x_3$ can be chosen before the space 
form $M$ is chosen, and only once $M$ (and hence $\kappa$) is chosen 
do we know the form of $X$.  In particular, the surface can be defined
before the space form is chosen.  

\begin{remark}
The phrase "isothermic coordinates" means simply conformal 
curvature-line coordinates.  However, the phrase "isothermic 
surface" will mean for us any surface for which isothermic 
coordinates exist, even if those isothermic coordinates have 
not been determined yet.   
\end{remark}

{\bf Notation:} Because we will always choose $Q$ 
as in \eqref{choiceofQ}, we will indicate this by denoting $M$ as 
$M_\kappa$, with the subscript $\kappa$.  
We let $n$ denote the unit normal vector for $x$, once $M_\kappa$ 
is chosen.  $n_0$ denotes the unit normal with respect to Euclidean 
$3$-space $M_0$, where 
$\kappa = 0$.  We sometimes write $H_\kappa$ for the mean curvature 
of the surface $x$ with respect to the 
space form $M_\kappa$, to indicate that the mean 
curvature depends on the choice of space form.  $H_0$ is the 
mean curvature in the case of Euclidean $3$-space $M_0$. 

\begin{lemma}\label{lem-meancurv}
The mean curvature $H_\kappa$ of $x$ with respect to the 
space form $M_\kappa$ given by $Q$ as in \eqref{choiceofQ}, with 
$\triangle x = \partial_u \partial_u x + 
\partial_v \partial_v x$, is 
\[ H_\kappa = \tfrac{-1}{2} |x_u|^{-2} \Re\{ \triangle x \cdot n \} - 
\frac{\kappa}{1-\kappa x^2} (xn+nx) = \]\[ 
= \tfrac{-1}{2} (1-\kappa x^2) |x_u|^{-2} \Re\{ \triangle x \cdot n_0 \} - 
\kappa (x n_0+n_0 x) = 
\]\[ (1-\kappa x^2) H_0 - \kappa (x n_0+n_0 x) \; . \]  
Then $H_\kappa$ is constant exactly when 
$\partial_u H_\kappa=\partial_v H_\kappa=0$, which is equivalent to 
\begin{equation}\label{derivofCMC} 
(\partial_u H_0) \cdot (1-\kappa x^2) = \kappa \tfrac{k_2-k_1}{2} 
\partial_u(x^2) \; , \;\;\; 
(\partial_v H_0) \cdot (1-\kappa x^2) = \kappa \tfrac{k_1-k_2}{2} 
\partial_v(x^2) \; , 
\end{equation} where the $k_j \in \mathbb{R}$ are 
the principal curvatures with respect to the Euclidean space form 
$M_0$, i.e. $\partial_u n_0 = - k_1 \partial_u x$ and 
$\partial_v n_0 = - k_2 \partial_v x$.  
\end{lemma}

\begin{proof}
Letting $x_{1u}$ denote $\tfrac{d}{du} (x_1)$, and 
similarly taking other notations, 
the unit normal vector to the surface is $\mathcal{T}_n$, where $n = 
(1-\kappa x^2) n_0$ and 
\[ n_0 = \frac{1}{2} \cdot \frac{(x_{2u} x_{3v}-x_{3u} x_{2v})i+ 
    (x_{3u} x_{1v}-x_{1u} x_{3v})j+ 
    (x_{1u} x_{2v}-x_{2u} x_{1v})k}{\sqrt{(x_{2u} x_{3v}-x_{3u} x_{2v})^2+
    (x_{3u} x_{1v}-x_{1u} x_{3v})^2+(x_{1u} x_{2v}-x_{2u} x_{1v})^2}} \; . \]
The first fundamental form $(g_{ij})$ satisfies 
$\langle \mathcal{T}_{x_u} , \mathcal{T}_{x_v} \rangle = 0 = g_{12} 
= g_{21}$, and 
\[ g_{11} = \langle \mathcal{T}_{x_u} , \mathcal{T}_{x_u} \rangle = 
\frac{4 |x_u|^2}{(1-\kappa x^2)^2} = \frac{4 |x_v|^2}{(1-\kappa x^2)^2}
 = \langle \mathcal{T}_{x_v} , \mathcal{T}_{x_v} \rangle = g_{22} \; . \]
Then using \eqref{Xprimeprime-eqn}, 
with the symbol $^\prime$ denoting either $\partial_u$ or $\partial_v$, 
we have 
(where the superscript "$T$" denotes the part of a 
vector tangent to $T_X M$)
\[ b_{11} = \langle X_{uu}^T, \mathcal{T}_n \rangle = 
\langle X_{uu}, \mathcal{T}_n \rangle = 
\frac{-4}{(1-\kappa x^2)^2} \Re\{ x_{uu} \cdot n \} + 
\frac{4 \kappa x_u^2}{(1-\kappa x^2)^3} (xn+nx) \; , \]
\[ b_{12} = b_{21} = \langle X_{uv}^T, \mathcal{T}_n \rangle = 
\langle X_{uv}, \mathcal{T}_n \rangle = 0 \; , \] \[ b_{22} 
 = \langle X_{vv}^T, \mathcal{T}_n \rangle = 
\langle X_{vv}, \mathcal{T}_n \rangle = 
\frac{-4}{(1-\kappa x^2)^2} \Re\{ x_{vv} \cdot n \} + 
\frac{4 \kappa x_v^2}{(1-\kappa x^2)^3} (xn+nx) \; . \]  
The result follows, using $H_0 = (k_1+k_2)/2$.  
\end{proof}

\begin{remark}
Thomsen proved in the 1920's that isothermic Willmore 
surfaces $x$ in the conformal $3$-sphere (i.e. surfaces that 
are critical with respect to the functional 
$\int (H^2-K) dA$) have a $Q \in \mathbb{R}^{4,1}$ so that $x$ 
becomes minimal in the space form represented by $Q$.  (See the 
third volume of Blaschke's texts \cite{Blaschke}.)
\end{remark}

\subsection{Spheres}
The spheres in any of the space forms $M_\kappa$ are the surfaces 
$x$ such that $|x-C_0|$ is constant for some constant 
$C_0 \in \text{Im}H$.  In the case that $\kappa=0$, if the 
sphere has radius $r_0$, then $r_0H_0=1$ (in particular, 
$H_0$ is positive with respect to the orientation of $n_0$ in 
the above proof).  
Thus the sphere can be written as $x=(-1/H_0) n_0+C_0$ for some 
constant $C_0$.  Then the equation $H_\kappa = (1-\kappa x^2) 
H_0 - \kappa (x n_0+n_0 x)$ gives the following formula 
\begin{equation}\label{Hkappa-vs-H0}
  H_\kappa = H_0 -\frac{\kappa}{4 H_0}-H_0 \kappa C_0^2 
\end{equation} for the relationships between the different 
mean curvatures for a sphere considered in the different 
space forms $M_\kappa$.  

A point 
\[ \mathcal{S} = \begin{pmatrix}
z & z_\infty \\ z_0 & -z
\end{pmatrix}
\] in $\mathbb{R}^{4,1}$ with positive norm 
\[ ||\mathcal{S}|| = \sqrt{-z^2-z_0 z_\infty} > 0 \] 
determines a sphere $\tilde{\mathcal{S}}$ in the 
space form $M_\kappa$ as follows (see Figure 14): Set 
\begin{equation}\label{eq:S-tilde-sphere} 
\tilde{\mathcal{S}} = \{ Y \in M_\kappa \, | \, 
\langle Y,\mathcal{S} \rangle = 0 \} \; . \end{equation} 

Note that $Y \in \tilde{\mathcal{S}}$ implies 
$Y$ is perpendicular to $\mathcal{S}-Y$, so $\tilde{\mathcal{S}}$ 
is the base of the tangent cone from $\mathcal{S}$ to $PL^4$, as 
pictured in Figure 14.  

So we have now seen how both points {\em and} spheres in 
the space forms can be described by just points in the single 
space $\mathbb{R}^{4,1}$, which is a valuable property, from the viewpoint 
of M\"obius geometry.  Note that $\tilde{\mathcal{S}}$ is invariant 
under real scalings of $\mathcal{S}$, and that if $\mathcal{S}$ 
satisfies $z_0=-z_\infty$, then $\tilde{\mathcal{S}}$ is a great 
hypersphere in $M_1= \mathbb{S}^3$.  Also, note that if 
$||\mathcal{S}|| = 0$, then $\mathcal{S}$ is a point in 
$\mathbb{S}^3$ and $\tilde{\mathcal{S}}$ 
is just a real scalar multiple of $\mathcal{S}$, hence
$\tilde{\mathcal{S}}$ simply gives back the same point $\mathcal{S}$.  

Let $\ell$ be the horizontal line segment from $\mathcal{S}$ to 
the timelike axis $\{(0,0,0,0,t) \, | \, t \in \mathbb{R} \}$.  Then $m = 
\ell \cap L^4$ is a single point, which, when considered as being in
the projectivized light cone 
$PL^4$, gives the center of $\tilde{\mathcal{S}}$ in the space form 
$M_1 = \mathbb{S}^3$.  

\begin{lemma}\label{lem:anglebetweenspheres}
Let $\tilde{\mathcal{S}}_1,\tilde{\mathcal{S}}_2$ be two intersecting 
spheres in $\mathbb{S}^3$ produced from $\mathcal{S}_1,\mathcal{S}_2$, 
respectively, and suppose $||\mathcal{S}_1||=||\mathcal{S}_2||=1$.  Let 
$\alpha$ be the intersection 
angle between $\tilde{\mathcal{S}}_1$ and $\tilde{\mathcal{S}}_2$.  
Then $\cos \alpha = \pm \langle \mathcal{S}_1,\mathcal{S}_2 \rangle$, 
where the sign on the right hand side depends on the orientations of 
$\tilde{\mathcal{S}}_1$ and $\tilde{\mathcal{S}}_2$.  
\end{lemma}

\begin{proof}
As $\kappa = 1$, any $p \in \mathbb{S}^3 = M_1$ has 
$x_5$ coordinate equal 
to $1$.  Take $p \in \tilde{\mathcal{S}}_1 \cap \tilde{\mathcal{S}}_2 
\subset M_1$, so $x_5(p)=1$.  Scale $\mathcal{S}_1$ and $\mathcal{S}_2$ 
so that $x_5(\mathcal{S}_1) = x_5(\mathcal{S}_2) = 1$.  Then 
$\mathcal{S}_1-p$ and $\mathcal{S}_2-p$ are normals 
(in the tangent space of $\mathbb{S}^3$) to 
$\tilde{\mathcal{S}}_1$ and $\tilde{\mathcal{S}}_2$, respectively, 
at $p$.  So 
\[ \cos \alpha = \left\langle \frac{\mathcal{S}_1-p}{||\mathcal{S}_1-p||} 
, \frac{\mathcal{S}_2-p}{||\mathcal{S}_2-p||} \right\rangle = \]\[ 
\frac{1}{||\mathcal{S}_1-p||} \frac{1}{||\mathcal{S}_2-p||} 
(\langle \mathcal{S}_1 , \mathcal{S}_2 \rangle
-\langle \mathcal{S}_2 , p \rangle - \langle \mathcal{S}_1 , p \rangle
+\langle p , p \rangle) = \]\[ 
\frac{1}{||\mathcal{S}_1-p||} \frac{1}{||\mathcal{S}_2-p||} 
(\langle \mathcal{S}_1 , \mathcal{S}_2 \rangle-0-0+0) = 
\frac{1}{||\mathcal{S}_1||} \frac{1}{||\mathcal{S}_2||} 
\langle \mathcal{S}_1 , \mathcal{S}_2 \rangle \; . \]
Returning to the scalings for $\mathcal{S}_1$ and $\mathcal{S}_2$ 
so that $||\mathcal{S}_1||= ||\mathcal{S}_2||=1$,  
the lemma is proved.  
\end{proof}

\begin{remark}\label{rem:tangentspheres}
Lemma \ref{lem:anglebetweenspheres} 
implies that if $\mathcal S$ gives a sphere $\tilde{\mathcal{S}}$ 
containing 
$Y \in M_\kappa$, then $\{ \mathcal{S}+t Y \, | \, t \in R \}$ gives
a pencil of spheres at $Y$, i.e. the collection of spheres of 
arbitrary radius through $Y$ and tangent to $\tilde{\mathcal{S}}$.  
\end{remark}

\begin{lemma}\label{lem:inversionthruspheres}
Inversion through $\tilde{\mathcal{S}}$ is the map 
$f: p \to p - 2\langle p,\mathcal{S}\rangle \mathcal{S}$, 
when $||\mathcal{S}||=1$.  
\end{lemma}

\begin{proof}
First note that $p \in L^4$ implies $p - 
2\langle p,\mathcal{S}\rangle \mathcal{S} \in L^4$.  
Now let $C$ be a circle that intersects $\tilde{\mathcal{S}}$ 
perpendicularly.  We wish to show that $p \in C$ implies 
$f(p) \in C$.  Note that $C = \tilde{\mathcal{S}}_1 
\cap \tilde{\mathcal{S}}_2$ for some spheres $\tilde{\mathcal{S}}_1$ 
and $\tilde{\mathcal{S}}_2$.  
Then $\tilde{\mathcal{S}}_1 \perp \tilde{\mathcal{S}}$ and 
$\tilde{\mathcal{S}}_2 \perp \tilde{\mathcal{S}}$, and so 
$\langle \mathcal{S},\mathcal{S}_1 \rangle = 
\langle \mathcal{S},\mathcal{S}_2 \rangle = 0$, by the previous 
lemma.  Then $p \in C$ implies $p \in 
\tilde{\mathcal{S}}_1 \cap \tilde{\mathcal{S}}_2$, which implies 
$\langle p,\mathcal{S}_1 \rangle = \langle p,\mathcal{S}_2 \rangle
= 0$.  Thus 
$\langle p-2 \langle p,\mathcal{S} \rangle \mathcal{S},
\mathcal{S}_1 \rangle = 
\langle p-2 \langle p,\mathcal{S} \rangle \mathcal{S},
\mathcal{S}_2 \rangle = 0$, and so $f(p) \in C$.
\end{proof}

For further explanation of all this, see \cite{Udo-bk}.  

\begin{lemma}\label{sphere-lemma}
$\tilde{\mathcal{S}}$ is a sphere with 
\[ \text{mean curvature} \;\; H_0 = \frac{|z_0|}{2 ||\mathcal{S}||} 
\;\; \text{and center} \;\; \frac{z}{z_0} \] in $M_0$, and is a 
sphere with mean curvature $H_\kappa$ in $M_\kappa$, 
where $H_\kappa$ is as given in Equation \eqref{Hkappa-vs-H0}.  
\end{lemma}

\begin{proof}
Take $z=z_1 i + z_2 j + z_3 k$, and consider the case $\kappa = 0$.
Take \[ Y = 2 \begin{pmatrix} y & -y^2 \\ 1 & -y 
\end{pmatrix} \in \tilde{S} \; . \] Then $YS+SY=0$ implies 
\[ \sum_{j=1}^3 (z_0 y_j - z_j)^2 = ||\mathcal{S}||^2 \; , \] 
and thus $\tilde{\mathcal{S}}$ is a sphere of radius 
$2 ||\mathcal{S}||/|z_0|$.  Hence $H_0 = |z_0|/(2 ||\mathcal{S}||)$.  
The final statement of the lemma now follows from Equation 
\eqref{Hkappa-vs-H0} itself.  
\end{proof}

\subsection{Christoffel transformations}
We now define the Christoffel transformation $x^*$, which for a 
CMC surface in $\mathbb{R}^3$ gives the parallel CMC surface.  
Let $x$ be a surface in $\mathbb{R}^3$ with mean curvature $H_0$ and unit 
normal $n_0$.  The Christoffel transformation $x^*$ satisfies that 
\begin{itemize}
\item $x^*$ is defined on the same domain as $x$, 
\item $x^*$ has the same conformal structure as $x$, 
\item $x$ and $x^*$ have opposite orientations, 
\item and $x$ and $x^*$ have parallel tangent planes at corresponding 
points.  
\end{itemize}
One can check that it automatically follows that the principal curvature 
directions at corresponding points of $x$ and $x^*$ will themselves 
also be parallel.  

\begin{remark}\label{remarkonorientation}
Let us be careful about what we regard as ``opposite 
orientations'' here.  With respect to a common orientation for 
the two parallel tangent planes at a point on $x$ and its 
corresponding point on $x^*$, the two surfaces will have 
opposite orientations.  But if the 
two surfaces both envelop a common sphere congruence, 
for which each of the corresponding pairs of points of $x$ and $x^*$ 
tangentially touch the same sphere in the congruence, then $x$ and 
$x^*$ will have the same orientation with respect to 
the orientation given by a sphere in the sphere congruence.  
(The surfaces $x$ and $x^*$ generally do not envelop a common 
sphere congruence, but they will when $x$ is CMC and $x^*$ is 
positioned to be the parallel CMC surface of $x$.)  
The first perspective might be more natural for parallel 
CMC surfaces, since one moves a constant distance along 
a normal line to get from one surface to the other, so that 
normal line provides a common orientation of the 
two surfaces' tangent planes, by which the surfaces have oppositie 
orientation.  However, the second perspective might be regarded as 
more natural for the Darboux transformations that we consider 
later (since a surface and a Darboux transform of it will always 
envelop a common sphere congruence).  
\end{remark}

This description above of the Christoffel transformations 
turns out to be equivalent to the following 
definition, and the existence of the integrating factor $\rho$ 
below is equivalent to the existence of isothermic coordinates.  
Then, we will see that we can choose $x^*$ so that 
$dx^* = x_u^{-1} du-x_v^{-1} dv$.  

\begin{defn}\label{defn:ChristTrans}
A Christoffel transformation $x^*$ of an umbilic-free surface 
$x$ in $\mathbb{R}^3$ is a surface that satisfies 
$dx^* = \rho (dn_0+ H_0 dx)$ for some 
nonzero real-valued function $\rho$ on 
the surface $x$ (here $x^*$ is determined only up to 
translations and homotheties).  
\end{defn}

\begin{remark}
The Christoffel transform is also sometimes called the 
``dual surface'', and ``taking the Christoffel transform" 
can be called ``dualizing''.  
\end{remark}

\begin{remark}
We did not allow umbilic points on $x$ in the above definition, 
because they can be troublesome.  In particular, the case that $x$ 
is a round sphere (i.e. is completely umbilic) is very special.  
\end{remark}

\begin{lemma}\label{lemma8ptpt15}
Away from umbilics of $x$, the Christoffel transform 
$x^*$ exists if and only if $x$ is isothermic.  
\end{lemma}

\begin{proof}
First we prove one direction, by assuming $x$ is isothermic 
and then showing $x^*$ exists.  

Take $x$ to be isothermic, and take 
isothermic coordinates $u,v$ for $x$, 
so $x_{uv} = A x_u + B x_v$ for some $A,B$.  Then 
\[ d(x_u^{-1} du-x_v^{-1} dv) = 16 g_{11}^{-2} 
(x_u x_{uv} x_u+x_v x_{uv} x_v) du \wedge dv = 0 \; . \]  This implies 
that there exists an $x^*$ such that 
\[ dx^* = x_u^{-1} du-x_v^{-1} dv \; . \]  Also, 
\[ dn_0 + H_0 dx = \tfrac{1}{8} (b_{11}-b_{22}) (x_u^{-1} du-x_v^{-1} dv) 
\; , \] implying that $x^*$ is a Christoffel transform, since 
$b_{11}-b_{22} \neq 0$ at non-umbilic points.  

Now we prove the other direction, by assuming $x^*$ exists 
and then showing that $x$ has isothermic coordinates.  

For any choice of coordinates $u,v$ for $x=x(u,v)$, 
the Codazzi equations are
\[ (b_{11})_v - (b_{12})_u = 
\Gamma_{12}^1 b_{11} + (\Gamma_{12}^2 - \Gamma_{11}^1)  b_{12}
- \Gamma_{11}^2 b_{22} \; , \]
\[ (b_{12})_v - (b_{22})_u = 
\Gamma_{22}^1 b_{11} + (\Gamma_{22}^2 - \Gamma_{21}^1)  b_{12}
- \Gamma_{21}^2 b_{22} \; . \]
(See, for example, page 97 of \cite{Hopf}.)  Here the Christoffel 
symbols are 
\[ \Gamma_{ij}^h = \frac{1}{2} \sum_{k=1}^2 g^{hk} 
(\partial_{u_j} g_{ik}+\partial_{u_i} g_{jk}-\partial_{u_k} g_{ij})
\; , \] where $u_1=u$ and $u_2=v$.  Because we are avoiding 
any umbilic points of $x$, we may assume that $u$ and $v$ are 
curvature line coordinates for $x$ (see, for example, Appendix 
B-5 of \cite{UYbook}), and so 
$g_{12}=b_{12}=0$.  It follows that 
\[ \Gamma_{11}^1 = \frac{\partial_u g_{11}}{2 g_{11}} \; , \;\; 
\Gamma_{22}^2 = \frac{\partial_v g_{22}}{2 g_{22}} \; , \;\; 
\Gamma_{11}^2 = - \frac{\partial_v g_{11}}{2 g_{22}} \; , \]\[ 
\Gamma_{22}^1 = - \frac{\partial_u g_{22}}{2 g_{11}} \; , \;\; 
\Gamma_{12}^1 = 
\Gamma_{21}^1 = \frac{\partial_v g_{11}}{2 g_{11}} \; , \;\; 
\Gamma_{12}^2 = 
\Gamma_{21}^2 = \frac{\partial_u g_{22}}{2 g_{22}} \; . \]
Denoting the principal curvatures by $k_j$, the Codazzi equations 
simplify to 
\begin{equation}\label{eq:Codazzi-iso-Christoffel}
2 (k_1)_v = \frac{\partial_v g_{11}}{g_{11}} 
\cdot (k_2-k_1) \; , \;\;\; 
2 (k_2)_u = \frac{\partial_u g_{22}}{g_{22}} 
\cdot (k_1-k_2) \; . \end{equation}
Then existence of $x^*$ gives 
\[ d (\rho dn_0 + \rho H_0 dx) = 0 \; , \] from which it follows that 
\[ \begin{pmatrix}
0 & \frac{b_{11}}{g_{11}} - \frac{b_{22}}{g_{22}} \\ 
\frac{b_{22}}{g_{22}} - \frac{b_{11}}{g_{11}}  & 0 
\end{pmatrix}
\begin{pmatrix}
\rho_u \\ \rho_v
\end{pmatrix} = 
\rho \cdot 
\begin{pmatrix}
\left( \frac{b_{11}}{g_{11}} + \frac{b_{22}}{g_{22}} \right)_v \\ 
\left( \frac{b_{11}}{g_{11}} + \frac{b_{22}}{g_{22}} \right)_u
\end{pmatrix} \; . 
\] Then because $\rho_{uv} = \rho_{vu}$ (i.e. it does not matter which 
order we take mixed derivatives in), we have 
\[ \left( \frac{(k_2+k_1)_v}{k_1-k_2} \right)_u
= \left( \frac{(k_1+
k_2)_u}{k_2-k_1} \right)_v \; , \]  which implies
\[ 
\frac{2(((k_1)_v)_u+((k_2)_u)_v)}{k_1-k_2} 
+ 2 (k_2-k_1)^{-2} \left( 
(k_1)_v (k_2 - k_1)_u
+ (k_2)_u (k_2 - k_1)_v \right) = 0 \; . \]  
Using the Codazzi equations 
\eqref{eq:Codazzi-iso-Christoffel}, we have 
\[ \left( \log \frac{g_{11}}{g_{22}} \right)_{uv} = 0 \; . \]  
In particular, there exist positive functions $f_1(u)$ and $f_2(v)$ 
depending only on $u$ and $v$, respectively, so that 
\[ (f_1(u))^2 g_{11} = (f_2(v))^2 g_{22} \; . \]
Writing $u = u(\hat u)$ and $v = v(\hat v)$ for new curvature line 
coordinates $\hat u$ and $\hat v$, we have $\hat g_{12} = 
\hat b_{12} = 0$ and $\hat g_{11} = (u_{\hat u})^2 g_{11}$ and 
$\hat g_{22} = (v_{\hat v})^2 g_{22}$, for the fundamental 
form entries $\hat g_{ij}$ and $\hat b_{ij}$ 
in terms of $\hat u$ and $\hat v$.  We can choose $\hat u$ and 
$\hat v$ so that $u_{\hat u} = f_1(u(\hat u))$ and 
$v_{\hat v} = f_2(v(\hat v))$ hold.  Then $\hat g_{11} = \hat g_{22}$ 
and so $\hat u,\hat v$ are isothermic coordinates.
\end{proof}

\begin{figure}[phbt]
\begin{center}
\includegraphics[width=1.0\linewidth]{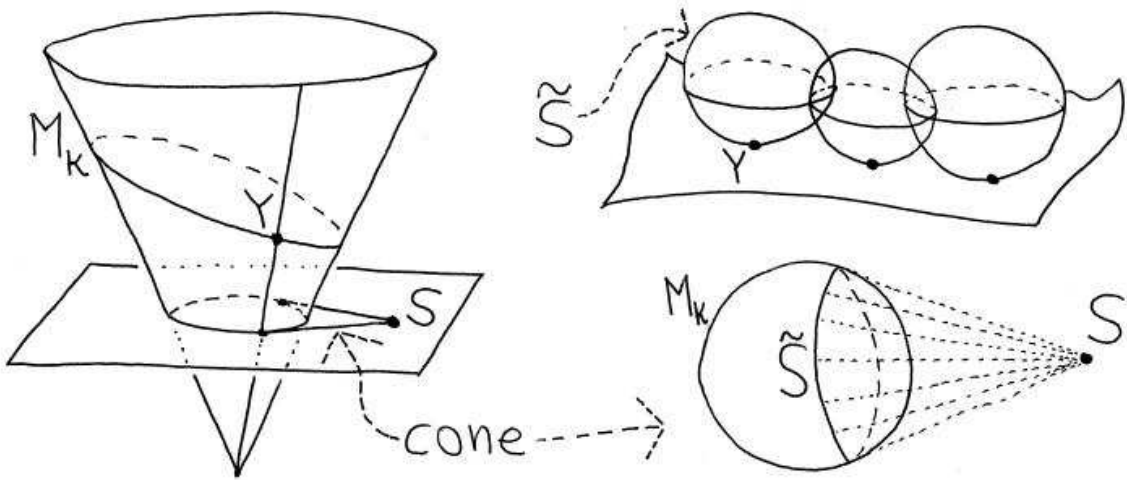}
\caption{A typical picture of an envelope 
on the right, and the corresponding 
picture in the $\mathbb{R}^{4,1}$ model on the left.} 
\end{center}\label{fig:envelop}
\end{figure}

\begin{corollary}\label{Cor:Christoffel}
Away from umbilic points, one Christoffel transformation $x^*$ of 
an isothermic surface $x=x(u,v)$ can be taken as a solution 
of $dx^* = x_u^{-1} du-x_v^{-1} dv$.  
\end{corollary}

Because of $dx^* = \rho (dn_0+H_0 dx)$, we have 
\[
 0=d^2x^* = d\rho \wedge (dn_0+H_0 dx) + \rho \cdot dH_0 \wedge dx \; ,
\] which gives, with respect to isothermic coordinates $(u,v)$, that 
\begin{equation}\label{rho-u-rho-v-eqn}
 \rho_u = - \frac{g_{11} \partial_u H_0}{g_{11} 
          H_0-b_{22}} \cdot \rho \; , \;\;\; 
 \rho_v = - \frac{g_{11} \partial_v H_0}{g_{11} 
          H_0-b_{11}} \cdot \rho \; .
\end{equation}
The existence of $x^*$ then automatically implies the 
compatibility condition $(\rho_u)_v = (\rho_v)_u$, with $\rho_u$ and 
$\rho_v$ as just above.  This pair of equations tells us that $\rho$ is 
uniquely determined once its value is chosen at a single point, and 
thus the solution $\rho$ is unique up to scalar multiplication by a 
constant factor.  Thus the Christoffel transformation in Corollary 
\ref{Cor:Christoffel} is essentially the unique choice, up to 
homothety and translation in $\mathbb{R}^3$.  
As a result of this, with essentially no 
loss of generality, we can now simply 
take the definition of $x^*$ as follows: 

\begin{defn}\label{defn:ChristTrans2}
The Christoffel transformation of a surface $x$ with isothermic 
coordinates $(u,v)$ is any $x^*$ (defined in 
$\mathbb{R}^3$ up to translation) such that $dx^* = x_u^{-1} du - x_v^{-1} dv$.  
\end{defn}

Definition \ref{defn:ChristTrans2} is slightly more general than 
Definition \ref{defn:ChristTrans}, because it can allow umbilic 
points in some cases.  

\begin{remark}
The function $\rho$ in Definition 
\ref{defn:ChristTrans} is generally a constant scalar multiple of the 
multiplicative inverse of the mean curvature of $x^*$, seen as follows: 
The Christoffel transform of the Christoffel transform $(x^*)^*$, 
with respect to Definition \ref{defn:ChristTrans2}, 
satisfies that \[ d((x^*)^*) = (x_u^*)^{-1} du - (x_v^*)^{-1} dv = 
(x_u^{-1})^{-1} du - (-x_v^{-1})^{-1} dv = x_u du + x_v dv = dx 
\; , \] so $(x^*)^*$ should be 
the original surface $x$, up to translation and homothety, 
with respect to Definition \ref{defn:ChristTrans}.  Thus, by 
scaling and translating appropriately, we may assume $(x^*)^* = x$.  Also, 
if the normal of $x$ is $n$, then the normal of $x^*$ is $-n$.  We have 
\[ dx = d((x^*)^*)=\rho^* (dn_0^*+H_0^* dx^*) = 
\rho^* (-dn_0+ H_0^* \rho (dn_0 + H_0 dx)) \; , \] and so 
\[ (1-\rho \rho^* H_0 H_0^*) dx = (H_0^* \rho \rho^* - \rho^*) dn_0 \; . \]  
Since $dx$ and $dn_0$ are linearly independent away from umbilic points, 
it follows that \[ \rho H_0^* = \rho^* H_0 = 1 \; . \]  
\end{remark}

\begin{remark}\label{rem:CMCparallelsurf} 
When $H_0$ is constant and we have isothermic coordinates, 
the equations in \eqref{rho-u-rho-v-eqn} tell us that $\rho$ 
is constant.  Thus if $x^{||}=x+H_0^{-1} n_0$ is the parallel 
CMC surface, then $x^*$ and $x^{||}$ differ 
by only a homothety and translation of $\mathbb{R}^3$.  Thus the 
Christoffel transformation is essentially the same as the 
parallel CMC surface to $x$, as expected.  
\end{remark}

\begin{remark}
The round cylinder gives one simple example of a 
Christoffel transform's orientation reversing property.  
For the cylinder $x(u,v) = (\cos u, \sin u, v)$ in $\mathbb{R}^3$, 
the normal vector is $n = (-\cos u, -\sin u, 0)$, and the 
Christoffel transform is $x^*(u,v) = (-\cos u, -\sin u, v)$ 
with its normal vector $n^* = (\cos u, \sin u, 0)$.  Thus 
$n^* = -n$.  (Note the comments in Remark \ref{remarkonorientation}.)  
\end{remark}

\begin{lemma}\label{minor-technical-lemma} 
\[ dx^* = \frac{2}{(k_1-k_2) |x_u|^2} (dn_0+H_0 dx) \; . \]  
\end{lemma}

\begin{proof}
\[ \left( \frac{2}{(k_1-k_2) |x_u|^2} (dn_0+H_0 dx) - x_u^{-1} du + 
x_v^{-1} dv \right) |x_u|^2 = \]  
\[ = \frac{2}{k_1-k_2} (-k_1 x_u du -k_2 x_v dv +\tfrac{k_1+k_2}{2} (x_u du + 
x_v dv)) + x_u du - x_v dv = 0 \; . \]
\end{proof}

We have already defined the Hopf differential here and in \cite{wisky}, for a 
surface in $\mathbb{R}^3$, as 
\[
  \hat Q dz^2 \; , \;\;\; \hat Q = \langle n_0,x_{zz} \rangle 
  \;\;\;\;\; (z=u+i v) \; . 
\]

\begin{corollary}\label{lem:Qiscte}
If $H_0$ is constant for the surface $x=x(u,v)$ in $\mathbb{R}^3$ with 
isothermic coordinates $(u,v)$, then the factor $\hat Q$ of the 
Hopf differential is a real constant.  
\end{corollary}

\begin{proof}
\[ \hat Q = \tfrac{1}{4} 
\langle n_0, x_{uu}-x_{vv} \rangle = (k_1-k_2) |x_u|^2 \; , \] 
which is constant by Lemma \ref{minor-technical-lemma} 
and Remark \ref{rem:CMCparallelsurf}.  It is clearly also real.  
\end{proof}

\subsection{Conserved quantities and CMC surfaces}  
In the next definition, we are once again considering 
general space forms $M$, so the normalization \eqref{choiceofQ} 
is not assumed.  

\begin{defn}\label{first-lcq-defn}
We set 
\[ \tau = \begin{pmatrix} x dx^* & -x dx^* x \\ dx^* & -dx^* x 
\end{pmatrix} \; . \]  If there exist smooth $Q$ 
and $Z$ in $\mathbb{R}^{4,1}$ depending on $(u,v)$ such that 
\begin{equation}\label{rossman-first-equation} 
 d(Q+\lambda Z) = (Q+\lambda Z) \lambda \tau - 
\lambda \tau (Q+\lambda Z) \end{equation} holds for 
all $\lambda \in \mathbb{R}$, then we call $Q+\lambda Z$ a 
{\em linear conserved quantity} of $x$.  
\end{defn}

We will describe geometric meanings of $Q$, $Z$ and $\tau$ 
later in this text.  

\begin{figure}[phbt]
\label{fig-stero-proj-in-light-cone}
\begin{center}
\includegraphics[width=0.7\linewidth]{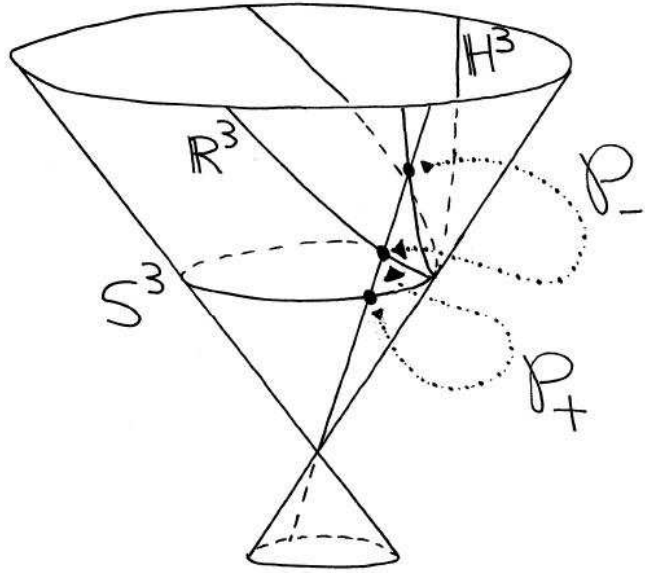}
\caption{${\mathcal P}_+$ and ${\mathcal P}_-$ are conformal 
maps from $\mathbb{S}^3$ and 
$\mathbb{H}^3$ to $\mathbb{R}^3$, showing that 
$\mathbb{S}^3$, 
$\mathbb{R}^3$ and $\mathbb{H}^3$ are M\"obius equivalent.}
\end{center}
\end{figure}

Some properties of linear conserved quantities are immediate.  
For example, $Q$ and $Z^2$ are constant, 
$X \tau = \tau X = 0$, $X \perp Z$ and $X \perp dZ$.  We now 
show these properties: 

\begin{lemma}
$Q$ is constant.  
\end{lemma}

\begin{proof}
Set $\lambda=0$ in the conserved quantity equation 
\eqref{rossman-first-equation}.  
\end{proof}

\begin{lemma}
$X \tau = \tau X = 0$.  
\end{lemma}

\begin{proof}
\[
 X \tau = \frac{2}{1-\kappa x^2} 
\begin{pmatrix}
x \\ 1 \end{pmatrix} 
\begin{pmatrix} 1 & -x \end{pmatrix} 
\begin{pmatrix} 
x \\ 1 \end{pmatrix} dx^* \begin{pmatrix} 
1 & -x \end{pmatrix} = 0 \; , 
\] since 
\[
 \begin{pmatrix}
1 & -x \end{pmatrix} \begin{pmatrix} 
x \\ 1 \end{pmatrix} = 0 \; . 
\] 
Similarly one can show $\tau X = 0$.  
\end{proof}

\begin{lemma}\label{lem-smooth:zsquared-is-cte}
If $Q+\lambda Z$ is a linear conserved quantity, then $Z^2$ is constant.  
\end{lemma}

\begin{proof}
First note that $d(Z^2)=Z \cdot dZ+dZ \cdot Z = 
Z (Q \tau - \tau Q)+(Q \tau - \tau Q) Z = (Q Z + Z Q) \tau - \tau 
(Q Z + Z Q)$, since $Z \tau = \tau Z$.  Because $Q Z + Z Q$ is real, 
we have $d(Z^2)=0$.  
\end{proof}

\begin{lemma}\label{lem:X-perp-to-Z-and-dZ}
$X$ is perpendicular to both $Z$ and $dZ$.  
\end{lemma}

\begin{proof}
$X Z + Z X$ is a real multiple of the identity, and is zero because 
$\tau (X Z+Z X) = \tau Z X = Z \tau X = Z \cdot 0 = 
0$.  Thus, $X \perp Z$.  Next, 
$X \cdot dZ + dZ \cdot X = X (Q \tau-\tau Q) + (Q \tau-\tau Q) X = 
X Q \tau - \tau Q X = (-Q X - 2 \langle X,Q \rangle I) \tau - \tau (-X Q 
- 2 \langle X,Q \rangle I) = (2 \tau - 2 \tau) \langle X,Q \rangle 
= 0$.  Thus $X \perp dZ$.  
\end{proof}

\begin{corollary}\label{getsusedinch11}
We have $Z^2 \leq 0$ (i.e. $Z^2=\alpha I$ for 
some $\alpha \leq 0$), and if $Z^2 = 0$, then $Z$ is parallel to $X$.  
\end{corollary}

\begin{proof}
Because $Z$ is perpendicular to $X$, and because $X$ is lightlike, 
$Z$ is either spacelike, or is a scalar multiple of $X$.  So 
$-Z^2 \geq 0$, and $-Z^2=0$ if and only if $Z$ is parallel to $X$.  
\end{proof}

Furthermore, 
when $Z \neq 0$, we will see that $Z^2 < 0$, see Equation 
\eqref{eq:norm-of-Z-is-positive}.  (That is, $Z^2$ cannot be 
zero.)  

\begin{remark}\label{rem:without-tau}
Necessary and sufficient conditions for existence of a 
linear conserved quantity can be stated without ever 
referring to $\tau$ if we wish, as follows: By definition, a linear 
conserved quantity exists if and only if there exist 
$Q=Q(u,v)$ and $Z=Z(u,v)$ in $\mathbb{R}^{4,1}$ such that the 
following three conditions hold:
\begin{enumerate}
\item $Q$ is constant, 
\item $dZ=Q \tau - \tau Q$, 
\item $Z \tau = \tau Z$.  
\end{enumerate}
Note that 
\[ XdX = \frac{4}{(1-\kappa x^2)^2}
\begin{pmatrix}
x dx & -x dx \cdot x \\ dx & -dx \cdot x
\end{pmatrix}
\] and that $(x^*)_u = x_u^{-1}$ and 
$(x^*)_v = -x_v^{-1}$ and $x_u^2=x_v^2$, so 
\[ X \cdot X_u = \frac{4 x_u^2}{(1-\kappa x^2)^2} \tau(\partial_u) 
\; , \;\;\; 
X \cdot X_v = \frac{-4 x_u^2}{(1-\kappa x^2)^2} \tau(\partial_v)
\; . \]  
Then from 
\[ (XQ+QX)dX-X(dX \cdot Q+Q \cdot dX)=Q \cdot XdX -XdX \cdot Q \] 
we have 
\[ (XQ+QX)X_u-X(X_u Q+Q X_u)
= \frac{4 x_u^2}{(1-\kappa x^2)^2} (Q \cdot \tau(\partial_u)-
\tau(\partial_u) \cdot Q) \] 
and 
\[ (XQ+QX)X_v-X(X_v Q+Q X_v)
= \frac{-4 x_u^2}{(1-\kappa x^2)^2} (Q \cdot \tau(\partial_v)-
\tau(\partial_v) \cdot Q) \; . \] 
We similarly have that the third condition above becomes 
\[ Z X X_u = X X_u Z \; , \;\;\; Z X X_v = X X_v Z \; . \]  
So we can now rewrite the 
three conditions above, without using $\tau$, as 
\begin{enumerate}
\item $Q$ is constant, 
\item $(XQ+QX)X_u-X(X_uQ+QX_u)=\frac{4 x_u^2}{(1-\kappa x^2)^2} Z_u$, 
\item $(XQ+QX)X_v-X(X_vQ+QX_v)=\frac{-4 x_u^2}{(1-\kappa x^2)^2} Z_v$, 
\item $Z X X_u = X X_u Z$, $Z X X_v = X X_v Z$.  
\end{enumerate}
\end{remark}

Properties like these will be utilized to prove Theorems 
\ref{cq-sphere-thm} and \ref{theBCtheorem} below.  
The first of these two theorems takes care of the 
special case that $x$ is a piece of a sphere.  

\begin{theorem}\label{cq-sphere-thm}
The surface $x$ in any space form is a part of a sphere if and only if 
it has a linear conserved quantity that is constant with respect to $\lambda$, 
that is, of the form $Q + \lambda \cdot 0$.  
\end{theorem}

\begin{proof}
Suppose that $x$ has a conserved quantity $Q$ of order $0$.  
(Here we are not assuming $Q$ is of the special form in 
\eqref{choiceofQ}.)  Then 
$Q + \lambda Z$ will also be a conserved quantity if $Z = \alpha Q$, 
for some constant $\alpha \in R$.  It follows from the above lemmas that 
$Q$ is constant and $Z$ is either spacelike or parallel to $X$, and $Z$ 
is perpendicular to $X$.  In particular, $Q$ is constant and perpendicular 
to $X$, and is therefore 
either spacelike or parallel to $X$.  Thus $X$ lies in the 
sphere given by $\mathcal{S}=Q$, as in \eqref{eq:S-tilde-sphere}.  
If $Q$ is lightlike, then $X$ would be a 
single point, and hence not a surface, so $Q$ must be spacelike 
(so the curvature $\kappa$ of the space form $M$ given by $Q$ 
is strictly negative).  Thus, in fact, 
$X$ is part of the virtual boundary sphere 
at infinity of the spaceform given by $Q$.  
It follows that $X$ will be part of a finite 
sphere in other choices for the space form.  

Computationally, this can be seen as follows: 
$Q + \lambda \cdot 0$ is a 
linear conserved quantity, and the equation for 
linear conserved quantities 
implies $Q \tau = \tau Q$.  With $Q$ in the form 
\eqref{choiceofQ}, it follows that the two equations 
\[ dx^* = -\kappa x dx^* x \; , \;\;\; x dx^* = - dx^* x 
\] hold.  This in turn implies $dx^* = \kappa x^2 dx^*$, and so 
one of 
\[ 1-\kappa x^2 = 0 \;\;\; \text{or} \;\;\; dx^* = 0 \] 
must hold.  However, because $dx^* = x_u^{-1} du - x_v^{-1} dv$ is 
never zero, we know that the first of these two 
equations must hold, and so $x$ is a portion of a sphere 
(and $\kappa < 0$).  

Conversely, in the case that $x$ is part of a sphere, then 
there exists a constant $\mathcal{S}=Q$ that is perpendicular to 
$X$, and it follows that $Q=Q+\lambda \cdot 0$ itself is a linear 
conserved quantity, by the four conditions at the end of 
Remark \ref{rem:without-tau}.  (Note that differentiation of 
$XQ+QX=0$ gives $dX \cdot Q + Q \cdot dX = 0$, because $Q$ is 
constant.)  
\end{proof}

\begin{theorem}\label{theBCtheorem} \cite{BurstCald}
An isothermic immersion $x=x(u,v)$ without umbilic points 
has constant mean curvature in a space form $M$ (produced 
by $Q \neq 0$) if and only if there exists (for that $Q$) a 
linear conserved quantity $Q + \lambda Z$.  
\end{theorem}

\begin{proof}
Assume that $x$ has a linear conserved quantity.  
We can take $Q$ as in \eqref{choiceofQ}, and 
denote the components of $Z$ by 
\[ Z = \begin{pmatrix} z & z_\infty \\ z_0 & -z \end{pmatrix} 
\in \mathbb{R}^{4,1} \; . \]  

The above lemmas tell us that $XZ+ZX=0$ and $X \perp dZ$, which, 
respectively, imply 
\[ x z- x^2 z_0+z x+z_\infty = 0 \;\;\; \text{and} \;\;\; 
x \, dz- x^2 dz_0+dz \, x+dz_\infty = 0 \; . \]  
Differentiating the first of these two equations, and then applying 
the second one, we have 
\[ dx \cdot z- (x dx+dx \cdot x) z_0+z dx = 0 \; , \] which implies 
$z$ must be of the form 
\[ z = z_0 \cdot x + h \cdot n_0 \] for some real-valued function $h$.  
Then 
\[ x (z_0 x + h n_0)-x^2z_0+(z_0 x + h n_0)x+
z_\infty=hxn_0+z_0x^2+hn_0x+z_\infty = 0 \; , \] so 
\[ z_\infty = - h (xn_0+n_0x)-z_0x^2 \; . \]  Thus 
\[ Z = z_0 \begin{pmatrix} x & -x^2 \\ 1 & -x \end{pmatrix} + 
h \begin{pmatrix} n_0 & -n_0x-xn_0 \\ 0 & -n_0 \end{pmatrix} \; . \]  
Because $Z^2$ is constant, \[ 
(z_0x+hn_0)^2-z_0h (xn_0+n_0x)-z_0^2 x^2 = - h^2/4 \] is constant, 
and so $h$ is constant, and then also $||Z||$ is constant and nonnegative.  
A direct computation, using 
$n_0 \, dx^* + dx^* \, n_0 = 0$, now shows that $Z \tau = \tau Z$, 
so the condition $Z \tau = \tau Z$ coming from Equation 
\eqref{rossman-first-equation} provides no extra information.  
The relation $dZ = Q \tau - \tau Q$ from 
\eqref{rossman-first-equation} gives that 
\[ dz_0 = \kappa (x \cdot dx^*+dx^* \cdot x) \; \text{and} \; 
dz_0 \cdot x + z_0 dx + h dn_0 = dx^*+\kappa x dx^* x \; . \]  
These two equations give us a pair of (real) equations that are 
linear with respect to both $h$ and $z_0$.  Solving 
simultaneously for $h$ and $z_0$ tells us that 
\begin{equation}\label{h-eqn} 
h = \frac{2 (1-\kappa x^2)}{x_u^2 (k_2-k_1)} \; , 
\end{equation} which we know to be constant, and that 
\begin{equation}\label{h-eqn2} 
z_0 = \tfrac{1}{2} h (k_2+k_1) = h \cdot H_0 \; . \end{equation} 
Equations \eqref{h-eqn}, \eqref{h-eqn2} and $h$ being 
constant then imply 
\[
 dz_0=h dH_0 = \frac{2 (1-\kappa x^2)}{x_u^2 (k_2-k_1)} dH_0 
  \; . \]
Then using that $dz_0 = \kappa (x dx^*+dx^* x)$, and that 
$dx^*=x_u^{-1} du-x_v^{-1}dv$, we find that \eqref{derivofCMC} holds, 
and so $H_\kappa$ is constant.  One direction of the theorem now follows.  

To prove the converse direction, assume that $x$ is a CMC surface 
with isothermic coordinate $z=u+iv$, 
then the Hopf differential is a constant multiple of $dz^2$ (see 
Corollary \ref{lem:Qiscte} here for the case when the space form is 
$\mathbb{R}^3$ and Equations (5.1.1) and (5.2.1) in \cite{wisky} for 
other space forms).  Thus, looking at the end of the proof of Lemma 
\ref{lem-meancurv}, we see that  
\[
 b_{11}-b_{22} = \frac{4 x_u^2 (k_2-k_1)}{1-\kappa x^2}
\] is constant, and so 
\[ h = \frac{2 (1-\kappa x^2)}{x_u^2 (k_2-k_1)} 
\] is also constant.  Take $Q$ as in \eqref{choiceofQ}, 
and then take 
\[
 Z = h H_0 \begin{pmatrix} x & -x^2 \\ 1 & -x 
\end{pmatrix} + h \begin{pmatrix} n_0 & -x n_0-n_0 x \\ 0 & -n_0 
\end{pmatrix} \; . 
\]  Then set the candidate for the conserved quantity to be 
$P=Q+\lambda Z$.  Noting that $dx^* = x_u^{-1} du- x_v^{-1} dv$, 
and $\tau$ is as in Definition \ref{first-lcq-defn}, 
a computation gives $dP+\lambda \tau P - P \lambda \tau = 0$, 
by Equation \eqref{derivofCMC}, so $P$ is indeed a linear conserved 
quantity.  
\end{proof}

In Theorem \ref{theBCtheorem}, for given $Q$, when $x$ is 
constant mean curvature and not totally umbilic, then $Z$ is unique.  
In fact, in the proof above we saw that $Z$ has the unique form 
\[ Z = h H_0 \begin{pmatrix}
x+H_0^{-1} n_0 & -x^2-H_0^{-1} (n_0 x + x n_0) \\ 1 & -x-H_0^{-1} n_0
\end{pmatrix} \; , \] where $h$ is the constant as in \eqref{h-eqn}.  
Furthermore, because $1-\kappa x^2$ is never zero, $h$ cannot be 
zero, so the norm of $Z$ satisfies 
\begin{equation}\label{eq:norm-of-Z-is-positive}
||Z||=\tfrac{1}{2} |h|>0 \; . \end{equation}  
In particular, $||Z|| \neq 0$.  

Also, by Lemma \ref{lem-meancurv}, the mean curvature satisfies 
\begin{equation}\label{eqn:Hwithscalarfactor}
 H_\kappa = -2 h^{-1} \langle Z,Q \rangle = 
 - \text{sgn}(h) \frac{1}{||Z||} \langle Z,Q \rangle \; . 
\end{equation}
In particular, if $||Z||=1$, then the mean curvature is 
$\pm \langle Z,Q \rangle$.  
Note that any constant scaling of the linear conserved quantity 
is still a linear conserved quantity, and will change 
the mean curvature by a constant multiple.  

Next, noting that $z_0 = h H_0$, Lemma \ref{sphere-lemma} tells 
us that $Z$ determines a sphere, as in \eqref{eq:S-tilde-sphere}, 
in $M_0$ with mean curvature 
\[ \pm \frac{|z_0|}{2 ||Z||} = \pm \frac{|h| |H_0|}{2 \cdot 
\tfrac{1}{2} |h|} = \pm |H_0| \; , \] so the mean curvature 
of this sphere is the same as the mean curvature of the surface.  

In particular, once we know that the surface and the sphere are 
tangent, Lemma \ref{lem-meancurv} implies that $Z$ 
determines a sphere congruence for which each sphere has the same 
mean curvature as the mean curvature at 
the corresponding point on the surface, regardless 
of the choice of space form (i.e. the choice of value $\kappa$).  
Thus $Z$ is the {\em mean curvature sphere congruence} 
(perhaps first defined by Sophie Germain in the first half of 
the 19'th century), once we 
know that the spheres determined by $Z$ contain the corresponding 
points $X$ in the surface and are tangent to the surface, which 
follow from the next lemma.  (In particular, it is not necessary 
that the surface be of constant mean curvature.  We also note 
one must check that $Z$ and $X$ have common orientation as well, 
and we leave that to the reader.)  In fact, the 
next lemma reconfirms Lemma \ref{lem:X-perp-to-Z-and-dZ}:  

\begin{lemma}
$\langle X,Z \rangle = \langle dX,Z \rangle = 0$.  
\end{lemma}

\begin{proof}
$\langle X,Z \rangle \cdot I = -(XZ+ZX)/2 = 0$ is now immediate.  
$\langle dX,Z \rangle \cdot I = -(dX \cdot Z+Z \cdot dX)/2 = 0$ 
follows from $dx \cdot n_0 + n_0 \cdot dx = 0$, i.e. 
$dx$ and $n_0$ are perpendicular.  
\end{proof}

\begin{lemma} The mean curvature sphere congruence 
$Z$ can be characterized as the {\em conformal Gauss map} of the surface 
$X$ (a notion introduced by Robert Bryant), i.e. the unique 
enveloped sphere congruence 
that induces the same conformal structure as $X$.  
\end{lemma}

\begin{proof}
That $Z$ is 
the conformal Gauss map can be seen from the following computation 
(we do not show uniqueness here): 
\[
 \langle dZ,dZ \rangle = -h^2 (H_0^2 dx^2+H_0 
     (dxdn_0+dn_0dx)+dn_0^2)
\]\[
 = -h^2 x_u^2 (\tfrac{1}{4} (k_1+k_2)^2 (du^2+dv^2)-(k_1+k_2) 
    (k_1du^2+k_2dv^2)+k_1^2 du^2+k_2^2 dv^2)
\]\[
 = - \tfrac{1}{4} h^2 x_u^2 (k_1-k_2)^2 (du^2+dv^2) \; . \]
\end{proof}

\begin{lemma} The mean curvature sphere congruence $Z$ can also 
      be characterized as the {\em central sphere congruence} 
      (a notion perhaps first defined by Darboux), i.e. the 
      sphere congruence whose spheres exchange the principal curvature 
      spheres via inversion.
\end{lemma}

\begin{proof}
Let $X=X(u,v)$ be a surface.  Take 
\[ T = T(u,v) = \begin{pmatrix}
\ell & \ell_\infty \\ \ell_0 & -\ell 
\end{pmatrix} \in \mathbb{R}^{4,1} \] such that 
$||T||=1$ (i.e. $T$ lies in the de Sitter space $\mathbb{S}^{3,1}$) and 
\[ \langle T, Q \rangle = \langle T, X \rangle = 
\langle T, dX \rangle = 0 \; , \] 
with $Q$ as in \eqref{choiceofQ}.  
These conditions are equivalent 
to 
\begin{itemize}
\item $\ell^2+\ell_0 \ell_\infty = -1$, 
\item $\ell_\infty \kappa + \ell_0 = 0$, 
\item $\ell x + x \ell + \ell_\infty - x^2 \ell_0 = 0$, 
\item $\ell \cdot dx + dx \cdot \ell - 
(x \cdot dx + dx \cdot x) \ell_0 = 0$.  
\end{itemize}
Set \[ S_t = T + t X = \begin{pmatrix}
\ell & \ell_\infty \\ \ell_0 & -\ell 
\end{pmatrix} + \frac{2 t}{1-\kappa x^2} \begin{pmatrix}
x & -x^2 \\ 1 & -x 
\end{pmatrix} =: \begin{pmatrix}
z & z_\infty \\ z_0 & -z 
\end{pmatrix} \; . \]  Then $S_t$ also lies in 
$\mathbb{S}^{3,1}$ and 
is perpendicular to both $X$ and $dX$.  By Remark 
\ref{rem:tangentspheres}, the 
$S_t$ represent all of the tangent spheres to $X$.  Then, by 
Equation \eqref{Hkappa-vs-H0} and Lemma 
\ref{sphere-lemma} and a direct computation, 
the mean curvature of the sphere $S_t$ with respect to the 
space form $M_\kappa$ is 
\[ \frac{|z_0|}{2} - \frac{\kappa}{2 |z_0|} - 
\kappa \frac{|z_0|}{2} \frac{z^2}{z_0^2} = \pm t \; . \]  
Then, if $k_j$ are the principal curvatures of $X$, 
$S_{k_1}$ and $S_{k_2}$ are the principal curvature spheres.  
By Lemma \ref{lem:inversionthruspheres}, 
when $Z$ is the central sphere congruence, 
we should have that 
\[ 
S_{k_2} = S_{k_1} - 2 \langle S_{k_1} , Z \rangle \cdot Z 
\; . \] However, 
as we wish to have an inversion that preserves orientation 
rather than reversing it, we change $S_{k_2}$ to $-S_{k_2}$.  
This does not change the sphere itself, as $S_{k_2}$ is defined 
only projectively anyways.  Thus the equation becomes 
\begin{equation}\label{eqn:applicationofsphereinversion} 
- S_{k_2} = S_{k_1} - 2 \langle S_{k_1} , Z \rangle \cdot Z 
\; . \end{equation}  
Now the image of $S_{k_1}$ under inversion and $S_{k_2}$ itself 
will have the same orientation.  

We have that $Z=S_t$ for some $t$, and so we can now 
compute from 
\eqref{eqn:applicationofsphereinversion} that 
\[  t = \frac{1}{2} (k_1+k_2) \; , \] 
i.e. $t$ is the mean curvature. 
Thus the central sphere congruence is the same as the mean 
curvature sphere congruence.  
\end{proof}

\begin{lemma} The mean curvature sphere congruence 
      $Z$ can be characterized as the 
      sphere congruence that has second order contact with 
      the surface in orthogonal directions. 
\end{lemma}

\begin{proof}
Principal curvature spheres, second order contact 
and orthogonality are examples of 
M\"obius invariant notions, because they are invariant under 
M\"obius transformations (such as mapping from one space form 
to another, as in Figure 15, 
or inverting through a sphere).  Because only 
M\"obius invariant notions appear in this proof, without 
loss of generality we may assume that the surface $X(u,v)$ lies 
in $M_0=\mathbb{R}^3$.  

Let $Z$ be the mean curvature sphere at a point 
$X(u_0,v_0)$ of the surface.  Then 
$X(u_0,v_0)$ is one point of the sphere $Z$.
Let $p$ be a different point in $Z$ and let $S$ be 
a sphere with center $p$ that intersects $Z$ transversally.  
We apply inversion $f_S$ of 
$\mathbb{R}^3$ through the sphere $S$, so that the point $p$ is mapped to 
infinity and the sphere $Z$ is thus mapped to a flat plane 
$f_S(Z)$.  The image $f_S(X(u,v))$ 
of $X(u,v)$ under inversion will satisfy $H=0$ at the 
point $f_S(X(u_0,v_0))$.  Thus the asymptotic directions 
of $f_S(X(u,v))$ 
at that point are perpendicular to each other, and are also the 
directions of second order contact with $f_S(Z)$.  
This completes the proof.  
\end{proof}

We now explain the conserved quantity equation in terms of the 
Calapso transformation, in order to motivate a definition used 
in the discrete setting.  But first, we consider: 

\subsection{Inverses of quaternionic matrices and $\text{Mob}(3)$} 
The M\"obius transformations are the maps from $\mathbb{S}^3$ 
to $\mathbb{S}^3$ that take $2$-spheres to $2$-spheres.  We now 
describe them algebraically, using quaternionic matrices. 

First we need the following lemma, which follows from the 
properties $\text{Re}(x) = \text{Re}(\bar x)$ and 
$\text{Re}(xy) = \text{Re}(yx)$, where $x$ and $y$ are quaternions: 

\begin{lemma}\label{lem:abcd}
For $a,b,c,d \in H$, we have \[ 
\text{Re}(abcd)=\text{Re}(\bar d \bar c \bar b \bar a)=
\text{Re}(\bar b \bar a \bar d \bar c)=\text{Re}(bcda) \; . \]
\end{lemma}

For later use, we also give the following lemma: 

\begin{lemma}\label{lem:forlateruse}
Suppose that $p,q,r,s \in \text{Im} H$ and 
$(p-q)(q-r)(r-s)(s-p) \in \mathbb{R}$.  Then 
\[ (p-q)(q-r)(r-s)(s-p) = 
(s-p)(r-s)(q-r)(p-q) = 
\]
\[ (q-r)(p-q)(s-p)(r-s) = 
(q-r)(r-s)(s-p)(p-q) \; . 
\]
\end{lemma}

Take a quaternionic $2 \times 2$ matrix 
\[ T = \begin{pmatrix}
a & b \\ c & d
\end{pmatrix} \; . \]  
Then the Study determinant $[T]$ of $T$ is the determinant of 
$T \cdot \bar T^t$, i.e. 
\[
 [T] = (a \bar a+b \bar b)
 (c \bar c+d \bar d)
- (a \bar c+b \bar d)
 (c \bar a+d \bar b) = 
|a|^2|d|^2+|b|^2|c|^2- b \bar d c \bar a - a \bar c d \bar b \; , 
\] and this is clearly a real number.  (Note that $\det T$ 
itself is not a well defined notion, as quaternions do not commute.)  
When $[T] \neq 0$, we can 
define the inverse of $T$ as 
\[ T^{-1} = \frac{1}{[T]} 
\begin{pmatrix}
|d|^2 \bar a-\bar c d \bar b & |b|^2 \bar c-\bar a b \bar d \\
|c|^2 \bar b-\bar d c \bar a & |a|^2 \bar d-\bar b a \bar c 
\end{pmatrix} \; . \]  
One can check that $T^{-1} T = T T^{-1} = I$, by using Lemma \ref{lem:abcd}.  

In general, for 
$A \in \mathbb{R}^{4,1}$, $T A T^{-1}$ might not lie in 
$\mathbb{R}^{4,1}$, i.e. 
we might not have 
\[
  T A T^{-1} = \begin{pmatrix}
x & x_\infty \\ x_0 & -x
\end{pmatrix} \; , \;\;\; x_0 , x_\infty \in \mathbb{R} \; , \;\;\; x \in 
\Im H \; . \]  To avoid such a problem, 
we define a set, call in $G$, as all $2 \times 2$ 
quaternionic matrices $T$ so that 
\begin{equation}\label{eqn:Mob3}
  \bar b d + \bar d b = \bar a c + \bar c a = 0 \; , 
\;\;\; \bar a d + \bar c b \in \mathbb{R} 
\end{equation}
and 
\begin{equation}\label{eqn:Mob3-nonsing}
  \bar a d+ \bar c b \neq 0 \; . \end{equation} 
Then \eqref{eqn:Mob3-nonsing} implies 
\[ [T] = (\bar a d+\bar c b) \overline{(\bar a d+\bar c b)}> 0 \] 
and $T \in G$ will have an inverse.  By 
Equation \eqref{eqn:Mob3} and the fact that 
$\bar a d+ \bar c b \neq 0$, we find, when 
\[ A = \alpha \begin{pmatrix}
x & x_\infty \\ x_0 & -x 
\end{pmatrix} \; , \;\;\; 
x \in \Im H \; , \;\;\; \alpha, x_0, x_\infty \in \mathbb{R} \; , 
\] that 
\[ [T] \cdot T A T^{-1} = \]\[ 
= \alpha (\bar a d+\bar c b) \cdot \left[
\begin{pmatrix}
a x \bar d - b x \bar c & a x \bar b - b x \bar a \\
c x \bar d - d x \bar c & c x \bar b -d x \bar a
\end{pmatrix}
+ x_0 \begin{pmatrix}
b \bar d & b \bar b \\
d \bar d & d \bar b
\end{pmatrix}
+ x_\infty \begin{pmatrix}
a \bar c & a \bar a \\
c \bar c & c \bar a 
\end{pmatrix} \right] \in \mathbb{R}^{4,1} \; . \]  
In particular, we have $T A T^{-1} \in \mathbb{R}^{4,1}$.  If, 
furthermore, $x_0=1$ and $x_\infty=-x^2$, i.e. $A \in L^4$, 
then we also have 
that $T A T^{-1}$ is in $L^4$, and this follows from the property 
\[ [ T_1 T_2 ] = [ T_1 ] \cdot [ T_2 ] \; , \;\;\; T_j \in 
G \; , \] for Study determinants.  

The set $G$ is a group.  For example, if 
$T_1$ and $T_2$ are in $G$, then so is $T_1T_2$.  

When $A$ is an element of $L^4$, and consequently $TAT^{-1}$ 
is as well, the ratio of the upper left and lower left 
entries of $T A T^{-1}$ will be 
$(a x + b)(c x + d)^{-1} \in \Im H$.  In this way, 
the $T \in G$ are related to M\"obius transformations by 
\begin{equation}\label{eqn:tmpmap} \left( T = \begin{pmatrix} a & b \\ c & d
\end{pmatrix} \right) * x = (a x+b)(c x+d)^{-1} \; . \end{equation} 
Thus, when $A \in L^4$, then $T A T^{-1}$ is 
essentially the same map as 
\eqref{eqn:tmpmap}.  

Note that, in Equation \eqref{eqn:tmpmap}, we place 
the inverse $(c x + d)^{-1}$ to the right side of $(ax+b)$.  
(Because quaternions do not commute, placing this term on the 
other side would not give the same result.)  

\begin{remark}
In fact, Equation \eqref{eqn:tmpmap} gives both orientation 
preserving and orientation reversing M\"obius transformations.  
For example, $\text{Im} H \ni x \to 
(0 \cdot x + 1)(1 \cdot x + 0)^{-1}=x^{-1}=-x/|x|^2 \in 
\text{Im} H$ is orientation preserving, while 
$\text{Im} H \ni x \to 
(0 \cdot x - 1)(1 \cdot x + 0)^{-1}=-x^{-1}=x/|x|^2 \in 
\text{Im} H$ is orientation reversing.  
\end{remark}

Furthermore, because 
\[
  \langle TAT^{-1}, TBT^{-1} \rangle = 
  \langle A,B \rangle
\]
for any $A,B \in \mathbb{R}^{4,1}$, the map 
$A \to TAT^{-1}$ is an isometry of $\mathbb{R}^{4,1}$.  

When $\kappa \neq 0$, i.e. when $Q$ as in 
\eqref{choiceofQ} is not null, $M_\kappa$ 
has a particular M\"obius transformation called the antipodal map, 
which we now describe: 
A point $X$ in $M_\kappa$ can be decomposed as 
\[ X = \frac{2}{1-\kappa x^2} \begin{pmatrix}
x & -x^2 \\ 1 & -x 
\end{pmatrix} = \mathcal{A} + \kappa^{-1} Q \; , 
\]
where 
\[ \mathcal{A} = \frac{2}{1-\kappa x^2} \begin{pmatrix}
x & -x^2 \\ 1 & -x 
\end{pmatrix} -\kappa^{-1} Q 
\] is perpendicular to $Q$.  The antipodal map is 
\[ \mathcal{A} + \kappa^{-1} Q  \to - \mathcal{A} + \kappa^{-1} Q \; , \] 
that is, we are moving the point $X$ to another point in $M_\kappa$ that is 
on the opposite side of $Q$.  Since 
\[ - \mathcal{A} + \kappa^{-1} Q = 
\frac{2}{1-\kappa (\kappa^{-1} x^{-1})^2} \begin{pmatrix}
\kappa^{-1} x^{-1} & -(\kappa^{-1} x^{-1})^2 \\ 1 & - \kappa^{-1} x^{-1} 
\end{pmatrix} \; , \]  
the antipodal map is the M\"obius transformation $x \to \kappa^{-1} x^{-1}$.  

\begin{remark}
M\"obius transformations of the ambient space preserve the 
conformal structure of 
the space, so will preserve the conformal structure of 
any surface inside the space as well.  Furthermore, 
M\"obius transformations 
will preserve contact order of any spheres tangent to the 
surface, and so will preserve the principal curvature 
spheres.  It follows that if $x(u,v)$ is an isothermic 
parametrization of a surface, it will remain an isothermic 
parametrization 
even after a M\"obius transformation is applied.  
\end{remark}

\begin{defn}
$\text{Mob(3)}$ is the collection of M\"obius transformations 
$L^4 \ni A \to TAT^{-1} \in L^4$, for $T \in G$.  
\end{defn}

\begin{remark}
$\text{Mob(3)}$ is a 10-dimensional object, while $G$ is an 
11-dimensional object.  
\end{remark}

Now we make some comments about M\"obius transformations in 
general dimension.  
M\"obius transformations of the $n$-dimensional sphere 
$\mathbb{S}^n$, $n \geq 2$, 
are maps that take points in $\mathbb{S}^n$ to points 
in $\mathbb{S}^n$ and also hyperspheres in $\mathbb{S}^n$ to 
hyperspheres in $\mathbb{S}^n$.  We denote 
the collection of M\"obius transformations of 
$\mathbb{S}^n$ by Mob($n$).  We have the following facts: 

\smallskip

{\bf Fact:} 
Let Conf$_g$($n$) denote the global conformal transformations of all of 
$\mathbb{S}^n$, and let Conf$_\ell$($n$) denote the 
local conformal transformations of local domains of $\mathbb{S}^n$.  Then 
\[ \text{Mob}(2)=\text{Conf}_g(2) \subset Conf_\ell(2) \; , 
\;\;\; \text{Conf}_g(2) \neq Conf_\ell(2) \]
and \[ \text{Mob}(n)=\text{Conf}_g(n) = Conf_\ell(n) \] for $n \geq 3$.  

\smallskip

The reason the case $n=2$ is different is that holomorphic functions 
from $\mathbb{C}$ to $\mathbb{C}$ can be pulled back by the inverse of 
stereographic projection to maps from $\mathbb{S}^2$ to $\mathbb{S}^2$, 
and those maps are generally in $\text{Conf}_\ell(2)$ but not in 
$\text{Conf}_g(2)$.  This occurs only in the case $n=2$.  

\smallskip 

{\bf Fact:} 
Let $f \in \text{Conf}_\ell(n)$ with $n \geq 3$.  Let $M 
\subset \mathbb{S}^n$ be a smooth hypersurface.  Then $p \in M$ 
is an umbilic point if and only if $f(p) \subset f(M)$ 
is an umbilic point as well.  

\smallskip

\begin{remark}
$O(n+1,1)$ is the set of orthogonal transformations of 
$\mathbb{R}^{n+1,1}$, and 
these transformations preserve the set of lines in the light cone, as well 
as the set of spacelike lines, and 
\[ \frac{O(n+1,1)}{ \{ \pm I \} } \subset \text{Mob}(n) \; , 
\]
and in fact, these two sets are equal.  
\end{remark}

\begin{remark}
Projective transformations are maps from 
projective space $P\mathbb{R}^{n+2}$ to $P\mathbb{R}^{n+2}$, 
i.e. from lines in $\mathbb{R}^{n+2}$ through the origin to lines in 
$\mathbb{R}^{n+2}$ through the origin, so that "lines of lines", which can 
generically be represented by lines in $\{x_{n+2}=1\}$ ($x_{n+2}$ is the 
final Cartesian coordinate of points 
$(x_1,...,x_{n+2}) \in \mathbb{R}^{n+2}$), get mapped to "lines of lines".  
The fundamental theorem of projective geometry, a nontrivial result, 
is this: {\em Any projective transformation comes from a linear map of 
$\mathbb{R}^{n+2}$.}  
Then, regarding $\mathbb{R}^{n+2}$ as $\mathbb{R}^{n+1,1}$ (i.e. 
changing the Euclidean metric to the Minkowski metric), 
those projective transformations that preserve the light cone are 
equivalent to Mob($n$).  
\end{remark}

\subsection{Calapso transformations}

In the following definition, the surface $x$ lies in some 
space form $M$, but since we are dealing with a M\"obius 
geometric notion, the choice of space form will not matter.  

\begin{defn}\label{defnthatis8pt43}
Let $x=x(u,v)$, with associated $X=X(u,v) \in M$, 
be an immersed surface with isothermic coordinates $u,v$.  A 
{\em Calapso transformation} $T \in \text{Mob}(3)$ is a 
solution of \[ T^{-1} dT = \lambda \tau \; . \]  Then the transform 
$\text{Im} H \ni x \to T * x \in \text{Im} H$, where $*$ denotes 
the M\"obius transformation as in \eqref{eqn:tmpmap}, or 
equivalently $L^4 \ni X \to T X T^{-1} \in L^4$, 
is a Calapso transform.  (We can also 
call it a {\em $T$-transform} or {\em conformal deformation}.)  
\end{defn}

\begin{remark}\label{rem:stayinMob3}
If we write $T$ as 
\[ T = \begin{pmatrix} 
a & b \\ c & d \end{pmatrix} \; , \] then the equation 
$dT=T \cdot \lambda \tau$ gives 
\[ (\bar b d + \bar d b)_u = 
(\bar b d + \bar d b)_v = 
(\bar a c + \bar c a)_u = 
(\bar a c + \bar c a)_v = 
(\bar a d + \bar c b)_u = 
(\bar a d + \bar c b)_v = 0 \; , \] 
and so if the initial condition for the solution $T$ lies 
in $\text{Mob}(3)$, then $T$ will always lie in 
$\text{Mob}(3)$.  
\end{remark}

\begin{remark}
Although we will not be particularly precise about this, we will 
generally use the word "transformation" when the object under consideration 
is a procedure toward a separate goal, and the word "transform" 
for the desired goal.  
\end{remark}

The Calapso transformation is classical, and was studied by 
Calapso, Bianchi and Cartan.  It preserves the conformal structure 
and is thus of interest in M\"obius geometry.  In the case that the 
starting surface is CMC, it is the same as 
the Lawson correspondence (see Remark 
\ref{rem:9point27}), which is an important transformation 
in the differential geometry of CMC surfaces.  

\begin{lemma}\label{lem:Calapexist}
If $x$ is isothermic, then Calapso transformations exist.  
\end{lemma}

\begin{proof}
The compatibility condition for the system 
\[ T^{-1} T_u = \lambda U \; , \;\;\; U = 
\begin{pmatrix} x \\ 1 \end{pmatrix} x_u^{-1} 
\begin{pmatrix} 1 & -x \end{pmatrix} \; , \] 
\[ T^{-1} T_v = \lambda V \; , \;\;\; V = - 
\begin{pmatrix} x \\ 1 \end{pmatrix} x_v^{-1} 
\begin{pmatrix} 1 & -x \end{pmatrix} \] 
to have a solution $T$ is 
\[ \lambda (U V - V U) + V_u - U_v = 0 \; , \]  and this 
condition holds 
precisely because of the conditions for isothermicity, that is 
\begin{equation}\label{eqn:isothermicityproperties}
x_u^2=x_v^2 \; , \;\; x_ux_v+x_vx_u=0 \; , \;\; 
x_{uv}=A x_u+B x_v \end{equation} for some functions $A,B$.  
By Remark \ref{rem:stayinMob3}, we have that $T$ always 
lies in $\text{Mob}(3)$ if it does at any one point, 
i.e. if the initial condition for $T$ is chosen to be in 
$\text{Mob}(3)$, completing the proof.  
\end{proof}

If the surface $x$ has a linear conserved quantity 
$P=Q+\lambda Z$, then 
\[ dP + \lambda \tau P - P \lambda \tau = 0 \] 
holds, i.e. 
$dP + T^{-1} dT \cdot P - P \cdot T^{-1} dT = 0$, which is equivalent 
to \begin{equation}\label{eqn:forpg68} d (T P T^{-1}) = 0 
\; , \end{equation}
that is to say, $T P T^{-1}$ is constant. It is 
$T P T^{-1}$ being constant that we will 
use to define discrete CMC surfaces, just as it defines 
smooth CMC surfaces, by Theorem \ref{theBCtheorem}.  

In M\"obius geometry (in the space $\mathbb{R}^{4,1}$), isothermic 
surfaces 
are deformable (Calapso transformations), and this deformation preserves 
second order invariants in M\"obius geometry, such as conformal class, 
and conformal class of the trace-free second fundamental form.  
(Note that for surfaces 
in Euclidean geometry, a nontrivial deformation will never 
preserve the second order invariants, i.e. the first and second 
fundamental forms, of Euclidean geometry.)  

\begin{remark}
Because \[ \lambda \tau = T^{-1} dT \; , \] 
$\lambda \tau$ 
can be thought of as the logarithmic derivative of the Calapso 
transformation.  
\end{remark}

\subsection{Darboux transformations} 
For smooth surfaces, a Darboux transform is one such that 
\begin{itemize}
\item there exists a sphere congruence enveloped by the original surface
      and the transform, 
\item the correspondence, given by the sphere congruence, from 
      the original surface to the other enveloping surface (i.e. 
      the transform), preserves curvature lines, 
\item this correspondence preserves conformality.  
\end{itemize}

However, we will define Darboux transformations in a different way, 
as in the following definition: 

\begin{defn}\label{defn:smoothDarbTrans}
Let $T$ be a Calapso transformation of $X$.  Then $\hat X$ in $PL^4$ is a 
Darboux transformation of $X$ if $T \cdot \hat{X} := T \hat{X} T^{-1}$ is 
constant in $PL^4$ for some choice of $\lambda$.  
\end{defn}

We can refer to the equation that $T \cdot \hat{X}$ is constant as 
{\em Darboux's linear system}.  

Here \[ T \hat{X} T^{-1} = 
T \left( \hat \alpha \begin{pmatrix} \hat x & -\hat x^2 \\ 1 & -\hat x 
\end{pmatrix} \right) T^{-1} \] being constant in $PL^4$ means that 
\begin{equation}\label{eqn-pre-riccati} 
d (r T \begin{pmatrix} \hat x & -\hat x^2 \\ 1 & -\hat x 
\end{pmatrix} T^{-1}) = 0 \end{equation} 
for some function $r \in \mathbb{R}$.  This is equivalent to the equation 
\begin{equation}\label{eqn-riccati} 
d\hat{x} = \lambda (\hat{x}-x) dx^* (\hat{x}-x) \; , \end{equation}
as we now show: 

\begin{lemma}\label{8ptpt50} 
Equations \eqref{eqn-pre-riccati} and \eqref{eqn-riccati} are equivalent.  
\end{lemma}

\begin{proof}
Equation \eqref{eqn-pre-riccati} is equivalent to the following 
four equations: 
\[ dr + r \lambda (dx^* (\hat x - x) + (\hat x - x) dx^*) = 0 \; , \] 
\[ (x dx^*+dx^* x) \hat x - \hat x (x dx^*+dx^* x) = 0 \; , \]  
\[ dr \cdot \hat x + r \lambda (x dx^* \hat x - x dx^* x - 
\hat x x dx^* + \hat x^2 dx^*) + r d\hat x = 0 \; , \]  
\[ -dr \cdot \hat x^2 + r \lambda (-x dx^* \hat x^2+x dx^* x \hat x+
\hat x x dx^* x - \hat x^2 dx^* x) - r \hat x d\hat x - 
r d\hat x \cdot \hat x= 0 \; . \]  
Note that $dx^* (\hat x - x) + (\hat x - x) dx^*$ is real-valued, 
so the first equation will define the real-valued function $r$.  
Also, note that $x dx^*+dx^* x$ is real as well, so the second equation 
automatically holds.  Substituting $dr$ from the first equation into 
the third equation, one arrives at Equation \eqref{eqn-riccati}.  
The fourth equation is then automatically true, again using that 
$x dx^*+dx^* x$ is real.  
\end{proof}

An equation of the form $y^\prime=f(y)$, where $f(y)$ is a quadratic 
polynomial, is called a Riccati equation, so Equation 
\eqref{eqn-riccati} is a Riccati-type partial differential equation 
(where $y$ becomes $\hat x$).  Because of this, at the end of this 
chapter we include a short appendix containing some well-known 
facts about the Riccati equation.  

Equation \eqref{eqn-riccati} is in turn equivalent to the matrix product 
\[ T \begin{pmatrix} \hat{x} \\ 1 \end{pmatrix} \] being constant, 
which means that 
\[ d (T \begin{pmatrix} \hat x \\ 1 \end{pmatrix} h) = 0 \] for some 
quaternionic-valued function $h \in H$.  

\begin{remark}
Note that we could rescale $\hat X$ in Definition 
\ref{defn:smoothDarbTrans} so that $T \cdot \hat X$ is not 
only constant in $PL^4$, but is constant in $L^4$ as well, 
if we wish.  
\end{remark}

\begin{remark}\label{rem:7point48}
When the surface $x$ has a linear conserved quantity $Q+\lambda Z$, 
one possibility for a Darboux transform 
is to take $\hat X = Q + \lambda_0 Z$ with 
$\lambda=\lambda_0$ 
chosen so that $||\hat X||=0$.  This would be a 
special case of a Baecklund 
transform (called a "complementary surface", and 
we will come back to this in Chapter \ref{section:p.c.q.s}, 
after we have defined polynomial conserved quantities).  
\end{remark}

\begin{remark}\label{rem:7point49}
We do not define Baecklund transforms until after we have defined 
polynomial conserved quantities in Chapter 
\ref{section:p.c.q.s}.  However, for 
now, let us just mention that 
more general Baecklund transforms can be obtained by this recipe: 
\begin{itemize}
\item we take a surface $x$ with a linear conserved 
quantity $P = Q + \lambda Z$, 
\item we pick a value $\lambda = \mu$, 
\item we pick an initial condition $\hat x_p$ for a possible surface 
$\hat x$, at some point $p$ in the domain of $x$, 
such that \[ \begin{pmatrix}
\hat x_p & -\hat x_p^2 \\ 1 & -\hat x_p \end{pmatrix} \perp P(\mu)_p \; , \]
\item we solve the Riccati equation \eqref{eqn-riccati} for $\hat x$.  
\end{itemize}
Actually, we can choose either $\mu$ or $\hat x_p$ first, and then choose the 
other.  
This gives a $3$-parameter family of Baecklund transformations, generally not 
preserving topology of the surface $x$ of course (when $x$ is not simply 
connected).  
\end{remark}

We now give a characterization of CMC surfaces in 
terms of Christoffel and Darboux transformations, see Theorem 
\ref{lem:characterizationOfCMC} below.
First we give some preliminary results.  

\begin{lemma}\label{vecAlemma}
If $v_1,v_2 \in \text{Im} H$ and $|v_1|=|v_2|$, then there exists 
$\vec a \in H$ such that $\vec a v_1 \vec a^{-1} = v_2$.  
\end{lemma}

\begin{proof}
The idea is to show that any rotation of $\mathbb{R}^3 \approx 
\text{Im}H$ can be 
written as $\text{Im} H \ni w \to \vec a w \vec a^{-1} \in 
\text{Im} H$ for some $\vec a \in H$.  Set 
$\vec a = \cos \theta + v \cdot \sin \theta$, for some arbitrary 
$v \in \text{Im} H$, $|v|=1$, 
and then $\vec a^{-1} = \cos \theta - v \cdot \sin \theta$.  If $w$ is 
parallel to $v$, then $w = \lambda v$ for some $\lambda \in 
\mathbb{R}$ and 
$\vec a w \vec a^{-1}=w$.  If $w$ is perpendicular to $v$, 
then $\vec a w \vec a^{-1} = 
\cos (2 \theta) w + \sin (2 \theta) (v \times w)$, which is a 
rotated image by angle $2 \theta$ of $w$ about $v$.  So $w \to 
\vec a w \vec a^{-1}$ represents an arbitrary rotation of $\mathbb{R}^3$.  
\end{proof}

The following is easily shown: 

\begin{lemma}\label{lem:beforecharofCMC}
The $\vec a$ in Lemma \ref{vecAlemma} is unique up to choices 
$r_1 \vec a + r_2 \vec a v_1$ for $r_1,r_2 \in \mathbb{R}$.  
\end{lemma}

We now come to that characterization of CMC surfaces: 

\begin{theorem}\label{lem:characterizationOfCMC}
A smooth surface $x$ in 
$\mathbb{R}^3$ has constant mean curvature if and only if 
some scaling and translation of the Christoffel transform $x^*$ equals 
a Darboux transform $\hat x$ (given by some specific value of $\lambda$).  
\end{theorem}

\begin{proof}
Assume $x$ is a CMC surface.  
Then $x^*$ is the parallel CMC surface, by Remark \ref{rem:CMCparallelsurf}.  
To show $x^*$ is a Darboux transformation, we must show, by Definition 
\ref{defn:smoothDarbTrans} and 
Equation \eqref{eqn-riccati}, that $dx^* = \lambda (x^*-x) dx^* (x^*-x)$ 
for some $\lambda \in \mathbb{R}$.  Because $x^*$ is the parallel 
CMC surface, we have 
$x^* = x + H_0^{-1} n_0$, and then taking $\lambda = -H_0^2/n_0^2$ 
gives that $x^*$ is a Darboux transform.  

Now we show the converse direction, proven 
by Udo Hertrich-Jeromin and Franz Pedit in the paper \cite{HP}.  
Assume $\hat x$ is a Darboux transform of $x$, 
and that $\hat x = a \cdot x^* + \vec b$ for some constants 
$a \in \mathbb{R} \setminus \{ 0 \}$ and $\vec b \in \text{Im}H$.  
So there exists $\lambda$ such that $d\hat x = 
\lambda (\hat x - x)dx^*(\hat x - x)$, that is, 
\[ a dx^* = \lambda (ax^*+\vec b - x)dx^*(ax^*+\vec b-x) \; . \]  
Thus 
\[ a dx^* = - \lambda 
(ax^*+\vec b - x)dx^*(ax^*+\vec b - x)^{-1}|ax^*+\vec b-x|^2 \; . 
\]  
Because $(ax^*+\vec b-x)dx^*(ax^*+\vec b-x)^{-1}$ has the same 
norm as that of $dx^*$, we have that $|\hat x - x|^2 = 
\pm a \lambda^{-1}$ is constant.  

Suppose \[ |\hat x - x|^2 = - a \lambda^{-1} \; . \]
Lemma \ref{lem:beforecharofCMC} implies 
\[ ax^*+\vec b - x = r_1 \cdot 1 + r_2 \cdot 1 \cdot x_u^{-1} = 
r_3 \cdot 1 + r_4 \cdot 1 \cdot x_v^{-1} \] for some 
$r_j \in \mathbb{R}$, so linear independence of $x_u^{-1}$ and $x_v^{-1}$ 
gives 
\[ \text{Im}H \ni a x^*+\vec b - x = r \in \mathbb{R} \; , \] for some 
real constant $r$.  Thus $r=0$ and $x=\hat x=a x^*+ \vec{b}$, which is a 
contradiction.  

Thus we have \[ |\hat x - x|^2 = + a \lambda^{-1} \]
is constant.  Now again, Lemma \ref{lem:beforecharofCMC} implies 
\[ a x^*+\vec b - x = r_1 n_0+r_2 n_0 x_u^{-1} = 
r_3 n_0 + r_4 n_0 x_v^{-1} \; . \]  
So $ax^*+\vec b - x = r \cdot n_0$ for some constant $r \in 
\mathbb{R}$.  So $dx^* = a^{-1} dx + r a^{-1} dn_0$.  Definition 
\ref{defn:ChristTrans} implies $x$ has CMC $H_0 = \pm r^{-1}$.  
\end{proof}

\begin{corollary}
Let $x$ be a CMC surface in $\mathbb{R}^3$.  
Let $\hat x$ be both a Christoffel 
and Darboux transform, as in Theorem \ref{lem:characterizationOfCMC}.
Then, $|\hat x -x|^2$ is constant, and $\hat x - x$ is perpendicular 
to $x$, and $\hat x $ is a parallel surface of $x$ up to scaling and 
translation.  
\end{corollary}

\subsection{Other transformations} 
Here we make some brief remarks about two other transformations.  
The interested reader 
can consult other sources for more complete information about them.  

If one disregards some degenerate cases, 
Ribaucour transforms (like Darboux transforms) preserve curvature lines, 
but (unlike Darboux transforms) they do not necessarily preserve 
the conformal structure.  A simple example of a Ribaucour transform of 
a surface in $\mathbb{R}^3$ is its reflection across a plane, which is not a 
Darboux transform.  So Ribaucour transformations are more general than 
Darboux transforms.  

In the case of a CMC $H \neq 0$ surface, a Goursat transformation is 
the composition of three transformations, first a Christoffel transformation, 
second a M\"obius transformation, and third another Christoffel 
transformation.  

In the case of a minimal surface, a Goursat transformation is as follows: 
lift a minimal surface to a null curve in $\mathbb{C}^3$, apply a complex 
orthogonal transformation to that null curve, and then project back 
to $\mathbb{R}^3$.  It is a 
M\"obius transformation for the Gauss map.  One example of this is a 
catenoid being transformed into a minimal surface that is defined on the 
universal cover of the annulus, and a picture of this can be 
found in Section 5.3 of \cite{Udo-bk}.  

\subsection{Appendix: comments on the Riccati equation}
As promised before when we discussed Darboux transformations, 
we include some basic facts here about the Riccati equation
\[ y^\prime(x)=a(x) (y(x))^2 + b(x) y(x) + c(x) \; , \;\;\; 
a(x) \neq 0 \; . \]
Set $v = a \cdot y$, and then 
\[ v^\prime = v^2 + R v + S \; , \;\;\; 
R = a^{-1} a^\prime + b \; , \;\; S = a c \; . \]  
Let $u$ satisfy $v = -u^\prime/u$.  Then 
\[ u^{\prime\prime} - R u^\prime + S u = 0 \; , \] 
which is a linear second order ordinary differential equation, 
so there is a method for finding all solutions $u$.  
Taking any such solution $u$, we have one solution 
\[ y_0 = \frac{-u^\prime}{a u} \] to the Riccati equation.  
From $y_0$ we can obtain all solutions $y$ to the Riccati 
equation as follows: Let $y$ be any solution, and define 
$z$ by $y = y_0 +z^{-1}$.  Then $(y_0+1/z)^\prime = 
a (y_0+1/z)^2+ b (y_0+1/z) + c$, and because $y_0$ itself 
is also a solution, we have 
\[ z^\prime + (2 a y_0 + b) z + a = 0 \; . \]  
This is a linear first order ordinary differential equation, 
so again all solutions $z$ can be found.  These solutions 
$z$ then give the general solutions $y = y_0+1/z$ of the 
Riccati equation.  

\begin{remark}
The Schwarzian derivative $S(w)$ of a function $w$ is 
\[ S(w) := \left( \frac{w^{\prime\prime}}{w^\prime} 
\right)^\prime 
- \frac{1}{2} \left( \frac{w^{\prime\prime}}{w^\prime} 
\right)^2 \; . \]  
It has the property that it is invariant under M\"obius 
transformations of $w$.  It is also related to CMC surface 
theory, and, in particular, it is very useful in the study 
of CMC $1$ surfaces in hyperbolic $3$-space 
$\mathbb{H}^3$ via the Weierstrass representation found by 
Bryant \cite{Bry} and developed further by 
Umehara-Yamada \cite{UY1}.  Consider the equation 
\[ S(w(x)) = f(x) \; . \]  
We wish to find a solution $w$.  Setting $y = 
w^{\prime\prime}/w^\prime$, we have the Riccati equation 
\[ y^\prime = \tfrac{1}{2} y^2+f \; . \]  We take 
$u$ as above solving 
\begin{equation}\label{eqnforRiccatieqn}
u^{\prime\prime}-R u^\prime+S u = u^{\prime\prime} 
+ \tfrac{1}{2} f u = 0 \; . \end{equation}  
We have $y = -2 u^\prime/u$.  
We then integrate $(w^{\prime\prime}/w^\prime) = 
-2 (u^\prime/u)$ to see that $w^\prime = c u^{-2}$ for 
some constant $c$.  Any other solution $\tilde u$ of 
\eqref{eqnforRiccatieqn} will give that $\tilde u ^\prime u 
- \tilde u u^\prime$ is constant, so we can take 
\[ w^\prime = \frac{\tilde u^\prime u 
- \tilde u u^\prime}{u^2} = \left( \frac{\tilde u}{u} 
\right)^\prime \; . \]  This implies that we can take 
the solution $w$ to be $w = \tilde u/u$.  
\end{remark}

\section{A conserved quantities approach to 
discrete CMC surfaces}
\label{chapondiscreteCMCsurfs}

Our purpose in this chapter is to present a definition for 
discrete constant mean curvature (CMC) $H$ surfaces in any of 
the three space forms Euclidean 3-space $\mathbb{R}^3$, spherical 
3-space $\mathbb{S}^3$ and hyperbolic 3-space $\mathbb{H}^3$.  
This new definition is equivalent to the previously known 
definitions \cite{BobPink} in the case of 
$\mathbb{R}^3$ (and we will show this 
in this text as well, in Lemmas \ref{oldsenselemma1} and 
\ref{lemmalemma9pt13}).  It 
also satisfies a Calapso transformation relation (the Lawson 
correspondence), suggesting the 
definition is also natural for the space form $\mathbb{S}^3$, and 
for CMC surfaces with $H \geq 1$ in $\mathbb{H}^3$.  The 
definition is the 
first one known for CMC surfaces with $-1 < H < 1$ in 
$\mathbb{H}^3$.  

This chapter falls under the category of 
``discrete differential geometry'', which is sometimes 
abbreviated as ``DDG'', and many researchers now work in this 
and related fields.  Here we list some of those researchers, but 
we first note that this list includes only people whose work 
is in some way related to the viewpoint presented in this text -- 
and even with this restriction is by no means a complete list: 
Sergey Agafonov, 
Andreas Asperl, 
Alexander Bobenko, 
Christoph Bohle, 
Folkmar Bornemann, 
Ulrike Buecking, 
Fran Burstall, 
Adam Doliwa, 
Charles Gunn, 
Udo Hertrich-Jeromin, 
Michael Hofer, 
Tim Hoffmann, 
Ivan Izmestiev, 
Michael Joswig, 
Axel Kilian, 
Yang Liu, 
Vladimir Matveev, 
Christian Mercat, 
Franz Pedit, 
Paul Peters, 
Ulrich Pinkall, 
Konrad Polthier, 
Helmut Pottmann, 
Jurgen Richter-Gebert, 
Wolfgang Schief, 
Jean-Marc Schlenkev, 
Nicholas Schmitt, 
Oded Schramm, 
Peter Schroeder, 
Boris Springborn, 
John Sullivan, 
Yuri Suris, 
Johannes Wallner, 
Wenping Wang, 
Max Wardetzky. 

\subsection{Discrete isothermic surfaces}
Consider a discrete surface $\ef_p \in \text{Im} H$ (recall that 
$\text{Im} H$ is the imaginary quaternions), which we can 
consider to be a discrete surface in Euclidean $3$-space, since 
$\text{Im} H$ is equivalent to $\mathbb{R}^3$ as a vector space (and we 
sometimes say this by writing 
$\text{Im} H \approx \mathbb{R}^3$).  Here $p$ is any 
point in a discrete lattice domain (locally always a subdomain of 
$\mathbb{Z}^2$).  Consider any quadrilateral in the lattice with vertices 
$p$, $q$, $r$, $s$ (i.e. the points $(m,n)$, $(m+1,n)$, 
$(m+1,n+1)$, $(m,n+1)$, respectively, for some $m,n \in 
\mathbb{Z}$) ordered counterclockwise about the quadrilateral.  

We change the notation "$x$" for surfaces in 
the previous chapter to "$\ef$" here.  This is for 
distinguishing between smooth surfaces, always denoted by "$x$", 
and discrete surfaces, always denoted by "$\ef$".  

It would be natural to 
assume that the points $\ef_p$, $\ef_q$, $\ef_r$ and $\ef_s$ are coplanar, 
so that they are the vertices of a planar quadrilateral in $\mathbb{R}^3$, 
and thus the surface is comprized of planar quadrilaterals connecting 
continuously along edges.  It is even better if the 
points $\ef_p$, $\ef_q$, $\ef_r$ and $\ef_s$ are concircular (i.e. 
all lie in one circle), because then we could extend the notion of 
a surface comprized of planar quadrilaterals to the cases that the 
ambient space is $\mathbb{S}^3$ or 
$\mathbb{H}^3$, cases which we will consider later in 
this chapter.  In fact, once the vertices are concircular, there is 
actually no 
further need to think about "planar faces", as all the 
necessary information 
is encoded in the circle itself. We will soon restrict to the 
concircular case, but for the moment we make no assumptions about the 
positioning of $\ef_p$, $\ef_q$, $\ef_r$ and $\ef_s$.  

We define the cross ratio of this quadrilateral as 
\[ q_{pqrs} = (\ef_q-\ef_p)(\ef_r-\ef_q)^{-1}
(\ef_s-\ef_r)(\ef_p-\ef_s)^{-1} \; . \]
(We are using $q$ to denote both the cross ratio and 
one vertex of the quadrilateral, but this will 
not cause confusion, since it will always be clear from 
context which meaning $q$ has in each case.)  

This cross ratio is not invariant with respect to 
conformal transformations of $\mathbb{R}^3$, but such an invariance 
{\em almost} holds, in the sense that we can produce a conformally 
invariant version of the cross ratio by changing it into a 
complex valued object, defined up to conjugation, as follows: 
\[ \hat{q}_{pqrs} = \text{Re}(q_{pqrs}) \pm i 
||\text{Im}(q_{pqrs})|| \; . \]  

\begin{lemma}
$\hat{q}_{pqrs}$ is a M\"obius invariant.  
\end{lemma}

\begin{proof}
Applying the following maps to the space $\text{Im}H$: 
\[ ai+bj+ck \to rai+rbj+rck \; , \] 
\[ ai+bj+ck \to ai+bj+ck + (a_0i+b_0j+c_0k) \; , \] 
\[ ai+bj+ck \to -ai+bj+ck \; , \] 
\[ ai+bj+ck \to (\cos(\theta) a-\sin(\theta) b)i+
(\sin(\theta) a+\cos(\theta) b)j+ck \; , \] 
\[ ai+bj+ck \to (\cos(\theta) a-\sin(\theta) c)i+bj+
(\sin(\theta) a+\cos(\theta) c)k \; , \] 
\[ ai+bj+ck \to ai+(\cos(\theta) b-\sin(\theta) c)j+
(\sin(\theta) b+\cos(\theta) c)k \; , \] 
\[ ai+bj+ck \to (ai+bj+ck)/(a^2+b^2+c^2) \; , \] 
where $\theta,r,a_0,b_0,c_0$ are any real constants, 
and $a,b,c$ represent coordinates of 
$\text{Im}H \approx \mathbb{R}^3$, we find that 
both $\text{Re}(q)$ and $||\text{Im}(q)||^2$ are preserved 
in all seven cases.  
These seven maps are a dilation, a translation, a 
reflection, three rotations, and an inversion, 
respectively, that generate the full M\"obius group 
(including orientation reversing transformations).  
It follows that $\hat q$ is a M\"obius invariant.  
\end{proof}

For $p_j, p_k \in \text{Im}H$, taking the corresponding $P_j,P_k 
\in M_\kappa$ as in \eqref{star8point1} and \eqref{star8point4}, 
we have the $\mathbb{R}^{4,1}$ inner product 
\begin{equation}\label{eqn:PjPk}
 \langle P_j,P_k \rangle = 
 \frac{2(p_j-p_k)^2}{(1-\kappa p_j^2)(1-\kappa p_k^2)} \; , 
\end{equation}
as in \eqref{star8point2}.  
As in Remark \ref{rem:scaling}, we can freely scale $P_j$ and 
$P_k$ to $\alpha_j P_j$ and $\alpha_k P_k$, and then 
$\langle P_j,P_k \rangle$ will scale to $\alpha_j\alpha_k \langle P_j, 
P_k \rangle$.  However, writing the cross ratio in terms of such inner 
products, we find it is invariant under such scalings.  A direct 
computation gives the following general formula for the cross ratio: 

\begin{lemma}\label{lem:generalformulaforcrossratio}
For $p_1,p_2,p_3,p_4 \in \text{Im}H$, we have 
$\hat q_{p_1p_2p_3p_4} =$ 
\[ = \frac{\langle P_1,P_2 \rangle 
\langle P_3,P_4 \rangle - \langle P_1,P_3 \rangle \langle P_2,P_4 \rangle
+ \langle P_1,P_4 \rangle \langle P_2,P_3 \rangle
\pm  \sqrt{\det (\langle P_i,P_j \rangle_{i,j=1,2,3,4} )}
}{2 \langle P_1,P_4 \rangle \langle P_2,P_3 \rangle} \; . 
\]
In particular, setting $s_{ij} = \langle P_i,P_j \rangle$, then 
\[ \hat q = \frac{s_{12}s_{34}-s_{13}s_{24}+s_{14}s_{23} \pm 
\sqrt{\mathcal{E}}}{2s_{14}s_{23}} \; , 
\]
where $\mathcal{E} = s_{12}^2s_{34}^2+s_{13}^2s_{24}^2+s_{14}^2s_{23}^2 
-2 s_{13}s_{14}s_{23}s_{24} - 2 s_{12}s_{14}s_{23}s_{34}
-2 s_{12}s_{13}s_{24}s_{34}$.  
\end{lemma}

Because 
\[ \mathcal{E} = \tfrac{1}{2} (s_{12} s_{34} - s_{14} s_{23})^2 
+ \tfrac{1}{2} (s_{12} s_{34} - s_{13} s_{24})^2 
+ \]\[ \tfrac{1}{2} (s_{13} s_{24} - s_{14} s_{23})^2 
- s_{12} s_{23} s_{34} s_{14} - 
s_{12} s_{24} s_{13} s_{34} - s_{13} s_{14} s_{23} s_{24} \; , 
\] it is not clear from straightforward algebraic considerations 
that $\mathcal{E} \leq 0$.  However, this does indeed hold, for 
geometric reasons: 

\begin{lemma}
$\mathcal{E} \leq 0$.  
\end{lemma}

\begin{proof}
Because the $P_j$ all lie in the light cone, 
$\text{span}\{P_1,P_2,P_3,P_4\}$ is a Minkowski space (i.e. the 
induced metric on this vector subspace is not positive definite).  
Therefore, we can choose a basis $e_1,e_2,e_3,e_4$ of this space 
so that 
\[ ||e_1||^2=||e_2||^2=||e_3||^2=-||e_4||^2=1 \;\; \text{and} \;\; 
\langle e_i,e_j \rangle = 0 \] for $i \neq j$.  Writing $P_j = 
a_{1j} e_1 + a_{2j} e_2 + a_{3j} e_3 + a_{4j} e_4$ in terms 
of the basis $e_1,e_2,e_3,e_4$, we have that 
\[ \mathcal{E} = \det (\langle P_i,P_j \rangle_{i,j=1}^4) = 
\]\[ = \det \left( \begin{pmatrix}
a_{11} & a_{12} & a_{13} & a_{14} \\ 
a_{21} & a_{22} & a_{23} & a_{24} \\ 
a_{31} & a_{32} & a_{33} & a_{34} \\ 
a_{41} & a_{42} & a_{43} & a_{44} 
\end{pmatrix}^t
\begin{pmatrix}
1 & 0 & 0 & 0 \\ 
0 & 1 & 0 & 0 \\ 
0 & 0 & 1 & 0 \\ 
0 & 0 & 0 & -1
\end{pmatrix}
\begin{pmatrix}
a_{11} & a_{12} & a_{13} & a_{14} \\ 
a_{21} & a_{22} & a_{23} & a_{24} \\ 
a_{31} & a_{32} & a_{33} & a_{34} \\ 
a_{41} & a_{42} & a_{43} & a_{44} 
\end{pmatrix} \right) 
\; . \] The lemma follows.  
\end{proof}

Now let us assume that for every quadrilateral with vertices 
$p,q,r,s$, the image points $\ef_p,\ef_q,\ef_r,\ef_s$ are 
concircular, with corresponding $F_p,F_q,F_r,F_s \in M_\kappa$.  
This makes the cross ratios all real-valued.  
In fact, once the cross ratio is real, then the value $q$ of 
the cross ratio, along with the values of $F_p$ and 
$F_q$ and $F_s$, determine that $F_r$ is 
\begin{equation}\label{eqn:circleparam} 
F_r = \alpha \left( F_p + \frac{1}{\langle F_q ,F_s \rangle} 
\{      (q-1) \langle F_p , F_s \rangle F_q + 
   (q^{-1}-1) \langle F_p , F_q \rangle F_s 
\} \right) \end{equation} 
for some real scalar $\alpha$, by 
Lemma \ref{lem:generalformulaforcrossratio}.  
In this way, the cross ratio gives a parametrization of the 
circle containing $\ef_p$, $\ef_q$ and $\ef_s$.  

\begin{remark}\label{rem:4ptsincircle}
If $\ef_p$, $\ef_q$, $\ef_r$ and $\ef_s$ all 
lie in the circle determined by the intersection of two distinct 
spheres $\tilde{\mathcal{S}}_1$ and $\tilde{\mathcal{S}}_2$ 
given by spacelike vectors $\mathcal{S}_1$ 
and $\mathcal{S}_2$, see \eqref{eq:S-tilde-sphere}, then 
$\ef_p, \ef_q, \ef_r, \ef_s \in \tilde{\mathcal{S}}_1 \cap 
\tilde{\mathcal{S}}_2$, or equivalently, 
\[ F_p, F_q, F_r, F_s \perp \text{span}\{ \mathcal{S}_1, 
\mathcal{S}_2 \} \; . \]
This implies that $F_p$, $F_q$, $F_r$ and $F_s$ all lie in a 
$3$-dimensional space.   
\end{remark}

Furthermore, we consider the following additional condition: 

\begin{defn}\label{defn:crossratiofactorizingfct}
When, for every quadrilateral, we can write the cross ratio as 
\[ q_{pqrs} = a_{pq}/a_{ps} \in \mathbb{R} \] so that the 
cross ratio factorizing function $a_{**}$ 
defined on the edges of $\ef$ satisfies 
\[ a_{pq}=a_{sr}  \in \mathbb{R} \; \; \text{and} \;\; a_{ps}=a_{qr} 
\in \mathbb{R} \; , \] 
then we say that $\ef$ is {\em discrete isothermic}.  
\end{defn}

Note that the $a_{**}$ are symmetric, 
i.e. $a_{pq}=a_{qp}$ for any adjacent $p$ and $q$.  

Definition \ref{defn:crossratiofactorizingfct} is equivalent to the 
Toda equation 
\[ q_{(m-1,n-1)}q_{(m,n)} = q_{(m,n-1)}q_{(m-1,n)} \] 
being satisfied, where the cross ratios 
\[ q_{(\hat m, \hat n)} := 
q_{(\hat m, \hat n),(\hat m+1, \hat n),(\hat m+1, \hat n+1),(\hat m, 
\hat n+1)} \] are all real.  

\begin{figure}[phbt]
\begin{center}
\includegraphics[width=0.6\linewidth]{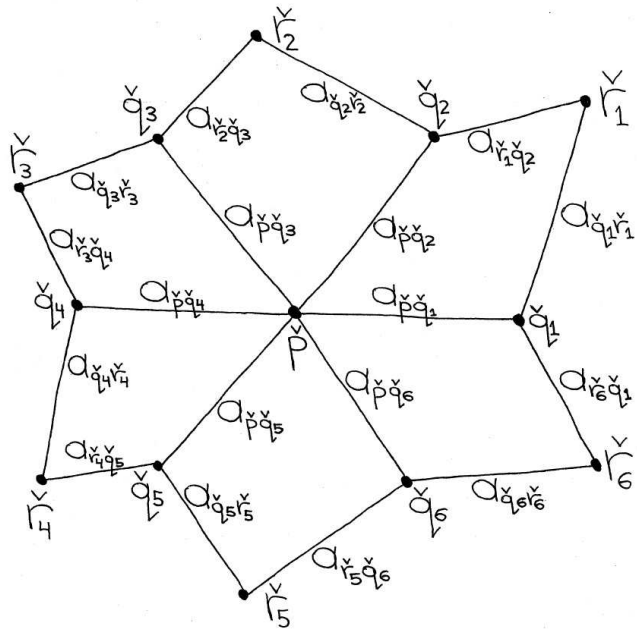}
\end{center}
\caption{Although we will not consider umbilics on 
discrete surfaces in this text, it is possible to define 
umbilics on discrete isothermic surfaces, 
as follows: We now do not consider the discrete surface 
as a map from a domain in the 
integer lattice $\mathbb{Z}^2$ (it will not be).  Let $\breve p$ be 
a vertex of a discrete surface consisting entirely of quadrilateral 
faces, with each face having concircular vertices.  
Thus all cross ratios on the faces are real, and we have 
a real cross ratio factorizing function $a$.  Suppose that 
$\breve p$ is a vertex of some even number of faces, and 
at least six faces, of the surface.  If the cross ratio 
factorizing condition in Definition 
\ref{defn:crossratiofactorizingfct} is satisfied, then we have a 
discrete isothermic surface with umbilic point $\breve p$.  
For example, if $\breve p$ has six adjacent faces as in the 
figure here, then we require that 
$a_{{\breve q}_{j-1} {\breve r}_{j-1}} = 
a_{{\breve p} {\breve q}_{j}} = 
a_{{\breve r}_{j} {\breve q}_{j+1}}$ for $j = 2,3,4,5$, 
and also $a_{{\breve q}_{6} {\breve r}_{6}} = 
a_{{\breve p} {\breve q}_{1}} = 
a_{{\breve r}_{1} {\breve q}_{2}}$ and 
$a_{{\breve q}_{5} {\breve r}_{5}} = 
a_{{\breve p} {\breve q}_{6}} = 
a_{{\breve r}_{6} {\breve q}_{1}}$.  Furthermore, this 
surface with an umbilic is then also discrete CMC if there exists a 
linear conserved quantity as in Definition \ref{defn:disclqcCMC}.  
}
\label{fig:discreteumbilics}
\end{figure}

\subsection{Isothermicity from the perspective of smooth surfaces}
One viewpoint on what a "discrete isothermic surface" is, as 
in Definition \ref{defn:crossratiofactorizingfct}, 
is as follows: Take a smooth surface $x$.  
Give it curvature line coordinates $x=x(u,v)$, so 
$x_u \perp x_v$.  (Curvature line 
coordinates always exist away from umbilics.)  Then the 
first and second fundamental forms are 
\[ I = \begin{pmatrix} g_{11} & 0 \\ 0 & g_{22} \end{pmatrix} \; , \;\;\; 
II = \begin{pmatrix} b_{11} & 0 \\ 0 & b_{22} \end{pmatrix} \; . \]  
One can always stretch the coordinates, so that $x = x(u,v) = 
x(\tilde u(u),\tilde v(v))$ for any monotonic functions 
$\tilde u$ depending 
only on $u$, and $\tilde v$ depending only on $v$.  
Note that $\langle x_{\tilde u} , x_{\tilde v} \rangle = 0$, and 
$x_{\tilde u \tilde v} = x_{uv} \tfrac{du}{d\tilde u} 
\tfrac{dv}{d\tilde v}$ implies $\langle x_{\tilde u \tilde v} , 
\vec N \rangle = 0$, so $(\tilde u, \tilde v)$ are also 
curvature line coordinates.  The surface is then isothermic 
if and only if there exist $\tilde u$, $\tilde v$ such that the 
metric becomes conformal, i.e. $\langle x_{\tilde u} , 
x_{\tilde u} \rangle = \langle x_{\tilde v} , 
x_{\tilde v} \rangle$, and this is equivalent to 
\[ \frac{g_{11}}{g_{22}} = \frac{a(u)}{b(v)} \; , \] 
where the function $a$ depends only on $u$, and $b$ 
depends only on $v$.  

Now consider the cross ratio $q_\epsilon$ of the four 
points $x(u,v)$, $x(u+\epsilon,v)$, $x(u+\epsilon,v+\epsilon)$ 
and $x(u,v+\epsilon)$.  Using that $x_u \perp x_v$ implies 
$x_u x_v^{-1} = - x_v^{-1} x_u$, we see that 
\begin{equation}\label{qepsilon} 
\lim_{\epsilon \to 0} q_\epsilon = - \frac{g_{11}}{g_{22}} \; . 
\end{equation}  So $x$ is isothermic if and only if 
\begin{equation}\label{star9point4} 
\lim_{\epsilon \to 0} q_\epsilon = - \frac{a(u)}{b(v)} \; , 
\end{equation} 
where again $a$ is some function that 
depends only on $u$, and $b$ depends only on $v$.  
This description of isothermicity does not involve any 
stretching by $\tilde u$ or $\tilde v$, which we would not 
be able to do in the discrete case anyways, and now Definition 
\ref{defn:crossratiofactorizingfct} is a natural discretization 
of \eqref{star9point4}:  The corresponding statement for discrete 
surfaces, where stretching of coordinates is no longer possible, is 
that the surface is discrete isothermic if and only if 
the cross ratio factorizing function can be chosen so that 
$a_{pq} = 
a_{rs}$ and $a_{ps}=a_{qr}$ for vertices $p,q,r,s$ (in order) about a 
given quadrilateral.  

There is another perspective on isothermicity, coming from 
a lemma proven by Bobenko and Pinkall \cite{BobPink2}: 

\begin{lemma}\label{lemmadiagonalBobPink}
Let $x(u,v)$ be a smooth surface in $\mathbb{R}^3$, and 
define the diagonal cross ratio \[ q_\epsilon^d = 
(x(u+\epsilon,v-\epsilon)-x(u-\epsilon,v-\epsilon)) 
(x(u+\epsilon,v+\epsilon)-x(u+\epsilon,v-\epsilon))^{-1} 
\;\; \times \]\[ 
(x(u-\epsilon,v+\epsilon)-x(u+\epsilon,v+\epsilon)) 
(x(u-\epsilon,v-\epsilon)-x(u-\epsilon,v+\epsilon))^{-1} \; . \] 
Then
\[ q_\epsilon^d = -1 + \mathcal{O}(\epsilon) \] if and only if 
$(u,v)$ are conformal coordinates for $x$, and 
\[ q_\epsilon^d = -1 + \mathcal{O}(\epsilon^2) \] if and only if 
$(u,v)$ are isothermic coordinates for $x$.  
\end{lemma}

The superscript "$d$" in $q_\epsilon^d$ stands for 
"diagonal", because we are using diagonal elements to 
define this cross ratio, unlike with the previous $q_\epsilon$.  
Also, $\mathcal{O}(\epsilon^k)$ denotes any function $f = 
f(\epsilon)$ such that the limit of $f(\epsilon) \epsilon^{-k}$, 
as $\epsilon$ approaches $0$, exists and is finite.  

\begin{proof}
Without loss of generality, we may assume $x(u,v)=\vec{0}$, 
and then for $\rho_u, \rho_v \in \{ \pm 1 \}$, we have 
\[ x(u + \rho_u \epsilon,v + \rho_v \epsilon) = 
\epsilon \rho_u x_u + \epsilon \rho_v x_v + 
\tfrac{1}{2} \epsilon^2 (x_{uu} + x_{vv}+2 \rho_u \rho_v x_{uv}) 
+ \mathcal{O}(\epsilon^3) \; , \] 
so 
\[ q_\epsilon^d = x_u x_v^{-1} x_u x_v^{-1} + \]\[ \epsilon 
(x_u x_v^{-1} x_{uv} x_v^{-1} + 
x_u x_v^{-1} x_u x_v^{-1} x_{uv} x_v^{-1} - 
x_{uv} x_v^{-1} x_u x_v^{-1} - x_u x_v^{-1} 
x_{uv} x_v^{-1} x_u x_v^{-1}) + \mathcal{O}(\epsilon^2) \; . \] 
If the coordinates are conformal, then $x_u x_v^{-1} x_u x_v^{-1} 
= -1$, and we have 
\[ q_\epsilon^d = -1 + \epsilon x_u^{-4} 
(x_u x_v x_{uv} (x_u + x_v) + x_u^2 x_{uv} (x_u-x_v)) 
+ \mathcal{O}(\epsilon^2) \; . \] 
Now, if the coordinates are isothermic, then $b_{12}=0$, and 
so there exist scalar functions $A$ and $B$ so that 
\[ x_{uv} = A x_u + B x_v \; . \]  From this it follows 
that $q_\epsilon^d = -1 + \epsilon \cdot 0 + 
\mathcal{O}(\epsilon^2)$.  
\end{proof}

This lemma leads to the following definition 
for discrete isothermic surfaces in the narrow sense: 
$\ef$ is discrete isothermic if \[ q_{pqrs} = -1 \] 
for all quadrilaterals, with vertices $\ef_p,\ef_q,\ef_r,\ef_s$.  

However, with this definition, transformations, such as the 
Calapso transform, of isothermic surfaces will not remain 
isothermic.  (Lemma \ref{lem:changeof-a-underCalapso} 
demonstrates this.)  
Hence the broader definition given in Definition 
\ref{defn:crossratiofactorizingfct} has been found to be more 
suitable.  

One could think of a discrete surface $x$ with cross ratios 
exactly $-1$ as being "isothermically parametrized", while a discrete 
surface $\ef$ with cross ratios satisfying Definition 
\ref{defn:crossratiofactorizingfct} is one that could 
have its coordinates stretched so that it becomes isothermic, 
were it a smooth surface.  

\subsection{Moutard lifts for smooth surfaces} 
Given a smooth immersion $x(u,v)$ so that 
$x_u \perp x_v$, the light cone lift 
\[ X = X(u,v) = \begin{pmatrix} x & -x^2 \\ 1 & -x 
\end{pmatrix} \in PL^4 \]
could also be represented by $\alpha \cdot X$ for any choice 
of nonzero real-valued function $\alpha=\alpha(u,v)$.  
If we choose $\alpha$ so that 
\begin{equation}\label{moutardlifteqn} 
\partial_u \partial_v (\alpha X) || X \; , 
\end{equation} or equivalently
$\alpha_u x_v + \alpha_v x_u + \alpha x_{uv} = 0$, 
then we say that $\alpha X$ is a {\em Moutard lift}.  

\begin{lemma}\label{lem:moutardliftsalwaysexist}
Moutard lifts always exist for any smooth isothermic immersion.  
\end{lemma}

\begin{proof}
Let $x(u,v)$ be a smooth isothermic immersion with isothermic 
coordinates $u,v$.  As we saw in the proofs of Lemma 
\ref{lem:Calapexist} and Lemma 
\ref{lemmadiagonalBobPink}, there 
exist real-valued functions $A,B$ so that 
\[
  x_{uv}=A x_u+Bx_v \; .
 \]
Taking the inner product of this with $x_u$ and with 
$x_v$, and using that $\langle x_u , x_u \rangle = 
\langle x_v , x_v \rangle$ and $\langle x_u , x_v \rangle=0$, 
we find that 
\[
 A=\partial_v (\tfrac{1}{2} \log \langle x_u , x_u \rangle)
 \; , \;\;\; 
 B=\partial_u (\tfrac{1}{2} \log \langle x_u , x_u \rangle) \; , 
\] and thus it follows that $A_u = B_v$.  
The existence of a solution $\alpha$ to the equation 
$\alpha_u x_v + \alpha_v x_u + \alpha x_{uv} = 0$ is equivalent 
to solving the system 
\[ \alpha_u = - \alpha B \; , \;\;\; \alpha_v = - \alpha A \; , \]  
because $x_{uv} = A x_u + B x_v$.  
The compatibility condition of this system 
is $A_u = B_v$, seen as follows: 
\[ \alpha_{uv} = \alpha_{vu} \] 
if and only if $(-\alpha B)_v = (-\alpha A)_u$, if and only if 
\[ \alpha_v B + \alpha B_v = \alpha_u A + \alpha A_u \; , \]  
if and only if $(-\alpha A) B + \alpha B_v = (-\alpha B) A + 
\alpha A_u$, if and only if 
\[ B_v = A_u \: . \]  This proves the lemma.  
\end{proof}

\begin{remark}
\label{rem:anotherproof}
Here is a hint of another way to prove Lemma 
\ref{lem:moutardliftsalwaysexist}: $x$ has isothermic coordinates, 
and so $e^{2 \hat u} = \langle x_u , x_u \rangle = 
\langle x_v , x_v \rangle$ for some real-valued function 
$\hat u=\hat u(u,v)$, which 
implies $2 \hat u_u e^{2 \hat u} = 2 \langle x_{uv} , 
x_v \rangle$, so 
$x_{uv} = *_1 \cdot x_u + \hat u_u x_v + *_2 \cdot \vec N$ 
for some functions $*_j$.  Similarly, 
now taking the derivative of $e^{2 \hat u}$ with respect 
to $v$, we have 
$x_{uv} = \hat u_v x_u + \hat u_u x_v + *_2 \cdot 
\vec N$.  Then $\langle x_{uv} , \vec N \rangle = - 
\langle x_u , \vec N_v \rangle 
= - \langle x_u , *_3 \cdot x_v \rangle = 0$ implies 
$x_{uv} = \hat u_v x_u + \hat u_u x_v$.  Then, taking the lift 
\[ X_1 = X_1(u,v) = \begin{pmatrix}
x & -x^2 \\ 1 & -x \end{pmatrix} \] of $x = x(u,v)$ into $L^4$, 
we have $(X_1)_{uv} = \hat u_v (X_1)_u + 
\hat u_u (X_1)_v$.  We then rescale $X_1$ to $X_2 := e^{-\hat u} 
X_1$.  A computation gives $(X_2)_{uv} = 
\lambda \cdot X_2$ with $\lambda = \hat u_u \hat u_v - 
\hat u_{uv}$.  $(X_2)_{uv} = 
\lambda X_2$ is the condition for a Moutard lift.  
This argument would still hold 
if $(u,v)$ were just curvature line coordinates, but not 
necessarily isothermic coordinates, for 
the isothermic surface $x$.  
In other words, even if just $e^{2 \hat u} = \langle x_u , 
x_u \rangle \cdot \alpha 
= \langle x_v , x_v \rangle \cdot \beta$, with functions 
$\alpha$ and $\beta$ such that $\alpha_v = 
\beta_u = 0$, this is enough to say there 
exists a Moutard lift $X_2$, i.e. $(X_2)_{uv} = \lambda X_2$.  
\end{remark}

\subsection{Moutard lifts for discrete surfaces} 
Recall that for each point $\ef \in \text{Im}H$, there is a 
unique lift $F \in M_\kappa$ (not necessarily Moutard).  
However, as noted in Remark \ref{rem:scaling} and in the previous 
section, we can freely multiply 
$F$ by any nonzero real scalar, giving a 
scalar freedom in the choice 
of lift.  Here, in the case of discrete surfaces, 
we describe particular choices for the lift $F$ 
that are convenient for computations, again 
called Moutard lifts, 
analogous to Moutard lifts for smooth surfaces.  

\begin{defn}\label{def:discrete-moutard}
We say that $F$ is a {\em Moutard lift} if, for 
the four vertices $p,q,r,s$ listed in 
counterclockwise order about any quadrilateral, we have 
\[ 
  (F_r-F_p) || (F_q - F_s) \; , 
\]
meaning that 
\begin{equation}\label{Fparallelity}
F_r -F_p = \mu (F_q-F_s)
\end{equation} for some real scalar $\mu$.  
\end{defn}

The discrete Moutard equation as in Definition 
\ref{def:discrete-moutard} can be justified as follows: 
consider a quadrilateral with lifts $F_p$, $F_q$, $F_r$ and 
$F_s$ at the vertices.  The discrete second derivative 
of $F$ (analogous to $X_{uv}$ in the smooth case) is 
\[ (F_r - F_s) - (F_q - F_p) \; , \] so the Moutard equation, 
i.e. the discrete version of Equation \eqref{moutardlifteqn}, 
can naturally be considered to be 
\[ F_r - F_s - F_q + F_p = \lambda_1 \tfrac{1}{4} (F_p+
F_q+F_r+F_s) \; , \;\;\; \lambda_1 \in \mathbb{R} \; , \] and the 
$\tfrac{1}{4}$ can be absorbed into the $\lambda_1$ as 
$\lambda_2 = \tfrac{1}{4} \lambda_1$.  
Then \[ F_p+F_r = \lambda (F_q+F_s) \; , \] 
where we have defined $\lambda$ by $\lambda = 
\tfrac{1+\lambda_2}{1-\lambda_2}$, i.e. 
$(F_p+F_r)||(F_q+F_s)$.  Since 
$F_*$ is only projectively defined and thus signs of 
any of the $F_*$ 
can always be switched (i.e. $F_* \to -F_*$), 
we could also write \[ (F_r-F_p) || (F_q-F_s) \] as 
in Definition \ref{def:discrete-moutard}.  

\begin{remark}
By a consideration similar to the one just above, we have 
discrete conjugate nets: 
A {\em conjugate net} for a smooth surface $x$ in 
$\mathbb{R}^3$ is 
coordinates so that the second fundamental form is diagonal 
(not necessarily conformal, nor necessarily curvature line 
coordinates), i.e. $x_{uv} \in \text{span}\{ x_u,x_v \}$.  
This last condition would be 
$(\ef_r-\ef_s)-(\ef_q-\ef_p) \in \text{span}\{ 
\ef_q-\ef_p,\ef_s-\ef_p \}$ 
for a discrete surface in $\mathbb{R}^3$, implying 
$\ef_r-\ef_p \in \text{span}\{ \ef_q-\ef_p,\ef_s-\ef_p \}$, 
and so $\ef_p$, $\ef_q$, $\ef_r$ and $\ef_s$ are coplanar.  
This is why we define {\em discrete conjugate nets} to be 
those discrete surfaces that have planar faces.  
\end{remark}

\begin{lemma}\label{lem:Moutardliftproperty}
For a Moutard lift $F$ of a discrete isothermic surface 
$\ef$, the cross ratios $q_{pqrs} = 
\frac{a_{pq}}{a_{ps}}$ satisfy 
\[ q_{pqrs} = \frac{a_{pq}}{a_{ps}} = 
 \frac{\langle F_p,F_q \rangle}{\langle F_p,F_s \rangle} \; . \]
\end{lemma}

\begin{proof}
For Moutard lifts, since 
$||F_r||=||F_p||=0$, we have \[
	 0 = \langle F_r+F_p,F_r-F_p \rangle = 
\mu \langle F_r+F_p,F_q-F_s \rangle     
\; , 	     \]
and so $(F_r+F_p) \perp (F_q-F_s)$.  Similarly, 
$(F_r-F_p) \perp (F_q+F_s)$.  So 
\[
  \langle F_p,F_r \rangle \cdot \langle F_q,F_s \rangle
= 
  \langle F_p,F_r-F_p \rangle \cdot \langle F_q-F_s,F_s \rangle
= \]\[
  \langle F_p,\mu (F_q-F_s) \rangle \cdot \langle 
  \mu^{-1}(F_r-F_p),F_s \rangle = 
  \langle F_p,F_q-F_s \rangle \cdot \langle F_r-F_p,F_s \rangle
= \]\[
  (\langle F_p,F_q \rangle - \langle F_p,F_s \rangle) \cdot 
  (\langle F_r,F_s \rangle - \langle F_p,F_s \rangle)
= 
  (\langle F_p,F_q \rangle - \langle F_p,F_s \rangle)^2 \; , 
\] since 
$\langle F_p,F_q \rangle = \langle F_r,F_s \rangle$, by 
$(F_r+\epsilon F_p) \perp (F_q-\epsilon F_s)$ for 
$\epsilon = \pm 1$.  Also, 
$\langle F_q,F_r \rangle = \langle F_p,F_s \rangle$.  
By Lemma \ref{lem:generalformulaforcrossratio} with the 
$p_*$ there being the projections of the $F_*$ 
here to $\text{Im}H$, 
and using that $\mathcal{E} = 0$, we have 
proven the lemma.  
\end{proof}

\begin{remark}\label{Moutardliftnotunique}
The Moutard lift is not completely unique, and it has more than 
just the freedom of a constant scalar multiple.  For example, 
if points $p$ corresponds to $(m,n)$ in the 
domain lattice in $\mathbb{Z}^2$, 
we could change a Moutard lift $F_p$ 
to $\alpha F_p$ when $m+n$ is even and $\beta F_p$ 
when $m+n$ is odd, 
for any nonzero constants $\alpha, \beta \in 
\mathbb{R}$, and this gives another Moutard lift.  
\end{remark}

Lemma \ref{lem:Moutardliftproperty} and 
Remark \ref{Moutardliftnotunique} imply that, by 
multiplying all $F_p$ by an appropriate constant real 
scalar, we may assume 
\begin{equation}\label{eqn:Moutard} 
F_p F_q + F_q F_p = a_{pq} \cdot I \end{equation} on all edges. 
Furthermore, any lift satisfying \eqref{eqn:Moutard} is Moutard, 
and all Moutard lifts satisfy \eqref{eqn:Moutard} up to the 
freedom given in Remark \ref{Moutardliftnotunique}.  

\begin{remark}\label{rem:matrixtoscalar}
Because $F_pF_q+F_qF_p$ is a scalar multiple of the identity, we
sometimes ignore that it is a matrix, and simply consider it as that
scalar $a_{pq}$.  
\end{remark}

\begin{lemma}
Let $\ef \in \mathbb{R}^3 \approx 
\text{Im} H$ be a discrete surface with 
concircular quadrilaterals.  Then there exists a Moutard 
lift if and only if $\ef$ is isothermic.  
In particular, we can then choose the Moutard lift so that 
Equation \eqref{eqn:Moutard} holds.  
\end{lemma}

\begin{proof}
First we assume $\ef$ is isothermic, and show that a Moutard 
lift exists.  Choose a particular quadrilateral $pqrs$, and 
assume a lift $F$ is chosen so that 
\eqref{eqn:Moutard} holds for both of the two edges $pq$ 
and $ps$ in that quadrilateral $pqrs$.  Then Equation 
\eqref{eqn:circleparam} implies we can choose $F_r$ to be 
(note that we are not requiring any condition like $F_r \in 
M_\kappa$ here) 
\[ F_r = F_p + \tfrac{1}{2} (\langle F_q,F_s \rangle)^{-1} 
((a_{ps}-a_{pq}) F_q + (a_{pq}-a_{ps}) F_s) \; . \]  
Noting that isothermicity implies $a_{pq}=a_{rs}$ and 
$a_{ps}=a_{qr}$, a computation gives that \eqref{eqn:Moutard} 
also holds on the edges $sr$ and $qr$.  It follows that a 
Moutard lift exists.  

We now assume that a Moutard lift 
$F$ exists, and then prove the surface $\ef$ is isothermic.  
Let $\ef_p$, $\ef_q$, $\ef_r$ and $\ef_s$ be the vertices 
of one quadrilateral of $\ef$ with cross ratio $q \in \mathbb{R}$.  
The assumption of concircularity implies that 
$F_r \in \text{span}\{ F_p, F_q, F_s \}$, 
by Remark \ref{rem:4ptsincircle}.  

Now recall that a point 
\[ p \in \mathbb{R}^3 \approx 
\text{Im} H \] has lift \[ (x_1,x_2,x_3,x_4,x_5) = 
(2 p_j,-(1-|p_j|^2),1+|p_j|^2) \approx 
P_j = 2 \begin{pmatrix}
p_j & -p_j^2 \\ 1 & -p_j 
\end{pmatrix} \in M_0 \subseteq L^4 \; , \] 
where 
$\R^3 = M_0$ is given by the $Q$ in \eqref{choiceofQ}
with $\kappa = 0$.  

We saw in \eqref{eqn:PjPk} that for $p_1,p_2 \in \R^3 
\approx \text{Im} H$, 
we have 
\[ \langle P_1,P_2 \rangle = 2 (p_1-p_2)^2 \; . \]  
We have $\ef_p,\ef_q,\ef_r,\ef_s \in \text{Im} H$, 
and we can find $\alpha_* \in \mathbb{R} \setminus \{ 0 \}$ 
so that $\alpha_p F_p$, $\alpha_q F_q$, $\alpha_r F_r$ and 
$\alpha_s F_s$ all lie in $M_0$, and then 
Lemma \ref{lem:forlateruse} gives that $q$ satisfies 
\begin{equation}\label{qqq} q^2 = 
\left( (\ef_p-\ef_q)\frac{(\ef_q-\ef_r)}{(\ef_q-\ef_r)^2}
(\ef_r-\ef_s)\frac{(\ef_s-\ef_p)}{(\ef_s-\ef_p)^2} 
\right)^2 = \end{equation} \[ 
= \frac{(\ef_p-\ef_q)^2(\ef_r-\ef_s)^2}{(\ef_q-\ef_r)^2(\ef_s-\ef_p)^2} = 
\frac{\langle \alpha_p F_p,\alpha_q F_q \rangle 
\langle \alpha_r F_r,\alpha_s F_s \rangle}
{\langle \alpha_q F_q,\alpha_r F_r \rangle \langle \alpha_s F_s,\alpha_p 
F_p \rangle} = \frac{\langle F_p,F_q \rangle \langle F_r,F_s \rangle}
{\langle F_q,F_r \rangle \langle F_s,F_p \rangle} \; . 
\]  A condition for $F_r$ to be in $L^4$ is, from 
Equation \eqref{Fparallelity}, 
\[ 0 = \langle F_r, F_r \rangle = 
\langle \mu (F_q-F_s) + F_p , \mu (F_q-F_s) + F_p \rangle = 
\mu^2 \langle F_q-F_s, F_q-F_s \rangle 
+ 2 \mu \langle F_q-F_s, F_p \rangle \; , \] 
which implies
\[ \mu = \frac{-2\langle F_q-F_s,F_p \rangle}{\langle 
F_q-F_s,F_q-F_s \rangle} \; , \] and so 
\[ F_r = \frac{-2\langle F_q-F_s,F_p \rangle}{\langle 
F_q-F_s,F_q-F_s \rangle} (F_q-F_s) + F_p \; , \] which implies 
\[ \langle F_r,F_s \rangle = \langle F_p,F_q \rangle 
\;\;\; \text{and} \;\;\; \langle F_r,F_q \rangle = 
\langle F_p,F_s \rangle \; . \]  This shows that the cross ratios 
of $\ef$ satisfy the condition in Definition 
\ref{defn:crossratiofactorizingfct}, completing the proof.  
\end{proof}

\begin{remark}
When the discrete surface is isothermic in the narrow sense, i.e. 
when the cross ratios are identically $-1$, there is 
a way to describe real values defined at the vertices 
so that they can be thought of as 
the "scalar factor" or "stretching factor" for the discrete 
"conformal metric", as follows: For a smooth surface $x(u,v)$ 
with isothermic coordinates $u,v$, we have as in 
Remark \ref{rem:anotherproof} 
that \[ X_2 = e^{-\hat u} X_1 \] is a Moutard lift, 
where $e^{2 \hat u}$ 
is the metric factor.  Now, in the case of a discrete isothermic surface 
$\ef$, one lift is 
\[ F_* = \begin{pmatrix} 
f_* & -f_*^2 \\ 1 & -f_* 
\end{pmatrix} \] 
($*$ now denotes vertices in the domain of $\ef$), 
and we can take a Moutard lift 
\[ \tilde F_* = s_* F_* \]  
satisfying \eqref{eqn:Moutard}.  Here $s_*$ will be the 
"discrete metric".  We can take $a_{pq} = \pm 1$, and then 
\[ |a_{pq}| = 2 | \langle \tilde F_p , \tilde F_q \rangle | \] 
(i.e. $\tilde F_*$ is a Moutard lift satisfying 
\eqref{eqn:Moutard}) implies 
\[ \tfrac{1}{2} = |s_p| \cdot |s_q| \cdot |\langle F_p , F_q \rangle| 
= \tfrac{1}{2} |s_p| \cdot |s_q| \cdot 
|\ef_p-\ef_q|^2 \; . \]  So $|s_*|$ behaves just like 
$e^{-\hat u}$ would in the case of a smooth isothermic surface.  
\end{remark}

We now give an application of Moutard lifts.  Suppose that 
$(0,0)$, $(\pm 1,0)$, $(0,\pm 1)$, $\pm (1,1)$ and $\pm (1,-1)$ 
are all in the lattice domain of a discrete surface $\ef$.  Then 
the {\em diagonal vertex star} of $\ef_{(0,0)}$ consists of the 
images $\ef_{(0,0)}$, $\ef_{(1,-1)}$, $\ef_{(1,1)}$, 
$\ef_{(-1,1)}$ and $\ef_{(-1,-1)}$ of the points 
$(0,0)$, $(1,-1)$, $(1,1)$, $(-1,1)$ and $(-1,-1)$.  The 
proof of the next lemma applies Moutard lifts.  

\begin{lemma}\label{lem:diag-vertex-star}
The five vertices of any diagonal vertex star on a discrete 
isothermic surface are cospherical.  
\end{lemma}

\begin{proof}
We can take the image $\ef_{(0,0)}$ of the 
point $(0,0)$ in the lattice domain to be the 
center of the diagonal vertex star.  
Let $F_{(i,j)}$ be a Moutard lift of $\ef_{(i,j)}$ satisfying 
Equation \eqref{eqn:Moutard}.  

Our goal is to show \[ 
\text{dim}(F_{(0,0)}, F_{(1,-1)}, F_{(1,1)}, F_{(-1,1)}, 
F_{(-1,-1)}) \leq 4 \; . \]  Then there exists a spacelike 
vector $\mathcal{S} \in \mathbb{R}^{4,1}$ which 
produces the sphere 
$\tilde{\mathcal{S}}$, via \eqref{eq:S-tilde-sphere}, that 
contains all five points $F_{(0,0)}, F_{(1,-1)}, 
F_{(1,1)}, F_{(-1,1)}, F_{(-1,-1)}$, and the proof would be 
completed.  

In the following computation, for the sake of simplicity, we ignore 
cases where some coefficients might be zero (those other 
cases can be dealt with separately).  

Because we chose a Moutard lift, we have 
$\langle F_q , F_r \rangle = \langle F_p , F_s \rangle$ 
on any quadrilateral, implying 
\[ \langle F_q , F_p - F_r \rangle = 
\langle F_p , F_q - F_s \rangle \; , \]
so 
\[ \langle F_q , F_s - F_q \rangle (F_r-F_p) = 
   \langle F_q , F_r - F_p \rangle (F_s-F_q) = 
\langle F_p , F_q - F_s \rangle (F_q-F_s)  \; , \]
and so 
\[ \langle F_q , F_s \rangle (F_r-F_p) = - \frac{1}{2} 
(a_{pq}-a_{ps}) (F_q-F_s)  \; . \]
This implies 
\[ \langle F_{(1,-1)} , F_{(0,0)} \rangle (F_{(1,0)}-F_{(0,-1)}) = 
- \frac{1}{2} (a_{(0,0)(1,0)}-a_{(0,-1)(0,0)}) (F_{(1,-1)}-F_{(0,0)}) \; , \]
\[ \langle F_{(0,-1)} , F_{(-1,0)} \rangle (F_{(0,0)}-F_{(-1,-1)}) = 
- \frac{1}{2} (a_{(-1,0)(0,0)}-a_{(0,-1)(0,0)}) (F_{(0,-1)}-F_{(-1,0)}) \; , \]
\[ \langle F_{(0,0)} , F_{(-1,1)} \rangle (F_{(0,1)}-F_{(-1,0)}) = 
- \frac{1}{2} (a_{(-1,0)(0,0)}-a_{(0,0)(0,1)}) (F_{(0,0)}-F_{(-1,1)}) \; , \]
\[ \langle F_{(1,0)} , F_{(0,1)} \rangle (F_{(1,1)}-F_{(0,0)}) = 
- \frac{1}{2} (a_{(0,0)(1,0)}-a_{(0,0)(0,1)}) (F_{(1,0)}-F_{(0,1)}) \; , \]
and then 
\[ F_{(1,0)}-F_{(0,-1)} = - 
\frac{a_{(0,0)(1,0)}-a_{(0,-1)(0,0)}}{2 \langle F_{(1,-1)} , F_{(0,0)} \rangle} 
(F_{(1,-1)}-F_{(0,0)}) \; , \]
\[ F_{(0,-1)}-F_{(-1,0)} = 
\frac{-2 \langle F_{(0,-1)} , F_{(-1,0)} \rangle}{a_{(-1,0)(0,0)}-a_{(0,-1)(0,0)}} 
(F_{(0,0)}-F_{(-1,-1)}) \; , \]
\[ F_{(-1,0)}-F_{(0,1)} = 
\frac{a_{(-1,0)(0,0)}-a_{(0,0)(0,1)}}{2 \langle F_{(-1,1)} , F_{(0,0)} \rangle} 
(F_{(0,0)}-F_{(-1,1)}) \; , \]
\[ F_{(0,1)}-F_{(1,0)} = 
\frac{2 \langle F_{(1,0)} , F_{(0,1)} \rangle}{a_{(0,0)(1,0)}-a_{(0,0)(0,1)}} 
(F_{(1,1)}-F_{(0,0)}) \; . \]
Adding these last four equations, we see that a linear combination of 
those five points 
$F_{(0,0)}, F_{(1,-1)}, F_{(1,1)}, F_{(-1,1)}, F_{(-1,-1)}$ equals zero, 
proving the result.  
\end{proof}

The conclusion of Lemma \ref{lem:diag-vertex-star} is in fact 
equivalent to the discrete surface being isothermic, and this then makes 
it obvious that discrete isothermicity is 
invariant under M\"obius transformations.  

\begin{defn}\label{defnofcentralS}
We say that the sphere (containing the vertex star) 
in Lemma \ref{lem:diag-vertex-star} 
is the {\em central sphere} of the discrete 
isothermic surface at the central vertex 
of the diagonal vertex star.  
\end{defn}

\subsection{Christoffel transforms}  When $\ef$ is a 
discrete isothermic surface in 
$\mathbb{R}^3 \approx \text{Im}H$, we 
can define the Christoffel transform $\ef^*$ (also 
in $\mathbb{R}^3$) of $\ef$ as follows: 

\begin{defn}\label{discreteChristoffeltransform}
Let $\ef$ be a discrete isothermic surface in $\R^3$.  Then 
the {\em Christoffel transform} $\ef^*$ of $\ef$ satisfies 
\begin{equation}\label{eqn:cristoffeltransform}
 d\ef_{pq}^* d\ef_{pq} = a_{pq} \; . \end{equation}
\end{defn}
 
Here, for any object $\frak{F}$ defined on vertices, 
$d\frak{F}_{pq}$ denotes the difference 
\[
  d\frak{F}_{pq} := \frak{F}_q - \frak{F}_p \]
of the values of $\frak{F}$ at the vertices $q$ and $p$.  

To see that this definition is natural, we consider the 
Christoffel transform $x^*$ of a smooth surface $x$ in 
$\mathbb{R}^3$ 
with isothermic coordinates $u,v$.  In the smooth case, 
we may assume $x$ and $x^*$ satisfy 
\[ dx = x_u du+x_v dv \; , \;\;\; dx^* = 
x_u^{-1} du-x_v^{-1} dv \; , \] as seen in the previous 
chapter.  So 
\[ dx^*(\partial_u) dx(\partial_u)=1 \;\;\; 
\text{and} \;\;\; 
dx^*(\partial_v) dx(\partial_v)=-1 \; . 
\]
We also have \[ \lim_{\epsilon \to 0} q_\epsilon=-1 = 
\frac{dx^*(\partial_u) dx(\partial_u)}{dx^*(\partial_v) dx(\partial_v)}
\; , \] by 
Equation \eqref{qepsilon}.  In the discrete case, we loosened 
the $-1$ in the right-hand side of 
Equation \eqref{qepsilon} to the $a_{pq}/a_{ps}$ in the 
right-hand side of $q_{pqrs} = a_{pq}/a_{ps}$, as 
in Definition \ref{defn:crossratiofactorizingfct}.  
Because of this, it is natural to consider that 
\[ \frac{a_{pq}}{a_{ps}} = 
\frac{d\ef^*_{pq} d\ef_{pq}}{d\ef^*_{ps} d\ef_{ps}} \; , 
\]
where $d\ef_{pq}$, $d\ef^*_{pq}$, $d\ef_{ps}$, $d\ef^*_{ps}$ now 
represent discrete analogs of 
$dx(\partial_u)$, $dx^*(\partial_u)$, $dx(\partial_v)$, 
$dx^*(\partial_v)$, and so 
Definition \ref{discreteChristoffeltransform} becomes natural.  

We can then prove the following: 

\begin{lemma} \cite{BobPink} 
If $\ef$ is a discrete isothermic surface, then 
there exists a Christoffel transform $\ef^*$ of $\ef$.  
\end{lemma}

\begin{proof}
$\ef^*$ exists if and only if the compatibility 
condition 
\begin{equation}\label{eqn:fstarcompatibility}
 d\ef^*_{pq}+d\ef^*_{qr}=d\ef^*_{ps}+d\ef^*_{sr} 
\end{equation}
holds, that is to say, we can apply ``discrete integration'' of 
$d\ef^*$ to obtain $\ef^*$.  

We now prove that Equation \eqref{eqn:fstarcompatibility} holds with 
$d\ef^*$ defined as in Equation \eqref{eqn:cristoffeltransform}.  
By Equation \eqref{eqn:cristoffeltransform}, Equation 
\eqref{eqn:fstarcompatibility} is equivalent to 
\[ a_{pq} d\ef_{pq}^{-1} + a_{qr} d\ef_{qr}^{-1} = 
a_{ps} d\ef_{ps}^{-1} + a_{sr} d\ef_{sr}^{-1} \; . \] 
Because $a_{pq}=a_{sr}$ and $a_{ps}=a_{qr}$ (by isothermicity), 
this equation is equivalent to 
\[ \frac{a_{pq}}{a_{ps}} ( d\ef_{pq}^{-1} - d\ef_{sr}^{-1}) = 
d\ef_{ps}^{-1} - d\ef_{qr}^{-1} \; . \] 
By Lemma \ref{lem:forlateruse}, the cross 
ratio is $a_{pq} a_{ps}^{-1} = 
d\ef_{pq} d\ef_{qr}^{-1} d\ef_{rs} d\ef_{sp}^{-1} = 
d\ef_{qr}^{-1} d\ef_{rs} d\ef_{sp}^{-1} d\ef_{pq} = 
d\ef_{qr}^{-1} d\ef_{pq} d\ef_{sp}^{-1} d\ef_{rs}$, and so the 
equation becomes 
\[ d\ef_{qr}^{-1} d\ef_{rs} d\ef_{sp}^{-1} + 
d\ef_{qr}^{-1} d\ef_{pq} d\ef_{sp}^{-1} = d\ef_{ps}^{-1} - 
d\ef_{qr}^{-1} \; , \]  that is, 
$d\ef^{-1}_{qr} (d\ef_{rs}+d\ef_{pq}) 
d\ef^{-1}_{sp} = d\ef^{-1}_{ps} - d\ef^{-1}_{qr}$, i.e. 
\[ d\ef_{rs} + d\ef_{pq} + d\ef_{qr} + d\ef_{sp} = 0 \; , \]  
and this follows from the fact that $\ef$ exists and 
so $d\ef$ is closed. 
\end{proof}

\begin{lemma}\label{lemmalemma8point21}
Let $\ef$ be a discrete isothermic surface.  Then 
the Christoffel transform $\ef^*$ of $\ef$ is isothermic with 
the same cross ratios as $\ef$.  
\end{lemma}

\begin{proof}
Let $q,q^*$ be the cross ratios of $\ef$, $\ef^*$ respectively.  Then 
\[ q^* = d\ef_{pq}^* (d\ef_{qr}^*)^{-1} d\ef_{rs}^* (d\ef_{sp}^*)^{-1} 
= a_{pq} d\ef_{pq}^{-1} (a_{qr} d\ef_{qr}^{-1})^{-1} 
a_{rs} d\ef_{rs}^{-1} (a_{sp} d\ef_{sp}^{-1})^{-1} = \]\[ 
(a_{pq}/a_{qr}) (a_{rs}/a_{sp}) d\ef_{pq}^{-1} (d\ef_{qr}^{-1})^{-1} 
d\ef_{rs}^{-1} (d\ef_{sp}^{-1})^{-1} = q^2 
(d\ef_{sp}^{-1} d\ef_{rs} d\ef_{qr}^{-1} d\ef_{pq})^{-1} \; . \]  
Then Lemma \ref{lem:forlateruse} implies 
\[ q^* = q^2 (d\ef_{pq} d\ef_{qr}^{-1} d\ef_{rs} d\ef_{sp}^{-1})^{-1} = 
q^2 \cdot q^{-1} = q \; . \]
\end{proof}

\subsection{Calapso transforms} 
Like in the smooth case, we can define Calapso 
transformations $T$ in the discrete case.  
We first define $\tau$ as 
\[
 \tau_{pq} = \begin{pmatrix} \ef_p \\ 1 \end{pmatrix} 
(\ef_q^* - \ef_p^*) 
\begin{pmatrix} 1 & -\ef_q \end{pmatrix} \; . \]  
Note that $\tau_{pq}$ does not have symmetry with respect to 
$p$ and $q$, and this was just a choice that was made, 
and there is no particular geometric motivation for choosing 
$\ef_p$ in the leftward vector and $\ef_q$ in the 
rightward vector.  Then taking any lift 
\[   F_p = \alpha_p \begin{pmatrix}
\ef_p & -\ef_p^2 \\ 1 & -\ef_p 
\end{pmatrix} \] at all $p$, 
a short computation gives 
\begin{equation}\label{gen-eqn-for-tau} \tau_{pq} = \begin{pmatrix}
\ef_p d\ef_{pq}^* & -\ef_p d\ef_{pq}^* \ef_q \\ 
d\ef_{pq}^* & -d\ef_{pq}^* \ef_q
\end{pmatrix} = 
-a_{pq} \frac{F_p F_q}{F_p F_q+F_qF_p} \; . \end{equation}  
Note that, although $F_p F_q+F_qF_p$ is a matrix, we are regarding 
it as a scalar here, like in Remark \ref{rem:matrixtoscalar}.  

If $F$ is a Moutard lift, then we can assume \eqref{eqn:Moutard}, 
and so we have 
\begin{equation}\label{eqn:Moutard-tau} 
\tau_{pq} = -F_p F_q \; . \end{equation}

For adjacent vertices $p,q$,  we define $T=T^\lambda$ by 
\begin{equation}\label{eqn:TpTq} 
T_q = T_p (I+\lambda \tau_{pq}) \; . \end{equation}  

This is not a commutative operation, as we will see in the proof of 
the next lemma, i.e. we cannot switch $p$ and $q$ 
and expect this equation to still hold.  So we must decide on a 
direction for each edge.  Let us do this by fixing one vertex p and then 
for any edge $\hat p \hat q$, where $\hat q$ is 
farther from $p$ than $\hat p$ is, apply the above equation to define 
$T_{\hat q} = T_{\hat p} (I+\lambda \tau_{\hat p \hat q})$.  It 
will turn out that this noncommutativity will not 
affect the Calapso transform (see the definition of the Calapso 
transform below), 
because $T$ is in fact defined up to real scalar factors 
even without this normalization of directions, 
so it is not a problem, 
but let us normalize these directions that we use in 
\eqref{eqn:TpTq} in order to 
choose a particular $T$.  (We will also use this normalization 
in the proofs of Lemmas \ref{lem:changeof-a-underCalapso} 
and \ref{nextlem:changeof-a-underCalapso}.)  

A direct computation shows that $I + \lambda \tau_{pq}$ 
(with $a,b,c,d$ now regarded as the entries in the matrix $I + 
\lambda \tau_{pq}$) 
satisfies \eqref{eqn:Mob3} and \eqref{eqn:Mob3-nonsing} when 
$1-\lambda a_{pq} \neq 0$, so $I + \lambda \tau_{pq} 
\in \text{Mob}(3)$ and then $T$ is as well, when the 
initial condition chosen for the solution $T$ is taken in 
$\text{Mob}(3)$.  

\begin{defn}
We say that $T F T^{-1}$ is a {\em Calapso transform}.  
\end{defn}

We can write $T F T^{-1}$ as $T_p F_p T_p^{-1}$ when we wish to 
specify which vertex $p$ is being used, and as $T^\lambda F 
(T^\lambda)^{-1}$ when we wish to 
specify which value of $\lambda$ has been chosen.  

We will see in the proof of the next lemma that 
$T$ is only defined up to real scalar factors, i.e. the $T$ are actually 
multivalued, and become well defined only when considered 
in a projectivized space.  
But, as noted above, this freedom does not affect the resulting 
Calapso transform $T F T^{-1}$.  

\begin{lemma}
If $\ef$ is a discrete isothermic surface, then a solution $T \in 
\text{Mob}(3)$ to \eqref{eqn:TpTq} exists.  
\end{lemma}

\begin{proof}
First we note that 
\[ (I+\lambda \tau_{pq})(I+\lambda \tau_{qp}) \]
is a real scalar multiple of $I$, so that $T$ is defined 
up to a real scalar factor when applying \eqref{eqn:TpTq} 
back and forth along a single edge.  
To see this, we need to see that 
\[ \tau_{pq}+\tau_{qp} \] is a real scalar multiple of $I$.  
Taking a Moutard lift 
$F$ of $\ef$ so that \eqref{eqn:Moutard} holds, 
$\tau_{pq} + \tau_{qp} = - F_p F_q - F_q F_p= 
-a_{pq} I$ is a real scalar multiple of $I$.  

For a quadrilateral with vertices $p,q,r,s$ 
in counterclockwise order, we have, if $T$ exists, that  
\[ 
T_r = T_q (I+\lambda \tau_{qr}) = 
T_p (I+\lambda \tau_{pq}) (I+\lambda \tau_{qr}) = \]
\[ T_p (I+\lambda \tau_{ps}) (I+\lambda \tau_{sr}) \; .
\]
So existence of $T$ would be implied by 
\begin{equation}\label{eqn:tau-relation} 
   (I+\lambda \tau_{pq}) (I+\lambda \tau_{qr}) = 
   (I+\lambda \tau_{ps}) (I+\lambda \tau_{sr}) \; , 
\end{equation}
that is to say, we want to show 
\[ \tau_{pq} \tau_{qr} = \tau_{ps} \tau_{sr} \] and 
\[ \tau_{pq} + \tau_{qr} - \tau_{sr} - \tau_{ps} = 
0 \; . \]  

The first of these two equations follows immediately from 
\[
  \begin{pmatrix} 1 & -\ef_q \end{pmatrix}
  \begin{pmatrix} \ef_q \\ 1 \end{pmatrix} = 0 \; , 
\] and the second one is not difficult to show if we 
use a Moutard lift satisfying 
\eqref{eqn:Moutard-tau}: Using such a lift $F$ means 
that we need only show 
\[ F_pF_q + F_qF_r = F_pF_s + F_sF_r \; , \]  
i.e. that 
\[ F_p (F_q-F_s)+(F_q-F_s)F_r = 0 \; . \]  
But by definition of the Moutard lift, $F_q-F_s$ and $F_p-F_r$ 
are parallel, so we need only show 
\[ F_p (F_p-F_r)+(F_p-F_r)F_r = 0 \; . \]  
This is clearly true, since $F_p^2=F_r^2 = 0$.  

Finally, as noted before, if $T_{p_0} \in \text{Mob(3)}$ at 
one vertex $p_0$, and $1-\lambda a_{pq}$ is never zero, then 
$T_p \in \text{Mob}(3)$ for all vertices $p$.  
\end{proof}

\begin{lemma}\label{lem:changeof-a-underCalapso}
Let $\ef$ be a discrete isothermic surface with lift $F$.  
The Calapso transform 
$F_p \to F_p^\mu := T_p^\mu F_p (T_p^\mu)^{-1}$ gives another 
isothermic surface $\ef^\mu$, and 
the cross ratio factorizing function $a_{pq}$ changes 
from $\ef$ to $\ef^\mu$ as follows: 
\[ a_{pq} \to a_{pq}^\mu = \frac{a_{pq}}{1-\mu a_{pq}} \; . \]  
\end{lemma}

\begin{proof}
Let $F$ be a Moutard lift satisfying \eqref{eqn:Moutard}.  
For a quadrilateral with vertices $p$, $q$, $r$ and $s$ listed in 
counterclockwise order around the quadrilateral, and noting that
\[ (I + \lambda \tau_{pq})(I + \lambda \tau_{qp}) = 
(1- \lambda a_{pq}) I \; , \] we have (assume $pq$ is directed from 
$p$ to $q$) 
\[ \langle F_p^\lambda , F_q^\lambda \rangle = 
\langle T_pF_pT_p^{-1}, T_qF_qT_q^{-1} \rangle = \]\[ \tfrac{-1}{2} 
[ T_p F_p (I+\lambda \tau_{pq}) F_q T_q^{-1} + \frac{1}{1-\lambda a_{pq}}
  T_q F_q (I+\lambda \tau_{qp}) F_p T_p^{-1} ] = \]
\[ 
\tfrac{-1}{2} 
[ T_p F_p F_q T_q^{-1} + \frac{1}{1-\lambda a_{pq}}
  T_q F_q F_p T_p^{-1} ] = \]\[ \tfrac{-1}{2} T_p 
[ F_p F_q 
  \frac{1}{1-\lambda a_{pq}} \cdot I + \frac{1}{1-\lambda a_{pq}} 
  I \cdot F_q F_p ] T_p^{-1} = 
  \frac{1}{1-\lambda a_{pq}} \langle F_p,F_q \rangle \; . \]
Also (assume as well that $qr$ is directed from $q$ to $r$), 
\[ \langle F_p^\lambda , F_r^\lambda \rangle = 
\langle T_pF_pT_p^{-1}, T_rF_rT_r^{-1} \rangle = \tfrac{-1}{2} 
\left[ T_pF_pT_p^{-1} T_rF_rT_r^{-1} + T_rF_rT_r^{-1} T_pF_pT_p^{-1} 
\right] = \]\[
\tfrac{-1}{2} \left[ T_pF_pT_p^{-1} T_qT_q^{-1} T_rF_rT_r^{-1} + 
T_rF_rT_r^{-1} T_qT_q^{-1} T_pF_pT_p^{-1} \right] = \]\[ 
\tfrac{-1}{2} \left[ T_pF_p \cdot I \cdot I \cdot F_rT_r^{-1} + 
\frac{1}{1-\lambda a_{pq}} \frac{1}{1-\lambda a_{qr}} T_rF_rF_pT_p^{-1} 
\right] = \]\[ \tfrac{-1}{2} 
T_q \left[ T_q^{-1} T_pF_p F_rT_r^{-1} T_q + 
\frac{1}{1-\lambda a_{pq}} 
\frac{1}{1-\lambda a_{qr}} T_q^{-1}T_rF_rF_pT_p^{-1}T_q \right] 
T_q^{-1} = \]\[ 
\frac{1}{1-\lambda a_{pq}} \frac{1}{1-\lambda a_{qr}} 
\langle F_p,F_r \rangle \; . \]  Similarly, 
\[ \langle F_p^\lambda , F_s^\lambda \rangle = 
\frac{1}{1-\lambda a_{ps}} \langle F_p,F_s \rangle 
\; , \]
\[ \langle F_q^\lambda , F_r^\lambda \rangle = 
\frac{1}{1-\lambda a_{qr}} \langle F_q,F_r \rangle 
\; , \]
\[ \langle F_q^\lambda , F_s^\lambda \rangle = 
\frac{1}{1-\lambda a_{pq}} 
\frac{1}{1-\lambda a_{qr}} \langle F_q,F_s \rangle 
\; . \]
We now renotate the subscripts $p,q,r,s$ by $p_1,p_2,p_3,p_4$, 
respectively.  Then, 
using the lift $F$ satisfying \eqref{eqn:Moutard} chosen here in 
Lemma \ref{lem:generalformulaforcrossratio} (note that we do not need to 
require $F \in M_\kappa$), and noting that we have 
$s_{12}=s_{34}=-\tfrac{1}{2} a_{p_1p_2}$ and $s_{14}=s_{23}= 
-\tfrac{1}{2} a_{p_1p_4}$, and because 
$q = a_{p_1p_2}/a_{p_1p_4}$, Lemma \ref{lem:generalformulaforcrossratio} 
implies $s_{13}s_{24} = \tfrac{1}{4} (a_{p_1p_2}-a_{p_1p_4})^2$.  
A computation, again using Lemma 
\ref{lem:generalformulaforcrossratio}, then shows that the 
corresponding cross ratio on the 
Christoffel transform $\ef^\mu$ is 
\[ a_{pq}^\mu/a_{ps}^\mu \; , \] where 
\[ a_{pq}^\mu = \frac{a_{pq}}{1-\mu a_{pq}} \; , \;\;\; 
a_{ps}^\mu = \frac{a_{ps}}{1-\mu a_{ps}} \; . \]  
Thus $\ef^\mu$ is an isothermic surface, and the lemma is proven.  
\end{proof}

In the above proof we saw that 
$\langle F_p^\mu,F_q^\mu \rangle = (1-\mu a_{pq})^{-1} 
\langle F_p,F_q \rangle = - \tfrac{1}{2} a_{pq}^\mu$ and 
$\langle F_p^\mu,F_s^\mu \rangle = (1-\mu a_{ps})^{-1} 
\langle F_p,F_s \rangle = - \tfrac{1}{2} a_{ps}^\mu$, 
so this corollary follows: 

\begin{corollary}\label{cor:Mout-to-Mout}
If $F_p$ is a Moutard lift of a discrete isothermic surface 
$\ef$ satisfying 
\eqref{eqn:Moutard}, then so is $F_p^\mu$, for any $\mu \in 
\mathbb{R} \setminus \{ 0 \}$.  
\end{corollary}

In order to state the next lemma, 
we define $T^{\lambda,\mu}$ by 
\[ T^{\lambda,\mu}_q = T^{\lambda,\mu}_p (I+\mu \tau_{pq}^\lambda) \; , \]  
where 
\[ \tau_{pq}^\lambda = \frac{-a_{pq}^\lambda F_p^\lambda 
F_q^\lambda}{F_p^\lambda F_q^\lambda+F_q^\lambda F_p^\lambda} \; , 
\;\;\; a_{pq}^\lambda = \frac{a_{pq}}{1-\lambda a_{pq}} \; , \;\;\; 
F_p^\lambda = T_p^\lambda F_p (T_p^\lambda)^{-1} \; . \]

\begin{lemma}\label{nextlem:changeof-a-underCalapso}
Let $\ef$ be a discrete isothermic surface with associated $T$.  Then 
$T$ is a 1-parameter group, that is, we can choose 
$T^{\lambda,\mu}$ so that \[ T^{\mu+\lambda} = 
T^{\lambda,\mu} T^\lambda \] for any $\lambda, \mu \in \mathbb{R}$.  
\end{lemma}

\begin{proof}
Without loss of generality, assume $F$ is a Moutard lift 
satisfying \eqref{eqn:Moutard}, and so Corollary \ref{cor:Mout-to-Mout}
implies $\tau_{pq}^\lambda = - F_p^\lambda F_q^\lambda$.  
First note that $T_q^\lambda = T_p^\lambda 
(I+\lambda \tau_{pq})$.  
We wish to show $T^{\mu+\lambda} = T^{\lambda,\mu} T^\lambda$, i.e. 
\begin{equation}\label{eqn:first-1param-gp-eqn}
 T^{\lambda,\mu}_q T_q^\lambda = T^{\lambda,\mu}_p T^\lambda_p 
(I+(\mu+\lambda) \tau_{pq}) \; , \end{equation} 
where the edge $pq$ is directed from $p$ to $q$.  Note that 
\begin{equation}\label{eqn:Teqn1} 
(T_p^\lambda)^{-1} T_q^\lambda=I+\lambda \tau_{pq} \; , 
\end{equation} and inverting gives 
\begin{equation}\label{eqn:Teqn2}
(T_q^\lambda)^{-1} T_p^\lambda=\frac{1}{1-\lambda a_{pq}} 
(I+\lambda \tau_{qp}) \; , 
\end{equation} since $F$ is a Moutard lift satisfying 
\eqref{eqn:Moutard}.  Then 
\[ T^{\lambda,\mu}_q T_q^\lambda = T^{\lambda,\mu}_p (I- \mu 
F_p^\lambda F_q^\lambda) T_q^\lambda = \]\[ 
= T^{\lambda,\mu}_p (I-\mu T_p^\lambda F_p 
(T_p^\lambda)^{-1} T_q^\lambda F_q (T_q^\lambda)^{-1}) 
T_q^\lambda \; . \]  Then, using the properties 
$F_p \tau_{pq} = \tau_{pq} F_q = 0$ and \eqref{eqn:Teqn1}, we have 
\[ T^{\lambda,\mu}_q T_q^\lambda = T_p^{\lambda,\mu} (I+\mu 
T_p^\lambda \tau_{pq} (T_q^\lambda)^{-1}) T_q^\lambda = \]\[ 
= T_p^{\lambda,\mu} T_p^\lambda ((T_p^\lambda)^{-1} T_q^\lambda + \mu 
\tau_{pq}) = 
T_p^{\lambda,\mu} T_p^\lambda ((I + \lambda \tau_{pq}) + \mu 
\tau_{pq}) = \]\[ = 
T_p^{\lambda,\mu} T_p^\lambda (I + (\lambda+\mu) \tau_{pq}) \; . \]  
Thus we have shown \eqref{eqn:first-1param-gp-eqn}.  
\end{proof}

Now we recall that, 
for general lifts $F$ that are not necessarily Moutard, we have 
\[ \tau_{pq} = \frac{-a_{pq} F_p F_q}{F_p F_q+F_q F_p} \; , \;\;\; 
\tau_{pq}^\mu = \frac{-a_{pq}^\mu F_p^\mu F_q^\mu}{F_p^\mu F_q^\mu 
+ F_q^\mu F_p^\mu} \; , \] and, by Equations 
\eqref{eqn:Teqn1} and \eqref{eqn:Teqn2}, we have  
\begin{equation}\label{eqn:invert1plustau} 
  (I+\mu \tau_{pq})^{-1} = \frac{1}{1-\mu a_{pq}} (I + 
  \mu \tau_{qp}) \; . \end{equation} 

\begin{remark}
Equation \eqref{eqn:invert1plustau} 
is not symmetric in $p$ and $q$.  In fact, as noted before, $\tau$ 
itself is not symmetric in $p$ and $q$.  
However, the most essential object, the family of 
flat connections $\Gamma_{pq}^\lambda$, is symmetric in 
$p$ and $q$ (see Remark \ref{symmetricityofGamma}).  We will discuss 
flat connections in the next Section \ref{flatconnections}.  
\end{remark}

Furthermore, if $F$ is Moutard satisfying \eqref{eqn:Moutard}, this 
is true of $F^\mu$ as well, by Corollary 
\ref{cor:Mout-to-Mout}, and we have 
$\tau_{pq}^\mu = -F_p^\mu F_q^\mu = - T_p^\mu F_p (T_p^\mu)^{-1} 
T_q^\mu F_q (T_q^\mu)^{-1} = -T_p^\mu F_p (I+\mu \tau_{pq}) 
F_q (T_q^\mu)^{-1} = T_p^\mu (-F_p F_q) (T_q^\mu)^{-1}$, so we have 
\begin{equation}\label{eqn:Teqn3} 
\tau_{pq}^\mu = T_p^\mu \tau_{pq} (T_q^\mu)^{-1} \; . \end{equation} 
This equation will be used later, when we show that if $\ef$ 
has a polynomial conserved quantity of type $n$, then so do 
its Calapso transformations (see Lemma 
\ref{lem:Calapso-preserves-type}).  
In particular, if $\ef$ is a discrete isothermic 
CMC surface in some space form, then so are its Calapso 
transformations (in different space forms in general).  But 
since we have not defined the notions 
of polynomial conserved quantities and discrete CMC surfaces 
yet, we come back to this later.  

\subsection{Flat connections}\label{flatconnections}
Let us first review 
what a connection is in the smooth case.  We will see how 
isothermic surfaces have a 1-parameter family of flat connections.  
Although we do not show it here (see 
\cite{BursPinkPed} for such an argument), 
the converse is also true: existence of a family 
of flat connections implies that the surface is isothermic.  

Recall that the Riemannian connection of a 
Riemannian manifold is the unique connection satisfying 
\begin{equation}\label{eqn:flatconnection1} 
\nabla_{fX+Y}Z=f \nabla_X Z+\nabla_Y Z \; , \end{equation} 
\begin{equation}\label{eqn:flatconnection2} 
\nabla_{X}(fY+Z)=X(f) Y+f \nabla_X Y+\nabla_X Z \; , \end{equation} 
\begin{equation}\label{eqn:flatconnection3} 
\nabla_X Y - \nabla_Y X=[X,Y] \; , \end{equation} 
\begin{equation}\label{eqn:flatconnection4} 
X \langle Y,Z \rangle =\langle \nabla_X Y,Z \rangle + \langle Y, 
                           \nabla_X Z \rangle \; , \end{equation} 
where $X,Y,Z$ are any smooth tangent vector fields of the manifold, and 
$f$ is any smooth function from the manifold to $\mathbb{R}$.  
The first two relations \eqref{eqn:flatconnection1}, 
\eqref{eqn:flatconnection2} define general affine connections, 
and adding in the last two 
conditions \eqref{eqn:flatconnection3}, \eqref{eqn:flatconnection4} 
makes the connection a Riemannian connection.  

Taking an $n$-dimensional manifold $M^n$ with affine connection 
$\nabla$, and taking a basis $X_1$, $X_2$, ... , $X_n$ of vector 
fields for the tangent spaces, we define $\Gamma_{ij}^k$ and 
$R_{lij}^k$ by \[ 
\nabla_{X_i} X_j = \sum_{k=1}^n \Gamma_{ij}^k X_k \; , \] 
\begin{equation}\label{eqn:curvaturetensor} 
\nabla_{X_i} \nabla_{X_j} X_l - \nabla_{X_j} \nabla_{X_i} X_l 
- \nabla_{[X_i,X_j]} X_l = \sum_{k=1}^n R_{lij}^k X_k \; . 
\end{equation}  
We define the one forms $\omega^i$ and $\omega_j^i$ by 
(here $\delta_j^i$ is the Kronecker delta function) 
\[ \omega^i(X_j) = \delta_j^i \; , \;\;\; \omega_j^i = 
\sum_{k=1}^n \Gamma_{kj}^i \omega^k \; . \]  
The one forms $\omega_j^i$ are called the connection 
one forms.  Then 
\[ d\omega_l^i + \sum_{p=1}^n \omega_p^i \wedge 
\omega_l^p = \tfrac{1}{2} \sum_{j,k=1}^n R_{ljk}^i \omega^j \wedge 
\omega^k \; . \]  
When the connection is the Riemannian connection, the 
$R_{ljk}^i$ give the Riemannian curvature tensor.  
When, for an affine connection, all of the $R_{ljk}^i$ are zero, 
then we say that $\nabla$ is a flat connection.  
For a more thorough explanation of the above equations, there are 
many textbooks one could look at, for example \cite{Helgason}.  

For a 
smooth isothermic surface $x$, we can regard $\mathbb{R}^{4,1}$ 
as $5$-dimensional fibers of a trivial vector bundle defined on $x$.  
We now 
define $\nabla = d + \lambda \tau$ for any choice of $\lambda 
\in \mathbb{R}$, i.e. \begin{equation}\label{ourconnection} 
\nabla_Z Y = d_Z Y + \lambda (\tau(Z) \cdot Y - Y \cdot \tau(Z)) \; , 
\end{equation}
where $Y \in \mathbb{R}^{4,1}$ depends on the parameters $u,v$ for 
the isothermic surface $x$, and $Z$ lies in the tangent space 
of the surface.  This is a bit different than the 
considerations above, because now the bundle is not 
the tangent bundle of the surface $x$, and so $Y$ is not necessarily 
tangent to $x$.  But in any case, $\tau(Z)$ is defined, because 
$Z$ lies in the tangent space of $x$.  We note that 
\eqref{eqn:flatconnection1} and \eqref{eqn:flatconnection2} hold, 
and so this $\nabla$ is an affine connection.  

We wish to see that this 
$\nabla$ in \eqref{ourconnection} is a flat connection for all 
$\lambda \in \mathbb{R}$.  That is, we wish to have 
\begin{equation}\label{flateqn1}
\nabla_{\partial_u} \nabla_{\partial_v} Y - 
\nabla_{\partial_v} \nabla_{\partial_u} Y - 
\nabla_{[\partial_u,\partial_v]} Y = 0 
\end{equation}
for any $Y \in \mathbb{R}^{4,1}$ depending on $u$ and $v$, 
and for any $\lambda \in 
\mathbb{R}$.  Because $[\partial_u,\partial_v]=0$, 
a computation shows that 
\eqref{flateqn1} will hold if 
\[ d(\lambda \tau) + (\lambda \tau) \wedge 
(\lambda \tau) = 0 \] holds for all 
$\lambda \in \mathbb{R}$, i.e. 
\begin{equation}\label{flateqn2}
\partial_u (\tau(\partial_v)) - \partial_v (\tau(\partial_u)) = 
\tau(\partial_u) \tau(\partial_v) - 
\tau(\partial_v) \tau(\partial_u) = 0 \; . 
\end{equation}

Then $\nabla$ is a family 
of flat connections parametrized by $\lambda$.  
Let us now confirm that $\nabla$ is flat: 

\begin{lemma}
Equation \eqref{flateqn2} holds.  
\end{lemma}

\begin{proof}
The proof is a direct computation using the properties in 
\eqref{eqn:isothermicityproperties}.  
\end{proof}

Connections are equivalent to having a notion of parallel transport 
along each given curve in the surface, and a connection is 
flat if and only if the parallel transport map depends only 
on the homotopy class of each curve (with fixed endpoints). 
In particular, if the surface $x$ is simply connected, parallel 
transport is independent of path if and only if the connection 
is flat, which can be seen as follows:  One direction is 
immediately clear from Equation 
\eqref{eqn:curvaturetensor}, by choosing the $X_i$ 
there to be constant vector fields (that is, by choosing $X_i$ 
by using parallel translation, i.e. $\nabla_* X_i = 0$), 
and then all $R_{l i j}^k$ become 
$0$.  To see the other direction, 
suppose that the connection is flat.  Then Equation 
\eqref{eqn:curvaturetensor} implies 
\[ \nabla_{\partial_u} \nabla_{\partial_v} Y = 
\nabla_{\partial_v} \nabla_{\partial_u} Y \] for any 
vector field $Y$.  Then we can apply an argument like in the proof of 
Proposition 3.1.2 in \cite{wisky} to conclude that if 
$Y$ is constructed so that $\nabla_{\partial_u} 
Y = 0$ along one curve where $v=v_0$ is constant and so that 
$\nabla_{\partial_v} Y = 0$ everywhere, then also 
$\nabla_{\partial_u} Y = 0$ everywhere, and so $Y$ is a 
vector field that is parallel on any curve in $x$.  

Thus, because the connection in Equation \eqref{ourconnection} 
is flat, every vector at one point of a simply-connected 
$x$ can be extended to a 
parallel vector field defined over all of $x$ that 
is independent of choice of path.  Let us denote such a vector 
field by 
\[ Y = \phi^{-1} \cdot Y_0 := \phi^{-1} Y_0 \phi \; , \] 
where $\phi$ is a map from the domain of $x$ (with 
isothermic coordinates $u,v$) to Mob(3), and 
$Y_0$ is any fixed vector in $\mathbb{R}^{4,1}$.  
The condition that $Y$ is parallel is 
\begin{equation}\label{eqn:phiparalleltransport} 
0 = \nabla_Z (\phi^{-1} \cdot Y_0) \end{equation} 
for all vectors $Z$ tangent to the surface $x$, at any point of 
$x$.  Equation 
\eqref{eqn:phiparalleltransport} holds if and only if 
\[ 0 = d_Z (\phi^{-1} Y_0 \phi) + 
\lambda (\tau(Z) \cdot \phi^{-1} Y_0 \phi - 
\phi^{-1} Y_0 \phi \cdot \tau(Z)) \] 
for all $Z$, which then holds if and only if 
\[ [ \mathcal{R}(Z) , Y_0 ] = 0 \] for all $Z$, where 
\[ \mathcal{R} = (d\phi) \cdot \phi^{-1} - \lambda \phi 
\tau \phi^{-1} \; . \] 
This is true for all $Z$ tangent to 
$x$, and for any choice of $Y_0 \in \mathbb{R}^{4,1}$.  
It would certainly suffice to have $\mathcal{R} = 0$, i.e. 
\begin{equation}\label{eqn:phiequalslambdaphitau} 
d\phi = \lambda \phi \tau \; . \end{equation}  
So we can take $\phi$ to be the Calapso transformation $T$, 
as in Definition 
\ref{defnthatis8pt43} and Lemma \ref{lem:Calapexist}.  

\begin{remark}\label{rem:TvsPhi}
Note that when $Y_0 \in L^4$, then $Y$ is actually 
a Darboux transform of the surface.  
\end{remark}

Equation \eqref{eqn:phiequalslambdaphitau} 
is how we can describe parallel transportation in 
terms of $\tau$.  

To get a connection for a discrete isothermic surface $\ef$, 
it is not the 
connection $\nabla$ that we will discretize, but rather the notion 
of parallel transport and Equation 
\eqref{eqn:phiequalslambdaphitau}: the discrete version of 
Equation \eqref{eqn:phiequalslambdaphitau} is 
\[ \phi_q-\phi_p= \lambda \phi_p \tau_{pq} \] 
along edges $pq$ directed from $p$ to $q$, 
i.e. \begin{equation}\label{eqnstarsubstar} 
\phi_p^{-1} \phi_q = I + \lambda \tau_{pq} \; . \end{equation} 
Note that this is exactly the same equation as \eqref{eqn:TpTq}.  

Let $Y_0$ be a fixed vector in $\mathbb{R}^{4,1}$.  
Analogous to the smooth case as above, for a solution $\phi$ to 
\eqref{eqnstarsubstar}, we form 
a vector field defined on the vertices of $\ef$ by 
\[ Y_p = \phi_p^{-1} \cdot Y_0 = \phi_p^{-1} Y_0 \phi_p \; , \] 
we then have the following: obviously 
$Y_0 = \phi_q (\phi_q^{-1} Y_0 \phi_q) \phi_q^{-1} = 
\phi_p (\phi_p^{-1} Y_0 \phi_p) \phi_p^{-1}$, and so 
\[ 
\phi_q Y_q \phi_q^{-1} = \phi_p Y_p \phi_p^{-1} 
\;\; \text{implies} \;\; 
\phi_p^{-1} \phi_q Y_q (\phi_p^{-1} \phi_q)^{-1} = Y_p 
\; , \] thus $(I+\lambda \tau_{pq}) Y_q (I+\lambda \tau_{pq})^{-1} 
= (1-\lambda a_{pq})^{-1} (I+\lambda \tau_{pq}) Y_q 
(I+\lambda \tau_{qp})$ by \eqref{eqn:invert1plustau}.  Thus 
\begin{equation}\label{gammaconnexion} 
\Gamma_{pq} \cdot Y_q = Y_p \; , \end{equation} 
where we define $\Gamma_{pq}=\Gamma_{pq}^\lambda$, 
as long as $\lambda a_{pq} \neq 1$, 
by (the symbol $\Gamma$ now plays a different role than 
it did at the beginning of this section) 
\begin{equation}\label{eqn:star-sylvia} 
\Gamma_{pq} \cdot Y_q = (1-\lambda a_{pq})^{-1} 
(I + \lambda \tau_{pq}) Y_q (I + \lambda \tau_{qp}) \; . \end{equation} 
Equation \eqref{gammaconnexion} defines parallel 
transport along edges, and thus provides a connection for the 
surface.  We conclude that $\Gamma_{pq}$ is a 
{\em flat $\text{Mob}(3)$-connection} on the discrete isothermic 
net, with the solution $\phi$ being a {\em gauge transformation} 
identifying this connection with the trivial connection.  

\begin{remark}\label{symmetricityofGamma}
The connection $\Gamma_{pq}$ is symmetric 
in the following sense: If, instead, $pq$ had been 
directed from $q$ to $p$, then $(\phi_q^{-1} \phi_p)^{-1} 
Y_q \phi_q^{-1} \phi_p = Y_p$ implies $(I+\lambda \tau_{qp})^{-1} 
Y_q (I + \lambda \tau_{qp}) = (1-\lambda a_{pq})^{-1} 
(I+\lambda \tau_{pq}) Y_q (I+\lambda \tau_{qp})$, 
and the definition of $\Gamma_{pq}$ 
in \eqref{eqn:star-sylvia} would not change; that is, $\Gamma_{pq}$ 
is independent of choice of direction along the edge $pq$.  
\end{remark}

Now, parallel sections $Y \in \mathbb{R}^{4,1}$ are those that 
satisfy \eqref{gammaconnexion} 
for some $\lambda \in \mathbb{R}$, and then 
$Y_q \to Y_p$ is parallel transport along edges.  

Note that $\Gamma_{pq}^\lambda 
\Gamma_{qp}^\lambda = 1$, by \eqref{eqn:invert1plustau}, 
and for this reason we call $\Gamma^\lambda$ 
a {\em connection}.  By 
\eqref{eqn:tau-relation} and \eqref{eqn:invert1plustau}, we have 
\[
 \Gamma_{pq}^\lambda \Gamma_{qr}^\lambda 
 \Gamma_{rs}^\lambda \Gamma_{sp}^\lambda = 1 \; , 
\]and for this reason we call it a {\em flat} connection.  
Finally, we call the $\Gamma_{pq}^\lambda$ 
as in \eqref{eqn:star-sylvia} the {\em isothermic family of 
connections} of $\ef$.  

\subsection{Linear conserved quantities} 
We can now discretize \eqref{rossman-first-equation} 
as follows: We say that $\ef$ is CMC (in the appropriate 
space form) if there exists a linear conserved 
quantity $P=Q+\lambda Z$ so that $T P T^{-1}$ is 
constant with respect to vertices in the domain of $\ef$.  
Here, $Q$ and $Z$ are maps defined on the lattice domain 
and taking values in $\mathbb{R}^{4,1}$.  (See Definition 
\ref{defn:disclqcCMC} below.)  We have proven that this 
holds in the smooth case (see Equation \eqref{eqn:forpg68}), and 
we take it as a definition in the discrete case.  
We will see in Lemma 
\ref{lemmalemma9pt13} the equivalence of this definition with 
previous definitions of discrete CMC surfaces.  
That $T P T^{-1}$ is constant is equivalent to 
\[ T_q P_q T_q^{-1} = T_p P_p T_p^{-1} \] for all 
adjacent vertices $p$ and $q$, which is equivalent to 
\[ (I+\lambda \tau_{pq}) P_q = P_p (I+\lambda \tau_{pq}) \; , \] 
which becomes the equation 
\begin{equation}\label{linearconservedquantity-rossman} 
(I+\lambda \tau_{pq}) (Q+\lambda Z)_q = (Q+\lambda Z)_p (I+\lambda 
\tau_{pq}) \; . 
\end{equation} 

\begin{remark}
Note that \eqref{linearconservedquantity-rossman} is 
equivalent to saying that $P$ is a parallel section of the 
flat connection $\Gamma_{pq}^\lambda$, for all $\lambda$.  
\end{remark}

Looking at the coefficients in front of the $\lambda^k$ in Equation 
\eqref{linearconservedquantity-rossman} 
for $k=0,1,2$, we immediately have the following lemma: 

\begin{lemma}\label{lem:lcq-in-parts}
Equation \eqref{linearconservedquantity-rossman} is equivalent to 
$dQ_{pq}=0$ and $dZ_{pq}=Q_p \tau_{pq}-\tau_{pq} Q_q$ and 
$\tau_{pq} Z_q=Z_p \tau_{pq}$.  
\end{lemma}

Noting that $Q$ is constant, we now come to a formal definition: 

\begin{defn}\label{defn:disclqcCMC}
If a linear conserved quantity $Q+\lambda Z$, $Q \neq 0$, exists for an 
isothermic discrete surface $\ef$, we say that $\ef$ is of 
constant mean curvature (CMC) in the space form $M$ determined by $Q$.  
\end{defn}

The first fact we give about these linear conserved quantities is this: 

\begin{lemma}\label{Zisctecte}
$||Z||$ is constant, that is, $||Z_p||$ does not depend on the choice 
of vertex $p$.  
\end{lemma}

\begin{proof}
We give an argument similar to the argument in the proof of Lemma 
\ref{lem-smooth:zsquared-is-cte}.  
Let $p$ and $q$ be adjacent vertices.  Then (with $\tau = \tau_{pq}$) 
\[ Z_q^2-Z_p^2=(Z_q-Z_p)Z_q + Z_p (Z_q-Z_p) = 
(Q \tau - \tau Q) Z_q + Z_p (Q \tau - \tau Q) = \]\[ 
= Q Z_p \tau - \tau Q Z_q + Z_p Q \tau - \tau Z_q Q = 
(Q Z_p + Z_p Q) \tau - \tau (Q Z_q+Z_q Q) \; . \]
We know that $Q Z_p + Z_p Q$ and $Q Z_q+Z_q Q$ are real 
multiples of the identity matrix, so it will suffice 
to prove $Q Z_p + Z_p Q = Q Z_q+Z_q Q$, which we do as follows: 
\[ (Q Z_p + Z_p Q - Q Z_q - Z_q Q) \tau = - (Q dZ_{pq} + dZ_{pq} Q) 
\tau = \]\[ = - (Q (Q\tau - \tau Q) + (Q \tau - \tau Q) Q) \tau 
= - (Q^2 \tau - \tau Q^2) \tau = 0\; . \]  
\end{proof}

Then, in analogy to \eqref{eqn:Hwithscalarfactor}, we define the 
mean curvature to be 
\[
  H = - \langle Z,Q \rangle 
\]
when we have normalized the conserved quantity by a scalar factor 
so that $||Z||=1$, which we can do because we know from the above 
lemma that $||Z||$ is constant.  This normalization also changes 
$Q$ by a scalar factor, thus potentially 
changing the curvature of the ambient 
space.  Even if we do not normalize the 
linear conserved quantity, we can still define the mean cruvature, 
like as in \eqref{eqn:Hwithscalarfactor}.  

\begin{remark}
One can see, in the case of $M_0 = 
\mathbb{R}^3$, that the above definition is equivalent 
to the definition found by Bobenko and Pinkall \cite{BobPink}: 
$\ef$ is CMC if $|\ef_p-\ef_p^*|^2$ is constant, and then that 
constant is $H_0^{-2}$.  This is proven in 
\cite{BHRS}.  Also, the property of being discrete 
CMC is preserved by Calapso transformations 
(see Lemma \ref{lem:Calapso-preserves-type} below), so the definition 
here is the right one for the space form $M_1=
\mathbb{S}^3$, and also for the space form $M_{-1} = 
\mathbb{H}^3$ when the mean curvature $H_{-1}$ has absolute 
value at least $1$.  
\end{remark}

\begin{remark}
Unlike the case of smooth surfaces, $Z$ will not be called the 
{\em central sphere congruence}.  We will call it the 
{\em mean curvature sphere congruence}, for any space form.  
In the discrete case, the central sphere congruence 
and mean curvature sphere congruence are generally not the same.  
(See Definition \ref{defnofcentralS}.)  
\end{remark}

Lemma \ref{lem:lcq-in-parts} gives the following two corollaries.  
The proofs are 
not hard.  One just needs to note that there exists an 
imaginary quaternion $n_p$ such that we can write 
\[ Z_p = \begin{pmatrix}
C_p \ef_p + n_p & B_p \\ C_p & - C_p \ef_p - n_p 
\end{pmatrix} \; , \;\;\; B_p, C_p \in \mathbb{R} \; , \] 
and then use Lemma \ref{lem:lcq-in-parts} to compute $B_p$.  Here 
we have defined $C_p$ as the lower left entry of $Z_p$ and then 
chosen $n_p$ to be the upper left entry minus 
$C_p$ times $\ef_p$.  

\begin{corollary}\label{cor:AinMarch2008}
Assume $\ef$ has a linear conserved quantity.  If $\kappa = 0$ 
and $Q$ is as in \eqref{choiceofQ}, then 
\[ Z_p = \begin{pmatrix}
H \ef_p + n_p & -n_p \ef_p-\ef_p n_p-H \ef_p^2 \\ H & - H \ef_p - n_p 
\end{pmatrix} \; , \] 
for some constant $H \in \mathbb{R}$.  Furthermore, $|n_p|^2$ is constant 
(because $||Z_p||$ is constant), and 
\[ d\ef^*_{pq} = d(H \ef+n)_{pq} \; , \;\;\; d\ef_{pq} n_q + n_p d\ef_{pq} = 0 \]
and 
\[ H \ef_q^2-H \ef_p^2+n_q \ef_q+\ef_q n_q-n_p \ef_p-\ef_p n_p = 
d\ef^* \ef_q + \ef_p d\ef^* \; . \]
\end{corollary}

We note that the equation $d\ef_{pq} n_q + n_p d\ef_{pq} = 0$ 
could have been replaced with the equivalent equation 
$d\ef^*_{pq} n_q + n_p d\ef^*_{pq} = 0$ in the above corollary.  

\begin{corollary}\label{cor:BinMarch2008}
Assume $\ef$ has a linear conserved quantity and $Q$ is as in 
\eqref{choiceofQ} for some $\kappa$.  Then 
\[ Z_p = \begin{pmatrix}
H_p \ef_p + n_p & -n_p \ef_p-\ef_p n_p-H_p \ef_p^2 \\ H_p & - H_p \ef_p - n_p 
\end{pmatrix} \; , \;\;\; H_p \in \mathbb{R} \; , \] 
for some function $H_p$ from the lattice domain of $\ef$ 
to $\mathbb{R}$.  Furthermore, $|n_p|^2$ is constant.
\end{corollary}

In light of Lemma \ref{lem:lcq-in-parts}, we now give three properties 
of linear conserved quantities: 

\begin{lemma}\label{useful-lemma1} 
Let $F$ be a Moutard lift of a discrete isothermic surface 
$\ef$ having a linear conserved quantity $Z+\lambda Q$.  
Suppose further that $F$ satisfies \eqref{eqn:Moutard-tau}.  
Then $dZ_{pq}=Q \tau_{pq} - \tau_{pq} Q$ is equivalent to 
\begin{equation}\label{eqn:threestars} 
Z_q = Z_p - (Q F_p+F_p Q) F_q + (Q F_q+F_q Q) F_p 
\end{equation}
for all adjacent $p,q$.  
\end{lemma}

\begin{proof}
Because $Q F_q+F_q Q$ is a real scalar multiple of $I$ for any $p$, 
we have $-(Q F_p+F_p Q) F_q + (Q F_q+F_q Q) F_p 
= -(Q F_p+F_p Q) F_q + F_p (Q F_q+F_q Q)
= F_p F_q Q-Q F_p F_q = Q \tau - \tau Q$. 
\end{proof}

\begin{corollary}
Once $Z$ is determined at one vertex $p$, it is uniquely 
determined via \eqref{eqn:threestars} at all vertices.  
\end{corollary}

\begin{lemma}\label{useful-lemma2} 
Assume the conditions in Lemma \ref{useful-lemma1}.  Then 
$Z_p \tau_{pq}=\tau_{pq} Z_q$ for all adjacent $p,q$ 
is equivalent to $Z_p F_p + F_p Z_p = 0$ for all $p$.  
\end{lemma}

\begin{proof}
Setting $\tau = \tau_{pq}$, 
$Z_p \tau = \tau Z_q$ implies $(F_p Z_p+Z_p F_p) \tau = 
F_p Z_p \tau = F_p \tau Z_q = 0 \cdot Z_q = 0$.  So 
$F_p Z_p+Z_p F_p=0$.  Conversely, $(F_p Z_p+Z_p F_p) F_q - F_p 
(F_q Z_q+Z_q F_q) = 0 \cdot F_q - F_p \cdot 0 = 0$, implying 
$Z_p F_p F_q = F_p F_q Z_q$ by Lemma 
\ref{useful-lemma1}, and then $Z_p \tau = \tau Z_q$.  
\end{proof}

\begin{remark}
Suppose that $\ef$ has a conserved quantity $P = 
Q + \lambda \cdot 0$ of order $0$ 
with $||P||^2$ not equal to zero.  Then $\ef$ is contained in 
a sphere, like for the case of smooth surfaces (Theorem 
\ref{cq-sphere-thm}), and this can be seen as follows: 
$P=Z=Q$ (i.e. $Q$ is both the 
highest and lowest coefficient of $P$) 
is constant in the case of order $0$, 
with $||Z||^2 \neq 0$ by assumption.  Thus the upcoming Lemma 
\ref{lem:sixfacts} 
tells us $||Z||^2 > 0$ and $Z \perp F_p$ for all $p$.  So $Z$ 
gives a sphere via \eqref{eq:S-tilde-sphere} and $\ef_p$ lies in that 
sphere for all $p$.  
\end{remark}

\subsection{On uniqueness of linear conserved quantities} 

When the domain of $\ef$ is 
\[ \{ (m,n) \in \mathbb{Z}^2 \, | \, 1 \leq m,n \leq k \} \; , \]  
or any translation of that domain, we say $\ef$ is a $k$ by 
$k$ net.  The vertex star of a vertex $\ef_{(m,n)}$ consists 
of it and its four neighboring vertices $\ef_{(m+1,n)}$, 
$\ef_{(m,n+1)}$, $\ef_{(m-1,n)}$, $\ef_{(m,n-1)}$.  When all 
five points in a vertex star are contained in a single 
sphere, as say that the vertex star is spherical.  

\begin{lemma}\label{possiblenewentry}
(\cite{BHRS}) 
Any $5$ by $5$ isothermic net whose centermost 
vertex star is not spherical 
has a linear conserved quantity.
\end{lemma}

\begin{proof}
We take a Moutard lift $F$ such that $\tau_{pq}=-F_pF_q$.  
We need to find a constant $Q$ and a variable $Z$ so that 
\begin{equation}\label{one}
Z_q=Z_p - (Q F_p+F_p Q) F_q+F_p (Q F_q + F_q Q)
\end{equation}
and 
\begin{equation}\label{two}
Z_p \tau_{pq} = \tau_{pq} Z_q
\end{equation}
hold.  Let us take the 
domain of the mesh to be $\{(m,n) \, | \, |m|, |n| \leq 2 \}$.  
By assumption, the centermost vertex star is nonspherical, so 
\begin{equation}\label{three}
  \dim \text{span} \{ F_{0,0},F_{1,0},F_{0,1},F_{-1,0},F_{0,-1} \} = 5 \; . 
\end{equation}
(Note that we have abbreviated the notation $F_{(i,j)}$ to 
$F_{i,j}$ here, because that will be convenient in this proof.)  
Set
\begin{equation}\label{eqn-star-3}
  Q = q_{0,0} F_{0,0} + q_{1,0} F_{1,0} + q_{0,1} F_{0,1} + q_{-1,0} 
  F_{-1,0} + q_{0,-1} F_{0,-1} 
\end{equation}
and 
\begin{equation}\label{eqn-star-4}
  Z_{0,0} = 
  c_{0,0} F_{0,0} + c_{1,0} F_{1,0} + c_{0,1} F_{0,1} + c_{-1,0} 
  F_{-1,0} + c_{0,-1} F_{0,-1} \; . 
\end{equation}
Define $\vec q = (q_{0,0},q_{1,0},q_{0,1},q_{-1,0},q_{0,-1})^t$ 
and $\vec c = (c_{0,0},c_{1,0},c_{0,1},c_{-1,0},c_{0,-1})^t$ and 
\[
 A = \begin{pmatrix}
     \langle F_{0,0},F_{0,0} \rangle & \langle F_{1,0},F_{0,0} \rangle &
      \langle F_{0,1},F_{0,0} \rangle & \langle F_{-1,0},F_{0,0} \rangle
      & \langle F_{0,-1},F_{0,0} \rangle \\ 
     \langle F_{0,0},F_{1,0} \rangle & \langle F_{1,0},F_{1,0} \rangle &
      \langle F_{0,1},F_{1,0} \rangle & \langle F_{-1,0},F_{1,0} \rangle
      & \langle F_{0,-1},F_{1,0} \rangle \\ 
     \langle F_{0,0},F_{0,1} \rangle & \langle F_{1,0},F_{0,1} \rangle &
      \langle F_{0,1},F_{0,1} \rangle & \langle F_{-1,0},F_{0,1} \rangle
      & \langle F_{0,-1},F_{0,1} \rangle \\ 
     \langle F_{0,0},F_{-1,0} \rangle & \langle F_{1,0},F_{-1,0} \rangle &
      \langle F_{0,1},F_{-1,0} \rangle & \langle F_{-1,0},F_{-1,0} \rangle
      & \langle F_{0,-1},F_{-1,0} \rangle \\ 
     \langle F_{0,0},F_{0,-1} \rangle & \langle F_{1,0},F_{0,-1} \rangle &
      \langle F_{0,1},F_{0,-1} \rangle & \langle F_{-1,0},F_{0,-1} \rangle
      & \langle F_{0,-1},F_{0,-1} \rangle 
 \end{pmatrix} \; , 
\]
\[
 \tilde A = \begin{pmatrix}
     \langle F_{0,0},F_{1,0} \rangle & \langle F_{1,0},F_{1,0} \rangle &
      \langle F_{0,1},F_{1,0} \rangle & \langle F_{-1,0},F_{1,0} \rangle
      & \langle F_{0,-1},F_{1,0} \rangle \\ 
     \langle F_{0,0},F_{0,1} \rangle & \langle F_{1,0},F_{0,1} \rangle &
      \langle F_{0,1},F_{0,1} \rangle & \langle F_{-1,0},F_{0,1} \rangle
      & \langle F_{0,-1},F_{0,1} \rangle \\ 
     \langle F_{0,0},F_{-1,0} \rangle & \langle F_{1,0},F_{-1,0} \rangle &
      \langle F_{0,1},F_{-1,0} \rangle & \langle F_{-1,0},F_{-1,0} \rangle
      & \langle F_{0,-1},F_{-1,0} \rangle \\ 
     \langle F_{0,0},F_{0,-1} \rangle & \langle F_{1,0},F_{0,-1} \rangle &
      \langle F_{0,1},F_{0,-1} \rangle & \langle F_{-1,0},F_{0,-1} \rangle
      & \langle F_{0,-1},F_{0,-1} \rangle 
 \end{pmatrix} \; , 
\]
\[
 E = \begin{pmatrix}
     \langle F_{0,0},F_{2,0} \rangle & \langle F_{1,0},F_{2,0} \rangle &
      \langle F_{0,1},F_{2,0} \rangle & \langle F_{-1,0},F_{2,0} \rangle
      & \langle F_{0,-1},F_{2,0} \rangle \\ 
     \langle F_{0,0},F_{0,2} \rangle & \langle F_{1,0},F_{0,2} \rangle &
      \langle F_{0,1},F_{0,2} \rangle & \langle F_{-1,0},F_{0,2} \rangle
      & \langle F_{0,-1},F_{0,2} \rangle \\ 
     \langle F_{0,0},F_{-2,0} \rangle & \langle F_{1,0},F_{-2,0} \rangle &
      \langle F_{0,1},F_{-2,0} \rangle & \langle F_{-1,0},F_{-2,0} \rangle
      & \langle F_{0,-1},F_{-2,0} \rangle \\ 
     \langle F_{0,0},F_{0,-2} \rangle & \langle F_{1,0},F_{0,-2} \rangle &
      \langle F_{0,1},F_{0,-2} \rangle & \langle F_{-1,0},F_{0,-2} \rangle
      & \langle F_{0,-1},F_{0,-2} \rangle 
 \end{pmatrix} \; , 
\]
\[
 G = \begin{pmatrix}
     \langle F_{0,0},F_{0,0} \rangle & \langle F_{0,0},F_{1,0} \rangle &
      \langle F_{0,0},F_{0,1} \rangle & \langle F_{0,0},F_{-1,0} \rangle
      & \langle F_{0,0},F_{0,-1} \rangle \\ 
     \langle F_{0,0},F_{0,0} \rangle & \langle F_{0,0},F_{1,0} \rangle &
      \langle F_{0,0},F_{0,1} \rangle & \langle F_{0,0},F_{-1,0} \rangle
      & \langle F_{0,0},F_{0,-1} \rangle \\ 
     \langle F_{0,0},F_{0,0} \rangle & \langle F_{0,0},F_{1,0} \rangle &
      \langle F_{0,0},F_{0,1} \rangle & \langle F_{0,0},F_{-1,0} \rangle
      & \langle F_{0,0},F_{0,-1} \rangle \\ 
     \langle F_{0,0},F_{0,0} \rangle & \langle F_{0,0},F_{1,0} \rangle &
      \langle F_{0,0},F_{0,1} \rangle & \langle F_{0,0},F_{-1,0} \rangle
      & \langle F_{0,0},F_{0,-1} \rangle 
 \end{pmatrix} \; , 
\]
\[
 B = \begin{pmatrix}
     \langle F_{0,0},F_{0,0} \rangle & 0 &
      0 & 0
      & 0 \\ 
     0 & \langle F_{0,0},F_{1,0} \rangle & 0 &
      0 & 0 \\ 
     0 & 0 & \langle F_{0,0},F_{0,1} \rangle & 0 
      & 0 \\ 
     0 & 0 & 0 & \langle F_{0,0},F_{-1,0} \rangle & 0  \\ 
     0 & 0 & 0 & 0 & \langle F_{0,0},F_{0,-1} \rangle 
 \end{pmatrix} \; , 
\]
\[
 C = \begin{pmatrix}
     \langle F_{1,0},F_{2,0} \rangle & 0 &
      0 & 0 \\ 
     0 & \langle F_{0,1},F_{0,2} \rangle & 0 &
      0 \\ 
     0 & 0 & \langle F_{-1,0},F_{-2,0} \rangle & 0 \\ 
     0 & 0 & 0 & \langle F_{0,-1},F_{0,-2} \rangle  \end{pmatrix} \; , 
\]
\[
 D = \begin{pmatrix}
     \langle F_{0,0},F_{2,0} \rangle & 0 &
      0 & 0 \\ 
     0 & \langle F_{0,0},F_{0,2} \rangle & 0 &
      0 \\ 
     0 & 0 & \langle F_{0,0},F_{-2,0} \rangle & 0 \\ 
     0 & 0 & 0 & \langle F_{0,0},F_{0,-2} \rangle
 \end{pmatrix} \; . 
\]

We need to know that $A$ is invertible, which follows from 
\eqref{three}, in this way: We can write $A$ as
\[ A = 
\begin{pmatrix}
F_{0,0} & F_{1,0} & 
F_{0,1} & 
F_{-1,0} & 
F_{0,-1} \\
\end{pmatrix}^t \cdot 
\begin{pmatrix}
1 & 0 & 0 & 0 & 0 \\
0 & 1 & 0 & 0 & 0 \\
0 & 0 & 1 & 0 & 0 \\
0 & 0 & 0 & 1 & 0 \\
0 & 0 & 0 & 0 & -1 \\
\end{pmatrix} \cdot 
\begin{pmatrix}
F_{0,0} & F_{1,0} & 
F_{0,1} & 
F_{-1,0} & 
F_{0,-1} \\
\end{pmatrix} 
\] 
(here we are regarding $F_{0,0}$, $F_{\pm 1,0}$, $F_{0,\pm 1}$ 
as $5$-vectors, not 2 by 2 quaternionic matrices), 
so $\det A \neq 0$ if and only if 
\[ \det 
\begin{pmatrix}
F_{0,0} & F_{1,0} & 
F_{0,1} & 
F_{-1,0} & 
F_{0,-1} \\
\end{pmatrix} \neq 0 \; , 
\]
but this last condition follows from \eqref{three}.  

Because of 
\[
 \langle F_{0,0}, Z_{0,0} \rangle = 0 \; , 
\]
\[
  \langle F_{1,0}, Z_{0,0} \rangle = 2
          \langle F_{1,0}, Q \rangle \langle F_{0,0}, F_{1,0} \rangle
 \; , 
\]
\[
  \langle F_{2,0}, Z_{0,0} \rangle = 2
          \langle F_{2,0}, Q \rangle \langle F_{1,0}, F_{2,0} \rangle - 2
          \langle F_{0,0}, Q \rangle \langle F_{1,0}, F_{2,0} \rangle + 2
          \langle F_{1,0}, Q \rangle \langle F_{0,0}, F_{2,0} \rangle
 \; , 
\] and other similar equations, we then have the two equations
\[
 A \vec c = 2 B A \vec q \; , \;\;\; 
 E \vec c = (-2 C G+ 2 C E+2 D \tilde A) \vec q \; , 
\]
which give in turn that 
\[
 (E A^{-1} B A+C G-C E-D \tilde A) \vec q = 0 \; . 
\]
Because $E A^{-1} B A+C G-C E-D \tilde A$ is a 4 by 5 matrix, 
this linear system has a nonzero solution $\vec q$.  We then 
define $Q$ using that solution $\vec q$, as in \eqref{eqn-star-3}.  
Then $A \vec c = 2 B A \vec q$ determines $\vec c$, which we 
use to define $Z_{0,0}$, as in \eqref{eqn-star-4}.  
We then propagate $Z$ using Equation \eqref{one}.  

It only remains to check that Equations \eqref{one} 
and \eqref{two} hold everywhere.  By Lemma 
\ref{useful-lemma2}, Equation \eqref{two} is equivalent to 
showing $F \perp Z$, which is now a property on 
vertices.  We leave out the details of computing $F \perp Z$ 
here, but note that such types of computations will be shown in 
detail in the proof of the next lemma.  
\end{proof}

\begin{remark}
In Lemma \ref{possiblenewentry}, we showed existence but 
not uniqueness of the linear conserved quantity.  It would be interesting 
to find natural geometric conditions 
that would make the linear conserved quantity unique.  
\end{remark}

\begin{lemma}
(\cite{BHRS}) 
For any nonspherical $3$ by $3$ isothermic net and any $Q \in 
\mathbb{R}^{4,1} \setminus \{ 0 \}$, 
there exists a $Z$ so that $\lambda Z + Q$ is a linear conserved 
quantity of the net.  
\end{lemma}

\begin{proof}
This proof will follow along the same lines as the previous proof, 
but now will be simpler because the size of the net is smaller.  

Let us take the domain mesh to be $\{ (m,n) \, | \, |m|,|n| \leq 1 
\}$.  Like in the previous proof, we take a Moutard lift $F$ such 
that $\tau_{pq} = - F_pF_q$.  
Using the notation in the previous proof, $\vec q$ is now 
given by the given choice of $Q$.  If the centermost vertex star 
$F_{(0,0)}$, $F_{(1,0)}$, $F_{(0,1)}$, $F_{(-1,0)}$, 
$F_{(0,-1)}$ would 
be spherical, then the whole 3 by 3 net would be spherical as well.  
Since this is not so, the central vertex star is not spherical, and 
thus the matrix $A$ in the previous proof is invertible.  
We can then solve $A \vec c = 2 B A \vec q$ for $\vec c$.  
Setting $p=(0,0)$, we have $Z_p$ defined by this $\vec c$, 
and we can propagate $Z$ like in the previous proof so that 
$Z_q$ and $Z_s$ are defined, where $q=(1,0)$ and $s=(0,1)$.  
The fact that $A \vec c = 2 B A \vec q$ holds implies that 
\[ \langle F_q,Z_q \rangle = 0 \;\;\; \text{i.e.} \;\;\; 
\langle Z_p,F_q \rangle = 2 \langle Q , F_q \rangle \langle F_p , F_q 
\rangle \; , \]
and 
\[ \langle F_s , Z_s \rangle = 0 \;\;\; \text{i.e.} \;\;\; 
\langle Z_p,F_s \rangle = 2 \langle Q , F_s \rangle \langle F_p , 
F_s \rangle \; . \]  
We also set $r=(1,1)$ and propagate $Z$ to $Z_r$.  Then, because 
$F_{\hat p}F_{\hat q}+F_{\hat q}F_{\hat p} =a_{{\hat p}{\hat q}} 
\cdot I$ for any edge ${\hat p}{\hat q}$ (i.e. $\langle 
F_{\hat p},F_{\hat q} \rangle = (-1/2) a_{{\hat p}{\hat q}}$), and 
because $(F_r-F_p)||(F_q-F_s)$, i.e. $F_r-F_p=
\alpha (F_q-F_s)$ for some scalar $\alpha$, we have 
that \begin{equation}\label{eqn:startwostars} 
F_r-F_p = \alpha (F_q-F_s) \; , \;\;\; 
\alpha = \frac{\langle F_q-F_s,F_p \rangle}{\langle 
F_q,F_s \rangle} \; , \end{equation} seen as follows: 
$\langle F_q,F_r \rangle = \langle F_p, F_s \rangle$, 
$\langle F_q,F_p-F_r \rangle = \langle F_p, F_q-F_s \rangle$, 
$\langle F_q,F_r-F_p \rangle (F_s-F_q) = 
\langle F_p, F_q-F_s \rangle (F_q-F_s)$, 
$\langle F_q,F_s-F_q \rangle (F_r-F_p) = 
\langle F_p, F_q-F_s \rangle (F_q-F_s)$, implying 
\eqref{eqn:startwostars}.  

We now wish to show $\langle F_r,Z_r \rangle = 0$.  First, 
we find an expression for $Z_r$: 
\[ Z_r = Z_q+2 \langle Q,F_q \rangle F_r - 2 \langle Q, 
F_r \rangle F_q = \] 
\[ = Z_p+2 \langle Q,F_p \rangle F_q - 2 \langle Q, 
F_q \rangle F_p + 2 \langle Q,F_q \rangle F_r - 2 \langle Q, 
F_r \rangle F_q = \]\[ 
 = Z_p + 2 \langle Q,F_q \rangle (F_r-F_p) - 2 \langle Q, 
F_r-F_p \rangle F_q \; , \]
thus, by \eqref{eqn:startwostars}, 
\[ Z_r = Z_p+2 \alpha \langle Q,F_q \rangle (F_q-F_s) - 2 
\alpha \langle Q, 
F_q-F_s \rangle F_q = \]\[ 
 = Z_p-2 \alpha 
\langle Q,F_q \rangle F_s + 2 \alpha \langle Q, 
F_s \rangle F_q \; . \]  
Then \eqref{eqn:startwostars} gives 
\[ \langle Z_r,F_r \rangle = \langle Z_p - 2 \alpha 
\langle Q,F_q \rangle F_s + 2 \alpha \langle Q,F_s 
\rangle F_q, \alpha F_q - \alpha F_s+F_p \rangle = \]\[ 
\alpha \langle Z_p,F_q \rangle - \alpha \langle 
Z_p,F_s \rangle - 2 \alpha^2 
\langle Q,F_q \rangle \langle F_q,F_s \rangle - 
\]\[ 2 \alpha 
\langle Q,F_q \rangle \langle F_s,F_p \rangle 
- 2 \alpha^2 
\langle Q,F_s \rangle \langle F_s,F_q \rangle + 2 
\alpha \langle Q,F_s \rangle \langle F_p,F_q \rangle = 
\]\[ 
\alpha \langle Z_p,F_q \rangle - 2 \alpha^2 
\langle Q,F_q+F_s \rangle \langle F_q,F_s \rangle -\alpha 
\langle Z_p,F_s \rangle - 2 \alpha 
\langle Q,F_q \rangle \langle F_s,F_p \rangle + 2 
\alpha \langle Q,F_s \rangle \langle F_p,F_q \rangle \; . 
\]
Thus, by \eqref{eqn:threestars}, 
\[ 
\langle Z_r, F_r \rangle = 
2 \alpha \langle Q,F_q \rangle \langle F_p,F_q \rangle - 2 \alpha 
\langle Q,F_s \rangle \langle F_p,F_s \rangle + 2 \alpha 
\langle Q,F_s \rangle \langle F_p,F_q \rangle - 
\]\[ 2 \alpha 
\langle Q,F_q \rangle \langle F_p,F_s \rangle - 2 
\alpha^2 \langle Q,F_q+F_s \rangle \langle F_q,F_s 
\rangle = \]\[ 
2 \alpha \langle Q,F_q+F_s \rangle \langle F_p,F_q \rangle - 
2 \alpha 
\langle Q,F_s+F_q \rangle \langle F_p,F_s \rangle - 
2 \alpha^2 
\langle Q,F_q+F_s \rangle \langle F_q,F_s \rangle = \]\[ 
2 \alpha \langle Q,F_q+F_s \rangle \langle F_p,F_q-F_s 
\rangle - 2 \alpha^2 
\langle Q,F_q+F_s \rangle \langle F_q,F_s \rangle = \]\[ 
2 \alpha \langle Q,F_q+F_s \rangle (\langle F_p,F_q-F_s 
\rangle - \alpha 
\langle F_q,F_s \rangle) = \]\[ 
2 \alpha \langle Q,F_q+F_s \rangle (\langle F_p,F_q-F_s 
\rangle - \frac{\langle F_p,F_q-F_s \rangle}{\langle 
F_q,F_s \rangle} \langle F_q,F_s \rangle) = 0 \; . \]
Repeating similar computations on the other three quadrilaterals 
completes the proof.  
\end{proof}

\begin{figure}[phbt]
\begin{center}
\includegraphics[width=0.52\linewidth]{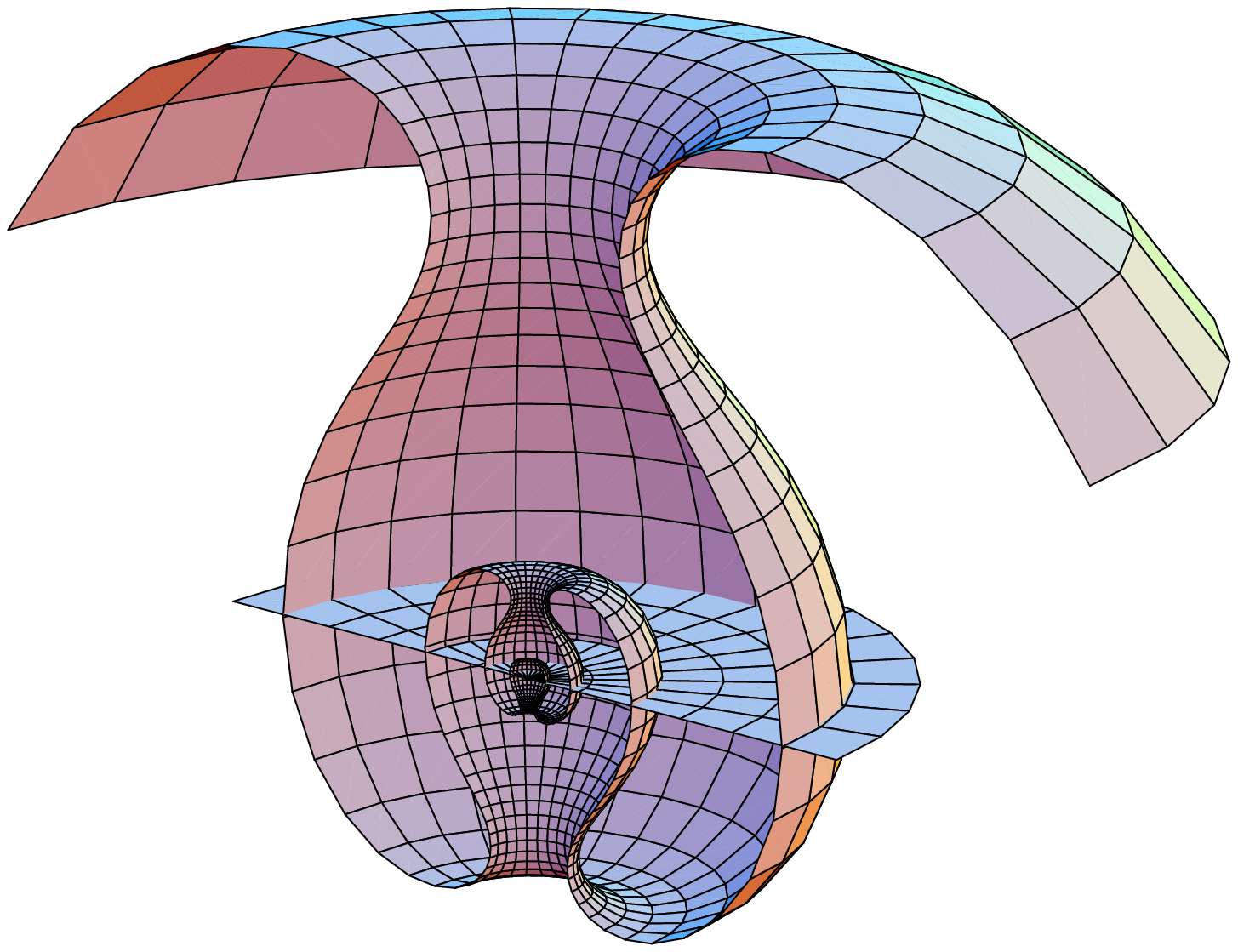}
\includegraphics[width=0.46\linewidth]{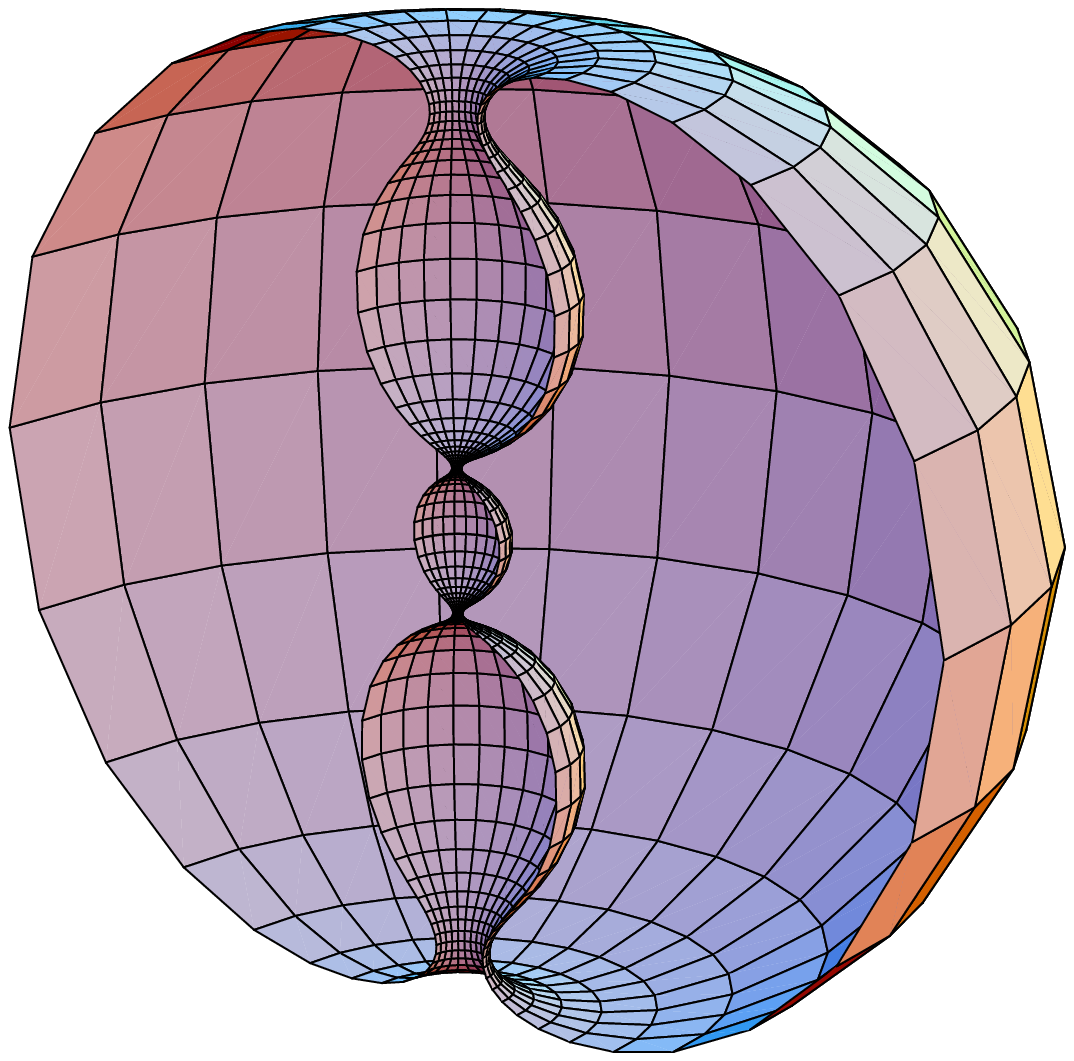}
\caption{A discrete minimal surface of revolution in 
$\mathbb{H}^3$ (in 
two copies of the upper-half space model -- one above the central 
plane, and one below), and a discrete minimal 
surface of revolution in $\mathbb{S}^3$ (where 
$\mathbb{S}^3$ has been stereographically projected to 
$\mathbb{R}^3$).}
\end{center}
\label{fig:babo}
\end{figure}

\subsection{Discrete CMC surfaces of revolution}
We take $Q$ as in 
\eqref{choiceofQ}.  
Let us first make the following assumption about the 
vertices of the discrete surface, implying we have a 
discrete surface of revolution: 
\[
 \text{Assumption 1:} \;\; \ef_{(m,n)} = r_m (c_n i + s_n j) 
     + h_m k \; , 
\] where $c_n = \cos(2 \pi \theta_n/N)$, 
$s_n = \sin(2 \pi \theta_n/N)$ and $r_m, h_m \in 
\mathbb{R}$, with 
$N$ a natural number and $\theta_n \in \mathbb{R}$.  

The cross ratio for the quadrilateral with vertices 
coming from 
$(m,n)$, $(m+1,n)$, $(m+1,n+1)$ and $(m,n+1)$ is 
\[ q = q_{m,n} = 
   \frac{-dh_{m,m+1}^2-dr_{m,m+1}^2}{4 r_mr_{m+1} \sin^2 (\pi 
   d\theta_{n,n+1}/N)} \; , \] 
where $dr_{m,m+1}=r_{m+1}-r_m$, $dh_{m,m+1}=h_{m+1}-h_m$ 
and $d\theta_{n,n+1}=\theta_{n+1}-\theta_n$.  So we can take 
\[ a_{(m,n),(m+1,n)} = 
     - \alpha \frac{dh_{m,m+1}^2+dr_{m,m+1}^2}{r_m r_{m+1}} \; , \] 
\[ a_{(m,n),(m,n+1)} = 4 \alpha \sin^2 (\pi d\theta_{n,n+1}/N) \] 
for any choice of $\alpha \in \mathbb{R} \setminus \{ 0 \}$.  
Because $d\ef_{pq}^* d\ef_{pq} = a_{pq}$, we have 
\[ d\ef^*_{pq}=\pm \frac{\alpha}{r_pr_q} d\ef_{pq} \; , \] where 
"$+$" is used for $m$-edges and "$-$" for $n$-edges.  
The $m$-edges are those between $\ef_{(m,n)}$ and 
$\ef_{(m+1,n)}$, and 
the $n$-edges are those between $\ef_{(m,n)}$ and 
$\ef_{(m,n+1)}$.  

We then have 
\[ \tau_{pq} = \pm \frac{\alpha}{r_pr_q} \begin{pmatrix}
\ef_p d\ef_{pq} & -\ef_p d\ef_{pq} \ef_q \\ 
d\ef_{pq} & -d\ef_{pq} \ef_q
\end{pmatrix} \; . \] 

Now assume $\ef$ is a discrete CMC surface, that is: 
\[
\text{Assumption 2:} \;\; \ef \; \text{has a linear 
conserved quantity} \;\; Q+\lambda Z \; . 
\]
Then, by Corollary \ref{cor:BinMarch2008}, we have 
\[ Z_p = \begin{pmatrix}
n_p+H_p\ef_p & -\ef_pn_p-n_p\ef_p-H_p\ef_p^2 \\
H_p & - n_p - H_p \ef_p
\end{pmatrix} \; . \]

\begin{defn}
We say that the surface of revolution $\ef_{m,n}$ has a 
{\em constant hyperbolic speed parametrization} if 
the cross ratio $q_{m,n}$
is a constant (i.e. indep of $m$ and $n$).  
\end{defn}

We now restrict to the case in the above definition: 
\[
 \text{Assumption 3:} \;\; \ef \; \text{is a constant hyperbolic 
speed parametrization} \; . 
\]
This third assumption is not so essential for the arguments 
here, but we include it as it is geometrically natural.  

The last two equations in Lemma \ref{lem:lcq-in-parts} 
now give the four equations 
\[
 d\ef_{pq} n_q + n_p d\ef_{pq} = 0 \; , \]
\[
 H_q = H_p + (\kappa \ef_p d\ef_{pq} + 
              \kappa d\ef_{pq} \ef_q) \cdot \frac{\pm \alpha}{r_p r_q} 
              \; , \]
\[
 n_q+H_q \ef_q = n_p+H_p \ef_p + (d\ef_{pq} + \kappa 
          \ef_p d\ef_{pq} \ef_q) \cdot 
          \frac{\pm \alpha}{r_p r_q} \; , \]
\[
 \ef_qn_q+n_q\ef_q+H_q\ef_q^2=\ef_pn_p+n_p\ef_p+H_p\ef_p^2 
    + (d\ef_{pq} \ef_q+\ef_p d\ef_{pq}) 
    \cdot \frac{\pm \alpha}{r_p r_q} \; . \]

We now assume that $n_p$, and hence the conserved quantity 
as well, has the same rotation symmetry as the surface itself:
\[
 \text{Assumption 4:} \;\; n_p = \rho_p (c_n i + s_n j) + \eta_p k \;\; 
   \text{and} \;\; \rho_p, \eta_p \in \mathbb{R} 
   \; \text{depend only on} \; m \; . \]
The fact that $||Z_p||^2$ is constant implies $|n_p|^2$ is constant, 
and thus $\rho_p^2+\eta_p^2$ is also constant.  We have the following 
further facts:
\[
  \ef_pn_p+n_p\ef_p+H_p\ef_p^2 = 
           -2 (r_p\rho_p+h_p\eta_p) - H_p (r_p^2+h_p^2) \; , 
\]
and when $p=(m,n)$ and $q=(m,n+1)$, we have 
\[
 \ef_pd\ef_{pq}+d\ef_{pq}\ef_q = 0 \; , \;\;\; 
 d\ef_{pq}+\kappa \ef_p d\ef_{pq} \ef_q = 
 r_p (1+\kappa (r_p^2+h_p^2)) (dc_{n,n+1}i+ds_{n,n+1}j) \]
for $dc_{n,n+1}=c_{n+1}-c_n$ and $ds_{n,n+1}=s_{n+1}-s_n$.  
When $p=(m,n)$ and $q=(m+1,n)$, we have 
\[ 
 \ef_pd\ef_{pq}+d\ef_{pq} \ef_q = r_m^2+h_m^2-r_{m+1}^2-h_{m+1}^2 \]
and 
\[
 d\ef_{pq}+\kappa \ef_p d\ef_{pq} \ef_q = 
 (dr_{m,m+1}+\kappa (r_{m+1}(r_m^2+h_m^2)-r_m(r_{m+1}^2+
 h_{m+1}^2)))(c_ni+s_nj)+ \]\[ + 
(dh_{m,m+1}+\kappa (h_{m+1}(r_m^2+h_m^2)-h_m(r_{m+1}^2+h_{m+1}^2)))k \; . \]  
Now the full list of equations becomes:
\begin{enumerate}
\item $\rho_m^2+\eta_m^2$ is constant, where we now denote $\rho_p$ and 
$\eta_p$ by $\rho_m$ and $\eta_m$, respectively, 
\item $H_p$ depends only on $m$, 
\item $\rho_m+H_m r_m = -\alpha r_m^{-1} (1+\kappa (r_m^2+h_m^2))$, 
\item $(\rho_{m+1}+\rho_m)dr_{m,m+1}+(\eta_{m+1}+\eta_m) dh_{m,m+1}=0$, 
\item $dr_{m,m+1} d\eta_{m,m+1} - dh_{m,m+1} d\rho_{m,m+1} =0$, 
\item $H_{m+1}-H_m = \alpha \kappa r_m^{-1} r_{m+1}^{-1} (r_m^2+
                h_m^2-r_{m+1}^2-h_{m+1}^2)$, 
\item $d\rho_{m,m+1}+H_{m+1}r_{m+1}-H_mr_m=\alpha r_m^{-1}r_{m+1}^{-1} 
                (dr_{m,m+1}+\kappa (r_{m+1}(r_m^2+h_m^2)-
                 r_m(r_{m+1}^2+h_{m+1}^2)))$, 
\item $d\eta_{m,m+1}+H_{m+1}h_{m+1}-H_mh_m=\alpha r_m^{-1}r_{m+1}^{-1} 
                (dh_{m,m+1}+\kappa (h_{m+1}(r_m^2+h_m^2)-
                 h_m(r_{m+1}^2+h_{m+1}^2)))$, 
\item $2 (r_m\rho_m+h_m\eta_m-r_{m+1}\rho_{m+1}-h_{m+1}\eta_{m+1})+
      H_m(r_m^2+h_m^2)-H_{m+1}(r_{m+1}^2+h_{m+1}^2)=
      \alpha r_m^{-1}r_{m+1}^{-1} (r_m^2+h_m^2-r_{m+1}^2-h_{m+1}^2)$. 
\end{enumerate}

The first condition above should follow from the other equations, 
and we can just assume the second condition.  We then have the system
\begin{equation}\label{eqn:bigmatrixeqn} P \cdot 
 \begin{pmatrix}
 H_m \\ H_{m+1} \\ \eta_m \\ \eta_{m+1} \\ \rho_m \\ \rho_{m+1} 
 \\ H_\kappa \\ 1 
 \end{pmatrix}
 = 
 \begin{pmatrix}
 0 \\ 0 \\ 0 \\ 0 \\ 0 \\ 0 \\ 0 \\ 0 
 \end{pmatrix} \; , 
\end{equation} 
\[ P = 
 \begin{pmatrix}
 r_m^2 & 0 & 0 & 0 & r_m & 0 & 0 & A \\
 0 & r_{m+1}^2 & 0 & 0 & 0 & r_{m+1} & 0 & B \\
 0 & 0 & dh_{m,m+1} & dh_{m,m+1} & dr_{m,m+1} & dr_{m,m+1} & 0 & 0 \\
 0 & 0 & -dr_{m,m+1} & dr_{m,m+1} & dh_{m,m+1} & -dh_{m,m+1} & 0 & 0 \\
 -1 & 1 & 0 & 0 & 0 & 0 & 0 & \kappa C \\
 -r_m & r_{m+1} & 0 & 0 & -1 & 1 & 0 & D \\
 -h_m & h_{m+1} & -1 & 1 & 0 & 0 & 0 & E \\
 r_m^2+h_m^2 & -r_{m+1}^2-h_{m+1}^2 & 2h_m & -2h_{m+1} & 
           2r_m & -2r_{m+1} & 0 & C 
 \end{pmatrix} \; , 
\]
where
\[
 A = \alpha (1+\kappa (r_m^2+h_m^2)) \; , 
\]
\[
 B = \alpha (1+\kappa (r_{m+1}^2+h_{m+1}^2)) \; , 
\]
\[
 C = \alpha r_m^{-1} r_{m+1}^{-1} (r_{m+1}^2+
          h_{m+1}^2-r_m^2-h_m^2) \; , 
\]
\[
 D = - \alpha r_m^{-1} r_{m+1}^{-1} (dr_{m,m+1}+
          \kappa (r_{m+1}(r_m^2+h_m^2)-r_m(r_{m+1}^2+h_{m+1}^2))) \; , 
\]
\[
 E = - \alpha r_m^{-1} r_{m+1}^{-1} (dh_{m,m+1}+
      \kappa (h_{m+1}(r_m^2+h_m^2)-h_m(r_{m+1}^2+h_{m+1}^2))) \; . 
\]
The fifth and eighth rows of the product on the left-hand side 
of \eqref{eqn:bigmatrixeqn} being zero implies that 
\[
 H_{m+1}-H_m = 2 \kappa (r_m \rho_m+h_m \eta_m-r_{m+1} \rho_{m+1} 
               - h_{m+1} \eta_{m+1}) + \]\[ + H_m (r_m^2+h_m^2) \kappa 
               - H_{m+1} (r_{m+1}^2+h_{m+1}^2) \kappa \; , 
\]
and so 
\[
 2 H_\kappa := H_{m+1}(1+\kappa (r_{m+1}^2+h_{m+1}^2)) + 2 \kappa 
     (r_{m+1} \rho_{m+1}+h_{m+1} \eta_{m+1}) = \]\[ = 
     H_m(1+\kappa (r_m^2+h_m^2)) + 2 \kappa 
     (r_m \rho_m+h_m \eta_m) \]
is constant.  

We can then choose the constant $\alpha$ so that $\eta_m^2+\rho_m^2=1$, 
and then start with some initial conditions and 
propagate through values of $m$ via 
\eqref{eqn:bigmatrixeqn} to produce the vertex 
data for a discrete CMC surface of revolution.  

\begin{figure}[phbt]
\begin{center}
\includegraphics[width=0.014\linewidth]{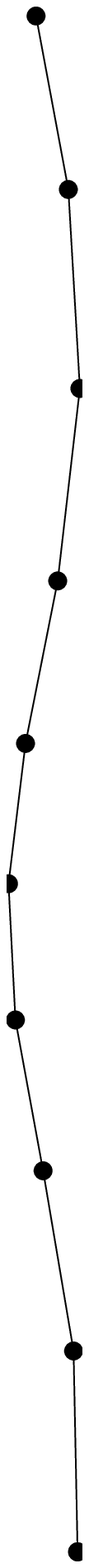}
\includegraphics[width=0.13\linewidth]{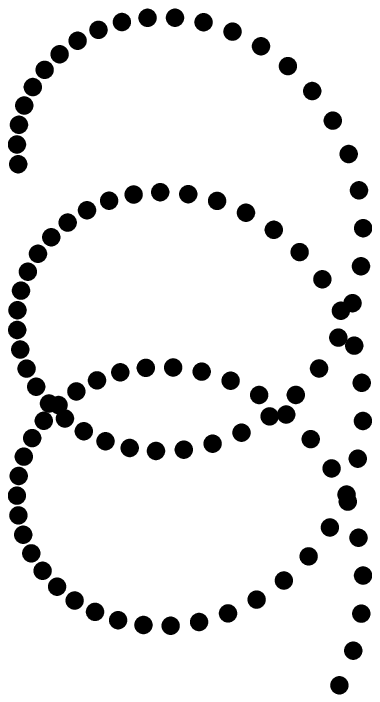}
\includegraphics[width=0.234\linewidth]{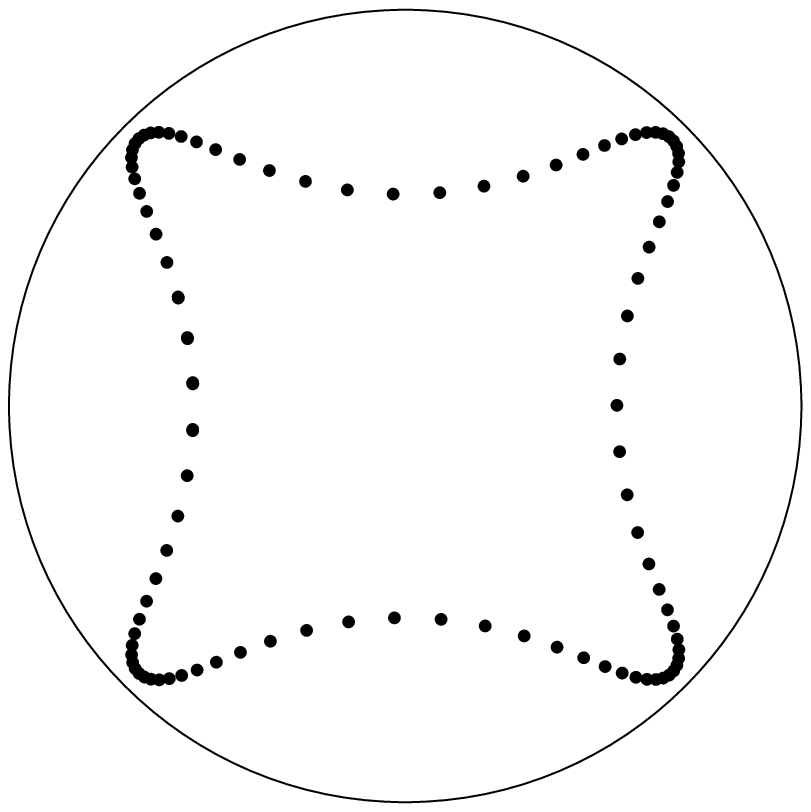}
\includegraphics[width=0.234\linewidth]{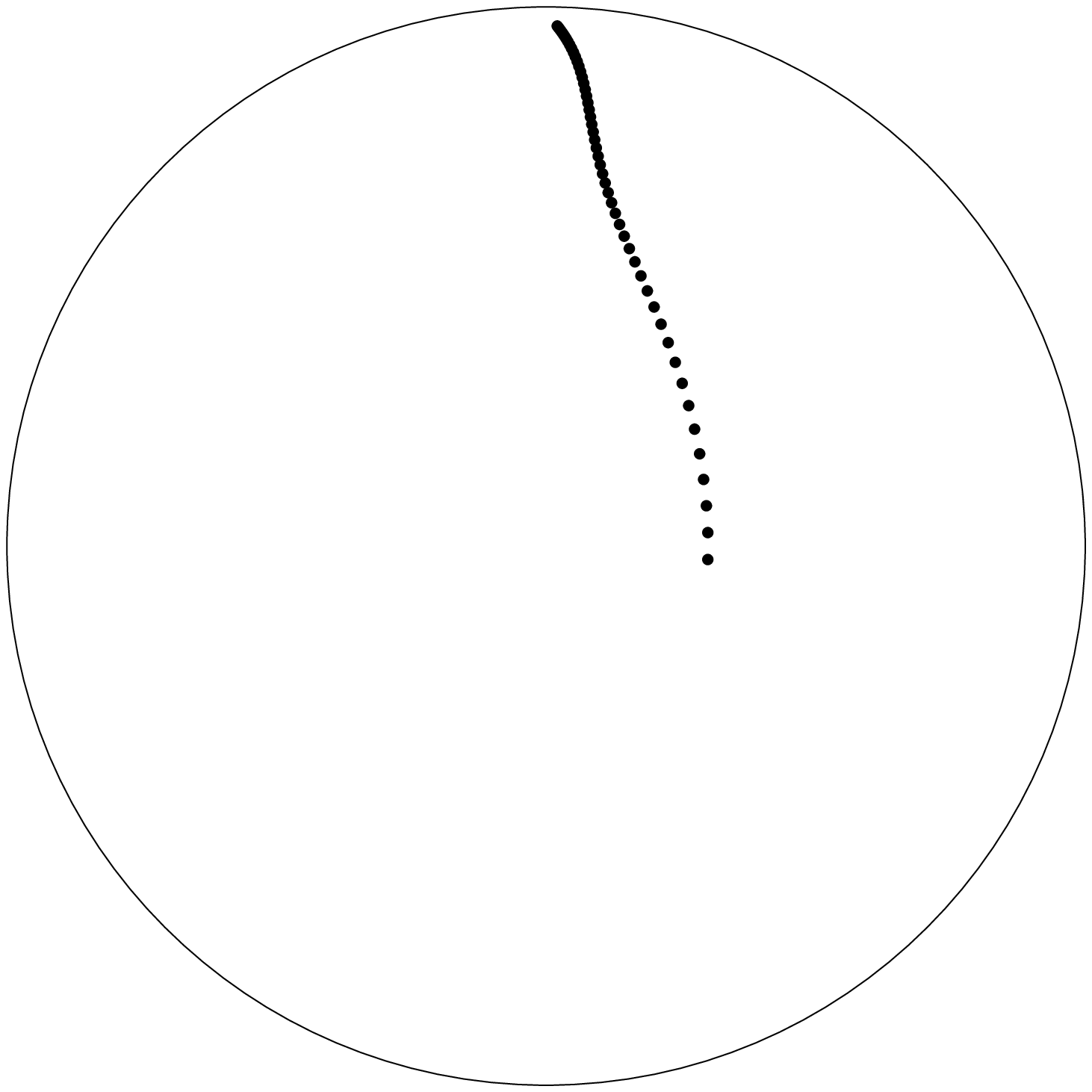}
\includegraphics[width=0.234\linewidth]{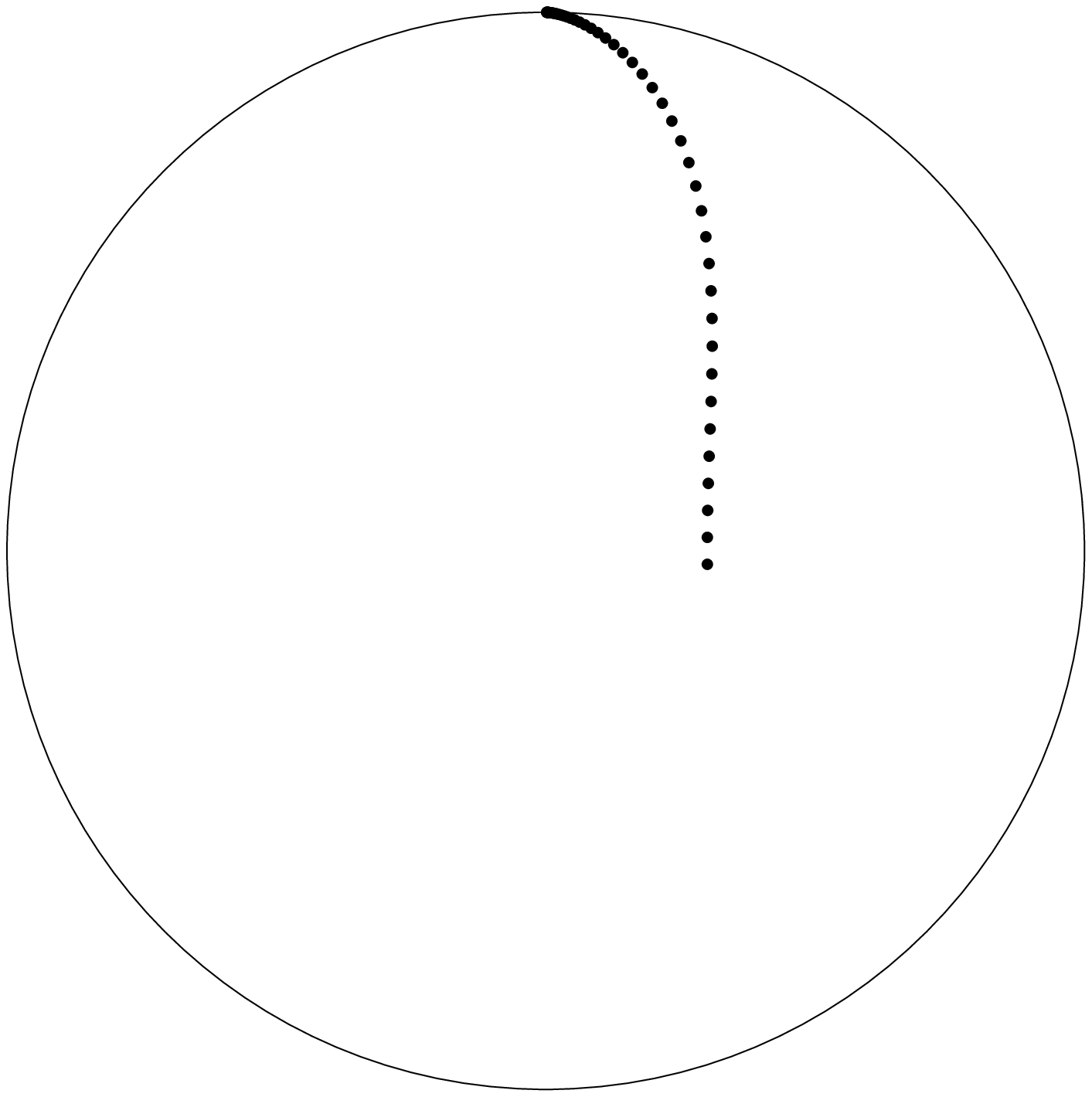}
\caption{Discrete profile curves for discrete CMC surfaces of 
revolution.  The meanings of these graphics are explained in 
Example \ref{discrete-examples}.}
\end{center}\label{fig:6pics}
\end{figure}

\begin{example}\label{discrete-examples} 
{\em In Figure 18, we show discrete CMC surfaces of revolution.  
The first two curves are profile curves for discrete 
nonminimal CMC surfaces of revolution in $\mathbb{R}^3$, 
the first being unduloidal 
and the second nodoidal.  (For each of these two curves, the axis 
of rotation producing the surface is a vertical line drawn to the left 
of the curve, and is not shown in the figure.)  The third picture 
shows the profile curve 
for a discrete CMC surface of revolution in $\mathbb{S}^3$, 
where $\mathbb{S}^3$ is 
stereographically projected to 
$\mathbb{R}^3$, and the circle shown is a geodesic of 
$\mathbb{S}^3$ that is also the axis of the surface -- and 
furthermore, this example has a periodicity that causes it to close 
on itself and form a 
torus.  Half of this surface in 
$\mathbb{S}^3$ is shown on the right-hand side 
of Figure 17 as well, under a 
different stereographic projection to $\mathbb{R}^3$.  
The final two pictures in Figure 18 
show profile curves for discrete CMC surfaces of revolution in 
$\mathbb{H}^3$.  These surfaces, with $H>1$ and $H=1$ 
respectively, are shown in the Poincare model, and 
the first is unduloidal 
while the second looks similar to a smooth embedded catenoid cousin.  
(For these two curves, the corresponding axis of revolution is the 
vertical line between the uppermost and lowermost points of the 
circle shown, 
and this circle lies in the boundary sphere at infinity of 
$\mathbb{H}^3$.)  
Also, on the left-hand side of Figure 17, 
we see a minimal surface that lies in both copies of 
$M_{-1} = \mathbb{H}^3 \cup \mathbb{H}^3$, 
and the horizontal plane shown there is the 
virtual boundary at infinity of two copies of the 
halfspace model for $\mathbb{H}^3$.  This 
example was first known in \cite{BHRS}, 
because the notion of discrete CMC for this case was not 
defined before then.}
\end{example}

\section{Discrete spacelike CMC surfaces in $\mathbb{R}^{2,1}$}
\label{discretemaximalsurface}

In Chapter \ref{maximalsurface}, we looked at smooth maximal surfaces in 
Minkowski $3$-space.  In this chapter, we consider one way to 
define discrete versions of them, and more generally, 
to define discrete spacelike CMC surfaces in $\mathbb{R}^{2,1}$.  
We start by reviewing the smooth case.  

\subsection{Smooth CMC surfaces in $\mathbb{R}^3$ and 
$\mathbb{R}^{2,1}$, 
without quaternions}

Consider a smooth surface 
\[ x(u,v) = (x_1(u,v),x_2(u,v),x_3(u,v)) \] 
in $\mathbb{R}^3$ or $\mathbb{R}^{2,1}$, with unit normal $n$.  Suppose the 
surface is spacelike, in the case of $\mathbb{R}^{2,1}$.  
Also, suppose that the 
coordinates $u,v$ are isothermic.  Conformality implies the first 
fundamental form is 
\[ I = \begin{pmatrix} 
E & 0 \\ 0 & E \end{pmatrix} \] with $E = \langle x_u , x_u \rangle$, 
where $\langle \cdot , \cdot \rangle$ denotes the inner product 
associated with $\mathbb{R}^3$ or $\mathbb{R}^{2,1}$.  Then 
the second fundamental form is 
\[ II = \begin{pmatrix} 
\langle n , x_{uu} \rangle & \langle n , x_{vu} \rangle \\ 
\langle n , x_{uv} \rangle & \langle n , x_{vv} \rangle \end{pmatrix} 
= \begin{pmatrix} 
b_{11} & b_{12} \\ 
b_{21} & b_{22} \end{pmatrix} \; , \] 
and having isothermic coordinates implies $n_u = - k_1 x_u$ and 
$n_v = - k_2 x_v$, where $k_1$ and $k_2$ are the principal 
curvatures, and so 
\[ II = \begin{pmatrix}
k_1 E & 0 \\ 0 & k_2 E \end{pmatrix} \; . \] 
The Hopf differential function is, with $z=u+iv$,  
\[ \hat{Q} = \langle n , x_{zz} \rangle = 
\frac{1}{4} \langle n , x_{uu}-x_{vv}-2 i x_{uv} \rangle = 
\frac{1}{4} \langle n , x_{uu}-x_{vv} \rangle \]\[ = 
\frac{1}{4} (b_{11}-b_{22}) = \frac{E}{4} (k_1-k_2) \; . \]  

If the mean curvature $H$ is constant, then 
Corollary \ref{lem:Qiscte} and the second equation in 
\eqref{eq:GaussformaxinL3} imply $\hat{Q}_{\bar z}=0$, 
so $\hat{Q} = (E/4) (k_1-k_2) \in \mathbb{R}$ is constant.  

\begin{lemma}\label{hereislem10pt1}
If $x$ is isothermic in $\mathbb{R}^3$ or 
$\mathbb{R}^{2,1}$ with isothermic 
coordinates $u,v$, then $x^*$ exists, 
solving $dx^* = - \frac{x_u}{E} du + \frac{x_v}{E} dv$.  
\end{lemma}

\begin{proof}
This was already proven in the case of $\mathbb{R}^3$ in 
Lemma \ref{lemma8ptpt15}, so let us be brief here: 
We want to show "$d^2x^*=0$", i.e. 
\[ d (-\frac{x_u}{E} du + \frac{x_v}{E} dv) = \; 0 , \] 
i.e. $2 x_{uv} E - x_u E_v - x_v E_u = 0$.  We can see this 
by noting that $b_{12}=0$ implies $x_{uv} = A x_u+B x_v$ for some 
reals $A$ and $B$, and that $\langle x_u , x_v \rangle = 0$.  
\end{proof}

The $x^*$ in Lemma \ref{hereislem10pt1} is the same as 
the $x^*$ in Definition \ref{defn:ChristTrans2}, 
but scaled by a factor of $1/4$.  This is a non-essential change.  

\begin{proposition}\label{lem:smoothcase-relatedtoR21}
Let $x$ be an isothermic immersion in 
$\mathbb{R}^3$ or $\mathbb{R}^{2,1}$, with $x^*$ 
as in the previous lemma.  
Then $x$ is CMC $H$ if and only if $dx^* = h (Hdx+dn)$ for some constant 
$h$.  
\end{proposition}

\begin{proof}
Let us again be brief, because the $\mathbb{R}^3$ case 
was already dealt with in Remark \ref{rem:CMCparallelsurf}: 
\[ -\frac{x_u}{E}du+\frac{x_v}{E}dv= h (Hdx+dn) \; , \;\;\; 
h \;\, \text{constant} \] is equivalent to 
\[ k_1+k_2 = 2 H \; , \;\; \text{and} \;\; 
h = 2 E^{-1} (k_1-k_2)^{-1} \; \text{is constant} \; . \]  
The first of these is clearly true, and $h$ is constant if and only if 
the Hopf differential function $\hat Q$ is constant, which is true 
if and only if $x$ is CMC.  
\end{proof}

\begin{corollary}\label{cor:fortheR21case}
An isothermic immersion $x$ in 
$\mathbb{R}^3$ or $\mathbb{R}^{2,1}$ is CMC if and only if 
\[ -\frac{x_u}{E}du+\frac{x_v}{E}dv= h (Hdx+dn) \] for some 
real constants $h$ and $H$.  
\end{corollary}

\subsection{Discrete isothermic CMC surfaces in $\mathbb{R}^3$, 
without quaternions}

Let $\ef$ be a discrete isothermic surface in $\text{Im} H \approx 
\mathbb{R}^3$ as in Chapter 
\ref{chapondiscreteCMCsurfs}, with cross ratio factorizing function 
$a_{pq}$.  Starting with the equation 
\[ d\ef_{pq}^* = a_{pq} \frac{-d\ef_{pq}}{|d\ef_{pq}|^2} \] 
for the Christoffel transformation, we have the following lemma, 
which follows from Corollary \ref{cor:AinMarch2008}: 

\begin{lemma}
A discrete isothermic surface $\ef$ in 
$\mathbb{R}^3$ is CMC if and only if there 
exist constants $h,H \in \mathbb{R}$ and $n_p$ with $|n_p|^2=1$ and 
$d\ef_{pq} n_q+n_p d\ef_{pq} = 0$ so that 
\[ h (dn_{pq} + H d\ef_{pq}) = \frac{-a_{pq} d\ef_{pq}}{|d\ef_{pq}|^2} \; . \]
\end{lemma}

However, 
$d\ef_{pq} n_q+n_p d\ef_{pq} = 0$ is still a quaternionic equation.  But this 
equation is equivalent to the pair of equations 
$d\ef_{pq} \wedge n_q + n_p \wedge d\ef_{pq} = 0$ 
and $\langle d\ef_{pq} , n_p+n_q \rangle_{\mathbb{R}^3} 
= 0$.  Then we can restate the previous 
lemma, without any use of quaternions, as: 

\begin{theorem}\label{thm10ptpt5}
A discrete isothermic surface $\ef$ in 
$\mathbb{R}^3$ is CMC if and only if there 
exist constants $h,H \in 
\mathbb{R}$ and vectors $n_p$ so that 
\begin{itemize}
\item $|n_p|^2=1$, 
\item $d\ef_{pq} \wedge n_q + n_p \wedge d\ef_{pq} = 0$, 
\item $\langle d\ef_{pq} , n_p+n_q \rangle_{\mathbb{R}^3} = 0$, and 
\item $h (dn_{pq} + H d\ef_{pq}) = \frac{-a_{pq} d\ef_{pq}}{|d\ef_{pq}|^2}$.
\end{itemize}
\end{theorem}

Not all four items in the above theorem are independent of each other.  
For example, the second item follows from the fourth item, 
because the second item is just telling us that $d\ef_{pq}$ is parallel 
to $dn_{pq}$.  

\subsection{Discrete CMC surfaces in $\mathbb{R}^{2,1}$}

We now propose possible definitions for discrete isothermic 
surfaces and discrete spacelike CMC surfaces in $\mathbb{R}^{2,1}$.  

Let $\ef$ be a map from a domain in $\mathbb{Z}^2$ to $\mathbb{R}^{2,1}$.  
Let $p=(m,n)$, 
$q=(m+1,n)$, 
$r=(m+1,n+1)$ and $s=(m,n+1)$ be four vertices in the domain of $\ef$, 
for some $m,n \in \mathbb{Z}$.  Let $\ef_p$, $\ef_q$, $\ef_r$ and 
$\ef_s$ be the images of $p$, $q$, $r$ and $s$ under $\ef$.

To define the cross ratio factorizing function 
$a_{pq}$ in the case of $\mathbb{R}^{2,1}$, we need to define some 
analogue of the cross ratio, call it $q=q_{pqrs}$.  Then 
we can define the $a_{pq}$ in the usual way.  

We now consider how to define the cross ratio on quadrilaterals.  We could 
consider quadrilaterals in spacelike planes, without rotating those 
planes to horizontal.  However, in the argument below we choose to 
rotate the planes to 
horizontal, so that the metric will be exactly the Euclidean metric 
that is so familiar to us.  

We assume that the points $\ef_p, \ef_q, \ef_r, \ef_s$ lie in a "circle" in 
a spacelike plane of $\mathbb{R}^{2,1}$.  In general, such a circle is 
\[ \left\{ \left. 
\begin{pmatrix}
\cos \beta & \sin \beta & 0 \\ -\sin \beta & \cos \beta & 0 \\ 0 & 0 & 1
\end{pmatrix}
\begin{pmatrix}
\cosh \gamma & 0 & \sinh \gamma  \\ 0 & 1 & 0 \\ 
\sinh \gamma & 0 & \cosh \gamma  \end{pmatrix}
\begin{pmatrix}
\rho \cos \theta \\ \rho \sin \theta \\ 0
\end{pmatrix}
+
\begin{pmatrix}
x_0 \\ y_0 \\ z_0
\end{pmatrix} \, \right| \, \theta \in [0,2\pi ) \right\} \; , \] 
where $x_0,y_0,z_0,\rho,\gamma,\beta$ are all real constants.  
By a rigid motion of $\mathbb{R}^{2,1}$, 
we can move this circle to the horizontal circle 
\[ \{ (\rho \cos \theta, \rho \sin \theta, 0) \, | \, \theta \in [0,2\pi) \} 
\; . \]   
Then the points $\ef_p,\ef_q,\ef_r,\ef_s$ are moved to 
points $(\rho \cos \theta_*, 
\rho \sin \theta_*, 0)$ for $*=p,q,r,s$, respectively.  

Then we can compute the cross ratio in the usual way for 
the space $\mathbb{R}^3$ (that is, we can replace the 
metric for $\mathbb{R}^{2,1}$ with 
the metric for $\mathbb{R}^3$ and then compute the cross ratio, 
which is allowed because the circle is now horizontal 
in $\mathbb{R}^{2,1}$):
\[
q_{pqrs} = \sin (\frac{\theta_p-\theta_q}{2}) 
\csc (\frac{\theta_q-\theta_r}{2})
\sin (\frac{\theta_r-\theta_s}{2})
\csc (\frac{\theta_s-\theta_p}{2}) \; . \]

\begin{remark}
This $q_{pqrs}$ is invariant under isometries of 
$\mathbb{R}^{2,1}$ (by definition), 
but is not M\"obius invariant (unlike the case of 
$\mathbb{R}^3$).  
\end{remark}

Once the $q_{qprs}$ are defined, then the $a_{pq}$ can be defined by 
\[ q_{pqrs} = a_{pq}/a_{ps} \; , \] and then we could use the same 
equations as for the 
$\mathbb{R}^3$ case, that is, the equations in 
Definition \ref{defn:crossratiofactorizingfct}, 
to determine when the surface is discrete isothermic, with 
spacelike quadrilaterals.  

Then, after restricting to discrete isothermic surfaces, we 
could define discrete spacelike CMC surfaces in 
$\mathbb{R}^{2,1}$ by imitating 
the equations from the case of discrete CMC surfaces in 
$\mathbb{R}^3$, as found 
in Theorem \ref{thm10ptpt5}.  This is 
justified by looking at smooth CMC surfaces in 
$\mathbb{R}^3$ and $\mathbb{R}^{2,1}$, 
which have exactly the same equations --  only the 
ambient metric changes, 
see Corollary \ref{cor:fortheR21case}.

So the equations we want for defining a discrete spacelike 
CMC surface in $\mathbb{R}^{2,1}$ are as follows: 
there exist $h,H \in \mathbb{R}$ and normals $n_p$ so that 
\begin{enumerate}
\item $\langle n_p , n_p \rangle_{\mathbb{R}^{2,1}} = -1$, 
\item $d\ef_{pq} \wedge n_q + n_p \wedge d\ef_{pq} = 0$, 
\item $\langle d\ef_{pq} , n_p+n_q \rangle_{\mathbb{R}^{2,1}} = 
0$, and 
\item $h (dn_{pq} + H d\ef_{pq}) = \frac{-a_{pq} d\ef_{pq}}{|d\ef_{pq}|^2}$, 
\end{enumerate}
where here $\langle \cdot , \cdot \rangle_{\mathbb{R}^{2,1}}$ 
represents the $\mathbb{R}^{2,1}$ inner product, and $\wedge$ 
is the $\mathbb{R}^{2,1}$ cross product, and $| \cdot |$ is the 
$\mathbb{R}^{2,1}$ norm.  

\section{Polynomial conserved quantities and Darboux transforms}
\label{section:p.c.q.s} 

\subsection{Polynomial conserved quantities} 
Equation \eqref{linearconservedquantity-rossman} 
can be extended to define discrete isothermic surfaces 
$\ef$ with polynomial conserved quantities, as follows: 

\begin{defn}
\[ P = Q+\lambda P_1+\lambda^2 P_2+ ... + \lambda^{n-1} P_{n-1} + 
\lambda^n Z \] is a {\em polynomial conserved quantity} if 
\begin{equation}\label{polynomialconservedquantity} 
(I+\lambda \tau_{pq}) 
P_q = P_p (I+\lambda \tau_{pq}) \; , 
\end{equation} 
where $Q$, $Z$ and the $P_j$ are maps from the lattice domain to 
$\mathbb{R}^{4,1}$.  
\end{defn}

We sometimes write $Z$ as $P_n$ as well.  
If such a polynomial conserved quantity exists, we say that 
$\ef$ is a {\em special surface of type $n$}, and the above 
Equation \eqref{polynomialconservedquantity} is equivalent to 
\begin{equation}\label{eqn:pcq-constancy} 
  T^\lambda_p P_p (T^\lambda_p)^{-1} \end{equation} being constant 
with respect to the vertices $p$.  

Equation \eqref{polynomialconservedquantity} can be restated as 
$\Gamma_{pq}^\lambda \cdot P_q = P_p$, i.e. 
$P$ is a parallel section, i.e. $P$ is conserved by the connection 
$\Gamma_{pq}^\lambda$, so we can call it a "conserved quantity".  

\subsection{Polynomial conserved quantities for smooth surfaces}\label{therehi} 
Before further exploring discrete surfaces with polynomial 
conserved quantities, let us consider the case of smooth surfaces.  
Definition \ref{first-lcq-defn} and Equation 
\eqref{rossman-first-equation} 
can be extended to define smooth surfaces with polynomial 
conserved quantities \[ P = Q+\lambda P_1+\lambda^2 P_2+ ... + 
\lambda^{n-1} P_{n-1} + \lambda^n Z \; , \] 
where $Q$, $Z$ and the $P_j$ are maps from the domain of definition of 
$x=x(u,v)$ to $\mathbb{R}^{4,1}$, as follows:  

\begin{defn}
$P$ is a {\em polynomial conserved quantity of type $n$} 
if 
\begin{equation}\label{polynomialconservedquantity-smooth} 
dP = \lambda P \tau-\lambda \tau P \; . 
\end{equation}
\end{defn}

We now state a result about the polynomial conserved quantities 
of Darboux transforms of smooth surfaces (recall that Darboux 
transformations were defined in Definition \ref{defn:smoothDarbTrans}): 

\begin{lemma}\label{lem:smooth-case-type-up-one}
If the initial isothermic surface 
$x=x(u,v)$ has a polynomial conserved quantity of order $n$, 
then any Darboux transform 
$\hat x = \hat x(u,v)$ has a polynomial conserved quantity of order 
at most $n+1$.  
\end{lemma}

\begin{proof}
Let $X$ be a lift of the initial surface $x$ with Calapso transformation 
$T$ and polynomial conserved quantity $P$ of order $n$.  
Then $TPT^{-1}$ is constant.  Let $\hat X$ be a lift of the Darboux 
transform $\hat x$ of $x$, i.e. $T \hat X T^{-1}$ is constant in 
$PL^4$ for some particular choice of $\lambda$, and 
let us refer to that choice of $\lambda$ as $\lambda = \mu$.  
(From now on we take $\mu$ to be that fixed value, and $\lambda$ 
will denote a free real parameter.)  We define 
\[
  A=I - \frac{\lambda}{\mu} \frac{X \hat X}{X\hat X+\hat X X} 
 \] 
(since $X\hat X+\hat X X$ is a scalar multiple of the identity, 
we regard it as a scalar in the denominator here), 
and we can check that 
\[
 A^{-1} = \frac{1}{(\mu-\lambda)(X\hat X+\hat X X)} (\mu X \hat X 
+ (\mu-\lambda) \hat X X) \; , \] which follows immediately from 
the property $X^2=\hat X^2=0$.  

Since we are free to rescale $X$ and $\hat X$, let us rescale them 
so that 
\[ X = \begin{pmatrix} x & -x^2 \\ 1 & -x \end{pmatrix} \; , \;\;\; 
   \hat X = \frac{1}{\delta^2} 
    \begin{pmatrix} \hat x & -\hat x^2 \\ 1 & -\hat x 
    \end{pmatrix} \; , \] 
where $\delta := \hat x - x$.  Then 
\[ X \hat X = \frac{1}{\delta^2} 
   \begin{pmatrix} x \delta & -x \delta \hat x \\ \delta & 
   -\delta \hat x \end{pmatrix} \; , \;\;\; 
   \hat X X = \frac{1}{\delta^2} 
   \begin{pmatrix} -\hat x \delta & \hat x \delta x \\ - \delta & 
   \delta x \end{pmatrix} \; , \]
and also 
\[ X \hat X + \hat X X = -I \; , \;\;\; x \delta \hat x = 
    \hat x \delta x \; . \]  

We also have that the logarithmic derivatives $\tau$ and $\hat \tau$ 
of Calapso transformations of $x$ and $\hat x$ satisfy 
\[ \tau (\partial_u) = \begin{pmatrix}
   x x_u^{-1} & -x x_u^{-1} x \\ x_u^{-1} & - x_u^{-1} x 
\end{pmatrix} \; , \;\;\; 
\tau (\partial_v) = - \begin{pmatrix}
   x x_v^{-1} & -x x_v^{-1} x \\ x_v^{-1} & - x_v^{-1} x 
\end{pmatrix} \; , \]\[ 
\hat \tau (\partial_u) = \begin{pmatrix}
   \hat x \hat x_u^{-1} & -\hat x \hat x_u^{-1} \hat x \\ 
   \hat x_u^{-1} & - \hat x_u^{-1} \hat x 
\end{pmatrix} \; , \;\;\; 
\hat \tau (\partial_v) = - \begin{pmatrix}
   \hat x \hat x_v^{-1} & -\hat x \hat x_v^{-1} \hat x \\ 
    \hat x_v^{-1} & - \hat x_v^{-1} \hat x 
\end{pmatrix} \; .
\]
Furthermore, by Equation \eqref{eqn-riccati}, we have 
\begin{equation}\label{eqn:6.21repeated} 
\hat x_u = \mu \delta x_u^{-1} \delta \; , \;\;\; 
\hat x_v = - \mu \delta x_v^{-1} \delta \; . \end{equation}

Next we should show that 
$d(TA)=TA \cdot \lambda \hat \tau$, so $\hat T = TA$ solves 
$\hat{T}^{-1} d\hat{T} = \lambda \hat{\tau}$.  That is, we 
wish to show that 
\[ A^{-1} \lambda \tau (\partial_u) A + A^{-1} dA - 
\lambda \hat \tau (\partial_u) = \]\[ 
\frac{\mu \lambda}{\mu-\lambda} \begin{pmatrix}
x x_u^{-1} & -x x_u^{-1} x \\ x_u^{-1} & - x_u^{-1} x
\end{pmatrix}
- \lambda \begin{pmatrix}
\hat x \hat x_u^{-1} & -\hat x \hat x_u^{-1} \hat x \\ \hat x_u^{-1} & 
- \hat x_u^{-1} \hat x 
\end{pmatrix}
- \frac{\lambda (\delta \delta_u + \delta_u \delta)}{(\mu-\lambda) 
\delta^4} 
\begin{pmatrix}
x \delta & -x \delta \hat x \\ \delta & - \delta \hat x
\end{pmatrix} + \]\[ 
\frac{\lambda}{(\mu-\lambda) \delta^2} 
\begin{pmatrix}
x_u \delta + x \delta_u & -x_u \delta \hat x-x \delta_u \hat x-x \delta 
\hat x_u \\ \delta_u & - \delta_u \hat x- \delta \hat x_u
\end{pmatrix} + 
\frac{\lambda^2}{\mu (\mu-\lambda) \delta^4} 
\begin{pmatrix}
-\hat x \delta x_u \delta & \hat x \delta x_u \delta \hat x \\ 
-\delta x_u \delta & \delta x_u \delta \hat x 
\end{pmatrix} \]
is zero, and also $A^{-1} \lambda \tau (\partial_v) A + A^{-1} dA - 
\lambda \hat \tau (\partial_v)=0$, 
and this follows from the first equation in 
\eqref{eqn:6.21repeated}.  

Then we define \[ \hat P = \mu (\mu-\lambda) A^{-1} P A \; , \] 
and we can 
show that $\hat T \hat P \hat T^{-1}$ is constant, as follows: 
$d(\hat{T} \hat{P} \hat{T}^{-1}) = \mu (\mu-\lambda) d(T A \cdot A^{-1} 
P A \cdot A^{-1} T^{-1}) = \mu (\mu-\lambda) d(TPT^{-1}) = 0$.  

It is now clear that $\hat P$ is a polynomial conserved quantity 
of degree at most $n+2$.  To show that the degree is actually at 
most $n+1$, it suffices to show that $P_n$ is perpendicular to $X$, 
and so $X P_n X = 0$.  We omit an argument for this, but note that 
the analogous argument for the case of discrete 
surfaces can be found in detail below.  
\end{proof}

\begin{remark}\label{rem:9point4}
The Darboux transform in Lemma \ref{lem:smooth-case-type-up-one} 
is a Baecklund transform exactly when it is of type at most $n$.  
See Remarks \ref{rem:7point48} and \ref{rem:7point49}.  
See also Definition \ref{somebegginnerr} 
and Lemma \ref{beginnerlemma} (discrete case).  
\end{remark}

For an isothermic surface with a polynomial conserved quantity 
of order $n$, we define a complementary surface as follows: take a value 
$\lambda_0$ of $\lambda$ so that 
\[ 
|| P(\lambda_0) ||^2 = || Q+\lambda_0 P_1+\lambda^2_0 P_2+ ... + 
\lambda^{n-1}_0 P_{n-1} + \lambda^n_0 Z ||^2=0
\] and define the complementary surface to be $P(\lambda_0)$.  
This will be a Baecklund transformation, so will be of type at most 
$n$.  We say more about this in Section \ref{compsurfcompsurf}.  

Complementary surfaces can be of type $n$.  
But if a Baecklund transform is of type $n-1$ (Darboux transforms 
must be of type at least $n-1$, as seen in 
Lemma \ref{downdowndown}), then it must be a complementary surface, 
by Lemma 4.10 of \cite{BHRS}.  Examples of type $n-1$ 
Baecklund transforms can come from 
CMC $1$ surfaces in $\mathbb{H}^3$ and minimal surfaces in 
$\mathbb{R}^3$.  In fact, we have the following lemma: 

\begin{lemma}\label{lem:twospecialcases}
In the case $n=1$ (i.e. CMC surfaces), CMC $\pm \sqrt{-\kappa}$ 
surfaces in $M_\kappa$ are the only 
cases where a type $n-1=0$ Baecklund transform can exist.  
In particular, if such a Baecklund transform exists, then 
$\kappa \leq 0$.  
\end{lemma}

\begin{proof}
When the linear conserved quantity is normalized, we have 
\[ ||\lambda Z + Q||^2 = \lambda^2-2 H \lambda - \kappa \; , \] 
and the discriminant is 
\[ 2 \sqrt{H^2+\kappa} \; . \]  
When a type $0$ Baecklund transform exists, we have 
a higher order zero of $\lambda^2-2 H \lambda - 
\kappa$ (by Lemma 4.10 in \cite{BHRS}), so 
$H^2+\kappa=0$, i.e. $H^2=-\kappa$.  (See \cite{BHRS} for 
further details.)  
\end{proof}

\begin{remark}
Take a smooth surface and apply two Darboux transformations 
given by using $\lambda$ and $\mu$, respectively.  Then 
apply a Darboux transformations to each of those, but now 
using $\mu$ for the case of the surface first made using 
$\lambda$ and using $\lambda$ for the case of the surface 
first made using $\mu$.  By a permutability theorem, this 
second pair of Darboux transformations is just one single 
surface.  Fixing one point on the original surface and 
looking at the other three corresponding points on the four 
(actually only three) transformed surfaces, one has a 
quadrilateral with cross ratio equal to $\lambda/\mu$.  One 
can keep repeating this procedure to make more 
quadrilaterals.  This will result 
in a discrete surface starting from a single point on the 
original smooth surface, and comprized of corresponding 
points on the transformed surfaces (one point for each 
transformed surface).  Because the cross ratios take the 
form $\lambda/\mu$, this discrete surface is 
discrete isothermic.  
\end{remark}

\subsection{Darboux transforms for discrete surfaces} 
The Darboux transforms of discrete surfaces have similar enveloping
properties to the case of smooth surfaces.  In the discrete case, 
the eight vertices of two corresponding quadrilaterals (one on the
original surface and the corresponding 
one on the Darboux transform) all lie in one sphere.  
(This can be seen from the upcoming Lemmas 
\ref{lemmalemma9point8} and \ref{lemmalemma9point10}.)  

Assume $\ef$ is a discrete isothermic surface, 
and that $F$ is a lift of $\ef$.  
We have the Christoffel transformation $T^\lambda$ satisfying 
\[ T_q^\lambda = T_p^\lambda (I+\lambda \tau_{pq}) \; . \] 

\begin{defn}\label{defn:discDarbTrans}
$\hat F$ gives a Darboux transform $\hat \ef$ of $\ef$ if 
\begin{equation}\label{eqn:definingDarbtransfindiscretecase} 
T^\mu_p \hat F_p (T^\mu_p)^{-1} \end{equation} is constant 
in $PL^4$ with respect to vertices $p$, for some value $\mu$.  
\end{defn}

When the term $T^\mu_p \hat F_p (T^\mu_p)^{-1}$ in Equation 
\eqref{eqn:definingDarbtransfindiscretecase} is set to a constant, 
we have what is sometimes called Darboux's linear system.  

In this definition, 
$\hat \ef$ is a Darboux transformation if $T^\mu \hat F (T^\mu)^{-1}$ 
is constant.  Here "constant" means in the projectivized 
sense.  That is, there 
exists an $r_{pq} \in 
\mathbb{R}$ such that $T_p^\mu \hat F_p (T_p^\mu)^{-1} 
= r_{pq} \cdot T_q^\mu \hat F_q (T_q^\mu)^{-1}$.  

Once a choice of $T$ is made, 
it is possible to choose $\hat F$ so that $r_{pq}=1$ on all edges, 
but then $\hat F$ might not be a Moutard lift.  

Just like in the smooth case, where we obtained the Riccati 
equation \eqref{eqn-riccati}, we have: 

\begin{lemma}
Definition \ref{defn:discDarbTrans} is equivalent to 
\begin{equation}\label{eq:RiccatiDiscrete} 
d\hat \ef_{pq} = \mu (\hat \ef-\ef)_p df^*_{pq} 
(\hat \ef-\ef)_q \; . \end{equation}
\end{lemma}

\begin{proof}
We prove just one direction here.  The other direction can be 
proven by an argument analogous to the one in the proof of 
Lemma \ref{8ptpt50}, and we leave that to the reader.  

$T_q^\mu \hat F_q (T_q^\mu)^{-1}$ being parallel to 
$T_p^\mu \hat F_p (T_p^\mu)^{-1}$ is equivalent to the following 
four equations: 
\[ 1 + \mu d\ef_{pq}^* d\mathcal{D}_q = r (1 - 
\mu d\mathcal{D}_p d\ef^*_{pq}) \; , \]
\[ \hat \ef_q + \mu \ef_p d\ef_{pq}^* d\mathcal{D}_q = 
r (\hat \ef_p - \mu \hat \ef_p 
d\mathcal{D}_p d\ef^*_{pq}) \; , \]
\[ \hat \ef_q + \mu d\ef_{pq}^* d\mathcal{D}_q \hat \ef_q = 
r (\hat \ef_p - \mu d\mathcal{D}_p d\ef^*_{pq} 
\ef_q) \; , \]
\[ ( \hat \ef_q + \mu \ef_p d\ef_{pq}^* d\mathcal{D}_q ) 
\hat \ef_q = r \hat \ef_p (\hat \ef_p - \mu 
d\mathcal{D}_p d\ef^*_{pq} \ef_q ) \; , \] 
for some real $r$, where $d\mathcal{D} := 
\hat \ef - \ef$.  Defining $r$ by the first of the 
four equations, we have 
\[ 
(1-\mu d\mathcal{D}_p d\ef^*_{pq}) 
(\hat \ef_q + \mu \ef_p d\ef^*_{pq} d\mathcal{D}_q) 
\hat \ef_q = 
\] 
\[ 
(1+\mu d\ef^*_{pq} d\mathcal{D}_q) \hat \ef_p 
(1 - \mu d\mathcal{D}_p d\ef^*_{pq} ) \hat \ef_q  = 
\] 
\[ 
(1-\mu d\mathcal{D}_p d\ef^*_{pq}) \hat \ef_p 
(1 + \mu d\ef^*_{pq} d\mathcal{D}_q) \hat \ef_q = 
\] 
\[ 
(1+\mu d\ef^*_{pq} d\mathcal{D}_q) \hat \ef_p 
(\hat \ef_p - \mu d\mathcal{D}_p d\ef^*_{pq} \ef_q ) \; . 
\] 
In particular, 
\[ 
\hat \ef_q + \mu \ef_p d\ef^*_{pq} d\mathcal{D}_q = 
\hat \ef_p + \mu \hat \ef_p d\ef^*_{pq} d\mathcal{D}_q \; , 
\] 
\[ 
\hat \ef_q - \mu d\mathcal{D}_p d\ef^*_{pq} \hat \ef_q = 
\hat \ef_p - \mu d\mathcal{D}_p d\ef^*_{pq} \ef_q \; . 
\] 
Summing these last two equations gives Equation 
\eqref{eq:RiccatiDiscrete}.  
\end{proof}

\begin{lemma}\label{lemmalemma9point8}
If $\hat \ef$ is a Darboux transform of $\ef$, then for 
adjacent $p$ and $q$, the four points $\ef_p$, $\ef_q$, 
$\hat \ef_q$ and $\hat \ef_p$ are concircular.  
\end{lemma}

\begin{proof}
Because $\hat \ef$ is a Darboux transformation, 
by \eqref{eq:RiccatiDiscrete} we have 
\[ 1 = \mu (\hat \ef_p - \ef_p) 
d\ef_{pq}^* (\hat \ef_q - \ef_q)(d\hat \ef_{pq})^{-1} \]
for some $\mu \in \mathbb{R}$.  So 
\[ (\hat \ef_p - \ef_p)
(\ef_q - \ef_p)^{-1} (\hat \ef_q - \ef_q)
(\hat \ef_q - \hat \ef_p)^{-1} \in \mathbb{R} \; , \]
so the cross ratio of $\ef_p$, $\ef_q$, 
$\hat \ef_q$ and $\hat \ef_p$ is real.  
\end{proof}

\begin{lemma}\label{lemmalemma9point9}
If $\ef^*$ is both a Christoffel and a Darboux transform, then 
$|\ef-\ef^*|$ is constant.  
\end{lemma}

\begin{proof}
The previous lemma implies that $\ef_p$, $\ef_q$, 
$\ef_q^*$ and $\ef_p^*$ are concircular.  Because $\ef^*$ 
is a Christoffel transform, $\ef_q-\ef_p$ and $\ef_q^* 
- \ef_p^*$ are parallel.  
\end{proof}

\begin{lemma}\label{lemmalemma9point10}
A Darboux transform $\hat \ef$ of $\ef$ has the same cross ratios as $\ef$.  
\end{lemma}

\begin{proof}
Note that $p$ and $q$ can be switched in Equation 
\eqref{eq:RiccatiDiscrete}, which can be seen just by conjugating 
that equation, so we have 
\[ (\hat \ef - \ef)_p df^*_{pq} (\hat \ef - \ef)_q = 
(\hat \ef - \ef)_q df^*_{pq} (\hat \ef - \ef)_p \; . \]  
Then 
\[ \hat q = d\hat \ef_{pq} (d\hat \ef_{qr} )^{-1} 
d\hat \ef_{rs} (d\hat \ef_{sp} )^{-1} = \]\[ 
(\hat \ef - \ef)_p df^*_{pq} (\hat \ef - \ef)_q 
((\hat \ef - \ef)_q df^*_{qr} (\hat \ef - \ef)_r)^{-1} 
(\hat \ef - \ef)_r df^*_{rs} (\hat \ef - \ef)_s 
((\hat \ef - \ef)_s df^*_{sp} (\hat \ef - \ef)_p)^{-1} = \]\[ 
(\hat \ef - \ef)_p df^*_{pq} (\hat \ef - \ef)_q 
((\hat \ef - \ef)_r df^*_{qr} (\hat \ef - \ef)_q)^{-1} 
(\hat \ef - \ef)_r df^*_{rs} (\hat \ef - \ef)_s 
((\hat \ef - \ef)_p df^*_{sp} (\hat \ef - \ef)_s)^{-1} = \]\[ 
(\hat \ef - \ef)_p df^*_{pq} 
(df^*_{qr})^{-1} df^*_{rs} df^*_{sp} (\hat \ef - \ef)_p^{-1} = q^* \; . \]
Since $q^*=q$, by Lemma \ref{lemmalemma8point21}, 
the proof is completed.  
\end{proof}

Justification for the following definition can be found in 
\cite{BobPink2}, \cite{Udo1}, \cite{Udo-bk} and \cite{HHP}: 

\begin{defn}\label{oldsensedefn}
Let $\ef$ be a discrete isothermic surface in 
$\mathbb{R}^3$.  The $\ef$ is 
CMC in the {\em old sense} if there exists a Christoffel transformation 
that is also a Darboux transformation.  
\end{defn}

Note that Christoffel transformations are 
defined only up to translation and scaling.  

\begin{lemma}\label{oldsenselemma1}
If $\ef$ is CMC in the old sense, then $\ef$ has a linear conserved quantity 
(for $\mathbb{R}^3$).
\end{lemma}

\begin{proof}
By assumption, there exists an 
$\ef^*$ such that $a_{pq} = d\ef_{pq} d\ef_{pq}^*$ 
and there exists a $\mu \in \mathbb{R}$ such that 
\begin{equation}\label{oldsense1} 
d\ef^*_{pq} = \mu (\ef^*-\ef)_p d\ef^*_{pq} (\ef^*-\ef)_q \; . 
\end{equation}
Set $n_p = s \cdot (\ef^*-\ef)_p$ for some constant $s \in 
\mathbb{R}$.  For simplicity, we assume $\mu>0$, and leave the case 
$\mu < 0$ to the reader.  

Lemma \ref{lemmalemma9point9} 
implies $|n_p|^2$ is constant.  Since $|\ef^*_p-\ef_p|^2$ is 
constant, Equation \eqref{oldsense1} implies 
\begin{equation}\label{oldsense2} 
\mu^{-2} = (\ef^*_p-\ef_p)^4 \end{equation} 
for all vertices $p$.  
Take $Q$ as in \eqref{choiceofQ} with $\kappa=0$, and set 
\[ Z = \begin{pmatrix}
H \ef+n & -\ef n - n \ef - H \ef^2 \\ 1 & - H \ef - n 
\end{pmatrix} \; . \]
The goal is to find values for the real constants $s$ and $H$ so that 
\begin{equation}\label{oldsense3}
dZ = Q \tau - \tau Q
\end{equation}
and 
\begin{equation}\label{oldsense4}
\tau Z_q = Z_p \tau \; . 
\end{equation}
If we take $H=1$, then 
Equation \eqref{oldsense4} holds if and only if $d\ef_{pq} n_q + n_p d\ef_{pq} 
= 0$, and this follows from $(\ef_p^*-\ef_p)^2 = (\ef^*_q-\ef_q)^2$ 
and the fact that $d\ef_{pq}^*$ is parallel to $d\ef_{pq}$.  
Equation \eqref{oldsense3} holds if and only if 
\begin{equation}\label{oldsense5}
d\ef^*_{pq} = H d\ef_{pq} + dn_{pq}
\end{equation}
and 
\begin{equation}\label{oldsense6}
\ef_p n_p + n_p \ef_p - \ef_q n_q - n_q \ef_q + H \ef_p^2 - H \ef_q^2 
= -d\ef^*_{pq} \ef_q - \ef_p d\ef^*_{pq} 
\end{equation}
both hold.  Equation \eqref{oldsense5} holds if $H=s=1$.  Now assume 
that $H=s=1$.  Then Equation \eqref{oldsense6} is equivalent to 
\[
(\ef^*-\ef)_p d\ef_{pq} + d\ef_{pq} (\ef^*-\ef)_q = 0 \; , 
\]
which in turn is equivalent to 
\[
(\ef^*-\ef)_p d\ef^*_{pq} (\ef^*-\ef)_q \cdot \frac{1}{(\ef^*-\ef)^2} 
= - d\ef^*_{pq} \; , 
\]
and this last equation is the same as Equation \eqref{oldsense1} 
by \eqref{oldsense2} and the facts that $(\hat \ef-\ef)^2<0$ 
and $\mu > 0$.  This completes the proof.  
\end{proof}

Although the constant $H$ becomes $1$ in the above proof, 
this does not 
necessarily mean that $H$ is the mean curvature of the CMC surface 
$\ef$, because the linear conserved quantity 
might not be normalized so that 
$-Z^2$ takes the value needed to make $H$ the mean curvature.  

\begin{lemma}\label{lemmalemma9pt13}
If $\ef$ has a linear conserved quantity $Q + \lambda Z$ 
for $\mathbb{R}^3$ (i.e. $Q^2=0$) so that 
$\langle Z,Q \rangle \neq 0$, then $\ef$ is CMC 
in $\mathbb{R}^3$ in the old sense.  
\end{lemma}

\begin{proof}
We can assume the constant term $Q$ in the linear conserved 
quantity is as in \eqref{choiceofQ} with $\kappa = 0$.  Then 
Corollary \ref{cor:AinMarch2008} implies that 
there exists a constant $H \in \mathbb{R} 
\setminus \{ 0 \}$, and an 
$n_p \in \text{Im} H$ with $|n_p|^2$ constant, such that 
\[ d\ef^*_{pq} = d(H \ef + n)_{pq} 
\; , \;\;\; d\ef_{pq} n_q+n_p d\ef_{pq} = 0 \; . \]  
The goal is to find constants $\mu$ and $\alpha$ in 
$\mathbb{R}$, and a constant $b \in \text{Im} H$ so that 
\[ \alpha d\ef^*_{pq} = \mu (\alpha \ef^*+b-\ef)_p d\ef_{pq}^* 
(\alpha \ef^*+b-\ef)_q 
\; . \]  Here $\hat \ef = \alpha \ef^* + b$, and without loss of 
generality we can take $\ef^* = H \ef + n$.  

Take $b=0$ and $\alpha = H^{-1}$.  Then the goal becomes to find 
$\mu$ such that 
$H^{-1} d\ef^*_{pq} = \mu H^{-1} n_p d\ef^*_{pq} H^{-1} n_q$, and 
$\mu = -H/n^2$ will work.  
\end{proof}

With respect to Lemma \ref{lemmalemma9pt13}, we can treat the case 
$\langle Z,Q \rangle = 0$ separately, and we leave this to the 
reader.  This will lead to the equivalence of discrete minimal 
surfaces as defined here via linear conserved quantities, and 
discrete minimal surfaces as previously defined (see 
\cite{BobPink2}, \cite{Udo1}, \cite{Udo-bk}, \cite{HHP}), 
like this: 

\begin{defn}
$\ef$ is a discrete minimal surface 
in $\mathbb{R}^3$ in the old sense if the Christoffel 
transform $\ef^*$ takes values in a sphere.  
\end{defn}

Now let us turn our attention to a discrete version of Lemma 
\ref{lem:smooth-case-type-up-one}.  
Suppose that the discrete isothermic surface $\ef$ has a 
polynomial conserved quantity $P$ of order $n$.  
Let $\hat \ef$ be a Darboux transform of 
$\ef$ determined by the value 
$\mu \in \mathbb{R}$.  ($\lambda$ and $\mu$ play the same 
roles here as they did in the proof of Lemma 
\ref{lem:smooth-case-type-up-one}.)  
Consider a Christoffel transformation $\hat T^\lambda$ 
of $\hat \ef$ satisfying 
\[ 
\hat T_q^\lambda = \hat T_p^\lambda (1+\lambda \hat \tau_{pq}) 
\; . \]  
Let $F$ and $\hat F$ be lifts into $L^4$ of $\ef$ and 
$\hat \ef$, respectively.  We define \[ A = A_p := 
I - \frac{\lambda}{\mu} \frac{F_p \hat F_p}{F_p \hat F_p + 
\hat F_p F_p} \] for each vertex $p$.   
We want to show \[ (T^\lambda A)_q = (T^\lambda A)_p 
(I+\lambda \hat \tau_{pq}) \; , \] so that we can take 
$\hat T^\lambda = T^\lambda A$, i.e. we want 
\[ (I+\lambda \tau_{pq}) A_q = A_p (I+\lambda \hat \tau_{pq}) \; , \] 
i.e. 
\[ (I+\lambda \tau_{pq}) (I - \frac{\lambda}{\mu} \frac{F_q \hat F_q}{F_q 
\hat F_q + \hat F_q F_q}) = (I - \frac{\lambda}{\mu} 
\frac{F_p \hat F_p}{F_p \hat F_p + \hat F_p F_p}) (I+
\lambda \hat \tau_{pq}) \; . \]  
We can choose $F$ and $\hat F$ so that $F_p$, $F_q$, 
$\hat F_q$ and $\hat F_p$ are a Moutard lift of the concircular 
quadrilateral with vertices $\ef_p$, $\ef_q$, 
$\hat \ef_q$ and $\hat \ef_p$, satisfying the equivalents of 
\eqref{eqn:Moutard} and \eqref{eqn:Moutard-tau}.  We can also let 
$1/\mu$ take the role of the cross ratio factor 
$a_{\ef_p\hat \ef_p} = a_{\ef_q\hat \ef_q}$ on the edges 
$\ef_p\hat \ef_p$ and $\ef_q\hat \ef_q$.  
Then the above equation is equivalent to 
\[ (F_q-\hat F_p) \hat F_q + F_p (F_q-\hat F_p) = 0 \; . \] 
Then, by the definition of Moutard lifts (Definition 
\ref{def:discrete-moutard}), this is equivalent to 
\[ (\hat F_q-F_p) \hat F_q + F_p (\hat F_q - F_p) = 0 \; , \] 
and this final equation is obviously true.  

It is then easily checked that 
\[ A_p^{-1} = \frac{1}{(\mu-\lambda)(F_p \hat F_p + \hat F_p F_p)} 
(\mu F_p \hat F_p+(\mu-\lambda) \hat F_p F_p) \; , \]  
so $(\mu-\lambda)A^{-1}_p$ is linear in $\lambda$.  Noting 
that $A_p$ itself is also linear in $\lambda$, we have that 
\[ \hat P := \mu (\mu-\lambda) A^{-1} P A \] is a polynomial in 
$\lambda$ of degree at most $n+2$.  Note that 
$T^\lambda P (T^\lambda)^{-1}$ is constant.  Also, 
\[ \hat T^\lambda \hat P (\hat T^\lambda)^{-1} = 
\mu (\mu-\lambda) T^\lambda A \cdot A^{-1}PA 
\cdot (T^\lambda A)^{-1} = \mu(\mu-\lambda) T^\lambda 
P (T^\lambda)^{-1} \; , \] so 
$\hat T^\lambda \hat P (\hat T^\lambda)^{-1}$ is constant.  
Thus $\hat P$ is a polynomial conserved quantity of type at most 
$n+2$ for the Darboux transform $\hat \ef$.  

We will see in Corollary \ref{cor:FPnF} below that 
$F P_n F=0$, so $\hat F F P_n F \hat F=0$, which implies that the top term 
of $\hat P$ is zero, so $\hat P$ is of type at most $n+1$.  This 
proves the following theorem (analogous to Lemma 
\ref{lem:smooth-case-type-up-one} for the smooth case):

\begin{theorem}\label{lem:typeatmostnplus1} 
A Darboux transform of a discrete 
special surface of type $n$ is a discrete special surface 
of type at most $n+1$.  
\end{theorem}

We now give some results, with the aim of obtaining Corollary 
\ref{cor:FPnF}.  

\begin{lemma}\label{lem:referencedinnextlemmaA}
If $P$ is a polynomial conserved quantity of a discrete 
isothermic surface $\ef$, and if $F$ is a lift of $\ef$, 
then $F_p dP_{pq} F_q = 0$ for all edges $pq$.  
\end{lemma}

\begin{proof}
By Equation \eqref{polynomialconservedquantity}, we have 
\[ 0 = F_p ((1+\lambda \tau_{pq}) P_q-P_p (1+\lambda \tau_{pq})) F_q 
= F_p (P_q-P_p) F_q \; , \]  
because $F_p \tau_{pq} = \tau_{pq} F_q = 0$.  
\end{proof}

\begin{corollary}\label{cor:referencedinnextlemmaB}
If $P$ is a polynomial conserved quantity of $\ef$, 
and if $F$ is a lift of $\ef$, then 
\begin{equation}\label{eqn:DPsubpq} 
dP_{pq} = \frac{\lambda a_{pq}}{\langle F_p,F_q \rangle} 
(\langle P_q, F_q \rangle F_p - \langle P_p, F_p \rangle F_q) 
\end{equation} for all edges $pq$.  
\end{corollary}

\begin{proof}
Because Equation 
\eqref{eqn:DPsubpq} is not affected by the choice of lift $F$, 
we may assume $F$ is Moutard, and \eqref{eqn:Moutard} and 
\eqref{eqn:Moutard-tau} hold.  
First note that $F_pF_q \neq 0$ if $\ef_p \neq \ef_q$.  Secondly, 
note that if $\mathcal{S} \in \mathbb{R}^{4,1} \setminus 
\{ 0 \}$ is perpendicular to both $F_p$ and $F_q$, then 
$\mathcal{S}$ is 
spacelike and $\mathcal{S}^2$ is a negative real 
scalar times $I$.  

Suppose further that $F_p \mathcal{S} F_q=0$, then 
\[ 0 = \mathcal{S} F_p \mathcal{S} F_q = 
-\mathcal{S}^2 F_pF_q \; , \] which gives a contradiction.  
Therefore Lemma \ref{lem:referencedinnextlemmaA} implies 
\[ dP_{pq} = \alpha F_p - \beta F_q \]  for some reals $\alpha, \beta$.  
Now consider the following computation: 
\[ (F_p P_q+P_q F_p) F_p = F_p P_q F_p = F_p (I+\lambda \tau_{pq}) 
P_q F_p = F_p P_p (I+\lambda \tau_{pq}) F_p = \]\[ 
F_p P_p ((1-\lambda a_{pq}) \cdot I-\lambda \tau_{qp}) F_p = 
(1-\lambda a_{pq}) (F_p P_p + P_p F_p) F_p \; . \]
Thus 
\[ F_p (P_q-P_p) F_p = -\lambda a_{pq} F_p P_p F_p \] and then 
\[ \beta \langle F_p,F_q \rangle F_p = 
\lambda a_{pq} \langle F_p,P_p \rangle F_p \; . \]  
Thus \[ \beta = \frac{\lambda a_{pq} 
\langle F_p,P_p \rangle}{\langle F_p,F_q \rangle} \; . \]  
We can derive $\alpha$ similarly.  
\end{proof}

\begin{lemma}\label{pnperpF}
$P_n \perp F$.  
\end{lemma}

\begin{proof}
Looking at the equation for $dP_{pq}$ in Corollary 
\ref{cor:referencedinnextlemmaB}, there is 
no $\lambda^{n+1}$ term on the left, so the $\lambda^{n+1}$ term on the 
right must be zero.  This means 
\[ \langle P_{n,q},F_q \rangle F_p = \langle P_{n,p},F_p \rangle F_q
\; . \]  Since $F_p$ and $F_q$ are not parallel, it follows that 
\[ \langle P_{n,q},F_q \rangle = \langle P_{n,p},F_p \rangle = 0 
\; . \]  So $\langle P_{n,p},F_p \rangle = 0$ for all vertices $p$.  
\end{proof}

\begin{corollary}\label{cor:FPnF}
$F P_n F=0$.  
\end{corollary}

\begin{proof}
Lemma \ref{pnperpF} gives $F P_n + P_n F =0$, which implies $0 = 
F P_n F+P_n F^2 = F P_n F$.  
\end{proof}

We also have the following stronger version of Corollary 
\ref{cor:referencedinnextlemmaB}, proven in 
\cite{BHRS}: 

\begin{corollary}
The polynomial conserved quantity $P$ satisfies 
\[ dP_{pq} = 
\frac{\lambda a_{pq}}{\langle F_p,F_q \rangle} \{ 
\langle P_q,F_q \rangle F_p - 
\langle P_p,F_p \rangle F_q
 \} \]\[ = 
\frac{\lambda a_{pq}}{(1-\lambda a_{pq}) \langle F_p,F_q \rangle} \{ 
\langle P_p,F_q \rangle F_p - 
\langle P_q,F_p \rangle F_q
 \} \; . \] 
\end{corollary}


The next lemma 
follows from the following symmetry: If $\hat \ef$ is a Darboux 
transform of $\ef$, then $\ef$ is also a Darboux transform of $\hat \ef$.  
So if the order of the polynomial conserved quantity can only go up 
by at most one, then also it can only go down by at most one.  

\begin{lemma}\label{downdowndown}
A Darboux transformation of a special surface of type $n$ 
is special of type at least $n-1$.  
\end{lemma}

\begin{proof}
The proof of Lemma \ref{lemmalemma9point8} implies 
\eqref{eq:RiccatiDiscrete} is equivalent to 
\begin{equation}\label{specsurfdownone}
1 = \mu a_{pq} (\hat \ef_p-\ef_p)(\ef_q-\ef_p)^{-1} 
(\hat \ef_q - \ef_q)(\hat \ef_q-\hat \ef_p)^{-1} \; , 
\end{equation}
so we know this equation \eqref{specsurfdownone} holds.  
We wish to show the equation that results when $\ef$ and 
$\hat \ef$ are switched also holds, i.e. that 
\begin{equation}\label{specsurfdowntwo}
1 = \hat \mu \hat a_{pq} (\ef_p-\hat \ef_p) 
(\hat \ef_q-\hat \ef_p)^{-1} 
(\ef_q - \hat \ef_q)(\ef_q-\ef_p)^{-1} 
\end{equation}
holds.  But the equivalence of \eqref{specsurfdownone} 
and \eqref{specsurfdowntwo} follows from $a_{pq} = 
\hat a_{pq}$ (Lemma \ref{lemmalemma9point10}) and Lemma 
\ref{lem:forlateruse}, when taking 
$\hat \mu = \mu$.  Now Theorem 
\ref{lem:typeatmostnplus1} implies the lemma.  
\end{proof}

\subsection{More on Calapso transformations} 
Because the Calapso transformation in \eqref{eqn:pcq-constancy} 
is constant, and Calapso transformations give isometries of 
$\mathbb{R}^{4,1}$ (for each fixed value of $T$), 
we have that $||P_p||^2$ is independent of $p$.  With this, 
it is not difficult to prove the following lemma about 
order $n$ polynomial conserved quantities $P$ of discrete 
isothermic surfaces $\ef$.  Let $F \in PL^4$ be a 
lift of $\ef$.  

\begin{lemma}\label{lem:sixfacts}
The following hold: 
\begin{enumerate}
\item $||Z||^2$ and $||Q||^2$ are constant.  (Lemmas 
\ref{lem:lcq-in-parts} and 
\ref{Zisctecte} when $n=1$)
\item $dP_{pq} = \frac{\lambda a_{pq}}{\langle F_p,F_q \rangle} 
      \{ \langle P_q,F_q \rangle F_p - 
         \langle P_p,F_p \rangle F_q \}$.  (Lemma 
\ref{cor:referencedinnextlemmaB}) 
\item $Z_p \perp F_p$ for all $p$.  (Lemma 
\ref{useful-lemma2} when $n=1$)
\item $||Z||^2 \geq 0$, and $||Z||^2=0$ if and only if $Z$ and 
      $F$ are parallel.  (Corollary 
    \ref{getsusedinch11} in the case of smooth 
surfaces, when $n=1$)
\item $S_{pq} := Z_p + a_{pq} \frac{\langle P_{n-1,q},F_q \rangle}
       {\langle F_p,F_q \rangle} F_p = 
          Z_q + a_{pq} \frac{\langle P_{n-1,p},F_p \rangle}
       {\langle F_p,F_q \rangle} F_q$.  (the curvature sphere) 
\item If $P$ is linear, then $\langle Q,Z \rangle$ is constant as well.  
(Corollary \ref{cor:AinMarch2008} 
when the ambient space is $\mathbb{R}^3$)
\item When $||Z||^2>0$, $S_{pq}$ gives a sphere via 
      \eqref{eq:S-tilde-sphere} containing both $\ef_p$ and $\ef_q$.  
\end{enumerate}
\end{lemma}

Next we prove the following lemma: 

\begin{lemma}\label{lem:positivecrossratio} 
If a discrete isothermic surface 
$\ef$ has a linear conserved quantity $P = 
Q + \lambda Z$ and lies in a connected space form 
(this rules out two copies of 
$\mathbb{H}^3$), then $||Z||^2=0$ implies 
the cross ratios of $\ef$ 
are positive.  In particular, the quadrilaterals are not embedded.  
\end{lemma}

\begin{proof}
Let $M_\kappa$ be the connected space form, which we may assume 
is produced by a $Q$ as in \eqref{choiceofQ}.  
Take the lift $F$ of $\ef$ so that $F \in M_\kappa$.  Because $||Z||^2=0$, 
there exists a real-valued function 
$r$ so that $Z=r F$ (by part (4) of Lemma \ref{lem:sixfacts}).  
Then $d(rF) = Q \tau - \tau Q$ gives the three equations 
\[ \frac{2 r_q}{1-\kappa \ef_q^2} - 
\frac{2 r_p}{1-\kappa \ef_p^2} = \kappa (\ef_p d\ef_{pq}^* + 
d\ef_{pq}^* \ef_q) \; , 
\]
\[ \frac{2 r_q \ef_q^2}{1-\kappa \ef_q^2} - 
\frac{2 r_p \ef_p^2}{1-\kappa \ef_p^2} = d\ef_{pq}^* \ef_q + 
\ef_p d\ef_{pq}^* \; , 
\]
\[ \frac{2 r_q \ef_q}{1-\kappa \ef_q^2} - 
\frac{2 r_p \ef_p}{1-\kappa \ef_p^2} = d\ef_{pq}^* + 
\kappa \ef_p d\ef_{pq}^* \ef_q \; . 
\]
The first and second of these equations imply that $r$ is constant 
(i.e. $r_p=r_q$).  

If $\kappa = 0$, the third equation gives $a_{pq} = 
d\ef_{pq} d\ef_{pq}^* = 2 r d\ef_{pq}^2$, 
so all the $a_{pq}$ have the 
same sign, and thus the cross ratios are positive.  

If $\kappa \neq 0$, then 
\[ \kappa \frac{2 r \ef_p \ef_q^2}{1-\kappa \ef_q^2} - \kappa 
\frac{2 r \ef_p \ef_p^2}{1-\kappa \ef_p^2} = \kappa (\ef_p 
d\ef_{pq}^* \ef_q + \ef_p^2 d\ef_{pq}^*) \; , 
\]
and the third of the above three equations gives 
\[ 2 r d\ef_{pq} = d\ef_{pq}^* (1-\kappa \ef_p^2) 
(1-\kappa \ef_q^2) \; . \]
Note that $1-\kappa \ef^2$ never changes sign, 
since $\ef$ stays in the 
connected $3$-dimensional space form $M_\kappa$, 
so all the $a_{pq}$ have the 
same sign.
\end{proof}

\begin{example}
The situation in Lemma \ref{lem:positivecrossratio} 
does indeed occur.  Consider the following 
simple example: Let the domain of $\ef$ be $\mathbb{Z}^2$, and 
              \[ f_{m,n} = j \; , \;\;\; \text{when} \;\; 
m \equiv n \equiv 0 \;\; (\text{mod} \; 2) \; , \]
              \[ f_{m,n} = i+j \; , \;\;\; \text{when} \;\; 
m \equiv 1, n \equiv 0 \;\; (\text{mod} \; 2) \; , \]
              \[ f_{m,n} = 0 \; , \;\;\; \text{when} \;\; 
m \equiv n \equiv 1 \;\; (\text{mod} \; 2) \; , \]
              \[ f_{m,n} = i \; , \;\;\; \text{when} \;\; 
m \equiv 0, n \equiv 1\;\; (\text{mod} \; 2) \; . \]
               All cross ratios are $1/2>0$.  We can take 
all cross ratio factors on vertical edges to be 
$a_{(m,n)(m,n+1)}=2$ and 
               on all horizontal edges to be $a_{(m,n)(m+1,n)}=1$.
Setting 
     \[ Z=Z_{m,n} =-\begin{pmatrix} \ef_{m,n} & -\ef_{m,n}^2 \\ 1 & 
-\ef_{m,n} \end{pmatrix} \; , 
            \;\;\; Q = \begin{pmatrix} 0 & 1 \\ 0 & 0 \end{pmatrix} \; , \]
               then $P = Q + \lambda Z$ is a linear conserved quantity 
of $\ef$, the ambient space is $\mathbb{R}^3$, and $||Z||^2=0$.  
Lemma \ref{lem:positivecrossratio} now implies all quadrilaterals are 
not embedded, which is also immediately clear from the definition 
of $\ef$.  
\end{example}

\begin{lemma}\label{lem:Calapso-preserves-type}
Suppose $\ef$ is a discrete isothermic surface with a conserved 
quantity $P$ of order $n$. 
Let $T_p^\mu$ denote a 
Calapso transformation satisfying $T_q^\mu = 
T_p^\mu (I + \mu \tau_{pq})$.  Then 
the Calapso transform $\ef_p^\mu$ with lift $F_p^\mu = 
T_p^\mu F_p (T_p^\mu)^{-1}$ 
also has a polynomial conserved quantity of order 
$n$, defined by \[
 P_p^\mu = T_p^\mu (P_p(\lambda+\mu)) (T_p^\mu)^{-1} \; , \]  
where $P(\lambda+\mu)$ denotes $P|_{\lambda \to \lambda 
+ \mu}$.  
\end{lemma}

\begin{proof}
\[ dP^\mu_{pq}+\lambda \tau^\mu_{pq} P^\mu_q - P^\mu_p \lambda 
\tau^\mu_{pq} = \]\[ 
T_q^\mu P_q(\lambda+\mu) (T_q^\mu)^{-1} - T_p^\mu P_p(\lambda+\mu) 
(T_p^\mu)^{-1} + \]\[ 
+ \lambda T_p^\mu \tau_{pq} P_q(\lambda+\mu) (T_q^\mu)^{-1} 
- \lambda T_p^\mu P_p(\lambda+\mu) \tau_{pq} (T_q^\mu)^{-1} \; , \] 
by Equation \eqref{eqn:Teqn3}.  So then 
\[ dP^\mu_{pq}+\lambda \tau^\mu_{pq} P^\mu_q - P^\mu_p \lambda 
\tau^\mu_{pq} = \]\[ 
T_q^\mu P_q(\lambda+\mu) (T_q^\mu)^{-1} - T_p^\mu P_p(\lambda+\mu) 
(T_p^\mu)^{-1} 
- T_p^\mu [P_q(\lambda+\mu)-P_p(\lambda+\mu) + 
\]\[ 
   + \mu \tau_{pq} P_q(\lambda+\mu)-\mu P_p(\lambda+\mu) \tau_{pq} 
] (T_q^\mu)^{-1} = \]\[
T_p^\mu [\mu \tau_{pq} P_q (\lambda+\mu)-P_p (\lambda+\mu) 
\mu \tau_{pq} - \mu \tau_{pq} P_q (\lambda+\mu) + 
\mu P_p (\lambda+\mu) \tau_{pq}] (T_q^\mu)^{-1} 
= 0 \; . \]
\end{proof}

Note that the corresponding result for the case of smooth surfaces 
of the above lemma, in the case $n=1$, implies that the Calapso 
transformations and Lawson transformations of smooth CMC surfaces 
are the same.  We say more about this in Remark 
\ref{rem:9point27} just below.  

\begin{remark}\label{rem:9point27}
Now we 
remark about the Lawson correspondence for smooth surfaces 
$x$ with linear conserved quantities $P = Q + \lambda Z$.  
First we note that, for a smooth surface $x$ with lift $X$ in 
the light cone $L^4$, we have, for the Calapso transformation 
$X^\lambda := TXT^{-1}$, 
\[ dX^\mu = d (T^\mu X (T^\mu)^{-1}) = T^\mu 
(dX + \mu \tau X - X \mu \tau) (T^\mu)^{-1}  = T^\mu 
dX (T^\mu)^{-1} \]
(note that $\tau X = X \tau = 0$), and we similarly have, 
for the linear conserved quantity $P^\mu = Q^\mu + 
\lambda Z^\mu$ of the Calapso transform, \[ 
dZ^\mu = T^\mu 
(dZ + \mu \tau Z - Z \mu \tau) (T^\mu)^{-1} = T^\mu 
dZ (T^\mu)^{-1} \; , \]
by considering the version of Lemma \ref{lem:Calapso-preserves-type} 
for smooth surfaces, and noting that $Z\tau-\tau Z= 0$.  

Let us consider for a moment how to derive the version of Lemma 
\ref{lem:Calapso-preserves-type} for smooth surfaces.  
For the smooth case as well, we have the corresponding properties 
$T^{\mu+\lambda} = T^{\lambda,\mu} T^\lambda$ 
(like we saw in Lemma \ref{nextlem:changeof-a-underCalapso} 
in the case of discrete surfaces) 
and $\tau^\mu = T^\mu \tau (T^\mu)^{-1}$ 
(like Equation \eqref{eqn:Teqn3} 
for the case of discrete surfaces) 
for the smooth surface $x^\mu$ with lift 
$X^\mu=T^\mu X (T^\mu)^{-1}$.  
When $x$ is CMC in some space form, 
we have a linear conserved quantity $P = Q + \lambda Z$, 
and by \eqref{eqn:forpg68}, 
\[ d (T^{\mu+\lambda} P(\mu+\lambda) (T^{\mu+\lambda})^{-1}) = 
0 \; , \] 
so 
\[ d (T^{\mu,\lambda} (T^{\mu} P(\mu+\lambda) (T^{\mu})^{-1}) 
(T^{\mu,\lambda})^{-1}) = 0 \; , \] and so 
$P^\mu(\lambda) = T^\mu P(\lambda + \mu) (T^\mu)^{-1}$ is a linear 
conserved quantity for $x^\mu$.  We have just derived the version of 
Lemma \ref{lem:Calapso-preserves-type} for smooth surfaces (stated 
only for the case $n=1$ here).  

Thus $x^\mu$ is CMC in the space 
form determined by the constant term $Q^\mu$ in $P^\mu$.  
Note that 
\[ P^\mu(\lambda) = \lambda Z^\mu+Q^\mu = 
(\lambda+\mu) T^\mu Z (T^\mu)^{-1} + T^\mu Q (T^\mu)^{-1}
= \]\[ = 
\lambda T^\mu Z (T^\mu)^{-1} + T^\mu (\mu Z+Q) (T^\mu)^{-1} 
\; . \]  
Now, 
\[ ||dX^\mu||^2 = ||T^\mu dX (T^\mu)^{-1}||^2
= ||dX||^2 \; , \] so the metrics of $x$ and $x^\mu$ are the same.  
Also, 
\[ -\langle dX^\mu, dZ^\mu \rangle = 
- \langle 
T^\mu dX (T^\mu)^{-1} , 
T^\mu dZ (T^\mu)^{-1}
\rangle = - \langle dX,dZ \rangle \; , \] 
so the Hopf differentials of $x$ and $x^\mu$ are the same.  
Also, assuming we have normalized $P$ properly, the mean 
curvatures $H$ and $H^\mu$ of $x$ and $x^\mu$ are related by 
\[ H^\mu = -\langle T^\mu Z (T^\mu)^{-1} , 
T^\mu (\mu Z + Q) (T^\mu)^{-1} \rangle = -\mu + H \; . \]  
We conclude that $x \to x^\mu$ is the Lawson correspondence, 
with the surface $x$ in the space form determined by $Q$, and 
the surface $x^\mu$ in the space form determined by $Q^\mu$.  
\end{remark}

\subsection{Baecklund transforms}

\begin{defn}\label{somebegginnerr}
If the Darboux transform $\hat \ef$ (with any lift $\hat F$) 
of a discrete special surface $\ef$ of type $n$ satisfies 
\[
  P(\mu) \perp \hat F \; , 
\]
then we say that $\hat \ef$ is a {\em Baecklund transform} of $\ef$.  
\end{defn}

In this case, it follows that $\hat \ef$ is also a special surface 
of type at most $n$, i.e. not of type $n+1$, as we will now see.  
(This definition is also related to Remark \ref{rem:7point49}.)  

\begin{lemma}\label{lem:ifPhasazero}
If a polynomial 
conserved quantity $P = 
P(\lambda)$ satisfies $P(\mu)=0$ for some $\mu \in 
\mathbb{R}$, 
then there exists a polynomial conserved quantity of order one less.  
\end{lemma}

\begin{proof}
$P(\mu)=0$ implies $\tilde P(\lambda) = \frac{1}{\lambda-\mu} 
\cdot P(\lambda)$ is still a polynomial.  Then $T_p^\lambda 
P_p(\lambda) (T_p^\lambda)^{-1}$ is constant 
with respect to $p$, and so 
$T_p^\lambda \tilde P_p(\lambda) (T_p^\lambda)^{-1}$ is too.  
Also, $\text{ord}\tilde P = (\text{ord}P) - 1$.  
\end{proof}

The following lemma justifies the statement we made in Remark 
\ref{rem:9point4}.  

\begin{lemma}\label{beginnerlemma}
For a Darboux transform $\hat \ef$ 
of a type $n$ discrete special surface $\ef$ 
determined by the value 
$\lambda=\mu$, if $P(\mu) \perp \hat \ef$, then $\hat \ef$ 
is of type at most $n$.  
\end{lemma}

\begin{proof}
As in the proof of Theorem \ref{lem:typeatmostnplus1}, 
\[
 \hat P = \frac{\mu}{F\hat F+\hat F F} 
 \{ (\mu F \hat F+(\mu-\lambda) \hat F F) P (I - 
 \frac{\lambda}{\mu} \frac{F \hat F}{F\hat F+\hat F F}) \} \; , 
\] so 
\[
 \hat P(\mu) = \frac{\mu}{F\hat F+\hat F F} 
 \{ \mu F (\hat F P(\mu)) (I - 
    \frac{F \hat F}{F\hat F+\hat F F} ) \} = 
\]
\[
 \hat P(\mu) = \frac{\mu}{F\hat F+\hat F F} 
 \{ \mu F (-P(\mu) \hat F) 
 \frac{\hat F F}{F\hat F+\hat F F} ) \} = 0 \; , 
\]
since $\hat F^2=0$.  
Then Lemma \ref{lem:ifPhasazero} proves the result.  
\end{proof}

\begin{lemma}
If $P(\mu)_p \perp \hat F_p$ for one value of $p$, then this 
holds also for any other value of $p$.  
\end{lemma}

\begin{proof}
We suppose $P(\mu)_p \perp \hat F_p$ holds at one particular $p$, 
and then show that $P(\mu)_q \perp \hat F_q$ holds for any 
adjacent $q$.  The relation 
\[
   \hat F_p P(\mu)_p = - P(\mu)_p \hat F_p
\] implies that 
\[
 -P(\mu)_p (T_p^\mu)^{-1} T_q^\mu \hat F_q (T_q^\mu)^{-1} 
  T_p^\mu = (T_p^\mu)^{-1} T_q^\mu \hat F_q (T_q^\mu)^{-1} 
  T_p^\mu P(\mu)_p \; , 
\] and so 
\[
 -P(\mu)_p (I+\mu \tau_{pq}) \hat F_q = 
 (I+\mu \tau_{pq}) \hat F_q (I+\mu \tau_{pq})^{-1} 
 P(\mu)_p (I+\mu \tau_{pq}) \; . 
\]  Therefore 
$-P(\mu)_q \hat F_q = \hat F_q P(\mu)_q$, and $\hat F_q \perp 
P(\mu)_q$.  
\end{proof}

\subsection{Complementary surfaces}\label{compsurfcompsurf} 
As promised in Section 
\ref{therehi}, we say more about complementary 
surfaces here.  

If $P$ is a polynomial conserved quantity of order $n$, then 
$||P(\mu)||^2$ has at most $2n$ zeros $\mu_1,...,\mu_{2n}$.  
We can choose $\hat F = P(\mu_j)$ to get another surface, 
for some $j \in \{ 1,...,2n \}$, because 
$P(\mu_j)$ lies in the light cone 
$L^4$.  (Clearly, choices of $\mu$ for which 
$P(\mu)$ is not in the light cone cannot be allowed.)
Then $T^{\mu_j} \hat F (T^{\mu_j})^{-1} = 
T^{\mu_j} P(\mu_j) (T^{\mu_j})^{-1}$ is 
constant, by the definition of a conserved quantity, and 
thus $\hat F$ gives a Darboux transform.  Furthermore, 
\[
 \langle \hat F,P(\mu_j) \rangle = ||P(\mu_j)||^2 = 0 \; , 
\] and so in fact we have a Baecklund transform.  

\begin{defn}
We call the Baecklund transform given by $\hat F = P(\mu_j)$ 
a {\em complementary surface} of $\ef$.  
\end{defn}

\subsection{The spaces in which Darboux transformations live} 
In this section, we include some comments about the ambient 
spaces that Darboux and Baecklund transformations lie in.  
The comments here are less than perfectly organized, and more 
thorough arguments can be found in \cite{BHRS}.  

Let $\ef$ be an isothermic discrete surface with 
normalized linear conserved quantity 
$P=Q + \lambda Z$, $||Z||^2=1$, 
and let $\hat \ef$ be a Darboux transform of $\ef$.  
To shorten notation, 
define $\mathcal{B}=F \hat F + \hat F F$.  Then, as in 
the proof of Theorem \ref{lem:typeatmostnplus1}, the 
conserved quantity for $\hat \ef$ is 
\[ \hat P = \mu \mathcal{B}^{-1} (\mu \mathcal{B}-\lambda \hat F F) P 
\mu^{-1} \mathcal{B}^{-1} (\mu \mathcal{B}-\lambda F \hat F) = 
\lambda^2 \hat Z + \lambda \hat P_1 + \hat Q \; , \]  
where 
\[ \hat Q = \mu^2 Q \; , \;\;\; \hat P_1 = 
\mu^2 Z - \mu \mathcal{B}^{-1} (\hat F F Q+Q F \hat F) \] and 
\[ \hat Z = \mathcal{B}^{-2} \hat F F Q F \hat F- \mu \mathcal{B}^{-1} 
(\hat F F Z+ Z F \hat F) \; . \]  
Using $F^2=\hat F^2=0$, and that $F Z F = 0$ implies $F Z \hat F F Q F = 
\mathcal{B} F Z Q F$ and $F Q F \hat F Z F = \mathcal{B} F Q Z F$ and 
$F Z \hat F F = \mathcal{B} F Z$ and $F \hat F Z F = \mathcal{B} Z F$, 
we have $\hat Z^2 = \mu^2 Z^2$.  Thus, when we normalize $\hat P$ so that 
the leading coefficient has squared norm $+1$, the constant term will 
become $\mu \cdot Q$, so $\ef$ and $\hat \ef$ do not live in 
the same space form, in general, but at least the 
sectional curvatures of the two space 
forms (i.e. the two quadrics) containing $\ef$ and $\hat \ef$ 
have the same sign.  However, it 
does not really matter that they are not in the same space 
form, as Darboux transforms are a notion most naturally considered 
for ambient spaces with just a conformal structure, not 
with a Riemannian structure (and the two quadrics do have a common 
conformal structure).  

In the case that the Darboux transform is actually a Baecklund 
transform with linear conserved quantity $\tilde Q + 
\lambda \tilde Z$, let us suppose that 
\begin{equation}\label{star-ruby} 
\hat P = \lambda s (\lambda \tilde Z + \tilde Q) + t 
(\lambda \tilde Z + \tilde Q) 
\end{equation} for some constants $s$ and $t$.  
Hence $\hat Z = s \tilde Z$, and so 
\[ \tilde Z^2 = (\mu/s)^2 Z^2 \; . \]  
Also, $\tilde Q = t^{-1} \hat Q = t^{-1} \mu^2 Q$.  Then $\hat P_1 = 
s \tilde Q + t \tilde Z$ implies 
\[ \mu^2 Z - \mu \mathcal{B}^{-1} (\hat F F Q+Q F \hat F) = 
s t^{-1} \mu^2 Q + t s^{-1} (\mathcal{B}^{-2} \hat F F Q F \hat F - 
\mu \mathcal{B}^{-1} (\hat F F Z + Z F \hat F)) \; . \]  
Multiplying this on both the left and the right by $F$, we get 
\[ - \mu \mathcal{B}^{-1} (F \hat F F Q F+F Q F \hat F F) = 
s t^{-1} \mu^2 F Q F + 
t s^{-1} (\mathcal{B}^{-2} F \hat F F Q F \hat F F) \; . \]  
So 
\[ (2 \mu + s t^{-1} \mu^2 + t s^{-1}) (F Q F) = 0 \; . \]  
Since $F Q= \mathcal{G} I - Q F$ for some nonzero real scalar 
$\mathcal{G}$, we have 
$F Q F = \mathcal{G} F \neq 0$, so 
\[ (\sqrt{|s t^{-1} \mu^2|} \pm \sqrt{|t s^{-1}|})^2=0 \; . \]
So in fact $\sqrt{|s t^{-1} \mu^2|} - \sqrt{|t s^{-1}|} = 0$, and it 
follows that $s^2 \mu^2 = t^2$.  Thus $\tilde Z^2 = t^2 s^{-4} Z^2$ and 
$\tilde Q = t s^{-2} Q$, so when we normalize $\lambda \tilde Z+\tilde Q$, 
$\tilde Q$ is changed back to the original $Q$.  The conclusion is that 
a Baecklund transform lies in the same space form as the original surface, 
when using normalized conserved quantities, and when assuming 
\eqref{star-ruby}.  In fact, this conclusion is true even without assuming 
\eqref{star-ruby}, see Theorem 4.5 in \cite{BHRS}.  

We gave the arguments here assuming $\ef$ has a linear conserved quantity, 
but the corresponding results and arguments hold in the case 
that $\ef$ has a polynomial conserved quantity as well.  

\subsection{Envelopes} 

\begin{defn}\label{defninenvelopesection}
Let $\ef$ be a discrete isothermic surface with domain $\Sigma 
\subset \mathbb{Z}^2$ and lift $F:\Sigma \to L^4$.  
We say that 
$\ef$ envelops the discrete sphere congruence 
$Z: \Sigma \to \mathbb{R}^{4,1}$, $Z_p$ spacelike for all 
$p \in \Sigma$, if 
\begin{enumerate}
\item $\ef_p \perp Z_p$ for all $p \in \Sigma$ (incidence), 
\item $Z_p \equiv Z_q$ mod $\text{span}\{F_p,F_q\}$ for all edges 
$pq$ with $p,q \in \Sigma$ (touching).  
\end{enumerate}
\end{defn}

\begin{remark}
The top-term coefficient $Z$ of a polynomial conserved quantity of $\ef$ 
is an example of a sphere congruence of $\ef$, by parts (3) and (5) 
of Lemma \ref{lem:sixfacts}.  
\end{remark}

\begin{remark}
Although $Z_p$ itself is a single sphere, it determines 
a pencil of spheres $Z_p+sF_p$ for $s \in \mathbb{R}$.  
\end{remark}

Suppose that $\ef$ is a discrete isothermic surface with lift $F$ that 
envelopes a sphere congruence $Z$, and let $\hat f$ with lift $\hat F$ 
be a Darboux transform of $\ef$.  Let $\hat Z_p = Z_p + s_p F_p$ be 
spheres in the pencils 
produced by $Z$ so that $\hat Z_p$ has incidence 
with $\hat f_p$, i.e. $\langle \hat F_p, \hat Z_p \rangle = 0$.  Let $c_p$ 
be the circle containing both $\ef_p$ and $\hat \ef_p$ that is 
perpendicular to $\hat Z_p$.  

Now, $\ef$ (resp. $\hat \ef$) envelops $Z$ (resp. $\hat Z$) if 
and only if there 
exists a circle $c_{pq}$ (resp. $\hat c_{pq}$) tangent to both $c_p$ and $c_q$ 
for all edges $pq$.  
(This follows from the second enumerated item in Definition 
\ref{defninenvelopesection}, 
which implies there is a sphere common to both the pencil 
produced by $Z_p$ (resp. $\hat Z_p$) and the pencil 
produced by $Z_q$ (resp. $\hat Z_q$).)  
In particular, $c_{pq}$ exists, because $\ef$ envelops 
$Z$.  We then have: 

\begin{lemma}
$\hat \ef$ envelops $\hat Z$.  
\end{lemma}

\begin{proof}
Consider the circle through the four points $\ef_p$, $\ef_q$, 
$\hat \ef_q$ and 
$\hat \ef_p$, the circular arc of $c_p$ from $\ef_p$ and $\hat \ef_p$, 
the circular arc of $c_q$ from $\ef_q$ and $\hat \ef_q$, 
and the circular arc of 
$c_{pq}$ from $\ef_p$ and $\ef_q$.  Geometric 
considerations show that all four circles lie 
in one sphere (in fact, by applying a 
M\"obius transformation, we could assume they all lie in a 
Euclidean $2$-plane), 
and so there exists an 
arc of a circle $\hat c_{pq}$ from 
$\hat \ef_p$ to $\hat \ef_q$ tangent to both $c_p$ and $c_q$.  So $\hat \ef$ 
envelops $\hat Z$.  
\end{proof}


\section{Discrete minimal surfaces in 
$\mathbb{R}^3$ and discrete CMC $1$ surfaces 
in $\mathbb{H}^3$}

We have already given definitions of discrete minimal surfaces in 
$\mathbb{R}^3$ and discrete CMC $1$ surfaces in 
$\mathbb{H}^3$ in Chapter \ref{chapondiscreteCMCsurfs}.  
However, in this chapter we describe the ways these particular surfaces 
were first defined in the literature, without using conserved quantities.  
These ways are more directly related to the Weierstrass and Bryant 
representations for smooth minimal surfaces in $\mathbb{R}^3$ and 
smooth CMC $1$ surfaces in $\mathbb{H}^3$.  
They also provide us with a clear reason to describe 
discrete holomorphic functions, which 
are essential to those first definitions.  

\subsection{Discrete holomorphic functions} 
Let $g$ be a map from the lattice $\mathbb{Z}^2$ (or a 
subdomain of $\mathbb{Z}^2$) 
to $\mathbb{C}$.  Then $g$ is a discrete holomorphic function if the 
cross ratios of $g$ satisfy 
\[ (g_q-g_p) (g_r-g_q)^{-1} (g_s-g_r) (g_p-g_s)^{-1} = 
\frac{a_{pq}}{a_{ps}} \; , \]  
with $a_{pq} = a_{rs} \in \mathbb{R}$ and 
$a_{ps} = a_{qr} \in \mathbb{R}$, for all 
quadrilaterals, with vertices $p=(m,n), q=(m+1,n), r=(m+1,n+1), 
s=(m,n+1) \in \mathbb{Z}^2$ in the domain of $g$.  

Throughout this chapter, $a_{pq}$ will denote the cross ratio factorizing 
function of $g$.  

\begin{remark}
Note that this definition of discrete holomorphic functions is the 
same as the definition of those discrete isothermic surfaces 
that lie in a plane.  
\end{remark}

\begin{remark}
The above definition of discrete holomorphic functions is in the "broad" 
sense.  The definition in the "narrow" sense would be that 
$\tfrac{a_{pq}}{a_{ps}}$ is identically $-1$.  
\end{remark}

\begin{remark}
If one takes a discrete derivative or discrete integral of a discrete 
holomorphic function, one will not get another 
discrete holomorphic function, in general.  
\end{remark}

Letting $(m,n)$ denote an arbitrary point in the domain of $g$, 
examples of discrete holomorphic functions $g$ are 
\begin{enumerate}
\item $g_{m,n} = c (m+i n)$ for $m,n \in 
     \mathbb{Z}$ and $c$ a complex constant, 
\item $g_{m,n} = e^{c (m+i n)}$ for $m,n \in 
     \mathbb{Z}$ and $c$ a 
real or pure imaginary constant, 
\item M\"obius and Darboux transformations of any of the above examples,
\item discrete versions of $z^\gamma$ and $\log z$, as in Example 
\ref{exa:10point3} below.  
\end{enumerate}

\begin{example}\label{exa:10point3}
For $\alpha \in (0,2) \in \mathbb{R}$, 
the following discrete holomorphic function is a discrete 
version of $g=z^\alpha$, 
in the narrow sense.  It is defined by the recursion 
\[ \alpha \cdot g_{m,n} = 2 m 
\frac{(g_{m+1,n}-g_{m,n})(g_{m,n}-g_{m-1,n})}{g_{m+1,n}-g_{m-1,n}}
+ 2 n 
\frac{(g_{m,n+1}-g_{m,n})(g_{m,n}-g_{m,n-1})}{g_{m,n+1}-g_{m,n-1}} \; . \]
We start with 
\[ g_{0,0}=0 \; , \;\;\; g_{1,0} = 1 \; , \;\;\; g_{0,1} = i^\alpha \; . \]  
We can use this recursion 
to propagate along the positive axes $\{ g_{m,0} \}$ 
and $\{ g_{0,n} \}$ with $m,n > 0$.  
We can then compute general $g_{m,n}$, $m,n>0$, by using that the 
cross ratio is always $-1$.  It turns out that the $g_{m,n}$ then 
automatically satisfy the above recursion relation for all $m$ and $n$.  
Also, Agafonov 
\cite{Ag} showed that these power functions are embedded in a wedge, and 
there is also a discrete version of $\log z$.  Furthermore, Agafonov 
showed that these $g$ are Schramm circle packings \cite{Schramm}.  
\end{example}

\begin{remark}
Here is an example of a function that is holomorphic in the 
sense here, but is not holomorphic in Schramm's sense \cite{Schramm}: 
\[ g_{0,0} = 0 \; , \;\; g_{1,0} = 1 \; , \;\; g_{2,0} = 2+r \; , \]\[ 
g_{0,1} = i \; , \;\; g_{1,1} = 1+i \; , \;\; g_{2,1} = 2+r+i \; , \]\[ 
g_{0,2} = 2i \; , \;\; g_{1,2} = 1+2i \; , \;\; g_{2,2} = 2+r+2i \; , \]  
for any fixed $r>0$.  In fact, 
Schramm's circle patterns are a special case of the 
definition for discrete holomorphic functions that we use here.  
(If one includes both centers of circles and intersection points of 
circles, then Schramm's circle packings give discrete holomorphic 
functions.)  The definition here (unlike 
Schramm's definition) is loose, in the sense that  
it allows the following flexibility: 
Take a discrete holomorphic function $g_{m,n}$ in the sense here with 
cross ratios identically $-1$.  Fix $g_{m,n}$ where $m+n$ is odd, as 
defined by this function.  Then change the value of 
$g_{0,0}$ freely, and then one 
can find new values for all $g_{m,n}$ where $m+n$ is even (and $(m,n) 
\neq (0,0)$) so that the cross ratios are all still $-1$.  
\end{remark}

\subsection{Smooth minimal surfaces in $\mathbb{R}^3$} 
We can always take a smooth CMC surface to have isothermic coordinates 
$z=u+iv$, $u,v \in \mathbb{R}$ (away from umbilic points), 
and then the Hopf differential becomes $r dz^2$ for some real constant 
$r$.  Rescaling the coordinate $z$ by a constant real factor, 
we may assume $r=1$.  
So now assume we have an isothermic minimal surface with 
Hopf differential function $Q=1$.  
Then \[ \frac{Q dz^2}{dg} = 
\frac{dz}{g^\prime} \; , \] where $g$ is the stereographic 
projection of the Gauss map to the complex plane, and $g^\prime = 
dg/dz$.  The map $g$ taking $z$ in the domain of the immersion 
(of the surface) to $\mathbb{C}$ is holomorphic.  
Because we are avoiding umbilics, we have $g^\prime \neq 0$.  
We are only concerned with local behavior of the surface, 
so we ignore the possiblity that $g$ has poles or other singularities.  
Then the Weierstrass representation is (with $\sqrt{-1}$ regarded as lying 
in the complex plane, unlike the quaternion $i$) 
\[ x = \mbox{Re} \int_{z_0}^z (2 g , 1-g^2,\sqrt{-1}+ 
\sqrt{-1} g^2) \frac{dz}{g^\prime} \; . \]  
Associating $(1,0,0)$, $(0,1,0)$ and $(0,0,1)$ with 
the quaternions $i$, $j$ and $k$, respectively, we have the partial derivatives 
as in the following lemma: 

\begin{lemma}
\[ x_u = (i-gj) j \frac{1}{g_u} (i-g j) \; , \]  
\[ x_v = (i-gj) j \frac{-1}{g_v} (i-g j) \; . \]  
\end{lemma}

\begin{proof}
This proof uses the holomorphicity of $g$, and uses identification 
of the imaginary complex number $\sqrt{-1}$ with the imaginary 
quaternion $i$.  

Because $g$ is holomorphic, we have $\sqrt{-1} g_u = g_v$ and 
$g^\prime(=g_z)=g_u=-\sqrt{-1} g_v$.  

Then, 
\[ x=\frac{1}{2} \left( 
\int (2g,1-g^2,\sqrt{-1}(1+g^2)) \frac{dz}{g^\prime} + 
\int (2\bar g,1-\bar g^2,
       -\sqrt{-1}(1+\bar g^2)) \frac{d\bar z}{\overline{g^\prime}} 
\right) \; , \]  
so 
\[ x_u = \frac{1}{2} \left( \frac{2 g}{g^\prime} + \frac{2 \bar g}{\overline{g^\prime}} , 
\frac{1-g^2}{g^\prime} + \frac{1-\bar g^2}{\overline{g^\prime}} , 
\frac{\sqrt{-1} (1+g^2)}{g^\prime} - \frac{\sqrt{-1}(1+\bar g^2)}{\overline{g^\prime}} 
\right) = \]
\[ \left( \frac{g}{g^\prime} + \frac{\bar g}{\overline{g^\prime}} \right) i + 
\frac{1}{2} \left( \frac{1-g^2}{g^\prime} + \frac{1-\bar g^2}{\overline{g^\prime}} \right) j + 
\frac{1}{2} \left( \frac{1+g^2}{g^\prime} - \frac{1+\bar g^2}{\overline{g^\prime}} \right) i \cdot k = \]
\[ \left( \frac{g}{g^\prime} + \frac{\bar g}{\overline{g^\prime}} \right) i + 
\frac{1}{2} \left( \frac{1-g^2}{g^\prime} + \frac{1-\bar g^2}{\overline{g^\prime}} - 
\frac{1+g^2}{g^\prime} + \frac{1+\bar g^2}{\overline{g^\prime}} \right) j = \]
\[ \left( \frac{g}{g^\prime} + \frac{\bar g}{\overline{g^\prime}} \right) i + 
\left( \frac{1}{\overline{g^\prime}} - \frac{g^2}{g^\prime}  \right) j = 
\frac{g}{g^\prime} i - i \frac{\bar g}{\overline{g^\prime}} k^2 + 
\frac{1}{\overline{g^\prime}} j - \frac{g^2}{g^\prime} j = \]
\[ \left( \frac{-i}{\overline{g^\prime}} k + \frac{g}{g^\prime} i \right) (1+g k) = 
(k+g) \frac{i}{g^\prime} (1+i g j) = \]
\[ (i-g j) j \frac{1}{g^\prime} (i-g j) = (i - g j) j \frac{1}{g_u} (i - g j) \; . \] 
Similarly, we have 
\[ x_v = \frac{i}{2} \left( \frac{2 g}{g^\prime} - \frac{2 \bar g}{\overline{g^\prime}} , 
\frac{1-g^2}{g^\prime} - \frac{1-\bar g^2}{\overline{g^\prime}} , 
\frac{\sqrt{-1} (1+g^2)}{g^\prime} + \frac{\sqrt{-1}(1+\bar g^2)}{\overline{g^\prime}} 
\right) = \]
\[ - \left( \frac{g}{g^\prime} - \frac{\bar g}{\overline{g^\prime}} \right) + 
\frac{1}{2} \left( \frac{1-g^2}{g^\prime} - \frac{1-\bar g^2}{\overline{g^\prime}} \right) k - 
\frac{1}{2} \left( \frac{1+g^2}{g^\prime} + \frac{1+\bar g^2}{\overline{g^\prime}} \right) k = \]
\[ \left( - \frac{g}{g^\prime} + \frac{\bar g}{\overline{g^\prime}} \right) + 
\frac{1}{2} \left( \frac{1-g^2}{g^\prime} - \frac{1-\bar g^2}{\overline{g^\prime}} - 
\frac{1+g^2}{g^\prime} - \frac{1+\bar g^2}{\overline{g^\prime}} \right) k = \]
\[ \left( - \frac{g}{g^\prime} + \frac{\bar g}{\overline{g^\prime}} \right) + 
\left( - \frac{1}{\overline{g^\prime}} - \frac{g^2}{g^\prime}  \right) k = 
\left( \frac{g}{g^\prime} j - \frac{\bar g}{\overline{g^\prime}} j - 
\frac{1}{\overline{g^\prime}} i - \frac{g^2}{g^\prime} i \right) j = \]
\[ \left(\frac{1}{\overline{g^\prime}} - i \frac{g}{g^\prime} j \right) (-i-\bar g j) j = 
(i-g j) \frac{-i}{\overline{g^\prime}} j (i - g j) = \]
\[ (i-g j) j \frac{i}{g^\prime} (i-g j) = 
(i - g j) j \frac{-1}{g_v} (i - g j) \; . \] 
\end{proof}

\subsection{Discrete minimal surfaces in $\mathbb{R}^3$} 
The smooth case above suggests that the definition for discrete minimal 
surfaces should be 
\[ \ef_q - \ef_p = (i-g_p j) j \frac{a_{pq}}{g_q-g_p} (i-g_q j) \; , \] 
where the map 
$g$ from a domain in $\mathbb{Z}^2$ to 
$\mathbb{C}$ is a discrete holomorphic function.  
Here $p=(m,n)$ and $q$ is either $(m+1,n)$ or $(m,n+1)$.  
As in the smooth case, we avoid "umbilics", so 
\[ g_q-g_p \neq 0 \; . \]  
Taking this as the definition (see \cite{BobPink2} and \cite{Udo1}), 
we have the following two examples: 

\begin{example}
The discrete holomorphic function $c (m+in)$ for $c$ a complex constant 
will produce a minimal surface called a discrete Enneper surface, 
and graphics for this surface can be seen in \cite{BobPink2}.  
\end{example}

\begin{example}
The discrete holomorphic function $e^{c_1 m+i c_2 n}$ for 
choices of real constants $c_1$ and $c_2$ so that the cross ratio is 
identically $-1$ will produce a minimal surface called a discrete 
catenoid, and graphics for this surface also can be seen in \cite{BobPink2}.  
See also Figure 4 in this text.  
\end{example}

\subsection{Smooth CMC $1$ surfaces in 
$\mathbb{H}^3$}\label{lastsubsectionA} 
We can now similarly describe smooth and discrete CMC $1$ surfaces in 
$\mathbb{H}^3$.  Construction of smooth isothermic CMC $1$ surfaces 
starts with the Bryant equation ($g$ is an arbitrary holomorphic 
function such that $g^\prime \neq 0$) 
\[ dF = F \begin{pmatrix}  g & -g^2 \\ 1 & -g \end{pmatrix} 
\frac{dz}{g^\prime} \] 
with solution $F \in \text{SL}_2\mathbb{C}$, 
and the surface is then \[ F \cdot \bar F^t \in \mathbb{H}^3 \; . \]  
Here, hyperbolic $3$-space is 
\[ \mathbb{H}^3 = \{ (x_0,x_1,x_2,x_3) \in \mathbb{R}^{3,1} \, | 
\, x_0>0 , x_0^2-x_1^2-x_2^2-x_3^2 = 1 \}
= \]\[ \left\{ \begin{pmatrix} x_0+x_3 & x_1+ix_2 \\ x_1-ix_2 & 
x_0-x_3 \end{pmatrix} \right\} 
= \{ X \cdot \bar X^t \, | \, X \in \text{SL}_2\mathbb{C} \} \; . \]  

\begin{example}
Take any constant $q \in \mathbb{C} \setminus \{ 0 \}$.  
Then $g = q z$ 
gives 
\[ F = q^{-1/2} \begin{pmatrix}
\cosh (z) & q \sinh (z) - q z \cosh(z) \\ 
\sinh (z) & q \cosh (z) - q z \sinh(z) & 
\end{pmatrix} \; , \]
and $F \bar F^t$ gives a CMC $1$ Enneper cousin in 
$\mathbb{H}^3$.  
\end{example}

\begin{example}
To make CMC $1$ surfaces of revolution, called catenoid 
cousins, one can use $g=e^{\mu z}$ for $\mu$ either real or 
purely imaginary.  
\end{example}

\subsection{Discrete CMC $1$ surfaces in 
$\mathbb{H}^3$}\label{lastsubsectionB} 
Following \cite{Udo1}, the discrete version of the Bryant equation becomes 
\begin{equation}\label{starrats0} 
F_q-F_p = F_p \begin{pmatrix} g_p & -g_pg_q \\ 1 & -g_q \end{pmatrix} 
\frac{\lambda a_{pq}}{g_q-g_p} \; , 
\;\;\; \det F \in \mathbb{R} \; , \end{equation} and $g$ is again a 
discrete holomorphic function with $q_q-q_p \neq 0$.  Now the 
formula in \cite{Udo1} for the surface is different: it is obtained 
using the $\mathbb{R}^{4,1}$ lightcone model by setting 
\[ \begin{pmatrix} a \\ b \end{pmatrix} = \begin{pmatrix} 0 & 1 \\ 
j & 0 \end{pmatrix} 
F_p \begin{pmatrix} i \\ j \end{pmatrix} \] 
and then taking the vertices of the surface as 
\[ \ef_p = r_p 
\begin{pmatrix} -b \bar a & a \bar a \\ b \bar b & - a \bar b \end{pmatrix} 
\in \mathbb{H}^3 \subset L^4 \subset \mathbb{R}^{4,1} \; , \;\;\; 
r_p \in \mathbb{R} \setminus \{ 0 \} \; . \] 
Here $\mathbb{H}^3$ lies in the $4$-dimensional light cone 
$L^4$ in the following way, like in Chapter \ref{chaponsmoothCMCsurf}: 
\[ \mathbb{H}^3 = \left\{ X \in L^4 \, \left| \, 
      X \cdot \begin{pmatrix} -i & 0 \\ 0 & i \end{pmatrix} + 
      \begin{pmatrix} -i & 0 \\ 0 & i \end{pmatrix} \cdot X = 
      2 I \right. \right\} \; . \] 
Note that because the entries of $F$ are complex, 
not quaternionic, it follows that 
$b \bar a$ is purely imaginary quaternionic, so $\ef_p$ 
really does lie in $\mathbb{R}^{4,1}$, and thus in $L^4$.  
The scalar $r_p$ is chosen so that $\ef_p \in \mathbb{H}^3$.  

One can check that $\ef_p$ will be of the form 
\[ r \begin{pmatrix} -(\bar A C+\bar B D)j+i(AD-BC) & C\bar C+D \bar D \\
A \bar A + B \bar B & j (A \bar C + B \bar D)-i(AD-BC)
\end{pmatrix} \; , 
\] where 
\[ F = \begin{pmatrix} A & B \\ C & D 
\end{pmatrix} \; . 
\]  For $\ef_p$ to lie in $\mathbb{H}^3$, we should take 
\[ r = \frac{1}{AD-BC} \; . \]  This means that the coefficient of the 
$i$ term in the diagonal entries will be simply $\pm 1$.  So we can view 
the surface as lieing in the $4$-dimensional space 
$\mathbb{R}^{3,1}$, by simply 
dropping the $x_1$ part off of (this sum of matrices was seen at the 
beginning of Chapter \ref{sect:lcq-smooth}) 
\[ x_1 \begin{pmatrix} i & 0 \\ 0 & -i \end{pmatrix} + 
x_2 \begin{pmatrix} j & 0 \\ 0 & -j \end{pmatrix} + 
x_3 \begin{pmatrix} k & 0 \\ 0 & -k \end{pmatrix} + 
x_4 \begin{pmatrix} 0 & 1 \\ -1 & 0 \end{pmatrix} + 
x_0 \begin{pmatrix} 0 & 1 \\ 1 & 0 \end{pmatrix} \; . \] 

Now, the projection into the Poincare ball model is 
\[ (x_2,x_3,x_4,x_0) \to \frac{(x_2,x_3,x_4)}{1+x_0} = \]  
\begin{equation}\label{starrats1}  
\frac{(\text{Re}(-\bar A C-\bar B D),\text{Im}(-\bar A C-\bar B D),
\tfrac{1}{2}(-A\bar A-B\bar B+C\bar C+D\bar D))}{AD-BC + 
\tfrac{1}{2} (A\bar A+B\bar B+C\bar C+D\bar D)} \; . \end{equation} 

On the other hand, if we simply look at 
\[ \frac{1}{AD-BC} F \bar F^t = \frac{1}{AD-BC} 
\begin{pmatrix}
A \bar A+B \bar B & A \bar C+B \bar D \\ C \bar A+D \bar B & C \bar C+D \bar D
\end{pmatrix} = \]\[ 
\begin{pmatrix}
y_0+y_3 & y_1+\sqrt{-1} y_2 \\ y_1-\sqrt{-1} y_2 & y_0-y_3
\end{pmatrix} \; ,\] and then project to the Poincare ball, we have 
\[ \frac{(y_1,y_2,y_3)}{1+y_0} = \]\begin{equation}\label{starrats2} 
\frac{(\text{Re}(A \bar C+B \bar D),\text{Im}(A \bar C+B \bar D),
\tfrac{1}{2}(A\bar A+B\bar B-C\bar C-D\bar D))}{AD-BC + 
\tfrac{1}{2} (A\bar A+B\bar B+C\bar C+D\bar D)} \; . \end{equation} 
Note that \eqref{starrats1} and \eqref{starrats2} are essentially the same, 
up to a rigid motion of $\mathbb{H}^3$.  Thus we have proven: 

\begin{theorem} (\cite{HRSY}) 
The discrete CMC 1 surface $\ef$ in $\mathbb{H}^3$ given by $F$ solving 
\eqref{starrats0} is 
\[ \ef = \frac{1}{\det F} F \bar F^t \; , \] 
up to a rigid motion of $\mathbb{H}^3$.  
\end{theorem}

We can now make specific examples by choosing discrete holomorphic 
functions $g_{m,n}$.  For example, we can construct discrete versions of 
the smooth CMC $1$ Enneper cousins and catenoid cousins by using 
discrete versions of $g = q z$ and $g = e^{\mu z}$, $\mu \in 
(\mathbb{R} \cup \sqrt{-1} \mathbb{R}) \setminus \{ 0 \}$, respectively.  

\begin{remark}
Recently, there has been research on a notion of discrete surfaces 
called $s$-isothermic surfaces, and we comment briefly on this here.  
One can "bend" Schramm's circle packings to get surfaces, by 
changing half of the circles (in a checkered pattern) 
into spheres.  
This leads to the notion of discrete $s$-isothermic minimal surfaces.  


We can define discrete $s$-CMC surfaces in this way: 
an $s$-isothermic surface is {\em $s$-CMC} 
if it has a Christoffel transform that is also a Darboux 
transform.  (See Definition \ref{oldsensedefn}.)  

For more on $s$-isothermic surfaces, see 
\cite{Bob-Hoff-Spring} and \cite{bs-book}.  
\end{remark}


\end{document}